\documentclass{article}
\usepackage{graphicx}
\usepackage{amsmath}
\usepackage{amssymb}
\usepackage{amsthm}
\usepackage{stmaryrd}
\SetSymbolFont{stmry}{bold}{U}{stmry}{m}{n}
\usepackage{cite}
\usepackage[arrow,curve,matrix,arc,2cell]{xy}
\UseAllTwocells
\usepackage{pgf,tikz}
\usetikzlibrary{arrows}
\usepackage{tikz-cd}
\usepackage[utf8]{inputenc}
\usepackage[unicode]{hyperref}
\usepackage{mathrsfs}
\def\e#1\e{\begin{equation}#1\end{equation}}
\def\ea#1\ea{\begin{align}#1\end{align}}
\def\eq#1{{\rm(\ref{#1})}}
\theoremstyle{plain}
\newtheorem{thm}{Theorem}[section]
\newtheorem{lem}[thm]{Lemma}
\newtheorem{prop}[thm]{Proposition}
\newtheorem{cor}[thm]{Corollary}
\theoremstyle{definition}
\newtheorem{dfn}[thm]{Definition}
\newtheorem{ex}[thm]{Example}
\newtheorem{rem}[thm]{Remark}
\numberwithin{figure}{section}
\numberwithin{equation}{section}
\def\dim{\mathop{\rm dim}\nolimits}
\def\vdim{\mathop{\rm vdim}\nolimits}
\def\vir{{\rm vir}}

\def\inc{\mathop{\rm inc}\nolimits}

\def\Ker{\mathop{\rm Ker}}
\def\Coker{\mathop{\rm Coker}}
\def\Im{\mathop{\rm Im}}
\def\In{\mathop{\rm In}\nolimits}
\def\Ex{\mathop{\rm Ex}\nolimits}
\def\depth{\mathop{\rm depth}\nolimits}
\def\rank{\mathop{\rm rank}}

\def\Hom{\mathop{\rm Hom}\nolimits}

\def\qcoh{\mathop{\rm qcoh}}

\def\PreSh{\mathop{\rm PreSh}}
\def\Sh{\mathop{\rm Sh}}
\def\id{\mathop{\rm id}\nolimits}
\def\Id{{\mathop{\rm Id}\nolimits}}

\def\max{{\mathop{\rm max}\nolimits}}
\def\exp{\mathop{\rm exp}\nolimits}
\def\tf{{\rm tf}}
\def\sa{{\rm sa}}

\def\gp{{\rm gp}}

\def\rsm{{\rm sm}}
\def\rin{{\rm in}}
\def\rex{{\rm ex}}
\def\ain{{\text{\rm a-in}}}
\def\aex{{\text{\rm a-ex}}}
\def\Spec{\mathop{\rm Spec}\nolimits}
\def\MSpec{\mathop{\rm MSpec}\nolimits}
\def\Specc{\mathop{\rm Spec^c}\nolimits}
\def\Speccin{\mathop{\rm Spec^c_{in}}\nolimits}
\def\Top{{\mathop{\bf Top}}}
\def\CCRings{{\mathop{\bf CC^{\bs\iy}Rings}}}
\def\CRings{{\mathop{\bf C^{\bs\iy}Rings}}}
\def\CRingsco{{\mathop{\bf C^{\bs\iy}Rings_{co}}}}
\def\CRingslo{{\mathop{\bf C^{\bs\iy}Rings_{lo}}}}
\def\CRingscinlo{{\mathop{\bf C^{\bs\iy}Rings^c_{in,lo}}}}

\def\CRingsc{{\mathop{\bf C^{\bs\iy}Rings^c}}}
\def\PCRingsac{{\mathop{\bf PC^{\bs\iy}Rings^{ac}}}}
\def\CRingsac{{\mathop{\bf C^{\bs\iy}Rings^{ac}}}}

\def\CRingsclo{{\mathop{\bf C^{\bs\iy}Rings^c_{lo}}}}
\def\CRingscin{{\mathop{\bf C^{\bs\iy}Rings^c_{in}}}}
\def\CRingscfiin{{\mathop{\bf C^{\bs\iy}Rings^c_{fi,in}}}}
\def\CRingscfg{{\mathop{\bf C^{\bs\iy}Rings^c_{fg}}}}
\def\CRingscfp{{\mathop{\bf C^{\bs\iy}Rings^c_{fp}}}}
\def\PCRingsc{{\mathop{\bf PC^{\bs\iy}Rings^c}}}
\def\CPCRingscin{{\mathop{\bf CPC^{\bs\iy}Rings^c_{in}}}}
\def\CPCRingsc{{\mathop{\bf CPC^{\bs\iy}Rings^c}}}
\def\PCRingscin{{\mathop{\bf PC^{\bs\iy}Rings^c_{in}}}}
\def\CRingscfi{{\mathop{\bf C^{\bs\iy}Rings^c_{fi}}}}
\def\CRingscsc{{\mathop{\bf C^{\bs\iy}Rings^c_{sc}}}}
\def\CRingscscin{{\mathop{\bf C^{\bs\iy}Rings^c_{sc,in}}}}
\def\CRingscZ{{\mathop{\bf C^{\bs\iy}Rings^c_{\Z}}}}
\def\CRingscDe{{\mathop{\bf C^{\bs\iy}Rings^c_{\De}}}}
\def\CRingscsa{{\mathop{\bf C^{\bs\iy}Rings^c_{sa}}}}
\def\CRingscto{{\mathop{\bf C^{\bs\iy}Rings^c_{to}}}}
\def\CRingsctf{{\mathop{\bf C^{\bs\iy}Rings^c_{tf}}}}

\def\CRS{{\mathop{\bf C^{\bs\iy}RS}}}
\def\LCRS{{\mathop{\bf LC^{\bs\iy}RS}}}
\def\CRSc{{\mathop{\bf C^{\bs\iy}RS^c}}}
\def\CRScin{{\mathop{\bf C^{\bs\iy}RS^c_{in}}}}
\def\LCRSc{{\mathop{\bf LC^{\bs\iy}RS^c}}}
\def\LCRScin{{\mathop{\bf LC^{\bs\iy}RS^c_{in}}}}
\def\LCRScsa{{\mathop{\bf LC^{\bs\iy}RS^c_{sa}}}}
\def\LCRSctf{{\mathop{\bf LC^{\bs\iy}RS^c_{tf}}}}
\def\LCRScZ{{\mathop{\bf LC^{\bs\iy}RS^c_{\Z}}}}
\def\LCRScinex{{\mathop{\bf LC^{\bs\iy}RS^c_{in,ex}}}}
\def\LCRScsaex{{\mathop{\bf LC^{\bs\iy}RS^c_{sa,ex}}}}
\def\LCRSctfex{{\mathop{\bf LC^{\bs\iy}RS^c_{tf,ex}}}}
\def\LCRScZex{{\mathop{\bf LC^{\bs\iy}RS^c_{\Z,ex}}}}
\def\ACSch{{\mathop{\bf AC^{\bs\iy}Sch}}}
\def\ACSchc{{\mathop{\bf AC^{\bs\iy}Sch^c}}}
\def\ACSchcin{{\mathop{\bf AC^{\bs\iy}Sch^c_{in}}}}
\def\CSch{{\mathop{\bf C^{\bs\iy}Sch}}}
\def\CSchc{{\mathop{\bf C^{\bs\iy}Sch^c}}}
\def\CSchac{{\mathop{\bf C^{\bs\iy}Sch^{ac}}}}

\def\CSchcfp{{\mathop{\bf C^{\bs\iy}Sch_{fp}^c}}}
\def\CSchcfg{{\mathop{\bf C^{\bs\iy}Sch_{fg}^c}}}
\def\CSchcfi{{\mathop{\bf C^{\bs\iy}Sch_{fi}^c}}}
\def\CSchcin{{\mathop{\bf C^{\bs\iy}Sch^c_{in}}}}

\def\CSchcfiin{{\mathop{\bf C^{\bs\iy}Sch^c_{fi,in}}}}
\def\CSchcZ{{\mathop{\bf C^{\bs\iy}Sch^c_{\Z}}}}
\def\CSchcDe{{\mathop{\bf C^{\bs\iy}Sch^c_{\De}}}}
\def\CSchcsa{{\mathop{\bf C^{\bs\iy}Sch^c_{sa}}}}
\def\CSchcto{{\mathop{\bf C^{\bs\iy}Sch^c_{to}}}}
\def\CSchctf{{\mathop{\bf C^{\bs\iy}Sch^c_{tf}}}}
\def\CSchcinex{{\mathop{\bf C^{\bs\iy}Sch^c_{in,ex}}}}
\def\CSchcZex{{\mathop{\bf C^{\bs\iy}Sch^c_{\Z,ex}}}}
\def\CSchcDeex{{\mathop{\bf C^{\bs\iy}Sch^c_{\De,ex}}}}
\def\CSchcsaex{{\mathop{\bf C^{\bs\iy}Sch^c_{sa,ex}}}}
\def\CSchctoex{{\mathop{\bf C^{\bs\iy}Sch^c_{to,ex}}}}
\def\CSchctfex{{\mathop{\bf C^{\bs\iy}Sch^c_{tf,ex}}}}
\def\CSchcfiinex{{\mathop{\bf C^{\bs\iy}Sch^c_{fi,in,ex}}}}
\def\CSchcfiZex{{\mathop{\bf C^{\bs\iy}Sch^c_{fi,\Z,ex}}}}
\def\CSchcfitfex{{\mathop{\bf C^{\bs\iy}Sch^c_{fi,tf,ex}}}}
\def\CSchcfiZ{{\mathop{\bf C^{\bs\iy}Sch^c_{fi,\Z}}}}
\def\CSchcfitf{{\mathop{\bf C^{\bs\iy}Sch^c_{fi,tf}}}}
\def\CSta{{\mathop{\bf C^{\bs\iy}Sta}}}
\def\CStac{{\mathop{\bf C^{\bs\iy}Sta^c}}}
\def\DMCSta{{\mathop{\bf DMC^{\bs\iy}Sta}}}
\def\DMCStac{{\mathop{\bf DMC^{\bs\iy}Sta^c}}}
\def\DMCStacin{{\mathop{\bf DMC^{\bs\iy}Sta^c_{in}}}}
\def\DMCStacto{{\mathop{\bf DMC^{\bs\iy}Sta^c_{to}}}}
\def\DMCStacfi{{\mathop{\bf DMC^{\bs\iy}Sta^c_{fi}}}}
\def\DMCStacfiin{{\mathop{\bf DMC^{\bs\iy}Sta^c_{fi,in}}}}
\def\Sta{{\mathop{\bf Sta}\nolimits}}
\def\Man{{\mathop{\bf Man}}}

\def\Manc{{\mathop{\bf Man^c}}}
\def\Manac{{\mathop{\bf Man^{ac}}}}
\def\Manacin{{\mathop{\bf Man^{ac}_{in}}}}
\def\Mancac{{\mathop{\bf Man^{c,ac}}}}

\def\Mangc{{\mathop{\bf Man^{gc}}}}
\def\Mancin{{\mathop{\bf Man^c_{in}}}}
\def\Mancwe{{\mathop{\bf Man^c_{we}}}}
\def\Mancst{{\mathop{\bf Man^c_{st}}}}
\def\Mangcin{{\mathop{\bf Man^{gc}_{in}}}}
\def\cManc{{\mathop{\bf\check{M}an^c}}}
\def\cMancin{{\mathop{\bf\check{M}an^c_{in}}}}
\def\cMangc{{\mathop{\bf\check{M}an^{gc}}}}
\def\cMangcin{{\mathop{\bf\check{M}an^{gc}_{in}}}}
\def\Mon{{\mathop{\bf Mon}}}

\def\Orb{{\mathop{\bf Orb}}}
\def\Orbc{{\mathop{\bf Orb^c}}}
\def\Euc{{\mathop{\bf Euc}}}
\def\Eucc{{\mathop{\bf Euc^c}}}
\def\Euccin{{\mathop{\bf Euc^c_{in}}}}
\def\Eucac{{\mathop{\bf Euc^{ac}}}}
\def\Eucacin{{\mathop{\bf Euc^{ac}_{in}}}}
\def\Sets{{\mathop{\bf Sets}}}

\def\bs{\boldsymbol}
\def\ge{\geqslant}
\def\le{\leqslant\nobreak}
\def\O{{\mathcal O}}

\def\mk{\mathfrak{m}}
\def\bO{\boldsymbol{\mathcal O}}

\def\K{{\mathbin{\mathbb K}}}
\def\N{{\mathbin{\mathbb N}}}
\def\Q{{\mathbin{\mathbb Q}}}
\def\R{{\mathbin{\mathbb R}}}
\def\Z{{\mathbin{\mathbb Z}}}

\def\fC{{\mathbin{\mathfrak C}\kern.05em}}
\def\fD{{\mathbin{\mathfrak D}}}
\def\fE{{\mathbin{\mathfrak E}}}
\def\fF{{\mathbin{\mathfrak F}}}
\def\fG{{\mathbin{\mathfrak G}}}
\def\fH{{\mathbin{\mathfrak H}}}

\def\bfC{{\mathbin{\boldsymbol{\mathfrak C}}\kern.05em}}
\def\bfD{{\mathbin{\boldsymbol{\mathfrak D}}}}
\def\bfE{{\mathbin{\boldsymbol{\mathfrak E}}}}
\def\bfF{{\mathbin{\boldsymbol{\mathfrak F}}}}
\def\bfG{{\mathbin{\boldsymbol{\mathfrak G}}}}
\def\bfH{{\mathbin{\boldsymbol{\mathfrak H}}}}
\def\bfI{{\mathbin{\boldsymbol{\mathfrak I}}}}
\def\cA{{\mathbin{\cal A}}}

\def\cC{{\mathbin{\cal C}}}
\def\cD{{\mathbin{\cal D}}}
\def\cE{{\mathbin{\cal E}}}
\def\cF{{\mathbin{\cal F}}}

\def\cJ{{\mathbin{\cal J}}}

\def\cM{{\mathbin{\cal M}}}
\def\cP{{\mathbin{\cal P}}}
\def\cS{{\mathbin{\cal S}}}
\def\cT{{\cal T}}

\def\oM{{\mathbin{\smash{\,\,\overline{\!\!\mathcal M\!}\,}}}}
\def\fCmod{{\mathbin{{\mathfrak C}\text{\rm -mod}}}}
\def\fCmodco{{\mathbin{{\mathfrak C}\text{\rm -mod}_{\rm co}}}}
\def\fDmod{{\mathbin{{\mathfrak D}\text{\rm -mod}}}}

\def\bfCmod{{\mathbin{\bs{\mathfrak C}\text{\rm -mod}}}}
\def\bfDmod{{\mathbin{\bs{\mathfrak D}\text{\rm -mod}}}}
\def\OWmod{{\mathbin{\O_W\text{\rm -mod}}}}
\def\OXmod{{\mathbin{\O_X\text{\rm -mod}}}}

\def\bOWmod{{\mathbin{\bO_W\text{\rm -mod}}}}
\def\bOXmod{{\mathbin{\bO_X\text{\rm -mod}}}}
\def\bOYmod{{\mathbin{\bO_Y\text{\rm -mod}}}}
\def\m{{\mathfrak m}}

\def\ub{{\underline{b\kern -0.1em}\kern 0.1em}{}}

\def\ue{{\underline{e}}{}}
\def\uf{{\underline{f\!}\,}{}}

\def\ug{{\underline{g\!}\,}{}}
\def\uh{{\underline{h\!}\,}{}}
\def\ui{{\underline{i\kern -0.07em}\kern 0.07em}{}}
\def\uj{{\underline{j\kern -0.1em}\kern 0.1em}{}}
\def\uk{{\underline{k\kern -0.1em}\kern 0.1em}{}}
\def\um{{\underline{m\kern -0.1em}\kern 0.1em}{}}
\def\un{{\underline{n\kern -0.1em}\kern 0.1em}{}}
\def\ulp{{\underline{p\kern -0.15em}\kern 0.15em}{}}
\def\uq{{\underline{q\kern -0.15em}\kern 0.15em}{}}
\def\ur{{\underline{r\kern -0.15em}\kern 0.15em}{}}
\def\us{{\underline{s\kern -0.15em}\kern 0.15em}{}}
\def\ut{{\underline{t\kern -0.1em}\kern 0.1em}{}}
\def\uu{{\underline{u\kern -0.1em}\kern 0.1em}{}}

\def\uU{{\underline{U\kern -0.25em}\kern 0.2em}{}}
\def\uV{{\underline{V\kern -0.25em}\kern 0.2em}{}}
\def\uW{{\underline{W\!\!}\,\,}{}}
\def\uX{{\underline{X\!}\,}{}}
\def\uY{{\underline{Y\!\!}\,\,}{}}
\def\uZ{{\underline{Z\!}\,}{}}
\def\bS{{\bs S}}
\def\bT{{\bs T}}
\def\bU{{\bs U}}
\def\bV{{\bs V}}
\def\bW{{\bs W}}
\def\bX{{\bs X}}
\def\bY{{\bs Y}}
\def\bZ{{\bs Z}}
\def\al{\alpha}
\def\be{\beta}
\def\ga{\gamma}
\def\de{\delta}
\def\io{\iota}
\def\ep{\epsilon}
\def\la{\lambda}
\def\ka{\kappa}
\def\th{\theta}
\def\ze{\zeta}

\def\vp{\varphi}
\def\si{\sigma}
\def\om{\omega}
\def\De{\Delta}

\def\Om{\Omega}
\def\Ga{\Gamma}

\def\pd{\partial}
\def\ts{\textstyle}
\def\st{\scriptstyle}
\def\sst{\scriptscriptstyle}

\def\sm{\setminus}
\def\sh{\sharp}
\def\op{\oplus}

\def\op{\oplus}
\def\ot{\otimes}
\def\ov{\overline}

\def\bigot{\bigotimes}
\def\iy{\infty}
\def\es{\emptyset}
\def\lb{\llbracket}
\def\rb{\rrbracket}
\def\ra{\rightarrow}
\def\rra{\rightrightarrows}
\def\Ra{\Rightarrow}
\def\ab{\allowbreak}
\def\longra{\longrightarrow}

\def\hookra{\hookrightarrow}

\def\t{\times}

\def\bu{\bullet}
\def\ci{\circ}
\def\ti{\tilde}
\def\d{{\rm d}}
\def\ha{{\ts\frac{1}{2}}}
\def\an#1{\langle #1 \rangle}
\def\ban#1{\bigl\langle #1 \bigr\rangle}
\def\md#1{\vert #1 \vert}
\def\bmd#1{\big\vert #1 \big\vert}
\def\Gac{\Gamma^{\mathrm{c}}}
\def\Gacin{\Gamma^{\mathrm{c}}_\rin}
\def\simc{\mathbin{\sim}}
\DeclareMathOperator{\Prc}{Pr}
\begin{document}
\title{$C^\iy$-algebraic geometry with corners}
\author{Kelli Francis-Staite and Dominic Joyce}
\date{}
\maketitle

\begin{abstract}
If $X$ is a manifold then the set $C^\iy(X)$ of smooth functions $f:X\ra\R$ is a $C^\iy$-{\it ring}, a rich algebraic structure with many operations. $C^\iy$-{\it schemes\/} are schemes over $C^\iy$-rings, a way of using Algebro-Geometric techniques in Differential Geometry. They include smooth manifolds, but also many singular and infinite-dimensional spaces. They have applications to Synthetic Differential Geometry, and to derived manifolds.

In this book, a sequel to \cite{Joyc9}, we define and study new categories of $C^\iy$-{\it rings with corners\/} and $C^\iy$-{\it schemes with corners}, which generalize manifolds with corners in the same way that $C^\iy$-rings and $C^\iy$-schemes generalize manifolds. These will be used in future work \cite{Joyc4} as the foundations of theories of derived manifolds and derived orbifolds with corners.
\end{abstract}

\setcounter{tocdepth}{2}
\tableofcontents

\section{Introduction}
\label{cc1}

If $X$ is a manifold then the $\R$-algebra $C^\iy(X)$ of smooth functions $c:X\ra\R$ is a $C^\iy$-{\it ring}. That is, for each smooth function $f:\R^n\ra\R$ there is an $n$-fold operation $\Phi_f:C^\iy(X)^n\ra C^\iy(X)$ acting by $\Phi_f:c_1,\ldots,c_n\mapsto f(c_1,\ldots,c_n)$, and these operations $\Phi_f$ satisfy many natural identities. Thus, $C^\iy(X)$ actually has a far richer structure than the obvious $\R$-algebra structure.

$C^\iy$-{\it algebraic geometry\/} \cite{Dubu2,Joyc2,Joyc9,Kock,MoRe2} is a version of algebraic geometry in which rings are replaced by $C^\iy$-rings. Our primary reference is the second author's monograph \cite{Joyc9}. Its basic objects are $C^\iy$-{\it schemes}, which form a category $\CSch$. Any smooth manifold $X$ determines a $C^\iy$-scheme $\uX=\Spec C^\iy(X)$, giving an embedding $\Man\subset\CSch$ as a full subcategory. But $C^\iy$-schemes are far more general than manifolds, and $\CSch$ contains many singular or infinite-dimensional objects. $C^\iy$-algebraic geometry has been applied in the foundations of Synthetic Differential Geometry \cite{Dubu2,Kock,MQR,MoRe1,MoRe2}, and Derived Differential Geometry \cite[\S 4.5]{Luri}, \cite{Joyc3,Joyc4,Spiv}.

Manifolds with corners $\Manc$ are a generalization of manifolds locally modelled on $[0,\iy)^k\t\R^{n-k}$ rather than on $\R^n$. The goal of this book, which is a sequel to \cite{Joyc9}, is to extend $C^\iy$-algebraic geometry to manifolds with boundary and corners (and more generally `things with boundary and corners'). We introduce notions of $C^\iy$-{\it rings with corners}, and $C^\iy$-{\it schemes with corners}, which form a category $\CSchc$ with a full embedding~$\Manc\subset\CSchc$. 

One of our motivations for introducing $C^\iy$-schemes with corners is to provide the foundations for a theory of {\it derived manifolds and derived orbifolds with corners}, which will be done by the second author \cite{Joyc4}, by regarding them as special examples of derived $C^\iy$-schemes with corners. These will have applications in Symplectic Geometry, and elsewhere, as the natural geometric structures on some important classes of moduli spaces. Our results could also be applied to develop a theory of Synthetic Differential Geometry with corners.

In the rest of this introduction we summarize the contents of the chapters.

\subsubsection*{Chapter \ref{cc2}. $C^\iy$-rings and $C^\iy$-schemes}

We begin with background on $C^\iy$-rings and $C^\iy$-schemes, mostly following \cite{Joyc9}. A {\it categorical\/ $C^\iy$-ring\/} is a product-preserving functor $F:\Euc\ra\Sets$, where $\Euc\subset\Man$ is the full subcategory of the category of smooth manifolds $\Man$ with objects the Euclidean spaces $\R^n$ for $n\ge 0$. These form a category $\CCRings$, with morphisms natural transformations, which is equivalent to the category $\CRings$ of $C^\iy$-rings by mapping~$F\mapsto\fC=F(\R)$. 

Here a $C^\iy$-{\it ring\/} is data $\bigl(\fC,(\Phi_f)_{f:\R^n\ra\R\,\, C^\iy}\bigr)$, where $\fC$ is a set and $\Phi_f:\fC^n\ra\fC$ a map for all smooth maps $f:\R^n\ra\R$, satisfying some axioms. Usually we write this as $\fC$, leaving the $C^\iy$-operations $\Phi_f$ implicit. The motivating examples for $X$ a smooth manifold are $\fC=C^\iy(X)$, the set of smooth functions $c:X\ra\R$, with operations $\Phi_f(c_1,\ldots,c_n)(x)=f(c_1(x),\ldots,c_n(x))$. We discuss $\fC$-{\it modules}, which are just modules over $\fC$ considered as an $\R$-algebra, and the {\it cotangent module\/} $\Om_\fC$, generalizing $\Ga^\iy(T^*X)$, the vector space of smooth sections of $T^*X$, as a module over~$C^\iy(X)$.

We then define the subcategory $\CRS$ of $C^\iy$-{\it ringed spaces\/} $\uX=(X,\O_X)$ where $X$ is a topological space and $\O_X$ a sheaf of $C^\iy$-rings on $X$, and the full subcategory $\LCRS\subset\CRS$ of {\it local\/ $C^\iy$-ringed spaces\/} for which the stalks $\O_{X,x}$ are local $C^\iy$-rings. The global sections functor $\Ga:\LCRS\ra\CRings^{\bf op}$ has a right adjoint $\Spec:\CRings^{\bf op}\ra\LCRS$, the {\it spectrum functor}. An object $\uX\in\LCRS$ is called a $C^\iy$-{\it scheme\/} if we may cover $X$ with open $U\subseteq X$ with $(U,\O_X\vert_U)\cong\Spec\fC$ for $\fC\in\CRings$. These form a full subcategory $\CSch\subset\LCRS$.

There is a full embedding $\Man\hookra\CSch$ mapping $X\mapsto \uX=(X,\O_X)$, where $\O_X$ is the sheaf of local smooth functions $c:X\ra\R$.

In classical algebraic geometry $\Ga\ci\Spec\cong\Id:\mathop{\bf Rings}^{\bf op}\ra\mathop{\bf Rings}^{\bf op}$, so $\mathop{\bf Rings}^{\bf op}$ is equivalent to the category $\mathop{\bf ASch}$ of affine schemes. In $C^\iy$-algebraic geometry it is not true that $\Ga\ci\Spec\cong\Id$, but we do have $\Spec\ci\Ga\ci\Spec\cong\Spec$. Define a $C^\iy$-ring $\fC$ to be {\it complete\/} if $\fC\cong\Ga\ci\Spec\fD$ for some $\fD$. Write $\CRingsco\subset\CRings$ for the full subcategory of complete $C^\iy$-rings. Then $(\CRingsco)^{\bf op}$ is equivalent to the category $\ACSch$ of affine $C^\iy$-schemes.

We define $\O_X$-{\it modules\/} on a $C^\iy$-scheme $\uX$, and the {\it cotangent sheaf\/} $T^*\uX$, which generalizes cotangent bundles of manifolds.

\subsubsection*{Chapter \ref{cc3}. Manifolds with (g-)corners}

We discuss {\it manifolds with corners}, locally modelled on $[0,\iy)^k\t\R^{n-k}$. The definition of manifold with corners $X,Y$ is an obvious generalization of the definition of manifold. However, the definition of smooth map $f:X\ra Y$ is {\it not\/} obvious: what conditions should we impose on $f$ near the boundary and corners of $X,Y$? There are several non-equivalent definitions of morphisms of manifolds with corners in the literature. The one we choose is due to Melrose \cite[\S 1.12]{Melr2}, \cite[\S 1]{KoMe}, who calls them {\it b-maps}. This gives a category~$\Manc$. 

A smooth map $f:X\ra Y$ is called {\it interior\/} if $f(X^\ci)\subset Y^\ci$, where $X^\ci$ is the {\it interior\/} of $X$ (that is, if $X$ is locally modelled on $[0,\iy)^k\t\R^{n-k}$, then $X^\ci$ is locally modelled on $(0,\iy)^k\t\R^{n-k}$). Write $\Mancin\subset\Manc$ for the subcategory with only interior morphisms.

A manifold with corners $X$ has a {\it boundary\/} $\pd X$, a manifold with corners of dimension $\dim X-1$, with a (non-injective) morphism $\Pi_X:\pd X\ra X$. For example, if $X=[0,\iy)^2$, then $\pd X$ is the disjoint union $(\{0\}\t[0,\iy))\amalg([0,\iy)\t\{0\})$ with the obvious map $\Pi_X$, so two points in $\pd X$ lie over $(0,0)\in X$. There is a natural, free action of the symmetric group $S_k$ on the $k^{\rm th}$ boundary $\pd^kX$, permuting the boundary strata. The $k$-{\it corners\/} is $C_k(X)=\pd^kX/S_k$, of dimension $\dim X-k$. The {\it corners\/} of $X$ is $C(X)=\coprod_{k=0}^{\dim X}C_k(X)$, an object in the category $\cManc$ of {\it manifolds with corners of mixed dimension}.

We call $X$ a {\it manifold with faces\/} if $\Pi_X\vert_F:F\ra X$ is injective for each connected component $F$ of $\pd X$ (called a {\it face\/}).

Now boundaries are not functorial: if $f:X\ra Y$ is smooth, there is generally no natural morphism $\pd f:\pd X\ra\pd Y$. However, there is a natural, interior morphism $C(f):C(X)\ra C(Y)$, giving the {\it corner functor\/} $C:\Manc\ra\cMancin$ or $C:\cManc\ra\cMancin$. We explain that $C$ is right adjoint to the inclusion $\inc:\cMancin\hookra\cManc$. Thus, from a categorical point of view, boundaries and corners in $\Manc$ are determined uniquely by the inclusion $\Mancin\hookra\Manc$, that is, by the comparison between interior and smooth maps.

A manifold with corners $X$ has two notions of (co)tangent bundle: the {\it tangent bundle\/} $TX$ and its dual the {\it cotangent bundle\/} $T^*X$, and the {\it b-tangent bundle\/} ${}^bTX$ and its dual the {\it b-cotangent bundle\/} ${}^bT^*X$. Here $TX$ is the obvious notion of tangent bundle. There is a morphism $I_X:{}^bTX\ra TX$ which identifies $\Ga^\iy({}^bTX)$ with the subspace of vector fields $v\in \Ga^\iy(TX)$ which are tangent to every boundary stratum of $X$. Tangent bundles are functorial over smooth maps $f:X\ra Y$ (that is, $f$ lifts to $Tf:TX\ra TY$ linear on the fibres). B-tangent bundles are functorial over interior maps~$f:X\ra Y$.

We also discuss the categories $\Mangcin\subset\Mangc$ of {\it manifolds with generalized corners}, or {\it manifolds with g-corners}, introduced by the second author in \cite{Joyc6}. These allow more complicated local models than $[0,\iy)^k\t\R^{n-k}$. The local models $X_P$ are parametrized by {\it weakly toric monoids\/} $P$, with $X_P=[0,\iy)^k\t\R^{n-k}$ when $P=\N^k\t\Z^{n-k}$. So we provide an introduction to the theory of monoids. (All monoids in this book are commutative.)

The simplest manifold with g-corners which is not a manifold with corners is $X=\bigl\{(w,x,y,z)\in[0,\iy)^4:wx=yz\bigr\}$. Most of the theory above extends to manifolds with g-corners, except that tangent bundles $TX$ are not well behaved, and we no longer have~$C_k(X)\cong\pd^kX/S_k$. 

We call morphisms $g:X\ra Z$, $h:Y\ra Z$ in $\Mancin$ or $\Mangcin$ {\it b-transverse\/} if ${}^bT_xg\op{}^bT_yh:{}^bT_xX\op{}^bT_yY\ra{}^bT_zZ$ is surjective for all $x\in X$, $y\in Y$ with $g(x)=h(y)=z$. Manifolds with g-corners have the nice property that all b-transverse fibre products exist in $\Mangcin$, whereas fibre products only exist in $\Mancin,\Manc$ under rather restrictive conditions. We can think of $\Mangcin$ as a kind of closure of $\Mancin$ under b-transverse fibre products.

\subsubsection*{Chapter \ref{cc4}. (Pre) $C^\iy$-rings with corners}

We introduce $C^\iy$-{\it rings with corners}. To decide on the correct definition, the obvious starting point is categorical $C^\iy$-rings, i.e.\ product-preserving functors $F:\Euc\ra\Sets$. We define a {\it categorical pre $C^\iy$-ring with corners\/} to be a product-preserving functor $F:\Eucc\ra\Sets$, where $\Eucc\subset\Manc$ is the full subcategory with objects $[0,\iy)^m\t\R^n$ for $m,n\ge 0$. These form a category $\CPCRingsc$, with morphisms natural transformations.

Then we define the category $\PCRingsc$ of {\it pre $C^\iy$-rings with corners}, with objects $\bigl(\fC,\fC_\rex,(\Phi_f)_{f:\R^m\t[0,\iy)^n\ra\R\,\, C^\iy},(\Psi_g)_{g:\R^m\t[0,\iy)^n\ra[0,\iy)\,\, C^\iy}\bigr)$ for $\fC,\ab\fC_\rex$ sets and $\Phi_f:\fC^m\t\fC_\rex^n\ra\fC$, $\Psi_g:\fC^m\t\fC_\rex^n\ra\fC_\rex$ maps, satisfying some axioms. Usually we write this as $\bfC=(\fC,\fC_\rex)$, leaving the $C^\iy$-operations $\Phi_f,\Psi_g$ implicit. There is an equivalence $\CPCRingsc\ra\PCRingsc$ mapping $F\mapsto(\fC,\fC_\rex)=(F(\R),F([0,\iy)))$. The motivating example, for $X$ a manifold with (g-)corners, is $(\fC,\fC_\rex)=(C^\iy(X),\Ex(X))$, where $C^\iy(X)$ is the set of smooth maps $c:X\ra\R$, and $\Ex(X)$ the set of {\it exterior\/} (i.e. smooth, but not necessarily interior) maps~$c':X\ra[0,\iy)$.

We define the category of $C^\iy$-{\it rings with corners\/} $\CRingsc$ to be the full subcategory of $(\fC,\fC_\rex)$ in $\PCRingsc$ satisfying an extra condition, that if $c'\in\fC_\rex$ is invertible in $\fC_\rex$ then $c'=\Psi_{\exp}(c)$ for some $c\in\fC$. In terms of spaces $X$, this says that if $c':X\ra[0,\iy)$ is smooth and invertible (hence positive) then $c'=\exp c$ for some smooth~$c=\log c':X\ra\R$.

Now we could also have started with $\Mancin$. So we write $\Euccin\subset\Mancin$ for the full subcategory with objects  $[0,\iy)^m\t\R^n$, and define the category $\CPCRingscin$ of {\it categorical interior pre $C^\iy$-rings with corners\/} to be product-preserving functors $F:\Euccin\ra\Sets$, with morphisms natural transformations. We define a subcategory $\PCRingscin\subset\PCRingsc$ of {\it interior pre $C^\iy$-rings with corners}, and an equivalence $\CPCRingscin\ra\PCRingscin$ mapping $F\mapsto(\fC,\fC_\rex)=(F(\R),F([0,\iy))\amalg\{0\})$. The category $\CRingscin$ of {\it interior $C^\iy$-rings with corners\/} is the intersection $\PCRingscin\cap\CRingsc$ in $\PCRingsc$. The motivating example, for $X$ a manifold with (g-)corners, is $(\fC,\fC_\rex)=(C^\iy(X),\In(X)\amalg\{0\})$, for $\In(X)$ the interior maps~$c':X\ra[0,\iy)$.

Although a $C^\iy$-ring with corners $\bfC=(\fC,\fC_\rex)$ has a huge number of $C^\iy$-operations $\Phi_f,\Psi_g$, it is helpful much of the time to focus on a small subset of these, giving $\fC,\fC_\rex$ smaller, more manageable structures. In particular, we often think of $\fC$ as an $\R$-{\it algebra}, with addition and multiplication given by $\Phi_{f_+},\Phi_{f_\cdot}:\fC\t\fC\ra\fC$ for $f_+,f_\cdot:\R^2\ra\R$ given by $f_+(x,y)=x+y$, $f_\cdot(x,y)=xy$. And we think of $\fC_\rex$ as a {\it monoid}, with multiplication given by $\Psi_{g_\cdot}:\fC_\rex\t\fC_\rex\ra\fC_\rex$ for $g_\cdot:[0,\iy)^2\ra[0,\iy)$ given by $g_\cdot(x,y)=xy$. Note that monoids also control the corner structure of manifolds with g-corners, in the same way.

We define various important subcategories of $C^\iy$-rings with corners by imposing conditions on $\fC_\rex$ as a monoid. For example, a $C^\iy$-ring with corners $\bfC=(\fC,\fC_\rex)$ is {\it interior\/} if $\fC_\rex=\fC_\rin\amalg\{0\}$ where $\fC_\rin$ is a submonoid of $\fC_\rex$, that is, $\fC_\rex$ has no zero divisors. In terms of spaces $X$, we interpret $\fC_\rin$ as the monoid $\In(X)$ of interior maps $c':X\ra[0,\iy)$. If $\bfC,\bfD$ are interior, a morphism $\bs\phi=(\phi,\phi_\rex):\bfC\ra\bfD$ is {\it interior\/} if $\phi_\rex:\fC_\rex\ra\fD_\rex$ maps $\fC_\rin\ra\fD_\rin$. 

The subcategory $\CSchcfi\subset\CSchc$ of {\it firm\/} $C^\iy$-rings with corners are those whose sharpening $\fC_\rex^\sh$ is finitely generated.

\subsubsection*{Chapter \ref{cc5}. $C^\iy$-schemes with corners}

We define the subcategory $\CRSc$ of $C^\iy$-{\it ringed spaces with corners\/} $\bX=(X,\bO_X)$ where $X$ is a topological space and $\bO_X=(\O_X,\O_X^\rex)$ a sheaf of $C^\iy$-rings with corners on $X$, and the full subcategory $\LCRSc\subset\CRSc$ of {\it local\/ $C^\iy$-ringed spaces with corners\/} for which the stalks $\bO_{X,x}$ for $x\in X$ are local $C^\iy$-rings with corners. The global sections functor $\Gac:\LCRSc\ra(\CRingsc)^{\bf op}$ has a right adjoint $\Specc:(\CRingsc)^{\bf op}\ra\LCRSc$, the {\it spectrum functor}. An object $\bX$ in $\LCRSc$ is called a $C^\iy$-{\it scheme with corners\/} if we may cover $X$ with open $U\subseteq X$ with $(U,\bO_X\vert_U)\cong\Specc\bfC$ for $\bfC$ in $\CRingsc$. These form a full subcategory~$\CSchc\subset\LCRSc$.

One might expect that to define corresponding subcategories $\CRScin\subset\CRSc$, $\LCRScin\subset\LCRSc$ of {\it interior\/} (local) $C^\iy$-ringed spaces with corners, we should just replace $\CRingsc$ by $\CRingscin\subset\CRingsc$. However, as $\inc:\CRingscin\hookra\CRingsc$ does not preserve limits, a sheaf of interior $C^\iy$-rings with corners is only a presheaf of $C^\iy$-rings with corners. So we define $\bX=(X,\bO_X)\in\CRSc$ to be {\it interior\/} if $\bO_X$ is the sheafification, as a sheaf valued in $\CRingsc$, of a sheaf valued in $\CRingscin$. Then the stalks $\bO_{X,x}$ of $\bO_X$ lie in $\CRingscin$, so $\O_{X,x}^\rex=\O_{X,x}^\rin\amalg\{0\}$.

We define $\Gacin:\LCRScin\ra(\CRingscin)^{\bf op}$ by $\Gacin(\bX)=(\Ga(\O_X),\ab\Ga(\O_X^\rin)\ab\amalg\{0\})$, where $\Ga(\O_X^\rin)\subset\Ga(\O_X^\rex)$ is the subset of sections $s$ with $s\vert_x\in \O_{X,x}^\rin\subset\O_{X,x}^\rex$ for each $x\in X$. Then $\Gacin$ has a right adjoint $\Speccin:(\CRingscin)^{\bf op}\ra\LCRScin$, where $\Speccin=\Specc\vert_{(\CRingscin)^{\bf op}}$.
An object $\bX\in\LCRScin$ is called an {\it interior\/ $C^\iy$-scheme with corners\/} if we may cover $X$ with open $U\subseteq X$ with $(U,\bO_X\vert_U)\cong\Speccin\bfC$ for $\bfC\in\CRingscin$. These form a full subcategory $\CSchcin\subset\LCRScin$, with~$\CSchcin\subset\CSchc$.

We define full embeddings $\Manc,\Mangc\hookra\CSchc$ or $\Mancin,\Mangcin\hookra\CSchcin$ mapping $X\mapsto(X,(\O_X,\O_X^\rex))$ or $X\mapsto(X,(\O_X,\O_X^\rin\amalg\{0\}))$, where $\O_X,\O_X^\rex,\O_X^\rin$ are the sheaves of local smooth functions $X\ra\R$, and exterior functions $X\ra[0,\iy)$, and interior functions $X\ra[0,\iy)$, respectively. Thus, manifolds with (g-)corners may be regarded as special examples of $C^\iy$-schemes with corners. Manifolds with (g-)faces map to {\it affine\/} $C^\iy$-schemes with corners.

For $C^\iy$-schemes, $\Ga\ci\Spec\not\cong\Id$ but $\Spec\ci\Ga\ci\Spec\cong\Spec$. Thus we can define a full subcategory $\CRingsco\subset\CRings$ of {\it complete\/} $C^\iy$-rings, with $(\CRingsco)^{\bf op}$ equivalent to the category $\ACSch$ of affine $C^\iy$-schemes. 

For $C^\iy$-schemes with corners, the situation is worse: we have both $\Gac\ci\Specc\not\cong\Id$ and $\Specc\ci\Gac\ci\Specc\not\cong\Specc$. To see why, note that if $X$ is a manifold with corners then smooth functions $c:X\ra\R$ are essentially local objects --- they can be glued using partitions of unity. 

However, smooth functions $c':X\ra[0,\iy)$ have some {\it strictly global\/} behaviour: there is a locally constant function $\mu_{c'}:\pd X\ra\N\amalg\{\iy\}$ giving the order of vanishing of $c'$ along the boundary $\pd X$. So the behaviour of $c'$ near distant points $x,y$ in $X$ is linked, if $x,y$ lie in the image of the same connected component of $\pd X$. This means that smooth functions $c':X\ra[0,\iy)$ cannot always be glued using partitions of unity, and localizing a $C^\iy$-ring with corners $\bfC=(\fC,\fC_\rex)$ at an $\R$-point $x:\fC\ra\R$, as one does to define $\Specc$, does not see only local behaviour around~$x$.

Since $\Specc\ci\Gac\ci\Specc\not\cong\Specc$, we cannot define a subcategory of `complete' $C^\iy$-rings with corners equivalent to affine $C^\iy$-schemes with corners. As a partial substitute, we define {\it semicomplete\/} $C^\iy$-rings with corners $\bfC=(\fC,\fC_\rex)$, such that $\Gac\ci\Specc$ is an isomorphism on $\fC$ and injective on~$\fC_\rex$.

If $\bX,\bY\in\CSchc$, a morphism $\bs f:\bX\ra\bY$ in $\CSchc$ is a morphism in $\LCRSc$. Although locally we can write $\bX\cong\Specc\bfC$, $\bX\cong\Specc\bfD$, because of the lack of a good notion of completeness, we do {\it not\/} know that locally we can write $\bs f=\Specc\bs\phi$ for some $\bs\phi:\bfD\ra\bfC$ in $\CRingsc$. One problem this causes is that $\bs g:\bX\ra\bZ$, $\bs h:\bY\ra\bZ$ are morphisms in $\CSchc$, we do not know that the fibre product $\bX\t_{\bs g,\bZ,\bs h}\bY$ in $\LCRSc$ (which always exists) lies in $\CSchc$, if $\bs g,\bs h$ are not locally of the form $\Specc\bs\phi$. So we do not know that all fibre products exist in~$\CSchc$.

To get round this, we introduce the full subcategory $\CSchcfi\subset\CSchc$ of {\it firm\/} $C^\iy$-schemes with corners $\bX$, which are locally of the form $\Specc\bfC$ for $\bfC$ a firm $C^\iy$-ring with corners. Morphisms $\bs f:\bX\ra\bY$ in $\CSchcfi$ are always locally of the form $\Specc\bs\phi$, so we can prove that $\CSchcfi$ is closed under fibre products in $\LCRSc$, and thus all fibre products exist in~$\CSchcfi$.

In general, $\CSchc$ contains a huge variety of objects, many of which are very singular and pathological, and do not fit with our intuitions about manifolds with corners. So it can be helpful to restrict to smaller subcategories of better behaved objects in $\CSchc$, such as $\CSchcfi$. For example, for $\bX$ to be firm means that locally $\bX$ has only finitely many boundary strata, which seems likely to hold in almost every interesting application.

We define several subcategories of $\CSchcin,\CSchc$, and prove results such as existence of reflection functors between them, and existence of fibre products and finite limits in them. Two particularly interesting and well behaved examples are the full subcategories $\CSchcto\subset\CSchcin$, $\CSchctoex\subset\CSchc$ of {\it toric\/} $C^\iy$-schemes with corners, whose corner structure is controlled by toric monoids in the same way that manifolds with g-corners are.

\subsubsection*{Chapter \ref{cc6}. Boundaries, corners, and the corner functor}

One of the most important properties of manifolds with corners $X$ is the existence of boundaries $\pd X$, and clearly we want to generalize this to $C^\iy$-schemes with corners. Our starting point, as in Chapter \ref{cc3}, is that the boundary $\pd X=C_1(X)$ is part of the corners $C(X)=\coprod_{k=0}^{\dim X}C_k(X)$, and the corner functor $C:\cManc\ra\cMancin$ is right adjoint to the inclusion~$\inc:\cMancin\hookra\cManc$.

We construct a right adjoint {\it corner functor\/} $C:\LCRSc\ra\LCRScin$ to the inclusion $\inc:\LCRScin\hookra\LCRSc$. We prove that the restriction to $\CSchc$ maps to $\CSchcin$, giving $C:\CSchc\ra\CSchcin$ right adjoint to $\inc:\CSchcin\hookra\CSchc$. For $\bX$ in $\LCRSc$, points in $C(X)$ are pairs $(x,P)$ for $x\in X$ and $P\subset\O_{X,x}^\rex$ a prime ideal in the monoid $\O_{X,x}^\rex$ of the stalk $\bO_{X,x}$ of $\bO_X$ at $x$. This is an analogue, for $X\in\Manc$, of a point in $C(X)$ being $(x,\ga)$ for $x\in X$ and $\ga$ a local corner component of $X$ at~$x$.

The corner functors for $\Manc,\Mangc$ and $\CSchc$ commute with the embeddings $\Manc,\Mangc\hookra\CSchc$.

To get an analogue of the decomposition $C(X)=\coprod_{k\ge 0}C_k(X)$ for $C^\iy$-schemes with corners, we restrict to the full subcategories $\CSchcfi\subset\CSchc$, $\CSchcfiin\subset\CSchcin$ of {\it firm\/} $C^\iy$-schemes with corners, where $C$ maps $\CSchcfi\ra\CSchcfiin$. For $\bX$ firm there is a locally constant sheaf $\check M^\rex_{C(X)}$ of finitely generated monoids on $C(\bX)$, with $\check M^\rex_{C(X)}\vert_{(x,P)}=\O_{X,x}^\rex/[c'=1$ if $c'\in\O_{X,x}^\rex\sm P]$. We define a decomposition $C(\bX)=\coprod_{k\ge 0}C_k(\bX)$ with $C_k(\bX)$ open and closed in $C(\bX)$, by saying that $(x,P)\in C_k(\bX)$ if the maximum length of a chain of prime ideals in $\check M^\rex_{C(X)}\vert_{(x,P)}$ is $k+1$. This recovers the usual decomposition $C(X)=\coprod_{k\ge 0}C_k(X)$ if $X$ is a manifold with (g-)corners. We define the {\it boundary\/}~$\pd\bX=C_1(\bX)$. 

For toric $C^\iy$-schemes with corners $C$ maps $\CSchctoex\ra\CSchcto$, and $C$ preserves fibre products, and all fibre products exist in $\CSchcto$. We use this to give criteria for when fibre products exist in $\CSchctoex$. This is an analogue of results in \cite{Joyc6} giving criteria for when fibre products exist in $\Mangc$, given that b-transverse fibre products exist in~$\Mangcin$.

\subsubsection*{Chapter \ref{cc7}. Modules, and sheaves of modules}

If $\bfC=(\fC,\fC_\rex)$ is a $C^\iy$-ring with corners, we define a $\bfC$-{\it module\/} to be a module over $\fC$ as an $\R$-algebra. Similarly, if $\bX=(X,(\O_X,\O_X^\rex))$ is a $C^\iy$-scheme with corners, we consider $\O_X$-{\it modules\/} on $X$, which are just modules on the underlying $C^\iy$-scheme $\uX=(X,\O_X)$. So the theory of modules over $C^\iy$-rings with corners and $C^\iy$-schemes with corners lifts immediately from modules over $C^\iy$-rings and $C^\iy$-schemes in \cite[\S 5]{Joyc9}, with no extra work to do.

If $X$ is a manifold with corners then as in Chapter \ref{cc3} we have two notions of cotangent bundle $T^*X,{}^bT^*X$, where ${}^bT^*X$ is functorial only under interior morphisms. Similarly, if $\bfC=(\fC,\fC_\rex)$ is a $C^\iy$-ring with corners, we have the {\it cotangent module\/} $\Om_\fC$ of $\fC$ from \cite[\S 5]{Joyc9}. If $\bfC$ is interior we also define the {\it b-cotangent module\/} ${}^b\Om_\bfC$, which uses the corner structure. If $X$ is a manifold with (g-)faces and $\bfC=\bs C^\iy_\rin(X)$ then $\Om_\fC=\Ga^\iy(T^*X)$ and ${}^b\Om_\bfC=\Ga^\iy({}^bT^*X)$. We show that b-cotangent modules are functorial under interior morphisms and have exact sequences for pushouts and coequalizers.

If $\bX$ is a $C^\iy$-scheme with corners we define the {\it cotangent sheaf\/} $T^*\bX$, and if $\bX$ is interior the {\it b-cotangent sheaf\/} ${}^bT^*\bX$, by sheafifying the (b-)cotangent modules of $\bO_X(U)$ for open $U\subset X$. If $\bX=F_\Manc(X)$ for $X\in\Manc$ these are the sheaves of sections of $T^*X$ and ${}^bT^*X$. We show that Cartesian squares in subcategories such as $\CSchcto,\CSchcfiin\subset\CSchcin$ yield exact sequences of b-cotangent sheaves. On the corners $C(\bX)$ we construct an exact sequence relating ${}^bT^*C(\bX)$, $\bs\Pi_\bX^*({}^bT^*\bX)$ and~$\check M^\rex_{C(X)}\ot_\N\O_{C(X)}$.

\subsubsection*{Chapter \ref{cc8}. Further generalizations}

Finally we propose three directions in which this book can be generalized:
\smallskip

\noindent{\bf $C^\iy$-stacks with corners.} In classical algebraic geometry, schemes are generalized to (Deligne--Mumford or Artin) stacks. The second author \cite{Joyc9} extended the theory of $C^\iy$-schemes to $C^\iy$-{\it stacks}, including {\it Deligne--Mumford\/ $C^\iy$-stacks}. This corresponds to generalizing manifolds to orbifolds.

We discuss a theory of $C^\iy$-{\it stacks with corners}, including  {\it Deligne--Mumford\/ $C^\iy$-stacks with corners}. These generalize {\it orbifolds with (g-)corners}. Much of the theory follows from \cite{Joyc9} with only cosmetic changes.
\smallskip

\noindent{\bf $C^\iy$-rings and $C^\iy$-schemes with a-corners.} Our theory starts with the categories $\Mancin\subset\Manc$ of manifolds with corners defined in Chapter \ref{cc3}. The second author \cite{Joyc7} also defined categories $\Manacin\subset\Manac$ of {\it manifolds with analytic corners}, or {\it manifolds with a-corners}. Even the simplest objects $\lb 0,\iy)$ in $\Manac$ and $[0,\iy)$ in $\Manc$ have different smooth structures. 

Manifolds with a-corners have applications to p.d.e.s with boundary conditions of asymptotic type, and to moduli spaces with boundary and corners, such as moduli spaces of Morse flow lines, in which (we argue) manifolds with a-corners give the correct smooth structure.

This entire book could be rewritten over $\Manacin\subset\Manac$ rather than $\Mancin\subset\Manc$. We explain the first steps in this.
\smallskip

\noindent{\bf Derived manifolds and derived orbifolds with corners.} Classical Algebraic Geometry has been generalized to Derived Algebraic Geometry, which is now a major area of mathematics. It is less well known that classical Differential Geometry can be generalized to {\it Derived Differential Geometry}, the study of {\it derived manifolds\/} and {\it derived orbifolds}, regarded as special examples of {\it derived\/ $C^\iy$-schemes\/} and {\it derived\/ $C^\iy$-stacks}. Some references are Lurie \cite[\S 4.5]{Luri}, Spivak \cite{Spiv}, Borisov and Noel \cite{Bori,BoNo}, and the second author~\cite{Joyc3,Joyc4}. 

It is desirable to extend the subject to {\it derived manifolds with corners\/} and {\it derived orbifolds with corners}, regarded as special examples of {\it derived\/ $C^\iy$-schemes with corners\/} and {\it derived\/ $C^\iy$-stacks with corners}. This will be done by the second author in \cite{Joyc4}, with this book as its foundations. Derived orbifolds with corners will have important applications in Symplectic Geometry, as (we argue) they are the correct way to make Fukaya--Ohta--Oh--Ono's `Kuranishi spaces with corners' \cite{FOOO} into well behaved geometric spaces.
\smallskip

\noindent{\it Acknowledgements.} This book is based on the first author's PhD thesis \cite{Fran}, supervised by the second author. Some previous work on the subject was done by Elana Kalashnikov in her MSc thesis \cite{Kala}, also supervised by the second author. The first author would like to thank the Rhodes Trust for funding during her PhD.

\section{\texorpdfstring{Background on $C^\iy$-schemes}{Background on C∞-schemes}}
\label{cc2}

We begin with background material and results on $C^\iy$-rings and $C^\iy$-schemes, which we will later generalize to $C^\iy$-rings with corners and $C^\iy$-schemes with corners. Our main reference is the second author \cite[\S 2--\S 4]{Joyc9}, see also \cite{Joyc2}, Dubuc \cite{Dubu2}, Moerdijk and Reyes \cite{MoRe2}, and Kock~\cite{Kock}.

\subsection{\texorpdfstring{$C^\iy$-rings}{C∞-rings}}
\label{cc21}

Here are two equivalent definitions of $C^\iy$-ring.

\begin{dfn} Write $\Man$ for the category of manifolds, and $\Euc$ for the full subcategory of
$\Man$ with objects the Euclidean spaces $\R^n$. That is, the
objects of $\Euc$ are $\R^n$ for $n=0,1,2,\ldots,$ and the morphisms
in $\Euc$ are smooth maps $f:\R^m\ra\R^n$. Write $\Sets$ for the category of sets.
In both $\Euc$ and $\Sets$ we have notions of (finite) products of
objects (that is, $\R^{m+n}=\R^m\t\R^n$, and products $S\t T$ of
sets $S,T$), and products of morphisms. 

Define a {\it categorical\/ $C^\iy$-ring\/} to be a product-preserving functor $F:\Euc\ra\Sets$. Here $F$ should also preserve the empty product, that is, it maps $\R^0$ in $\Euc$ to the terminal object in $\Sets$, the point $*$. If $F,G:\Euc\ra\Sets$ are categorical $C^\iy$-rings, a {\it morphism\/} $\eta:F\ra G$ is a natural transformation $\eta:F\Ra G$. We write $\CCRings$ for the category of categorical $C^\iy$-rings. 
\label{cc2def1}
\end{dfn}

Categorical $C^\iy$-rings in this sense are an example of an {\it algebraic theory\/} in the sense of Ad\'amek, Rosick\'y and Vitale \cite{ARV}, and many of the basic categorical properties of $C^\iy$-rings follow from this.

\begin{dfn} A $C^\iy$-{\it ring\/} is a set $\fC$ together with
operations
\begin{equation*}
\smash{\Phi_f:\fC^n={\buildrel {\!\ulcorner\,\text{$n$ copies
}\,\urcorner\!} \over {\fC\t\cdots \t \fC}}\longra \fC}
\end{equation*}
for all $n\ge 0$ and smooth maps $f:\R^n\ra\R$, where by convention
when $n=0$ we define $\fC^0$ to be the single point $\{\es\}$. These
operations must satisfy the following relations: suppose $m,n\ge 0$,
and $f_i:\R^n\ra\R$ for $i=1,\ldots,m$ and $g:\R^m\ra\R$ are smooth
functions. Define a smooth function $h:\R^n\ra\R$ by
\begin{equation*}
h(x_1,\ldots,x_n)=g\bigl(f_1(x_1,\ldots,x_n),\ldots,f_m(x_1
\ldots,x_n)\bigr),
\end{equation*}
for all $(x_1,\ldots,x_n)\in\R^n$. Then for all
$(c_1,\ldots,c_n)\in\fC^n$ we have
\begin{equation*}
\Phi_h(c_1,\ldots,c_n)=\Phi_g\bigl(\Phi_{f_1}(c_1,\ldots,c_n),
\ldots,\Phi_{f_m}(c_1,\ldots,c_n)\bigr).
\end{equation*}
We also require that for all $1\le j\le n$, defining
$\pi_j:\R^n\ra\R$ by $\pi_j:(x_1,\ldots,x_n)\mapsto x_j$, we have
$\Phi_{\pi_j}(c_1,\ldots,c_n)=c_j$ for
all~$(c_1,\ldots,c_n)\in\fC^n$.

Usually we refer to $\fC$ as the $C^\iy$-ring, leaving the $C^\iy$-operations $\Phi_f$ implicit.

A {\it morphism\/} between $C^\iy$-rings $\bigl(\fC,(\Phi_f)_{
f:\R^n\ra\R\,\,C^\iy}\bigr)$, $\bigl({\mathfrak
D},(\Psi_f)_{f:\R^n\ra\R\,\,C^\iy}\bigr)$ is a map
$\phi:\fC\ra{\mathfrak D}$ such that $\Psi_f\bigl(\phi
(c_1),\ldots,\phi(c_n)\bigr)=\phi\ci\Phi_f(c_1,\ldots,c_n)$ for all
smooth $f:\R^n\ra\R$ and $c_1,\ldots,c_n\in\fC$. We will write
$\CRings$ for the category of $C^\iy$-rings.
\label{cc2def2}
\end{dfn}

Each $C^\iy$-ring has an underlying $\R$-algebra structure.

\begin{prop}
\label{cc2prop1}
There is an equivalence of categories from $\CRings$ to\/ $\CCRings,$ identifying $\fC$ in $\CRings$ with\/ $F:\Euc\ra\Sets$ in $\CCRings$ such that\/ $F\bigl(\R^n\bigr)=\fC^n$ for\/~$n\ge 0$.
\end{prop}

The following example motivates these definitions:

\begin{ex}
\label{cc2ex1}
{\bf(a)} Let $X$ be a manifold. Define a functor $F_X:\Euc\ra\Sets$ by $F_X(\R^n)=\Hom_\Man(X,\R^n)$, and $F_X(g)=g\ci:\Hom_\Man(X,\R^m)\ra\Hom_\Man(X,\R^n)$ for each morphism $g:\R^m\ra\R^n$ in $\Euc$. Then $F_X$ is a categorical $C^\iy$-ring. If $f:X\ra Y$ is a smooth map of manifolds define a natural transformation $F_f:F_Y\Ra F_f$ by $F_f(\R^n)=\ci f:\Hom_\Man(Y,\R^n)\ra\Hom_\Man(X,\R^n)$. Then $F_f$ is a morphism in $\CCRings$. Define a functor $F_\Man^\CCRings:\Man\ra\CCRings^{\bf op}$ to map $X\mapsto F_X$ and $f\mapsto F_f$.
\smallskip

\noindent{\bf(b)} Let $X$ be a manifold. Write $C^\iy(X)$ for the set of smooth functions $c:X\ra\R$. For non-negative integers $n$ and smooth $f:\R^n\ra\R$, define $C^\iy$-operations $\Phi_f:C^\iy(X)^n\ra C^\iy(X)$ by composition
\e
\bigl(\Phi_f(c_1,\ldots,c_n)\bigr)(x)=f\bigl(c_1(x),\ldots,c_n(x)\bigr),
\label{cc2eq1}
\e
for all $c_1,\ldots,c_n\in C^\iy(X)$ and $x\in X$. The composition and projection relations follow directly from the definition of $\Phi_f$, so that $C^\iy(X)$ forms a $C^\iy$-ring. If we consider the $\R$-algebra structure of $C^\iy(X)$ as a $C^\iy$-ring, this is the canonical $\R$-algebra structure on $C^\iy(X)$. If $f:X\ra Y$ is a smooth map of manifolds, then $f^*:C^\iy(Y)\ra C^\iy(X)$ mapping $c\mapsto c\ci f$ is a morphism of $C^\iy$-rings. 

Define $F_\Man^\CRings:\Man\ra\CRings^{\bf op}$ to map $X\mapsto C^\iy(X)$ and $f\mapsto f^*$. Moerdijk and Reyes \cite[Th.~I.2.8]{MoRe2} show that $F_\Man^\CRings$ is full and faithful, and takes transverse fibre products in $\Man$ to fibre products in~$\CRings^{\bf op}$.
\end{ex}

There are many more $C^\iy$-rings than those that come from manifolds. For example, if $X$ is a smooth manifold $X$ of $\dim X>0$, then the set $C^k(X)$ of $k$-differentiable maps $f:X\ra \R$ is a $C^\iy$-ring with operations $\Phi_f$ defined as in \eq{cc2eq1}, and each of these $C^\iy$-rings is different for all~$k=0,1,\ldots.$ 

\begin{ex} 
\label{cc2ex2}
Consider $X=*$ the point, so $\dim X=0$, then $C^\iy(*)=\R$ and Example \ref{cc2ex1} shows the $C^\iy$-operations $\Phi_f:\R^n\ra\R$ given by $\Phi_f(x_1,\ldots,x_n)=f(x_1,\ldots,x_n)$ make $\R$ into a $C^\iy$-ring. This is the initial object in~$\CRings$, and the simplest nonzero example of a $C^\iy$-ring. The zero $C^\iy$-ring is the set $\{0\}$ where all $C^\iy$-operations $\Phi_f:\{0\}\ra\{0\}$ send $0\mapsto 0$, and this is the final object in $\CRings$. 
\end{ex}

By Moerdijk and Reyes \cite[p.~21--22]{MoRe2} and Ad\'amek et al.\ \cite[Prop.s 1.21, 2.5 \& Th.~4.5]{ARV} we have:

\begin{prop} 
\label{cc2prop2}
The category $\CRings$ of\/ $C^\iy$-rings has all small limits and all small colimits. The forgetful functor\/ $\Pi:\CRings\ra\Sets$ preserves limits and directed colimits, and can be used to compute such (co)limits, however it does not preserve general colimits such as pushouts.
\end{prop}

For the pushout of morphisms $\phi:\fC\!\ra\!\fD$, $\psi:\fC\!\ra\!\fE$ in $\CRings$, we write $\fD\amalg_{\phi,\fC,\psi}\fE$ or $\fD\amalg_\fC\fE$. In the special case $\fC\!=\!\R$ the coproduct $\fD\amalg_\R\fE$ will be written as $\fD\ot_\iy\nobreak\fE$. Recall that coproduct of $\R$-algebras $A,B$ is the tensor product $A\ot B$, however $\fD\ot_\iy\fE$ is usually different from their tensor product $\fD\ot\fE$. For example, for $m,n>0$, then $C^\iy(\R^m)\ot_\iy C^\iy(\R^n)\cong C^\iy(\R^{m+n})$ as in \cite[p.~22]{MoRe2}, which contains $C^\iy(\R^m)\ot C^\iy(\R^n)$ but is larger than this, as it includes elements such as $\exp(fg)$ for $f\in C^\iy(\R^m)$ and~$g\in C^\iy(\R^n)$. 

The next definition and proposition come from Ad\'amek et al.~\cite[Rem.~11.21, Prop.s 11.26, 11.28, 11.30 \& Cor. 11.33]{ARV}.

\begin{dfn}
\label{cc2def3}
If $A$ is a set then by \cite[Rem.~11.21]{ARV} we can define the {\it free $C^\iy$-ring $\fF^A$ generated by\/} $A$. We may think of $\fF^A$ as $C^\iy(\R^A)$, where $\R^A=\bigl\{(x_a)_{a\in A}:x_a\in\R\bigr\}$. Explicitly, we define $\fF^A$ to be the set of maps $c:\R^A\ra\R$ which depend smoothly on only finitely many variables $x_a$, and operations $\Phi_f$ are defined as in \eq{cc2eq1}. Regarding $x_a:\R^A\ra\R$ as functions for $a\in A$,we have $x_a\in\fF^A$, and we call $x_a$ the {\it generators\/} of~$\fF^A$.

Then $\fF^A$ has the property that if $\fC$ is any $C^\iy$-ring then a choice of map $\al:A\ra\fC$ uniquely determines a morphism $\phi:\fF^A\ra\fC$ with $\phi(x_a)=\al(a)$ for $a\in A$. When $A=\{1,\ldots,n\}$ we have $\fF^A\cong C^\iy(\R^n)$.	
\end{dfn}

\begin{prop}
\label{cc2prop3}
{\bf(a)} Every object\/ $\fC$ in $\CRings$ admits a surjective morphism $\phi:\fF^A\ra\fC$ from some free $C^\iy$-ring $\fF^A$. We call\/ $\fC$ \begin{bfseries}finitely generated\end{bfseries} if this holds with\/ $A$ finite.
\smallskip

\noindent{\bf(b)} Every object\/ $\fC$ in $\CRings$ fits into a \begin{bfseries}coequalizer diagram\end{bfseries}
\e
\xymatrix@C=50pt{  \fF^B \ar@<.1ex>@/^.3pc/[r]^{\al} \ar@<-.1ex>@/_.3pc/[r]_{\be} & \fF^A \ar[r]^{\phi} & \fC, }
\label{cc2eq2}
\e
that is, $\fC$ is the colimit of\/ $\fF^B\rightrightarrows\fF^A$ in $\CRings,$ where $\phi$ is automatically surjective. We call\/ $\fC$ \begin{bfseries}finitely presented\end{bfseries} if this holds with\/ $A,B$ finite.
\end{prop} 

Actually as any relation $f=g$ in $\fC$ is equivalent to the relation $f-g=0$, we can simplify \eq{cc2eq2} by taking $\be$ to map $x_b\mapsto 0$ for all $b\in B$. But for the analogue for $C^\iy$-rings with corners in \S\ref{cc45}, this will not work. 

\begin{dfn}
\label{cc2def4}
Recall that a {\it local\/ $\R$-algebra}, is an $\R$-algebra $R$ with a unique maximal ideal $\mk$. The residue field of $R$ is the field isomorphic to $R/\mk$. A $C^\iy$-ring $\fC$ is called {\it local\/} if, regarded as an $\R$-algebra, $\fC$ is a local $\R$-algebra with residue field $\R$. The quotient morphism gives a (necessarily unique) morphism of $C^\iy$-rings $\pi:\fC\ra\R$ with the property that $c\in \fC$ is invertible if and only if $\pi(c)\ne 0$. Equivalently, if such a morphism $\pi:\fC\ra \R$ exists with this property, then $\fC$ is local with maximal ideal $\m_\fC\cong \Ker\pi$. Write $\CRingslo\subset\CRings$ for the full subcategory of local $C^\iy$-rings.

Usually morphisms of local rings are required to send maximal ideals into maximal ideals. However, if $\phi:\fC\ra\fD$ is any morphism of local $C^\iy$-rings, we see that $\phi^{-1}(\m_\fD)=\m_\fC$ as the residue fields in both cases are $\R$, so there is no difference between local morphisms and morphisms for $C^\iy$-rings.
\end{dfn}

The next proposition may be found in Moerdijk and Reyes~\cite[\S I.3]{MoRe2}.

\begin{prop}
\label{cc2prop4}
All finite colimits exist in $\CRingslo,$ and agree with the corresponding colimits in $\CRings$. 
\end{prop}

Localizations of $C^\iy$-rings were studied in \cite{Dubu1,Dubu2,MQR,MoRe1}, \cite[p.~23]{MoRe2} and~\cite{Joyc9}.

\begin{dfn}
\label{cc2def5}
A {\it localization\/} $\fC(s^{-1}:s\in S)=\fD$ of a $C^\iy$-ring $\fC$ at a subset $S\subset \fC$ is a $C^\iy$-ring $\fD$ and a morphism $\pi:\fC\ra\fD$ such that $\pi(s)$ is invertible in $\fD$ for all $s\in S$, which has the universal property that for any morphism of $C^\iy$-rings $\phi:\fC\ra\fE$ such that $\phi(s)$ is invertible in $\fE$ for all $s\in S$, there is a unique morphism $\psi:\fD\ra\fE$ with $\phi=\psi\ci\pi$. We call $\pi:\fC\ra\fD$ the {\it localization morphism\/} for~$\fD$. 
 
By adding an extra generator $s^{-1}$ and extra relation $s\cdot s^{-1}-1=0$ for each $s\in S$ to $\fC$, it can be shown that localizations $\fC(s^{-1}:s\in S)$ always exist and are unique up to unique isomorphism. When $S=\{c\}$ then $\fC(c^{-1})\cong (\fC\ot_\iy C^\iy(\R))/I$, where $I$ is the ideal generated by $\io_1(c)\cdot\io_2(x)-1$, $x$ is the generator of $C^\iy(\R)$, and $\io_1,\io_2$ are the coproduct morphisms $\io_1:\fC\ra\fC\ot_\iy C^\iy(\R)$ and~$\io_2:C^\iy(\R)\ra\fC\ot_\iy C^\iy(\R)$.
\end{dfn}

An example of this is that if $f\in C^\iy(\R^n)$ is a smooth function, and $U=f^{-1}(\R\setminus\{0\})\subseteq\R^n$, then using partitions of unity one can show that $C^\iy(U)\cong C^\iy(\R^n)(f^{-1})$, as in~\cite[Prop.~I.1.6]{MoRe2}.

\begin{dfn}
\label{cc2def6}
A $C^\iy$-ring morphism $x:\fC\ra\R$, where $\R$ is regarded as a $C^\iy$-ring as in Example \ref{cc2ex2}, is called an $\R$-{\it point\/}. Note that a map $x:\fC\ra\R$ is a morphism of $C^\iy$-rings whenever it is a morphism of the underlying $\R$-algebras, as in \cite[Prop.~I.3.6]{MoRe2}. We define $\fC_x$ as the localization $\fC_x=\fC(s^{-1}:s\in\fC$, $x(s)\ne 0)$, and denote the projection morphism by $\pi_x:\fC\ra\fC_x$. Importantly, \cite[Lem.~1.1]{MoRe1} shows $\fC_x$ is a local $C^\iy$-ring. 
\end{dfn}

We can describe $\fC_x$ explicitly as in the second author \cite[Prop.~2.14]{Joyc9}.

\begin{prop}
\label{cc2prop5}
Let\/ $x:\fC\ra\R$ be an $\R$-point of a $C^\iy$-ring $\fC,$ and consider the projection morphism $\pi_x:\fC\ra\fC_x$. Then $\fC_x\cong\fC/\Ker{\pi_x}$. This kernel is $\Ker{\pi_x}=I$ where
\e
I=\bigl\{\text{$c\in\fC:$ there exists\/ $d\in\fC$ with\/ $x(d)\ne 0$ in $\R$ and\/ $c\cdot d=0$ in $\fC$}\bigr\}.
\label{cc2eq3}
\e
\end{prop}

While this localization morphism $\pi_x:\fC\ra\fC_x$ is surjective, general localizations of $C^\iy$-rings need not have surjective localization morphisms. 

\begin{ex}
\label{cc2ex3}
Let $C^\iy_p(\R^n)$ be the set of germs of smooth functions $c:\R^n\ra\R$ at $p\in\R^n$ for $n\ge 0$ and $p\in\R^n$. We give $C^\iy_p(\R^n)$ a $C^\iy$-ring structure by using \eq{cc2eq1} on germs of functions. There are several equivalent definitions:
\begin{itemize}
\setlength{\itemsep}{0pt}
\setlength{\parsep}{0pt}
\item[(i)] $C^\iy_p(\R^n)$ is the set of $\sim$-equivalence classes $[U,c]$ of pairs $(U,c)$, where $p\in U\subseteq\R^n$ is open and $c:U\ra\R$ is smooth, and $(U,c)\sim(U',c')$ if there exists open $p\in U''\subseteq U\cap U'$ with~$c\vert_{U''}=c'\vert_{U''}$.
\item[(ii)] $C^\iy_p(\R^n)\cong C^\iy(\R^n)/I_p$, where $I_p\subset C^\iy(\R^n)$ is the ideal of functions vanishing near $p$. 
\item[(iii)] $C^\iy_p(\R^n)\cong C^\iy(\R^n)(f^{-1}:f\in C^\iy(\R^n)$, $f(p)\ne 0)$.
\end{itemize}
Then $C^\iy_p(\R^n)$ is local, with maximal ideal~${\mathfrak m}_{p}=\bigl\{[U,c]\in C^\iy_p(\R^n):c(x)=0\bigr\}$.
\end{ex}

We prove some facts about exponentials and logs in (local) $C^\iy$-rings.

\begin{prop}
\label{cc2prop6}
{\bf(a)} Let\/ $\fC$ be a $C^\iy$-ring. Then the\/ $C^\iy$-operation $\Phi_{\exp}:\fC\ra\fC$ induced by $\exp:\R\ra\R$ is injective.
\smallskip

\noindent{\bf(b)} Let\/ $\fC$ be a local\/ $C^\iy$-ring, with morphism $\pi:\fC\ra\R$. If\/ $a\in\fC$ with\/ $\pi(a)>0$ then there exists $b\in\fC$ with\/ $\Phi_{\exp}(b)=a$. This $b$ is unique by\/~{\bf(a)}.
\end{prop}

\begin{proof} For (a), let $a\in\fC$ with $b=\Phi_{\exp}(a)\in\fC$. Then $\Phi_{\exp}(-a)$ is the inverse $b^{-1}$ of $b$. The map $t\mapsto \exp(t)-\exp(-t)$ is a diffeomorphism $\R\ra\R$. Let $e:\R\ra\R$ be its inverse. Define smooth $f:\R^2\ra\R$ by $f(x,y)=e(x-y)$. Then $f(\exp t,\exp(-t))=t$. Hence in the $C^\iy$-ring $\fC$ we have
\begin{equation*}
\Phi_f(b,b^{-1})=\Phi_f(\Phi_{\exp}(a),\Phi_{\exp\ci -}(a))=\Phi_{f\ci(\exp,\exp\ci -)}(a)=\Phi_{\id}(a)=a.
\end{equation*}
But $b$ determines $b^{-1}$ uniquely, so $\Phi_f(b,b^{-1})=a$ implies that $b=\Phi_{\exp}(a)$ determines $a$ uniquely, and $\Phi_{\exp}:\fC\ra\fC$ is injective.

For (b), choose smooth $g,h:\R\ra\R$ with $g(x)=\log x$ for $x\ge\ha\pi(a)>0$, $h(\pi(a))>0$, and $h(x)=0$ for $x\le \ha\pi(a)$. Set $b=\Phi_g(a)$ and $c=\Phi_h(a)$. Then
\begin{equation*}
c\cdot(\Phi_{\exp}(b)-a)=\Phi_h(a)\cdot(\Phi_{\exp\ci g}(a)-a)=\Phi_{h(x)\cdot(\exp\ci g(x)-x)}(a)=0,
\end{equation*}
since $h(x)\cdot(x-\exp\ci g(x))=0$ as $h(x)=0$ for $x\le \ha\pi(a)$ and $\exp\ci g(x)-x=0$ for $x\ge\ha\pi(a)$. Also $\pi(c)=h(\pi(a))>0$, so $c$ is invertible. Thus~$\Phi_{\exp}(b)=a$.
\end{proof}

\subsection{\texorpdfstring{Modules and cotangent modules of $C^\iy$-rings}{Modules and cotangent modules of C∞-rings}}
\label{cc22}

We discuss modules and cotangent modules for $C^\iy$-rings, following~\cite[\S 5]{Joyc9}.

\begin{dfn}
\label{cc2def7}
A {\it module\/ $M$\/} over a $C^\iy$-ring $\fC$ is a module over $\fC$ as a commutative $\R$-algebra, and morphisms of $\fC$-modules are the usual morphisms of $\R$-algebra modules. Denote $\mu_M:\fC\t M\ra M$ the multiplication map, and write $\mu_M(c,m)=c\cdot m$ for $c\in\fC$ and $m\in M$. The category $\fCmod$ of  $\fC$-modules is an abelian category.

If a $\fC$-module $M$ fits into an exact sequence $\fC\ot\R^n\ra M\ra 0$ in $\fCmod$ then it is {\it finitely generated\/}; if it further fits into an exact sequence $\fC\ot\R^m\ra \fC\ot\R^n\ra M\ra 0$ it is {\it finitely presented\/}. This second condition is not automatic from the first as $C^\iy$-rings are generally not Noetherian.

For a morphism $\phi:\fC\ra\fD$ of $C^\iy$-rings and $M$ in $\fCmod$ we have $\phi_*(M)=M\ot_\fC\fD$ in $\fDmod$, giving a functor $\phi_*:\fCmod\ra\fDmod$. For $N$ in $\fDmod$ there is a $\fC$-module $\phi^*(N)=N$ with $\fC$-action $\mu_{\phi^*(N)}(c,n)=\mu_N(\phi(c),n)$. This gives a functor~$\phi^*:\fDmod\ra\fCmod$.
\end{dfn}

\begin{ex}
\label{cc2ex4}
Let $\Ga^\iy(E)$ be the collection of smooth sections $e$ of a vector bundle $E\ra X$ of a manifold $X$, so $\Ga^\iy(E)$ is a vector space and a module over $C^\iy(X)$. If $\la:E\ra F$ is a morphism of vector bundles over $X$, then there is a morphism of $C^\iy(X)$-modules $\la_*:\Ga^\iy(E)\ra \Ga^\iy(F)$ where $\la_*:e\mapsto \la\ci e$.

For each smooth map of manifolds $f:X\ra Y$ there is a morphism of $C^\iy$-rings $f^*:C^\iy(Y)\ra C^\iy(X)$. Each vector bundle $E\ra Y$ gives a vector bundle $f^*(E)\ra X$. Using $(f^*)_*:C^\iy(Y)$-mod$\,\ra C^\iy(X)$-mod from Definition \ref{cc2def7}, then $(f^*)_*\bigl(\Ga^\iy(E) \bigr)=\Ga^\iy(E)\ot_{C^\iy(Y)}C^\iy(X)$ is isomorphic to~$\Ga^\iy\bigl(f^*(E)\bigr)$ in~$C^\iy(X)$-mod.
\end{ex}

The definition of $\fC$-module only used the commutative $\R$-algebra structure of $\fC$, however the {\it cotangent module\/} $\Om_\fC$ of $\fC$ uses the full $C^\iy$-ring structure. 

\begin{dfn}
\label{cc2def8}
Take a $C^\iy$-ring $\fC$ and $M\in \fCmod$, then a $C^\iy$-{\it derivation} is a map $\d:\fC\ra M$ that satisfies the following: for any smooth  $f:\R^n\ra\R$ and elements $c_1,\ldots,c_n\in\fC$, then
\e
\d\Phi_f(c_1,\ldots,c_n)-\ts\sum\limits_{i=1}^n\Phi_{\frac{\pd f}{\pd x_i}}(c_1,\ldots,c_n)\cdot \d c_i=0.
\label{cc2eq4}
\e
This implies that $\d$ is $\R$-linear and is a derivation of $\fC$ as a commutative $\R$-algebra, that is, $\d(c_1c_2)=c_1\cdot\d c_2+c_2\cdot\d c_1$ for all~$c_1,c_2\in\fC$.

The pair $(M,\d)$ is called a {\it cotangent module\/} for $\fC$ if it is universal in the sense that for any $M'\in \fCmod$ with $C^\iy$-derivation $\d':\fC\ra M'$, there exists a unique morphism of $\fC$-modules $\la:M\ra M'$ with $\d'=\la\ci\d$. Then a cotangent module is unique up to unique isomorphism. We can explicitly construct a cotangent module for $\fC$ by considering the free $\fC$-module over the symbols $\d c$ for all $c\in\fC$, and quotienting by all relations \eq{cc2eq4} for smooth $f:\R^n\ra\R$ and elements $c_1,\ldots,c_n\in\fC$. We call this construction `the' cotangent module of $\fC$, and write it as~$\d_\fC:\fC\ra\Om_\fC$.

If we have a morphism of $C^\iy$-rings $\fC\ra\fD$ then $\Om_\fD=\phi^*(\Om_\fD)$ can be considered as a $\fD$-module with $C^\iy$-derivation $\d_\fD\ci\phi:\fC\ra\Om_\fD$. The universal property of $\Om_\fC$ gives a unique morphism $\Om_\phi:\Om_\fC\ra\Om_\fD$ of $\fC$-modules such that $\d_\fD\ci\phi=\Om_\phi\ci\d_\fC$. From this we have a morphism of $\fD$-modules $(\Om_\phi)_*:\Om_\fC\ot_\fC\fD\ra\Om_\fD$. If we have two morphisms of $C^\iy$-rings $\phi:\fC\ra\fD$, $\psi:\fD\ra\fE$ then uniqueness implies that~$\Om_{\psi\ci\phi}=
\Om_\psi\ci\Om_\phi:\Om_\fC\ra\Om_\fE$.
\end{dfn}

\begin{ex}
\label{cc2ex5}
As in Example \ref{cc2ex4}, if $X$ is a manifold, 
its cotangent bundle $T^*X$ is a vector bundle over $X$, and its global sections $\Ga^\iy(T^*X)$ form a $C^\iy(X)$-module, with $C^\iy$-derivation $\d:C^\iy(X)\ra \Ga^\iy(T^*X)$, $\d:c\mapsto\d c$ the usual exterior derivative and
equation \eq{cc2eq4} following from the chain rule. 

One can show that $(\Ga^\iy(T^*X),\d)$ has the universal property in Definition \ref{cc2def8}, and so form a cotangent module for $C^\iy(X)$. This is stated in \cite[Ex.~5.4]{Joyc9}, and proved in greater generality in Theorem \ref{cc7thm1}(a) below.

If we have a smooth map of manifolds $f:X\ra Y$, then $f^*(T^*Y),\ab T^*X$ are vector bundles over $X$, and the derivative $\d f:f^*(T^*Y)\ra T^*X$ is a vector bundle morphism. This induces a morphism of $C^\iy(X)$-modules $(\d f)_*:\Ga^\iy(f^*(T^*Y))\ra \Ga^\iy(T^*X)$, which is identified with $(\Om_{f^*})_*$ from Definition~\ref{cc2def8}
using that $\Ga^\iy(f^*(T^*Y))\cong \Ga^\iy(T^*Y)\ot_{C^\iy(Y)}C^\iy(X)$.
\end{ex}

This example shows that Definition \ref{cc2def8} abstracts the notion of sections of a cotangent bundle of a manifold to a concept that is well defined for any $C^\iy$-ring.

\begin{ex}
\label{cc2ex6}
{\bf(a)} Let $A$ be a set and $\fF^A$ the free $C^\iy$-ring from Definition \ref{cc2def3}, with generators $x_a\in\fF^A$ for $a\in A$. Then there is a natural isomorphism
\begin{equation*}
\Om_{\fF^A}\cong\ban{\d x_a:a\in A}_\R\ot_\R\fF^A.
\end{equation*}

\noindent{\bf(b)} Suppose $\fC$ is defined by a coequalizer diagram \eq{cc2eq2} in $\CRings$. Then writing $(x_a)_{a\in A}$, $(\ti x_b)_{b\in B}$ for the generators of $\fF^A,\fF^B$, we have an exact sequence
\begin{equation*}
\xymatrix@C=30pt{\ban{\d \ti x_b:b\in B}_\R\ot_\R\fC \ar[r]^{\ga} & \ban{\d x_a:a\in A}_\R\ot_\R\fC \ar[r]^(0.7)\de & \Om_\fC \ar[r] & 0 }
\end{equation*}
in $\fCmod$, where if $\al,\be$ in \eq{cc2eq2} map $\al:\ti x_b\!\mapsto\! f_b\bigl((x_a)_{a\in A}\bigr)$, $\be:\ti x_b\!\mapsto\! g_b\bigl((x_a)_{a\in A}\bigr)$, for $f_b,g_b$ depending only on finitely many $x_a$, then $\ga,\de$ are given by
\begin{equation*}
\ga(\d \ti x_b)=\sum_{a\in A}\phi\Bigl(\frac{\pd f_b}{\pd x_a}\bigl((x_a)_{a\in A}\bigr)-\frac{\pd g_b}{\pd x_a}\bigl((x_a)_{a\in A}\bigr)\Bigr)\d x_a,\;\>\de(\d x_a)=\d\ci\phi(x_a).
\end{equation*}
Hence as in \cite[Prop.~5.6]{Joyc9}, if $\fC$ is finitely generated (or finitely presented) in the sense of Proposition \ref{cc2prop3}, then $\Om_\fC$ is finitely generated (or finitely presented).
\end{ex}

Cotangent modules behave well under localization,~\cite[Prop.~5.7]{Joyc9}:

\begin{prop}
\label{cc2prop7}
Let\/ $\fC$ be a $C^\iy$-ring, $S\subseteq\fC,$ and\/ $\fD=\fC(s^{-1}:s\in S)$ be the localization of\/ $\fC$ at\/ $S$ with projection $\pi:\fC\ra\fD,$ as in Definition\/ {\rm\ref{cc2def5}}. Then 
$(\Om_\pi)_*:\Om_\fC\ot_\fC\fD\ra\Om_\fD$ is
an isomorphism of\/ $\fD$-modules.	
\end{prop}

Here is \cite[Th.~5.8]{Joyc9}:

\begin{thm}
\label{cc2thm1}
Suppose we are given a pushout diagram of\/ $C^\iy$-rings:
\begin{equation*}
\xymatrix@C=60pt@R=14pt{ *+[r]{\fC} \ar[r]_\be \ar[d]^\al & *+[l]{\fE} \ar[d]_\de \\
*+[r]{\fD} \ar[r]^\ga & *+[l]{\fF,\!} }
\end{equation*}
so that\/ $\fF=\fD\amalg_\fC\fE$. Then the following sequence of\/
$\fF$-modules is exact:
\e
\xymatrix@C=17pt{ \Om_\fC\ot_{\fC,\ga\ci\al}\fF
\ar[rrr]^(0.5){(\Om_\al)_*\op -(\Om_\be)_*} &&&
{\raisebox{5pt}{$\begin{subarray}{l}\ts \Om_\fD\ot_{\fD,\ga}\fF\,\, \op\\
\ts\,\,\,\Om_\fE\ot_{\fE,\de}\fF \end{subarray}$}}
\ar[rrr]^(0.65){(\Om_\ga)_*\op(\Om_\de)_*} &&& \Om_\fF \ar[r] & 0. }
\label{cc2eq5}
\e
Here\/ $(\Om_\al)_*:\Om_\fC\ot_{\fC,\ga\ci\al}\fF\ra
\Om_\fD\ot_{\fD,\ga}\fF$ is induced by\/ $\Om_\al:\Om_\fC\ra
\Om_\fD,$ and so on. Note the sign of\/ $-(\Om_\be)_*$
in\/~\eq{cc2eq5}.
\end{thm}

\subsection{Sheaves}
\label{cc23}

We explain presheaves and sheaves with values in a (nice) category $\cA$, following Godement \cite{Gode} and Mac Lane and Moerdijk \cite{MaMo}. Throughout we suppose $\cA$ is {\it complete}, that is, all small limits exist in $\cA$, and {\it cocomplete}, that is, all small colimits exist in $\cA$. The categories of sets, abelian groups, rings, $C^\iy$-rings, monoids etc.\ all satisfy this, as will (interior) $C^\iy$-rings with corners. Sometimes it is helpful to suppose that objects of $\cA$ are sets with extra structure, so there is a faithful functor $\cA\ra\Sets$ taking each object to its underlying set.

\begin{dfn}
\label{cc2def9}
A {\it presheaf\/} $\cE$ on a topological space $X$ valued in $\cA$ gives an object $\cE(U)\in \cA$ for every open set $U\subseteq X$, and a morphism 
$\rho_{UV}:\cE(U)\ra\cE(V)$ in $\cA$ called the {\it restriction map\/} for
every inclusion $V\subseteq U\subseteq X$ of open sets, satisfying
the conditions that
\begin{itemize}
\setlength{\itemsep}{0pt}
\setlength{\parsep}{0pt}
\item[(i)] $\rho_{UU}=\id_{\cE(U)}:\cE(U)\ra\cE(U)$ for all open
$U\subseteq X$; and
\item[(ii)] $\rho_{UW}=\rho_{VW}\ci\rho_{UV}:\cE(U)\ra\cE(W)$ for all
open~$W\subseteq V\subseteq U\subseteq X$.
\end{itemize}

A presheaf $\cE$ is called a {\it sheaf\/} if for all open covers $\{U_i\}_{i\in I}$ of $U$, then 
\begin{equation*}
\cE(U)\ra \prod_{i\in I}\cE(U_i)\rightrightarrows \prod_{i,j\in I}\cE(U_i\cap U_j)
\end{equation*}
forms an equalizer diagram in $\cA$. This implies:
\begin{itemize}
\setlength{\itemsep}{0pt}
\setlength{\parsep}{0pt}
\item[(iii)] $\cE(\es)=0$ where $0$ is the final object in $\cA$.
\end{itemize}
If there is a faithful functor $F:\cA\ra \Sets$ taking an object of $\cA$ to its underlying set that preserves limits, then a presheaf $\cE$ valued in $\cA$ on $X$ is a sheaf if it equivalently satisfies the following:
\begin{itemize}
\setlength{\itemsep}{0pt}
\setlength{\parsep}{0pt}
\item[(iv)](Uniqueness.) If $U\subseteq X$ is open, $\{V_i:i\in I\}$ is an open
cover of $U$, and $s,t\in F(\cE(U))$ with $F(\rho_{UV_i})(s)=F(\rho_{UV_i})(t)$ in
$F(\cE(V_i))$ for all $i\in I$, then $s=t$ in $F(\cE(U))$; and
\item[(v)](Gluing.) If $U\subseteq X$ is open, $\{V_i:i\in I\}$ is an open cover of
$U$, and we are given elements $s_i\in F(\cE(V_i))$ for all $i\in I$
such that $F(\rho_{V_i(V_i\cap V_j)})(s_i)=F(\rho_{V_j(V_i\cap
V_j)})(s_j)$ in $F(\cE(V_i\cap V_j))$ for all $i,j\in I$, then there
exists $s\in F(\cE(U))$ with $F(\rho_{UV_i})(s)=s_i$ for all $i\in I$.
\end{itemize}
If $s\in F(\cE(U))$ and $V\subseteq U$ is open we write~$s\vert_V=F(\rho_{UV})(s)$. 

If $\cE,\cF$ are presheaves or sheaves valued in $\cA$ on $X$, then a {\it morphism\/} $\phi:\cE\ra\cF$ is a morphism $\phi(U):\cE(U)\ra\cF(U)$ in $\cA$ for all open $U\subseteq X$ such that the following diagram commutes for all open
$V\subseteq U\subseteq X$
\begin{equation*}
\xymatrix@C=90pt@R=14pt{
*+[r]{\cE(U)} \ar[r]_{\phi(U)} \ar[d]^{\rho_{UV}} & *+[l]{\cF(U)}
\ar[d]_{\rho_{UV}'} \\ *+[r]{\cE(V)} \ar[r]^{\phi(V)} & *+[l]{\cF(V),\!} }
\end{equation*}
where $\rho_{UV}$ is the restriction map for $\cE$, and $\rho_{UV}'$
the restriction map for $\cF$. We write $\PreSh(X,\cA)$ and $\Sh(X,\cA)$ for the categories of presheaves and sheaves on a topological space $X$ valued in~$\cA$.
\end{dfn}

\begin{dfn}
\label{cc2def10}
For $\cE$ a presheaf valued in $\cA$ on a topological space $X$, then we can define the {\it stalk\/} $\cE_x\in\cA$ at a point $x\in X$ to be the direct limit of the $\cE(U)$ in $\cA$ for all $U\subseteq X$ with $x\in U$, using the restriction maps $\rho_{UV}$. 

If there is a faithful functor $F:\cA\ra \Sets$ taking an object of $\cA$ to its underlying set that preserves colimits, then explicitly it can be written as a set of equivalence classes of sections $s\in F(\cE(U))$ for any open $U$ which contains $x$, where the equivalence relation is such that $s_1\sim s_2$ for $s_1\in F(\cE(U))$ and $s_2\in F(\cE(V))$ with $x\in U, V$ if there is an open set $W\subset V\cap U$ with $x\in W$ and $s_1\vert_W=s_2\vert_W$ in~$F(\cE(W))$. 

The stalk is an object of $\cA$, and the restriction morphisms give rise to morphisms $\rho_{U,x}:\cE(U)\ra\cE_x$. A morphism of presheaves $\phi:\cE\ra\cF$ induces morphisms $\phi_x:\cE_x\ra\cF_x$ for all $x\in X$. If $\cE,\cF$ are sheaves then $\phi$ is an isomorphism if and only if $\phi_x$ is an isomorphism for all~$x\in X$.
\end{dfn}

\begin{dfn}
\label{cc2def11}
There is a {\it sheafification} functor $\PreSh(X,\cA)\ra\Sh(X,\cA)$, which is left adjoint to the inclusion $\Sh(X,\cA)\hookra\PreSh(X,\cA)$. We write $\hat\cE$ for the sheafification of a presheaf $\cE$. The adjoint property gives a morphism $\pi:\cE\ra\hat\cE$ and a universal property: whenever we have a morphism $\phi:\cE\ra\cF$ of presheaves on $X$ and $\cF$ is a sheaf, then there is a unique morphism $\hat\phi:\hat\cE\ra\cF$ with $\phi=\hat\phi\ci\pi$. Thus sheafification is unique up to canonical isomorphism.

Sheafifications always exists for our categories $\cA$, and there are isomorphisms of stalks $\cE_x\cong \hat\cE_x$ for all $x\in X$. If there is a faithful functor $F:\cA\ra \Sets$ taking an object of $\cA$ to its underlying set that preserves colimits and limits, it can be constructed (as in \cite[Prop.~II.1.2]{Hart}) by defining $\hat\cE(U)$ as the subset of all functions $t:U\ra \amalg_{x\in U}\cE_x$ such that for all $x\in U$, then $t(x)=F(\rho_{V,x})(s)\in \cE_x$ for some $s\in F(\cE(V))$ for open $V\subset U$, $x\in V$. 
\end{dfn}

If $f:X\ra Y$ is a continuous map of topological spaces, we can consider {\it pushforwards\/} and {\it pullbacks\/} of
sheaves by $f$. We will use both of these definitions when defining $C^\iy$-schemes (with corners).

\begin{dfn}
\label{cc2def12}
If $f:X\ra Y$ is a continuous map of topological
spaces, and $\cE$ is a sheaf valued in $\cA$ on $X$, then the {\it
direct image\/} (or {\it pushforward\/}) sheaf
$f_*(\cE)$ on $Y$ is defined by $\bigl(f_*(\cE)\bigr)(U)=\cE\bigl(f^{-1}(U)\bigr)$ for all
open $U\subseteq V$. Here, we have restriction maps $\rho'_{UV}=
\rho_{f^{-1}(U)f^{-1}(V)}:\bigl(f_*(\cE)\bigr)(U)\ra
\bigl(f_*(\cE)\bigr)(V)$ for all open $V\subseteq U\subseteq Y$ so that $f_*(\cE)$ is a sheaf valued in $\cA$ on~$Y$.

For a morphism $\phi:\cE\ra\cF$ in $\Sh(X,\cA)$ we can define $f_*(\phi):f_*(\cE)\ra f_*(\cF)$ by
$\bigl(f_*(\phi)\bigr)(U)=\phi\bigl(f^{-1}(U)\bigr)$ for all open
$U\subseteq Y$. This gives a morphism $f_*(\phi)$ in $\Sh(Y,\cA)$, and a functor
$f_*:\Sh(X,\cA)\ra\Sh(Y,\cA)$. 
For two continuous maps of topological spaces, $f:X\ra Y$, $g:Y\ra Z$, then~$(g\ci f)_*=g_*\ci
f_*$.
\end{dfn}

\begin{dfn}
\label{cc2def13}
For a continuous map $f:X\ra Y$ and a sheaf $\cE$ valued in $\cA$ on $Y$, we define the {\it pullback\/} ({\it inverse image\/}) of $\cE$ under $f$ to be the sheafification of the presheaf $U\mapsto \lim_{A\supseteq f(U)}\cE(A)$ for open $U\subseteq X$, where the direct limit is taken over all open $A\subseteq Y$ containing $f(U)$, using the restriction maps $\rho_{AB}$ in $\cE$. We write this sheaf as $f^{-1}(\cE)$. If $\phi:\cE\ra\cF$ is a morphism in $\Sh(Y,\cA)$, there is a {\it pullback morphism\/}~$f^{-1}(\phi):f^{-1}(\cE)\ra f^{-1}(\cF)$.
\end{dfn}

\begin{rem}
\label{cc2rem1}
For a continuous map $f:X\ra Y$ of topological spaces we have functors $f_*:\Sh(X,\cA)\ra\Sh(Y,\cA)$, and
$f^{-1}:\Sh(Y,\cA)\ra\Sh(X,\cA)$. Hartshorne \cite[Ex.~II.1.18]{Hart} gives a natural bijection
\e
\Hom_X\bigl(f^{-1}(\cE),\cF\bigr)\cong\Hom_Y\bigl(\cE,f_*(\cF)\bigr)
\label{cc2eq6}
\e
for all $\cE\in\Sh(Y,\cA)$ and $\cF\in\Sh(X,\cA)$, so that $f_*$ is
right adjoint to $f^{-1}$. This will be important in several proofs we consider below. 
\end{rem}

\subsection{\texorpdfstring{$C^\iy$-schemes}{C∞-schemes}}
\label{cc24}

We now recall the definition of (local) $C^\iy$-ringed spaces and $C^\iy$-schemes, following the second author~\cite{Joyc9}.

\begin{dfn}
\label{cc2def14}
A {\it $C^\iy$-ringed space\/} $\uX=(X,\O_X)$ is a
topological space $X$ with a sheaf $\O_X$ of $C^\iy$-rings on $X$.

A {\it morphism\/} $\uf=(f,f^\sh):(X,\O_X)\ra (Y,\O_Y)$ of $C^\iy$ ringed spaces consists of a continuous map $f:X\ra Y$ and a morphism $f^\sh:f^{-1}(\O_Y)\ra\O_X$ of sheaves of
$C^\iy$-rings on $X$, for $f^{-1}(\O_Y)$ the inverse image sheaf as in Definition \ref{cc2def13}. From  \eq{cc2eq6}, we know $f_*$ is right
adjoint to $f^{-1}$, so there is a natural bijection
\e
\Hom_X\bigl(f^{-1}(\O_Y),\O_X\bigr)\cong\Hom_Y\bigl(\O_Y,f_*(\O_X)\bigr).
\label{cc2eq7}
\e
We will write $f_\sh:\O_Y\ra f_*(\O_X)$ for the morphism of sheaves of $C^\iy$-rings on $Y$ corresponding to the morphism $f^\sh$ under \eq{cc2eq7}, so that
\e
f^\sh:f^{-1}(\O_Y)\longra\O_X\quad \leftrightsquigarrow\quad
f_\sh:\O_Y\longra f_*(\O_X).
\label{cc2eq8}
\e

Given two $C^\iy$-ringed space morphisms $\uf:\uX\ra\uY$ and $\ug:\uY\ra\uZ$ we can compose them to form
\begin{equation*}
\ug\ci\uf=\bigl(g\ci f,(g\ci f)^\sh\bigr)=\bigl(g\ci f,f^\sh\ci
f^{-1}(g^\sh)\bigr).
\end{equation*}
If we consider $f_\sh:\O_Y\ra f_*(\O_X)$, then the composition is
\begin{equation*}
(g\ci f)_\sh=g_*(f_\sh)\ci g_\sh:\O_Z\longra (g\ci f)_*(\O_X)=g_*\ci
f_*(\O_X).
\end{equation*}

We call $\uX=(X,\O_X)$ a {\it local\/ $C^\iy$-ringed space\/} if it is $C^\iy$-ringed space for which the stalks $\O_{X,x}$
of $\O_X$ at $x$ are local $C^\iy$-rings for all $x\in X$. As in Definition \ref{cc2def4}, since morphisms of local $C^\iy$-rings are automatically local morphisms, morphisms of local $C^\iy$-ringed spaces $(X,\O_X),(Y,\O_Y)$ are just morphisms of $C^\iy$-ringed spaces without any additional locality condition.

We will follow the notation of \cite{Joyc9} and write $\CRS$ for the category of $C^\iy$-ringed spaces, and $\LCRS$ for the full subcategory of local $C^\iy$-ringed spaces. We write underlined upper case letters such as $\uX,\uY,\uZ,\ldots$ to represent $C^\iy$-ringed spaces $(X,\O_X),(Y,\O_Y),(Z,\O_Z),\ldots,$ and underlined lower case letters $\uf,\ug,\ldots$ to represent morphisms of $C^\iy$-ringed spaces $(f,f^\sh),\ab(g,g^\sh),\ab\ldots.$ When we write `$x\in\uX$'
we mean that $\uX=(X,\O_X)$ and $x\in X$. If we write `$\uU$ {\it is open in\/} $\uX$' we will mean that $\uU=(U,\O_U)$ and $\uX=(X,\O_X)$ with $U\subseteq X$ an open set and $\O_U=\O_X\vert_U$.
\end{dfn}

\begin{ex}
\label{cc2ex7}
For a manifold $X$, we have a $C^\iy$-ringed space $\uX=(X,\O_X)$ with topological space $X$ and its sheaf of smooth functions $\O_X(U)=C^\iy(U)$ for each open subset $U\subseteq X$, with $C^\iy(U)$ defined in Example \ref{cc2ex1}. If $V\subseteq U\subseteq X$ then the restriction morphisms $\rho_{UV}:C^\iy(U)\ra C^\iy(V)$ are the usual restriction of a function to an open subset~$\rho_{UV}:c\mapsto c\vert_V$.

As the stalks $\O_{X,x}$ at $x\in X$ are local $C^\iy$-rings, isomorphic to the ring of germs as in Example \ref{cc2ex3}, then $\uX$ is a local $C^\iy$-ringed space.

For a smooth map of manifolds $f:X\ra Y$ with corresponding local $C^\iy$-ringed spaces $(X,\O_X),(Y,\O_Y)$ as above we define $f_\sh(U):\O_Y(U)= C^\iy(U)\ra\O_X(f^{-1}(U))=C^\iy(f^{-1}(U))$ for each open $U\subseteq Y$ by $f_\sh(U):c\mapsto c\ci f$ for all $c\in C^\iy(U)$. This gives a morphism $f_\sh:\O_Y\ra f_*(\O_X)$ of sheaves of $C^\iy$-rings on $Y$. Then $\uf=(f,f^\sh): (X,\O_X)\ra(Y,\O_Y)$ is a morphism of (local) $C^\iy$-ringed spaces with $f^\sh:f^{-1}(\O_Y)\ra\O_X$ corresponding to $f_\sh$ under~\eq{cc2eq8}.
\end{ex}

\begin{dfn}
\label{cc2def15}
Let $\fC$ be a $C^\iy$-ring, and write $X_\fC$ for the set of all $\R$-points $x$ of $\fC$, as in Definition \ref{cc2def5}. Write ${\cal T}_\fC$ for the topology on $X_\fC$ that has basis of open sets $U_c=\bigl\{x\in X_\fC:x(c)\ne 0\bigr\}$ for all~$c\in\fC$. For each $c\in\fC$ define a map $c_*:X_\fC\ra\R$ such that $c_*:x\mapsto x(c)$. 

For a morphism $\phi:\fC\ra\fD$ of $C^\iy$-rings, we can define $f_\phi:X_\fD\ra X_\fC$ by $f_\phi(x)=x\ci\phi$, which is continuous.
\end{dfn}

From \cite[Lem.~4.15]{Joyc9}, this definition implies that $\cT_\fC$ is the weakest topology on $X_\fC$ such that the $c_*:X_\fC\ra\R$ are continuous for all $c\in\fC$. Also $(X_\fC,\cT_\fC)$ is a regular, Hausdorff topological space.

\begin{dfn}
\label{cc2def16}
For a $C^\iy$-ring $\fC$, we will define the {\it spectrum\/} of $\fC$, written $\Spec\fC$. Here, $\Spec\fC$ is a $C^\iy$-ringed space $(X,\O_X)$, with $X$ the topological space $X_\fC$ from Definition \ref{cc2def15}. For open $U\subseteq X$, then $\O_{X}(U)$ is the set of functions $s:U\ra \coprod_{x\in U}\fC_x$, where we write $s_x$ for the image of $x$ under $s$, such that around each point $x\in U$ there is an open subset $x\in W\subseteq U$ and element $c\in\fC$ with $s_y=\pi_y(c)\in\fC_y$ for all $y\in W$. This is a $C^\iy$-ring with the operations $\Phi_f$ on $\O_{X}(U)$ defined using the operations $\Phi_f$ on $\fC_x$ for~$x\in U$. 

For $s\in \O_{X}(U)$, the restriction map of functions $s\mapsto s\vert_V$ for open $V\subseteq U\subseteq X$ is a morphism of $C^\iy$-rings, giving the restriction map $\rho_{UV}:\O_{X}(U)\ra\O_{X}(V)$. The stalk $\O_{X,x}$ at $x\in X$ is isomorphic to $\fC_x$, which is a local $C^\iy$-ring. Hence $(X,\O_{X})$ is a local $C^\iy$-ringed space.

For a morphism $\phi:\fC\ra\fD$ of $C^\iy$-rings, we have an induced morphism of local $C^\iy$-rings, $\phi_x:\fC_{f_\phi(x)}\ra\fD_x$. If we let $(X,\O_X)=\Spec\fC$, $(Y,\O_Y)=\Spec\fD$, then for open $U\subseteq X$ define $(f_\phi)_\sh(U):\O_{X}(U)\ra\O_{Y}(f_\phi^{-1}(U))$ by $(f_\phi)_\sh(U)s:x\mapsto \phi_x(s_{f_\phi(x)})$. This gives a morphism $(f_\phi)_\sh:\O_{X}\ra (f_\phi)_*(\O_{Y})$ of sheaves of $C^\iy$-rings on $X$. 
Then $\uf_\phi=(f_\phi,f_\phi^\sh): (Y,\O_{Y})\ra (X,\O_{X})$ is a morphism of local $C^\iy$-ringed spaces, where $f_\phi^\sh$ corresponds to $(f_\phi)_\sh$ under \eq{cc2eq8}. Then  $\Spec$ is a functor $\CRings^{\bf op}\ra\LCRS$, called the {\it spectrum functor}, where $\Spec\phi:\Spec\fD\ra\Spec\fC$ is defined by $\Spec\phi=\uf_\phi$. 
\end{dfn}

\begin{ex}
\label{cc2ex8}
For a manifold $X$ then $\Spec C^\iy(X)$ is isomorphic to the local $C^\iy$-ringed space $\uX$ constructed in Example \ref{cc2ex7}.
\end{ex}

Here is~\cite[Lem~4.28]{Joyc9}:

\begin{lem}
\label{cc2lem1}
Let\/ $\fC$ be a $C^\iy$-ring, set\/ $\uX=\Spec\fC=(X,\O_X),$ and let\/ $c\in\fC$. If we write $U_c=\{x\in X:x(c)\ne 0\}$ as in Definition\/ {\rm\ref{cc2def15},} then $U_c\subseteq X$ is open and\/~$\uU_c=(U_c,\O_X\vert_{U_c})\cong\Spec\fC(c^{-1})$. 
\end{lem}

\begin{dfn}
\label{cc2def17}
The {\it global sections functor\/} $\Ga:\LCRS\ra\CRings^{\bf op}$ takes $(X,\O_X)$ to $\O_X(X)$ and morphisms $(f,f^\sh):(X,\O_X)\ra(Y,\O_Y)$ to $\Ga:(f,f^\sh)\mapsto f_\sh(Y)$, for $f_\sh$ relating $f^\sh$ as in~\eq{cc2eq8}. 

For each $C^\iy$-ring $\fC$ we can define a morphism $\Xi_\fC:\fC\ra \Ga\ci\Spec\fC$. Here, for $c\in \fC$ then $\Xi_\fC(c):X_\fC\ra\coprod_{x\in X_\fC}\fC_x$ is defined by $\Xi_\fC(c)_x=\pi_x(c)\in\fC_x$, so $\Xi_\fC(c)\in\O_{X_\fC}(X_\fC)=\Ga\ci\Spec\fC$. This $\Xi_\fC$ is a $C^\iy$-ring morphism as it is built from $C^\iy$-ring morphisms $\pi_x:\fC\ra\fC_x$, and the $C^\iy$-operations on $\O_{X_\fC}(X_\fC)$ are defined pointwise in the $\fC_x$. This defines a natural transformation $\Xi:\Id_\CRings\Ra\Ga\ci\Spec$ of functors~$\CRings\ra\CRings$.
\end{dfn}

\begin{thm} 
\label{cc2thm2}
The functor\/ $\Spec:\CRings^{\bf op}\ra\LCRS$ is \begin{bfseries}right adjoint\end{bfseries} to $\Ga:\LCRS\ra\CRings^{\bf op}$. Here, $\Xi_\fC$ is the unit of the adjunction between $\Ga$ and\/ $\Spec$. This implies $\Spec$ preserves limits as in {\rm\cite[p.~687]{Dubu2}}. Hence if we have $C^\iy$-ring morphisms $\phi:\fF\ra\fD,$ $\psi:\fF\ra\fE$ in $\CRings$ then their pushout\/ $\fC=\fD\amalg_\fF\fE$ has image that is isomorphic to the fibre product\/~$\Spec\fC\cong\Spec\fD\t_{\Spec\fF}\Spec\fE$.
\end{thm}

We extend this theorem to $C^\iy$-schemes with corners in~\S\ref{cc53}.

\begin{rem}
\label{cc2rem2}
Our definition of spectrum functor follows \cite{Joyc9} and Dubuc \cite{Dubu2}, and is called the {\it Archimedean spectrum\/} in Moerdijk et al.\ \cite[\S 3]{MQR}. They also show it is a right adjoint to the global sections functor as above.
\end{rem}

\begin{dfn}
\label{cc2def18}
Objects $\uX\in\LCRS$ that are isomorphic to $\Spec\fC$ for some $\fC\in\CRings$ are called {\it affine\/ $C^\iy$-schemes}. Elements $\uX\in\LCRS$ that are locally isomorphic to $\Spec\fC$ for some $\fC\in\CRings$ (depending upon the open sets) are called $C^\iy$-{\it schemes\/}. We define $\CSch$ and $\ACSch$ to be the full subcategories of $C^\iy$-schemes and affine $C^\iy$-schemes in $\LCRS$ respectively.
\end{dfn}

\begin{rem}
\label{cc2rem3}
Unlike ordinary algebraic geometry, affine $C^\iy$-schemes are very general objects. All manifolds are affine, and all their fibre products are affine. But not all manifolds with corners are affine $C^\iy$-schemes with corners.
\end{rem}

\subsection{\texorpdfstring{Complete $C^\iy$-rings}{Complete C∞-rings}}
\label{cc25}

In ordinary algebraic geometry, if $A$ is a commutative ring then $\Ga\ci\Spec A\cong A$, and $\Spec:\mathop{\bf Rings}^{\bf op}\ra\mathop{\bf ASch}$ is an equivalence of categories, with inverse $\Ga$. For $C^\iy$-rings $\fC$, in general $\Ga\ci\Spec\fC\not\cong\fC$, and $\Spec:\CRings^{\bf op}\ra\ACSch$ is neither full nor faithful. But as in \cite[Prop.~4.34]{Joyc9}, we have:

\begin{prop}
\label{cc2prop8}
For each $C^\iy$-ring $\fC,$ $\Spec\Xi_\fC:\Spec\ci\Ga\ci\Spec\fC\ra\Spec\fC$ is an isomorphism in $\LCRS$.
\end{prop}

This motivates the following definition~\cite[Def.~4.35]{Joyc9}:

\begin{dfn}
\label{cc2def19}
A $C^\iy$-ring $\fC$ is called {\it complete\/} if $\Xi_\fC:\fC\ra\Ga\ci\Spec\fC$ is an isomorphism. We define $\CRingsco$ to be the full subcategory in $\CRings$ of complete $C^\iy$-rings. By Proposition \ref{cc2prop8} we see that $\CRingsco$ is equivalent to the image of the functor $\Ga\ci\Spec:\CRings\ra\CRings$, which gives a left adjoint to the inclusion of $\CRingsco$ into $\CRings$. Write this left adjoint as the functor~$\Pi_{\rm all}^{\rm co}=\Ga\ci\Spec:\CRings\ra\CRingsco$.
\end{dfn}

An example of a non-complete $C^\iy$-ring is the quotient $\fC=C^\iy(\R^n)/I_{\rm cs}$  of $C^\iy(\R^n)$ for $n>0$ by the ideal $I_{\rm cs}$ of compactly supported functions, and $\Pi_{\rm all}^{\rm co}(\fC)=0\not\cong\fC$. The next theorem comes from~\cite[Prop.~4.11 \& Th.~4.25]{Joyc9}.
 
\begin{thm}
\label{cc2thm3}
{\bf(a)} $\Spec\vert_{(\CRingsco)^{\bf op}}:(\CRingsco)^{\bf op}\ra\LCRS$ is full and faithful, and an equivalence\/~$\Spec\vert_{\cdots}:(\CRingsco)^{\bf op}\ra\ACSch$.
\smallskip

\noindent{\bf(b)} Let\/ $\uX$ be an affine $C^\iy$-scheme. Then $\uX\cong\Spec\O_X(X),$ where $\O_X(X)$ is a complete $C^\iy$-ring.
\smallskip

\noindent{\bf(c)} The functor $\Pi_{\rm all}^{\rm co}:\CRings\ra\CRingsco$ is left adjoint to the inclusion functor\/ $\inc:\CRingsco\hookra\CRings$. That is, $\Pi_{\rm all}^{\rm co}$ is a \begin{bfseries}reflection functor\end{bfseries}.

\smallskip

\noindent{\bf(d)} All small colimits exist in $\CRingsco,$ although they may not coincide with the corresponding small colimits in $\CRings$. 
\smallskip

\noindent{\bf(e)} $\Spec\vert_{(\CRingsco)^{\bf op}}=\Spec\ci\inc:(\CRingsco)^{\bf op}\ra\LCRS$ is right adjoint to $\Pi_{\rm all}^{\rm co}\ci\Ga:\LCRS\ra(\CRingsco)^{\bf op}$. Thus $\Spec\vert_{\cdots}$ takes limits in $(\CRingsco)^{\bf op}$ (equivalently, colimits in $\CRingsco$) to limits in~$\LCRS$.
\end{thm}

In the following theorem we summarize results found in Dubuc \cite[Th.~16]{Dubu2}, Moerdijk and Reyes \cite[\S~II.~Prop.~1.2]{MoRe1}, and the second author~\cite[Cor.~4.27]{Joyc9}.

\begin{thm}
\label{cc2thm4}
There is a full and faithful functor $F_\Man^\ACSch:\Man\ra\ACSch$ that takes a manifold\/ $X$ to the affine $C^\iy$-scheme $\uX=(X,\O_X),$ where $\O_X(U)=C^\iy(U)$ is the usual smooth functions on $U$. Here $(X,\O_X)\cong \Spec(C^\iy(X))$ and hence $\uX$ is affine. The functor $F_\Man^\ACSch$ sends transverse fibre products of manifolds to fibre products of\/ $C^\iy$-schemes. 
\end{thm}

\subsection{\texorpdfstring{Sheaves of $\O_X$-modules on $C^\iy$-ringed spaces}{Sheaves of Oᵪ-modules on C∞-ringed spaces}}
\label{cc26} 

This section follows \cite[\S 5.3]{Joyc9}. Our definition of $\O_X$-module is the usual definition of sheaf of modules on a ringed space as in Hartshorne \cite[\S II.5]{Hart} and Grothendieck \cite[\S 0.4.1]{Grot}, using the $\R$-algebra structure on our $C^\iy$-rings.

\begin{dfn}
\label{cc2def20}
For each $C^\iy$-ringed space $\uX=(X,\O_X)$ we define a category $\OXmod$. The objects are {\it sheaves of\/ $\O_X$-modules} (or simply $\O_X$-{\it modules}) $\cE$ on $X$. Here, $\cE$ is a functor on open sets $U\subseteq X$ such that $\cE:U\mapsto \cE(U)$ in $\O_X(U)\text{-mod}$ is a sheaf as in Definition \ref{cc2def9}. This means we have linear restriction maps $\cE_{UV}:\cE(U)\ra \cE(V)$ for each inclusion of open sets $V\subseteq U\subseteq X$, such that the following commutes
\begin{equation*}
\xymatrix@R=15pt@C=100pt{ *+[r]{\O_X(U)\t \cE(U)} \ar[d]^{\rho_{UV}\t\cE_{UV}}
\ar[r]_(0.6){} & *+[l]{\cE(U)}
\ar[d]_{\cE_{UV}} \\
*+[r]{\O_X(V)\t \cE(V)} \ar[r]^(0.6){} & *+[l]{\cE(V),\!} }
\end{equation*}
where the horizontal arrows are module multiplication. Morphisms in $\OXmod$ are sheaf morphisms $\phi:\cE\ra\cF$ commuting with the $\O_X$-actions. An $\O_X$-module $\cE$ is called a {\it vector bundle\/} if it is locally free, that is, around every point there is an open set $U\subseteq X$ with $\cE\vert_U\cong\O_X\vert_U\ot_\R\R^n$.
\end{dfn}

\begin{dfn}
\label{cc2def21}
We define the {\it pullback\/} $\uf^*(\cE)$ of a sheaf of modules $\cE$ on $\uY$ by a morphism $\uf=(f,f^\sh):\uX\ra\uY$ of $C^\iy$-ringed spaces as $\uf^*(\cE)=f^{-1}(\cE)
\ot_{f^{-1}(\O_Y)}\O_X$. Here $f^{-1}(\cE)$ is as in Definition
\ref{cc2def13}, so that $\uf^*(\cE)$ is a sheaf of modules on $\uX$.
Morphisms of $\O_Y$-modules $\phi:\cE\ra\cF$ give morphisms of $\O_X$-modules $\uf^*(\phi)=f^{-1}(\phi)\ot\id_{\O_X}:\uf^*(\cE)\ra\uf^*(\cF)$.
\end{dfn}

\begin{dfn}
\label{cc2def22}
Let $\uX=(X,\O_X)$ be a $C^\iy$-ringed space. Define a presheaf $\cP T^*\uX$ of $\O_X$-modules on $X$ such that $\cP T^*\uX(U)$ is the cotangent module $\Om_{\O_X(U)}$ of Definition \ref{cc2def8}, regarded as a module over the $C^\iy$-ring $\O_X(U)$. For open sets $V\subseteq U\subseteq X$ we have restriction morphisms $\Om_{\rho_{UV}}:\Om_{\O_X(U)}\ra\Om_{\O_X(V)}$ associated to the morphisms of $C^\iy$-rings $\rho_{UV}:\O_X(U)\ra\O_X(V)$ so that the following commutes:
\begin{equation*}
\xymatrix@R=15pt@C=110pt{ *+[r]{\O_X(U)\t \Om_{\O_X(U)}}
\ar[d]^{\rho_{UV}\t\Om_{\rho_{UV}}} \ar[r]_(0.6){\mu_{\O_X(U)}} &
*+[l]{\Om_{\O_X(U)}} \ar[d]_{\Om_{\rho_{UV}}} \\
*+[r]{\O_X(V)\t \Om_{\O_X(V)}} \ar[r]^(0.6){\mu_{\O_X(V)}} & *+[l]{\Om_{\O_X(V)}.\!} }
\end{equation*}
Definition \ref{cc2def8} implies $\Om_{\psi\ci\phi}=\Om_\psi\ci\Om_\phi$, so this is a well defined presheaf of $\O_X$-modules. The {\it cotangent sheaf\/ $T^*\uX$ of\/} $X$ is the sheafification of~$\cP T^*\uX$.

The universal property of sheafification shows that for open $U\subseteq X$ we have an isomorphism of
$\O_X\vert_U$-modules
\begin{equation*}
T^*\uU=T^*(U,\O_X\vert_U)\cong T^*\uX\vert_U.
\end{equation*}
For a morphism $\uf:\uX\ra\uY$ in $\CRS$ we have $\uf^*(T^*\uY)=f^{-1}(T^*\uY)
\ot_{f^{-1}(\O_Y)}\O_X$. The universal properties of sheafification imply that $\uf^*(T^*\uY)$ is the sheafification of the presheaf $\cP(\uf^*(T^*\uY))$, where
\begin{equation*}
U\longmapsto\cP(\uf^*(T^*\uY))(U)=
\ts\lim_{V\supseteq f(U)}\Om_{\O_Y(V)}\ot_{\O_Y(V)}\O_X(U).
\end{equation*}
This gives a presheaf morphism $\cP\Om_\uf:\cP(\uf^*(T^*\uY))\ra\cP T^*\uX$ on $X$, where
\begin{equation*}
(\cP\Om_\uf)(U)=\ts\lim_{V\supseteq f(U)}
(\Om_{\rho_{f^{-1}(V)\,U}\ci f_\sh(V)})_*.
\end{equation*}
Here, we have morphisms $f_\sh(V):\O_Y(V)\ra\O_X(f^{-1}(V))$ from $f_\sh:\O_Y\ra f_*(\O_X)$ corresponding to $f^\sh$ in $\uf$ as in \eq{cc2eq8}, and $\rho_{f^{-1}(V)\,U}:\O_X(f^{-1}(V))\ra\O_X(U)$ in $\O_X$ so that $(\Om_{\rho_{f^{-1}(V)\,U}\ci f_\sh(V)})_*:\Om_{\O_Y(V)} \ot_{\O_Y(V)}\O_X(U)\ra\Om_{\O_X(U)}=(\cP T^*\uX)(U)$ is constructed as in Definition \ref{cc2def8}. Then write $\Om_\uf:\uf^*(T^*\uY)\ra T^*\uX$ for the induced morphism of the associated sheaves. This corresponds to the morphism $\d f:f^*(T^*Y)\ra T^*X$ of vector bundles over a manifold $X$ and smooth map of manifolds $f:X\ra Y$ as in Example~\ref{cc2ex5}.
\end{dfn}

\subsection{\texorpdfstring{Sheaves of $\O_X$-modules on $C^\iy$-schemes}{Sheaves of Oᵪ-modules on C∞-schemes}}
\label{cc27}

We define a spectrum functor for modules, as in~\cite[Def.s 5.16, 5.17 \& 5.25]{Joyc9}.

\begin{dfn}
\label{cc2def23}
Let $\fC$ be a $C^\iy$-ring and set $\uX=(X,\O_X)=\Spec\fC$. Let $M\in\fCmod$ be a $\fC$-module. For each open subset $U\subseteq X$ there is a natural morphism $\fC\ra\O_X(U)$ in $\CRings$. Using this we make $M\ot_\fC\O_X(U)$ into an $\O_X(U)$-module. This assignment $U\mapsto M\ot_\fC\O_X(U)$ is naturally a presheaf $\cP\MSpec M$ of $\O_X$-modules. Define $\MSpec M\in\OXmod$ to be its sheafification. 

A morphism $\mu:M\ra N$ in $\fCmod$ induces $\O_X(U)$-module morphisms $M\ot_\fC\O_X(U)\ra N\ot_\fC\O_X(U)$ for all open $U\subseteq X$, and hence a presheaf morphism, which descends to a morphism $\MSpec\mu:\MSpec M\ra\MSpec N$ in $\OXmod$. This defines a functor $\MSpec:\fCmod\ra\OXmod$. It is an exact functor of abelian categories.

There is also a global sections functor $\Ga:\OXmod\ra\fCmod$ mapping $\Ga:\cE\mapsto\cE(X)$, where the $\O_X(X)$-module $\cE(X)$ is viewed as a $\fC$-module via the natural morphism~$\fC\ra\O_X(X)$.	

For any $M\in\fCmod$ there is a natural morphism $\Xi_M:M\ra\Ga\ci\MSpec M$ in $\fCmod$, by composing $M\ra M\ot_\fC\O_X(X)=\cP\MSpec M(X)$ with the sheafification morphism $\cP\MSpec M(X)\ra\MSpec M(X)=\Ga\ci\MSpec M$. Generalizing Definition \ref{cc2def19}, we call $M$ {\it complete\/} if $\Xi_M$ is an isomorphism. Write $\fCmodco\subseteq\fCmod$ for the full subcategory of complete $\fC$-modules.
\end{dfn}

Here is \cite[Th.s~5.19, Prop.~5.20, Th.~5.26, \& Prop.~5.31]{Joyc9}:

\begin{thm}
\label{cc2thm5}
{\bf(a)} In Definition\/ {\rm\ref{cc2def23},} $\MSpec:\fCmod\ra\OXmod$ is left adjoint to\/ $\Ga:\OXmod\ra\fCmod,$ generalizing Theorem\/~{\rm\ref{cc2thm2}}.
\smallskip

\noindent{\bf(b)} There is a natural isomorphism\/ $\MSpec\ci\Ga\Ra\Id_{\OXmod}$. This gives a natural isomorphism 
$\MSpec\ci\Ga\ci\MSpec\Ra\MSpec,$ generalizing Proposition\/~{\rm\ref{cc2prop8}}.
\smallskip

\noindent{\bf(c)} $\MSpec\vert_{\fCmodco}:\fCmodco\ra\OXmod$ is an equivalence of categories, generalizing Theorem\/~{\rm\ref{cc2thm3}(a)}.
\smallskip

\noindent{\bf(d)} The functor $\Pi_{\rm all}^{\rm co}=\Ga\ci\MSpec:\fCmod\ra \fCmodco$ is left adjoint to the inclusion functor\/ $\inc:\fCmodco\hookra\fCmod,$ generalizing Theorem\/~{\rm\ref{cc2thm3}(c)}.
\smallskip

\noindent{\bf(e)} There is a natural isomorphism $T^*\uX\cong\MSpec\Om_\fC$ in $\OXmod$.
\end{thm}

\begin{rem}
\label{cc2rem4}
In \cite[\S 5.4]{Joyc9}, following conventional algebraic geometry as in Hartshorne \cite[\S II.5]{Hart}, we define a notion of {\it quasicoherent sheaf\/} $\cE$ on a $C^\iy$-scheme $\uX$, which is that we may cover $\uX$ by open $\uU\subseteq\uX$ with $\uU\cong\Spec\fC$ and $\cE\vert_\uU\cong\MSpec M$ for $\fC\in\CRings$ and $M\in\fCmod$. But then \cite[Cor.~5.22]{Joyc9} uses Theorem \ref{cc2thm5}(c) to show that every $\O_X$-module is quasicoherent, that is, $\qcoh(\uX)=\OXmod$, which is not true in conventional algebraic geometry. So here we will not bother with the language of quasicoherent sheaves.	
\end{rem}

Here is \cite[Th.~5.32]{Joyc9}, where part (b) is deduced from Theorem~\ref{cc2thm1}:

\begin{thm}
\label{cc2thm6}
{\bf(a)} Let\/ $\uf:\uX\ra\uY$ and\/ $\ug:\uY\ra\uZ$ be
morphisms of\/ $C^\iy$-schemes. Then in $\OXmod$ we have
\begin{equation*}
\Om_{\ug\ci\uf}=\Om_\uf\ci \uf^*(\Om_\ug):(\ug\ci\uf)^*(T^*\uZ)\longra T^*\uX.
\end{equation*}
 
\noindent{\bf(b)} Suppose we are given a Cartesian square in\/~{\rm $\CSch$:}
\begin{equation*}
\xymatrix@C=80pt@R=14pt{ *+[r]{\uW} \ar[r]_\uf \ar[d]^\ue & *+[l]{\uY} \ar[d]_\uh \\
*+[r]{\uX} \ar[r]^\ug & *+[l]{\uZ,\!} }
\end{equation*}
so that\/ $\uW=\uX\t_\uZ\uY$. Then the following is exact in $\OWmod\!:$
\begin{equation*}
\xymatrix@C=15pt{ (\ug\ci\ue)^*(T^*\uZ)
\ar[rrrr]^(0.44){\ue^*(\Om_\ug)\op -\uf^*(\Om_\uh)} &&&&
\ue^*(T^*\uX)\op\uf^*(T^*\uY) \ar[rr]^(0.65){\Om_\ue\op \Om_\uf}
&& T^*\uW \ar[r] & 0.}
\end{equation*}
\end{thm}

\section{Background on manifolds with (g-)corners}
\label{cc3}

Next we discuss {\it manifolds with corners}, following Melrose \cite{Melr1,Melr2,Melr3} and the second author \cite{Joyc1}, \cite[\S 2]{Joyc6}, and a generalization of them, {\it manifolds with g-corners}, introduced in \cite{Joyc6}. There is more than one notion of smooth map between manifolds with corners. The one we choose was introduced by Melrose \cite{Melr1,Melr2,Melr3}, who calls them {\it b-maps}.

\subsection{Manifolds with corners}
\label{cc31}

We define manifolds with corners, following the second author \cite[\S 2]{Joyc6}. 

\begin{dfn}
\label{cc3def1}
Use the notation $\R^m_k=[0,\iy)^k\t\R^{m-k}$
for $0\le k\le m$, and write points of $\R^m_k$ as $u=(u_1,\ldots,u_m)$ for $u_1,\ldots,u_k\in[0,\iy)$, $u_{k+1},\ldots,u_m\in\R$. Let $U\subseteq\R^m_k$ and $V\subseteq \R^n_l$ be open, and $f=(f_1,\ldots,f_n):U\ra V$ be a continuous map, so that $f_j=f_j(u_1,\ldots,u_m)$ maps $U\ra[0,\iy)$ for $j=1,\ldots,l$ and $U\ra\R$ for $j=l+1,\ldots,n$. Then we say:
\begin{itemize}
\setlength{\itemsep}{0pt}
\setlength{\parsep}{0pt}
\item[(a)] $f$ is {\it weakly smooth\/} if all derivatives $\frac{\pd^{a_1+\cdots+a_m}}{\pd u_1^{a_1}\cdots\pd u_m^{a_m}}f_j(u_1,\ldots,u_m):U\ra\R$ exist and are continuous for all $j=1,\ldots,n$ and $a_1,\ldots,a_m\ge 0$, including one-sided derivatives where $u_i=0$ for $i=1,\ldots,k$.

By Seeley's Extension Theorem, this is equivalent to requiring $f_j$ to extend to a smooth function $f_j':U'\ra\R$ on open neighbourhood $U'$ of $U$ in~$\R^m$.
\item[(b)] $f$ is {\it smooth\/} if it is weakly smooth and every $u=(u_1,\ldots,u_m)\in U$ has an open neighbourhood $\ti U$ in $U$ such that for each $j=1,\ldots,l$, either:
\begin{itemize}
\setlength{\itemsep}{0pt}
\setlength{\parsep}{0pt}
\item[(i)] we may uniquely write $f_j(\ti u_1,\ldots,\ti u_m)=F_j(\ti u_1,\ldots,\ti u_m)\cdot\ti u_1^{a_{1,j}}\cdots\ti u_k^{a_{k,j}}$ for all $(\ti u_1,\ldots,\ti u_m)\in\ti U$, where $F_j:\ti U\ra(0,\iy)$ is weakly smooth and $a_{1,j},\ldots,a_{k,j}\in\N=\{0,1,2,\ldots\}$, with $a_{i,j}=0$ if $u_i\ne 0$; or 
\item[(ii)] $f_j\vert_{\smash{\ti U}}=0$.
\end{itemize}
\item[(c)] $f$ is {\it interior\/} if it is smooth, and case (b)(ii) does not occur.
\item[(d)] $f$ is {\it strongly smooth\/} if it is smooth, and in case (b)(i), for each $j=1,\ldots,l$ we have $a_{i,j}=1$ for at most one $i=1,\ldots,k$, and $a_{i,j}=0$ otherwise. 
\item[(e)] $f$ is a {\it diffeomorphism} if it is a smooth bijection with smooth inverse.
\end{itemize}
\end{dfn}

\begin{dfn}
\label{cc3def2}
Let $X$ be a second countable Hausdorff topological space. An {\it $m$-dimensional chart on\/} $X$ is a pair $(U,\phi)$, where $U\subseteq\R^m_k$ is open for some $0\le k\le m$, and $\phi:U\ra X$ is a
homeomorphism with an open set~$\phi(U)\subseteq X$.

Let $(U,\phi),(V,\psi)$ be $m$-dimensional charts on $X$. We call
$(U,\phi)$ and $(V,\psi)$ {\it compatible\/} if
$\psi^{-1}\ci\phi:\phi^{-1}\bigl(\phi(U)\cap\psi(V)\bigr)\ra
\psi^{-1}\bigl(\phi(U)\cap\psi(V)\bigr)$ is a diffeomorphism between open subsets of $\R^m_k,\R^m_l$, in the sense of Definition~\ref{cc3def1}(e).

An $m$-{\it dimensional atlas\/} for $X$ is a system
$\{(U_a,\phi_a):a\in A\}$ of pairwise compatible $m$-dimensional
charts on $X$ with $X=\bigcup_{a\in A}\phi_a(U_a)$. We call such an
atlas {\it maximal\/} if it is not a proper subset of any other
atlas. Any atlas $\{(U_a,\phi_a):a\in A\}$ is contained in a unique
maximal atlas, the set of all charts $(U,\phi)$ of this type on $X$
which are compatible with $(U_a,\phi_a)$ for all~$a\in A$.

An $m$-{\it dimensional manifold with corners\/} is a second
countable Hausdorff topological space $X$ equipped with a maximal
$m$-dimensional atlas. Usually we refer to $X$ as the manifold,
leaving the atlas implicit, and by a {\it chart\/ $(U,\phi)$ on\/}
$X$, we mean an element of the maximal atlas.

Now let $X,Y$ be manifolds with corners of dimensions $m,n$, and $f:X\ra Y$ a continuous map. We call $f$ {\it weakly smooth}, or {\it smooth}, or {\it interior}, or {\it strongly smooth}, if whenever $(U,\phi),(V,\psi)$ are charts on $X,Y$ with $U\subseteq\R^m_k$, $V\subseteq\R^n_l$ open, then $\psi^{-1}\ci f\ci\phi:(f\ci\phi)^{-1}(\psi(V))\ra V$ is weakly smooth, or smooth, or interior, or strongly smooth, respectively, as maps between open subsets of $\R^m_k,\R^n_l$ in the sense of Definition \ref{cc3def1}. We call $f:X\ra Y$ a {\it diffeomorphism\/} if it is a bijection and $f,f^{-1}$ are smooth.

Write $\Mancin,\Mancst\subset\Manc\subset\Mancwe$ for the categories with objects manifolds with corners, and morphisms interior maps, and strongly smooth maps, and smooth maps, and weakly smooth maps, respectively. We will be interested almost exclusively in the categories~$\Mancin\subset\Manc$.

We also write $\cManc$ for the category with objects disjoint unions $\coprod_{n=0}^\iy X_n$, where $X_n$ is a manifold with corners of dimension $n$ (we call these {\it manifolds with corners of mixed dimension}), allowing $X_n=\es$, and morphisms continuous maps $f:\coprod_{m=0}^\iy X_m\ra\coprod_{n=0}^\iy Y_n$, such that $f_{mn}:=f\vert_{X_m\cap f^{-1}(Y_n)}:X_m\cap f^{-1}(Y_n)\ra Y_n$ is a smooth map of manifolds with corners for all $m,n\ge 0$. We write $\cMancin\subset\cManc$ for the category with the same objects, and with morphisms $f$ such that $f_{mn}$ is interior for all $m,n\ge 0$. There are obvious full and faithful embeddings $\Manc\subset\cManc$, $\Mancin\subset\cMancin$.
\end{dfn}

\begin{rem}
\label{cc3rem1}
Some references on manifolds with corners are Cerf \cite{Cerf}, Douady \cite{Doua}, Gillam and Molcho \cite[\S 6.7]{GiMo}, Kottke and Melrose \cite{KoMe}, Margalef-Roig and Outerelo Dominguez \cite{MaOu}, Melrose \cite{Melr1,Melr2,Melr3}, Monthubert \cite{Mont}, and the second author \cite{Joyc1,Joyc6}. Just as objects, without considering morphisms, most authors define manifolds with corners $X$ as in Definition \ref{cc3def2}. However, Melrose \cite{Melr1,Melr2,Melr3} and authors who follow him impose an extra condition, that $X$ should be a {\it manifold with faces\/} in the sense of Definition \ref{cc3def8} below.

There is no general agreement in the literature on how to define smooth maps, or morphisms, of manifolds with corners: 
\begin{itemize}
\setlength{\itemsep}{0pt}
\setlength{\parsep}{0pt}
\item[(i)] Our `smooth maps' in Definitions \ref{cc3def1}--\ref{cc3def2} are due to Melrose \cite[\S 1.12]{Melr2}, \cite[\S 1]{KoMe}, who calls them {\it b-maps}. `Interior maps' are also due to Melrose.
\item[(ii)] The author \cite{Joyc1} defined and studied `strongly smooth maps' above (which were just called `smooth maps' in \cite{Joyc1}). 
\item[(iii)] Gillam and Molcho's {\it morphisms of manifolds with corners\/} \cite[\S 6.7]{GiMo} coincide with our `interior maps'.
\item[(iv)] Most other authors, such as Cerf \cite[\S I.1.2]{Cerf}, define smooth maps of manifolds with corners to be weakly smooth maps, in our notation.
\end{itemize}

We will base our theory of (interior) $C^\iy$-rings and $C^\iy$-schemes with corners on the categories $\Mancin\subset\Manc$. Section \ref{cc82} will discuss theories based on other categories of `manifolds with corners' including the category $\Manac$ of {\it manifolds with analytic corners}, or {\it manifolds with a-corners}, defined by the second author \cite{Joyc7}, which are rather different to the categories above.
\end{rem}

\begin{dfn}
\label{cc3def3}
Let $X$ be a manifold with corners (or a manifold with g-corners in \S\ref{cc33}). Smooth maps $g:X\ra[0,\iy)$ will be called {\it exterior maps}, to contrast them with interior maps. We write $C^\iy(X)$ for the set of smooth maps $f:X\ra\R$, and $\In(X)$ for the set of interior maps $g:X\ra[0,\iy)$, and $\Ex(X)$ for the set of exterior maps $g:X\ra[0,\iy)$. Thus, we have three sets:
\begin{itemize}
\setlength{\itemsep}{0pt}
\setlength{\parsep}{0pt}
\item[(a)] $C^\iy(X)$ of smooth maps $f:X\ra\R$;
\item[(b)] $\In(X)$ of interior maps $g:X\ra[0,\iy)$; and 
\item[(c)] $\Ex(X)$ of exterior (smooth) maps $g:X\ra[0,\iy)$, with $\In(X)\subseteq\Ex(X)$.
\end{itemize}
Much of the book will be concerned with algebraic structures on these three sets, and generalizations to other spaces $X$. In Chapter \ref{cc4} we give $(C^\iy(X),\In(X)\amalg\{0\})$ and $(C^\iy(X),\Ex(X))$ the (large and complicated) structure of `(pre) $C^\iy$-rings with corners'. But a lot of the time, it will be enough that $C^\iy(X)$ is an $\R$-algebra, and $\In(X),\Ex(X)$ are monoids under multiplication, as in~\S\ref{cc32}.
\end{dfn}

\subsection{Monoids}
\label{cc32}

We will use monoids to define manifolds with g-corners in \S\ref{cc33}, and in the study of $C^\iy$-rings with corners in Chapter \ref{cc4}. Here we recall some facts about monoids, in the style of log geometry. A reference is Ogus~\cite[\S I]{Ogus}. 

\begin{dfn}
\label{cc3def4}
A (commutative) monoid is a set $P$ equipped with an associative commutative binary operation $+:P\t P\ra P$ that has an identity element $0$. All monoids in this book will be commutative. A morphism of monoids $P\ra Q$ is a morphism of sets that respects the binary operations and sends the identity to the identity. Write $\Mon$ for the category of monoids. For any $n\in \N$ and $p\in P$ we will write $np=n\cdot p={\buildrel{\ulcorner\,\,\,\text{$n$ copies } \,\,\,\urcorner} \over
{\vphantom{i}\smash{p+\cdots+p}}}$, and set~$0\cdot p=0$.

The above is {\it additive notation\/} for monoids. We will also very often use {\it multiplicative notation}, in which the binary operation is written $\cdot:P\t P\ra P$, thought of as multiplication, and the identity element is written 1, and we write $p^n$ rather than $np$, with~$p^0=1$.

If $P$ is a monoid written in multiplicative notation, a {\it zero element\/} is $0\in P$ with $0\cdot p=0$ for all $p\in P$. Zero elements need not exist, but are unique if they do. One should not confuse zero elements with identities.

The rest of this definition will use additive notation.

A {\it submonoid\/} $Q$ of a monoid $P$ is a subset that is closed under the binary operation and contains the identity element. We can form the {\it quotient monoid} $P/Q$ which is the set of all $\simc$-equivalence classes $[p]$ of $p\in P$ such that $p\simc p'$ if there are $q,q'\in Q$ with $p+q=p'+q'\in P$. It has an induced monoid structure from the monoid $P$. There is a morphism $\pi:P\ra P/Q$. This quotient satisfies the following universal property: it is a monoid $P/Q$ with a morphism $\pi:P\ra P/Q$ such that $\pi(Q)=\{0\}$ and if $\mu:P\ra R$ is a monoid morphism with $\mu(Q)=\{0\}$ then $\mu=\nu\ci \pi$ for a unique morphism~$\nu:P/Q\ra R$. 

A {\it unit\/} in $P$ is an element $p\in P$ that has a (necessarily unique) inverse under the binary operation, $p'$, so that $p'+p=0$. Write $P^\t$ for the set of all units of $P$. It is a submonoid of $P$, and an abelian group. A monoid $P$ is an abelian group if and only if~$P=P^\t$.

An {\it ideal\/} $I$ in a monoid $P$ is a nonempty proper subset $\es\ne I\subsetneq P$ such that if $p\in P$ and $i\in I$ then $i+p\in I$, so it is necessarily closed under $P$'s binary operation. It must not contain any units. An ideal $I$ is called {\it prime\/} if whenever $a+b\in I$ for $a,b\in P$ then either $a$ or $b$ is in $P$. We say the complement $P\setminus I$ of a prime ideal $I$ is a {\it face\/} which is automatically a submonoid of $P$.  If we have elements $p_j\in P$ for $j$ in some indexing set $J$ then we can consider the {\it ideal generated} by the $p_j$, which we write as $\langle p_j\rangle_{j\in J}$. It consists of all elements in $P$ of the form $a+p_j$ for any $a\in P$ and any $j\in J$. Note that if any of the $p_j$ are units then the `ideal' generated by these $p_j$ is a misnomer, as $\langle p_j\rangle_{j\in J}$ is not an ideal and instead equal to~$P$.

For any monoid $P$ there is an associated abelian group $P^\gp$ and morphism $\pi^\gp:P\ra P^\gp$. This has the universal property that any morphism from $P$ to an abelian group factors through $\pi^\gp$, so $ P^\gp$ is unique up to canonical isomorphism. It can be shown to be isomorphic to the quotient monoid $(P\t P)/\De_P$, where $\De_P=\{(p,p):p\in P\}$ is the diagonal submonoid of $P\t P$, and~$\pi^\gp:p\mapsto [p,0]$.

For a monoid $P$ we have the following properties:
\begin{itemize}
\setlength{\itemsep}{0pt}
\setlength{\parsep}{0pt}
\item[(i)] If there is a surjective morphism $\N^k\ra P$ for some $k\ge 0$, we call $P$ {\it finitely generated}. This morphism can be uniquely written as $(n_1,\ldots, n_k)\mapsto n_1p_1+\cdots n_kp_k$ for some $p_1,\ldots, p_k\in P$ which we call the {\it generators} of $P$. This implies $P^\gp$ is finitely generated. If there is an isomorphism $P\cong \N^A$ for some set $A$ (e.g. if $P\cong\N^k$ for $k\ge 0$) then $P$ is called {\it free}.
\item[(ii)] If $P^\t=\{0\}$ we call $P$ {\it sharp}. Any monoid has an associated {\it sharpening\/} $P^\sh$ which is the sharp quotient monoid $P/P^\t$ with surjection~$\pi^\sh:P\ra P^\sh$.
\item[(iii)] If $\pi^\gp:P\ra P^\gp$ is injective we call $P$ {\it integral\/} or {\it cancellative}. This occurs if and only if $p+p'=p+p''$ implies $p'=p''$ for all $p,p',p''\in P$. Then $P$ is isomorphic to its image under $\pi^\gp$, so we consider it a subset of~$P^\gp$.
\item[(iv)] If $P$ is integral and whenever $p\in P^\gp$ with $np\in P\subset P^\gp$ for some $n\ge 1$ implies $p\in P$ then we call $P$ {\it saturated}.
\item[(v)] If $P^\gp$ is a torsion free group, then we call $P$ {\it torsion free}. That is, if there is $n\ge 0$ and $p\in P^\gp$ such that $np=0$ then $p=0$.
\item[(vi)] If $P$ is finitely generated, integral, saturated and torsion free then it is called {\it weakly toric}. It has {\it rank} $\rank P=\dim_\R(P\ot_\N\R)$. For a weakly toric $P$ there is an isomorphism $P^\t\cong\Z^l$ and $P^\sh$ is a toric monoid (defined below). The exact sequence $0\ra P^\t\ra P\ra P^\sh\ra 0$ splits, so that $P\cong P^\sh\t\Z^l$. Then the rank of $P$ is equal to $\rank P=\rank P^\gp=\rank P^\sh+l$. 
\item[(vii)] If $P$ is a weakly toric monoid and is also sharp we call $P$ {\it toric} (note that saturated and sharp together imply torsion free). For a toric monoid $P$ its associated group $P^\gp$ is a finitely generated, torsion-free abelian group, so $P^\gp\cong\Z^k$ for $k\ge 0$. Then the rank of $P$ is $\rank P=k$.
\end{itemize}
Parts (vi)--(vii) are not universally agreed in the literature. For example, Ogus \cite[p.~13]{Ogus}, calls our weakly toric monoids {\it toric monoids}, and our toric monoids {\it sharp toric monoids}.
\end{dfn}

\begin{ex}
\label{cc3ex1}
{\bf(a)} The most basic toric monoid is $\N^k$ under addition for $k=0,1,\ldots,$ with $(\N^k)^\gp\cong\Z^k$.
\smallskip

\noindent{\bf(b)} $\Z^k$ under addition for $k>0$ is weakly toric, but not toric.\smallskip

\noindent{\bf(c)} $\bigl([0,\iy),\cdot,1\bigr)$ under multiplication is a monoid that is not finitely generated. It has identity $1$ and zero element $0$. We have $[0,\iy)^\gp=\{1\}$, so $[0,\iy)$ is not integral, and $[0,\iy)^\t=(0,\iy)$, so $[0,\iy)$ is not sharp.
\end{ex}

\subsection{Manifolds with g-corners}
\label{cc33}

The second author \cite{Joyc6} defined the category $\Mangc$ of {\it manifolds with generalized corners}, or {\it manifolds with g-corners}. These extend manifolds with corners $X$ in \S\ref{cc31}, but rather than being locally modelled on $\R^n_k=[0,\iy)^k\t\R^{n-k}$, they are modelled on spaces $X_P$ for $P$ a weakly toric monoid in~\S\ref{cc32}.

\begin{dfn}
\label{cc3def5}
Let $P$ be a weakly toric monoid. As in \cite[\S 3.2]{Joyc6} we define $X_P=\Hom(P,[0,\iy))$ to be the set of monoid morphisms $x:P\ra[0,\iy)$, where the target is considered as a monoid under multiplication as in Example \ref{cc3ex1}(c). The {\it interior\/} of $X_P$ is defined to be $X_P^\ci=\Hom(P,(0,\iy))$ where $(0,\iy)$ is a submonoid of $[0,\iy)$, so that~$X_P^\ci\subset X_P$.

For $p\in P$ there is a corresponding function $\la_p:X_P\ra[0,\iy)$ such that $\la_p(x)=x(p)$. For any $p,q\in P$ then $\la_{p+q}=\la_p\cdot \la_q$ and $\la_0=1$. Define a topology on $X_P$ to be the weakest topology such that each $\la_p$ is continuous. Then $X_P$ is locally compact and Hausdorff and $X_P^\ci$ is an open subset of $X_P$. The {\it interior} $U^\ci$ of an open set $U\subset X_P$ is defined to be $U\cap X_P^\ci$. 

As $P$ is weakly toric, then we can take a presentation for $P$ with generators $p_1,\ldots, p_m$ and relations 
\begin{equation*}
a_1^jp_1+\cdots+a_m^jp_m=b_1^jp_1+\cdots+b_m^jp_m\quad\text{in $P$ for $j=1,\ldots,k,$}
\end{equation*}
for $a_i^j,b_i^j\in\N$, $i=1,\ldots,m$, $j=1,\ldots,k$. Then we have a continuous function $\la_{p_1}\t \cdots \t \la_{p_m}:X_P\ra [0,\iy)^m$ that is a homeomorphism onto its image \begin{equation*}X_P'=\bigl\{(x_1,\ldots,x_m)\in[0,\iy)^m:x_1^{a_1^j}\cdots x_m^{a_m^j}=x_1^{b_1^j}\cdots x_m^{b_m^j},\; j=1,\ldots,k\bigr\},\end{equation*} which is closed subset of $[0,\iy)^m$. 

Let $U$ be an open subset of $X_P$, and $U'=\la_{p_1}\t \cdots \t \la_{p_m}(U)\subset X_P'$. Then we say a continuous function $f:U\ra \R$ or $f:U\ra [0,\iy)$ is {\it smooth\/} if there exists an open neighbourhood $W'$ of $U'$ in $[0,\iy)^m$ and a smooth function $g:W'\ra \R$ or $g:W'\ra[0,\iy)$ that is smooth in the sense above, such that $f=g\ci \la_{p_1}\t \cdots \t \la_{p_m}$. As in \cite[Prop.~3.14]{Joyc6}, this is independent of the choice of generators $p_1,\ldots,p_m$ for~$P$. 

Suppose $Q$ is another weakly toric monoid, and consider open $V\subseteq X_Q$ and a continuous function $f:U\ra V$. Then $f$ is {\it smooth\/} if $\la_q\ci f:U\ra [0,\iy)$ is smooth for all $q\in Q$ in the sense above. We call smooth $f$ {\it interior\/} if $f(U^\ci)\subseteq V^\ci$, and a {\it diffeomorphism\/} if it is bijective with smooth inverse. 
\end{dfn}

\begin{ex}
\label{cc3ex2}
If $P=\N^k\t \Z^{n-k}$ then $P$ is weakly toric. We can take generators $p_1=(1,0,\ldots, 0), p_2=(0,1,0,\ldots, 0), \ldots, p_n=(0,\ldots, 0,1), p_{n+1}=(0,\ldots, 0, -1, \ldots, -1)$ with $p_{n+1}$ having $-1$ in the $k+1$ to $n+1$ entries, so the only relation is $p_{k+1}+\cdots+p_{n+1}=0$. Then $X_P$ is homeomorphic to 
\begin{equation*}
X_P'=\bigl\{(x_1,\ldots,x_{n+1})\in[0,\iy)^{n+1}:x_{k+1}\cdots x_{n+1}=1\bigr\}.
\end{equation*}
This means that for $(x_1,\ldots,x_{n+1})\in X_P'$ we have $x_{k+1},\ldots, x_{n+1}>0$ with $x_{n+1}=x_{k+1}^{-1}\cdots x_n^{-1}$. So there is a homeomorphism from $X_P\ra\R^n_k$ mapping $(x_1,\ldots,x_{n+1})\mapsto (x_1,\ldots, x_k, \log(x_{k+1}),\ldots, \log(x_n))$. In \cite[Ex.~3.15]{Joyc6} we show that this identification $X_P\cong\R^n_k$ identifies the topology, and the notion of smooth maps to $\R,[0,\iy)$ and between open subsets of $X_P\cong\R^n_k$ and $X_Q\cong\R^m_l$, in Definition \ref{cc3def1} for $\R^n_k$ and above for $X_P$. Thus, the spaces $X_P$ generalize the spaces $\R^n_k$ used as local models for manifolds with corners.
\end{ex}

Following \cite[\S 3.3]{Joyc6} we define manifolds with g-corners:

\begin{dfn}
\label{cc3def6}
Let $X$ be a topological space. Define a {\it g-chart} on $X$ to be a triple $(P,U,\phi)$, where $P$ is a weakly toric monoid, $U\subset X_P$ is open and $\phi:U\ra X$ is a homeomorphism with an open subset $\phi(U)\subset X$. If $\rank P=n$ we call $(P,U,\phi)$ {\it $n$-dimensional}. For set theory reasons (to ensure a maximal atlas is a set not a class) we suppose $P$ is a submonoid of $\Z^k$ for some $k\ge 0$.

We call $n$-dimensional g-charts $(P,U,\phi)$ and $(Q,V,\psi)$ on $X$ {\it compatible\/} if $\psi^{-1}\ci\phi:\phi^{-1}\bigl(\phi(U)\cap\psi(V)\bigr)\ra \psi^{-1}\bigl(\phi(U)\cap\psi(V)\bigr)$ is a diffeomorphism between open subsets of $X_P$ and $X_Q$. A {\it g-atlas} on $X$ is a family $\cA=\bigl\{(P_i,U_i,\phi_i):i\in I\bigr\}$ of pairwise compatible g-charts $(P_i,U_i,\phi_i)$ on $X$ with the same dimension $n$, with $X=\bigcup_{i\in I}\phi_i(U_i)$. We call $\cA$ {\it maximal\/} if it is not a proper subset of any other g-atlas. We define a {\it manifold with g-corners\/} $(X,\cA)$ to be a Hausdorff, second countable topological space $X$ with a maximal g-atlas~$\cA$. 

We define smooth maps, and interior maps, $f:X\ra Y$ between manifolds with g-corners $X,Y$ as in Definition \ref{cc3def2}. We write $\Mangc$ for the category of manifolds with g-corners and smooth maps, and $\Mangcin\subset\Mangc$ for the subcategory of manifolds with g-corners and interior maps.

As in Example \ref{cc3ex2}, if $P\cong\N^k\t \Z^{n-k}$ we may identify $X_P\cong\R^n_k$. So as in \cite[Def.~3.22]{Joyc6}, we may identify $\Manc\subset\Mangc$ and $\Mancin\subset\Mangcin$ as full subcategories, where a manifold with g-corners $(X,\cA)$ is a manifold with corners if $\cA$ has a g-subatlas of g-charts $(P_i,U_i,\phi_i)$ with~$P_i\cong\N^k\t \Z^{n-k}$.

We also write $\cMangc$ for the category with objects disjoint unions $\coprod_{n=0}^\iy X_n$, where $X_n$ is a manifold with g-corners of dimension $n$ (we call these {\it manifolds with g-corners of mixed dimension}), allowing $X_n=\es$, and morphisms continuous maps $f:\coprod_{m=0}^\iy X_m\ra\coprod_{n=0}^\iy Y_n$, such that $f_{mn}:=f\vert_{X_m\cap f^{-1}(Y_n)}:X_m\cap f^{-1}(Y_n)\ra Y_n$ is a smooth map of manifolds with g-corners for all $m,n\ge 0$. We write $\cMancin\subset\cManc$ for the category with the same objects, and with morphisms $f$ such that $f_{mn}$ is interior for all $m,n\ge 0$. There are obvious full embeddings $\Mangc\subset\cMangc$ and~$\Mangcin\subset\cMangcin$.
\end{dfn}

\begin{rem}
\label{cc3rem2}
Any weakly toric monoid $P$ is isomorphic to $P^\sh\t\Z^l$, where $P^\sh$ is toric and $l\ge 0$. Then $X_P\cong X_{P^\sh}\t X_{\Z^l}\cong X_{P^\sh}\t\R^l$. Hence manifolds with g-corners have local models $X_Q\t\R^l$ for toric monoids $Q$ and $l\ge 0$, where $X_{\N^k}\cong[0,\iy)^k$. Each toric monoid $Q$ has a natural point $\de_0\in X_Q$ called the {\it vertex\/} of $X_Q$, which acts by taking $0\in Q$ to $1\in [0,\iy)$ and all non-zero $q\in Q$ to zero. Given a manifold with g-corners $X$ and a point $x\in X$, there is a toric monoid $Q$ such that $X$ near $x$ is modelled on $X_Q\t\R^l$ near $(\de_0,0)\in X_Q\t\R^l$, where~$\rank Q+l=\dim X$.
\end{rem}

From \cite[Ex.~3.23]{Joyc6} we have the simplest example of a manifold with g-corners that is not a manifold with corners. 

\begin{ex}
\label{cc3ex3}
Let $P$ be the weakly toric monoid of rank 3 with
\begin{equation*}
P=\bigl\{(a,b,c)\in \Z^3:a\ge 0,\; b\ge 0,\; a+b\ge c\ge 0\bigr\}.
\end{equation*}
This has generators  $p_1=(1,0,0)$, $p_2=(0,1,1)$, $p_3=(0,1,0)$, and $p_4=(1,0,1)$ and one relation $p_1+p_2=p_3+p_4$. The local model it induces is
\e
X_P\cong X_P'=\bigl\{(x_1,x_2,x_3,x_4)\in[0,\iy)^4:x_1x_2=x_3x_4\bigr\}.
\label{cc3eq1}
\e

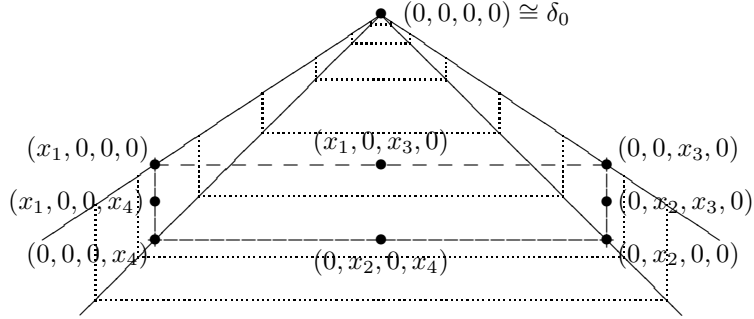
\begin{figure}[htb]
\centerline{$\splinetolerance{.8pt}
\begin{xy}
0;<1mm,0mm>:
,(0,0)*{\bu}
,(14,0)*{(0,0,0,0)\cong\de_0}
,(-30,-20)*{\bu}
,(-39,-18)*{(x_1,0,0,0)}
,(30,-20)*{\bu}
,(39.5,-18)*{(0,0,x_3,0)}
,(-30,-30)*{\bu}
,(-39,-32)*{(0,0,0,x_4)}
,(30,-30)*{\bu}
,(39.5,-32)*{(0,x_2,0,0)}
,(0,-30)*{\bu}
,(0,-33)*{(0,x_2,0,x_4)}
,(0,-20)*{\bu}
,(0,-17)*{(x_1,0,x_3,0)}
,(-30,-25)*{\bu}
,(-40.5,-25)*{(x_1,0,0,x_4)}
,(30,-25)*{\bu}
,(40.5,-25)*{(0,x_2,x_3,0)}
,(0,0);(-45,-30)**\crv{}
?(.8444)="aaa"
?(.85)="bbb"
?(.75)="ccc"
?(.65)="ddd"
?(.55)="eee"
?(.45)="fff"
?(.35)="ggg"
?(.25)="hhh"
?(.15)="iii"
?(.05)="jjj"
,(0,0);(45,-30)**\crv{}
?(.8444)="aaaa"
?(.85)="bbbb"
?(.75)="cccc"
?(.65)="dddd"
?(.55)="eeee"
?(.45)="ffff"
?(.35)="gggg"
?(.25)="hhhh"
?(.15)="iiii"
?(.05)="jjjj"
,(0,0);(-40,-40)**\crv{}
?(.95)="a"
?(.85)="b"
?(.75)="c"
?(.65)="d"
?(.55)="e"
?(.45)="f"
?(.35)="g"
?(.25)="h"
?(.15)="i"
?(.05)="j"
,(0,0);(40,-40)**\crv{}
?(.95)="aa"
?(.85)="bb"
?(.75)="cc"
?(.65)="dd"
?(.55)="ee"
?(.45)="ff"
?(.35)="gg"
?(.25)="hh"
?(.15)="ii"
?(.05)="jj"
,"a";"aa"**@{.}
,"b";"bb"**@{.}
,"c";"cc"**@{.}
,"d";"dd"**@{.}
,"e";"ee"**@{.}
,"f";"ff"**@{.}
,"g";"gg"**@{.}
,"h";"hh"**@{.}
,"i";"ii"**@{.}
,"j";"jj"**@{.}
,"a";"aaa"**@{.}
,"b";"bbb"**@{.}
,"c";"ccc"**@{.}
,"d";"ddd"**@{.}
,"e";"eee"**@{.}
,"f";"fff"**@{.}
,"g";"ggg"**@{.}
,"h";"hhh"**@{.}
,"i";"iii"**@{.}
,"j";"jjj"**@{.}
,"aa";"aaaa"**@{.}
,"bb";"bbbb"**@{.}
,"cc";"cccc"**@{.}
,"dd";"dddd"**@{.}
,"ee";"eeee"**@{.}
,"ff";"ffff"**@{.}
,"gg";"gggg"**@{.}
,"hh";"hhhh"**@{.}
,"ii";"iiii"**@{.}
,"jj";"jjjj"**@{.}
,(-30,-20);(30,-20)**@{--}
,(-30,-20);(-30,-30)**\crv{}
,(-30,-30);(30,-30)**\crv{}
,(30,-30);(30,-20)**\crv{}
\end{xy}$}
\caption{3-manifold with g-corners $X_P'\cong X_P$ in \eq{cc3eq1}}
\label{cc3fig1}
\end{figure}

Figure \ref{cc3fig1} is a sketch of $X_P'$ as a square-based 3-dimensional infinite pyramid. As in Remark \ref{cc3rem2}, $X_P'$ has vertex $(0,0,0,0)$ corresponding to $\de_0\in X_P$. It has 1-dimensional edges of points $(x_1,0,0,0),(0,x_2,0,0), \ab(0,0,x_3,0),(0,0,0,x_4)$, and 2-dimensional faces of points $(x_1,0,x_3,0),\! (x_1,0,0,x_4),\! \ab(0,x_2,x_3,0),\!\ab (0,\ab x_2,\ab 0,\ab x_4)$. Its interior $X_P^{\prime\ci}\ab\cong\R^3$ consists of points $(x_1,x_2,x_3,x_4)$ with $x_1,\ldots, x_4$ non-zero and $x_1x_2=x_3x_4$. Then $X_P\sm\{\de_0\}$ is a 3-manifold with corners, but $X_P$ is not a manifold with corners near $\de_0$, as it is not locally isomorphic to~$\R^n_k$. 
\end{ex}

We give more material on manifolds with g-corners in \S\ref{cc34}, \S\ref{cc35}, and~\S\ref{cc671}.

\subsection{Boundaries, corners, and the corner functor}
\label{cc34}

This section follows the second author \cite{Joyc1} and~\cite[\S 2.2 \& \S 3.4]{Joyc6}. 

\begin{dfn}
\label{cc3def7}
Let $X$ be a manifold with corners, or g-corners, of dimension $n$. For each $x\in X$ we define the {\it depth\/} $\depth_X(x)\in\{0,1,\ldots,n\}$ as follows: if $X$ is a manifold with corners then $\depth_X(x)=k$ if $X$ near $x$ is locally modelled on $\R^n_k$ near 0. And if $X$ is a manifold with g-corners then $\depth_X(x)=k$ if $X$ near $x$ is locally modelled on $X_P\t\R^{n-k}$ near $(\de_0,0)$, using Remark \ref{cc3rem2}, where $P$ is a toric monoid of rank $k$.

Write $S^k(X)=\{x\in X:\depth_X(x)=k\}$ for $k=0,\ldots,n$, the {\it codimension $k$ boundary stratum}. This defines a stratification $X=\coprod_{k=0}^nS^k(X)$ called the {\it depth stratification}. Then $S^k(X)$ has a natural structure of a manifold without boundary of dimension $n-k$. Its closure is $\ov{S^k(X)}=\coprod_{l=k}^nS^l(X)$. The {\it interior\/} of $X$ is~$X^\ci:=S^0(X)$.

Define a  {\it local $k$-corner component\/ $\ga$ of\/ $X$ at\/} $x\in X$ to be a local choice of connected component of $S^k(X)$ near $x$. That is, for any sufficiently small open neighbourhood $V$ of $x$ in $X$, then $\ga$ assigns a choice of connected component $W$ of $V\cap S^k(X)$ with $x\in\ov W$, and if $V',W'$ are alternative choices then $x\in \overline{W\cap W'}$. The local $1$-corner components are called {\it local boundary components\/} of~$X$.

As sets, we define the {\it boundary} and {\it $k$-corners} of $X$ for $k=0,\ldots,n$ by
\e
\begin{split}
\pd X&=\bigl\{(x,\be):\text{$x\in X$, $\be$ is a local boundary
component of $X$ at $x$}\bigr\},\\
C_k(X)&=\bigl\{(x,\ga):\text{$x\in X$, $\ga$ is a local $k$-corner 
component of $X$ at $x$}\bigr\},
\end{split}
\label{cc3eq2}
\e
for $k=0,1,\ldots, n$. This implies that $\pd X=C_1(X)$ and $C_0(X)\cong X$. In \cite[\S 2.2 \& \S 3.4]{Joyc6} we define natural structures of a manifold with corners, or g-corners (depending on which $X$ is) on $\pd X$ and $C_k(X)$, so that $\dim\pd X=n-1$ and $\dim C_k(X)=n-k$. The interiors are $(\pd X)^\ci\cong S^1(X)$ and $C_k(X)^\ci\cong S^k(X)$. We define smooth maps $\Pi_X:\pd X\ra X$, $\Pi_X:C_k(X)\ra X$ by $\Pi_X:(x,\be)\mapsto x$ and $\Pi_X:(x,\ga)\mapsto x$. These are generally neither interior, not injective.
\end{dfn}

The next definition will be important in Chapter \ref{cc5}, as a manifold with corners corresponds to an {\it affine\/} $C^\iy$-scheme with corners if it is a manifold with faces. We extend Definition \ref{cc3def8} to manifolds with g-corners in Definition~\ref{cc4def12}.

\begin{dfn}
\label{cc3def8}
A manifold with corners $X$ is called a {\it manifold with faces\/} if $\Pi_X\vert_F:F\ra X$ is injective for each connected component $F$ of $\pd X$. The {\it faces\/} of $X$ are the components of $\pd X$, regarded as subsets of $X$. Melrose \cite{Melr1,Melr2,Melr3} assumes his manifolds with corners are manifolds with faces. Write $\bf{Man^f_{in}}\subset\Mancin$, $\bf{Man^f}\subset\Manc$ for the full subcategories of manifolds with faces.
\end{dfn}

\begin{ex}
\label{cc3ex4}
Define the {\it teardrop\/} to be the subset $T=\bigl\{(x,y)\in\R^2:x\ge 0$, $y^2\le x^2-x^4\bigr\}$, as in  \cite[Ex.~2.8]{Joyc6}. As shown in Figure \ref{cc3fig2}, $T$ is a manifold with corners of dimension 2. The teardrop is not a manifold with faces as in Definition \ref{cc3def8}, since the boundary is diffeomorphic to $[0,1]$, and so is connected, but the map $\Pi_T:\pd T\ra T$ is not injective. 
\begin{figure}[htb]
\begin{xy}
,(-1.5,0)*{}
,<6cm,-1.5cm>;<6.7cm,-1.5cm>:
,(3,.3)*{x}
,(-1.2,2)*{y}
,(-1.5,0)*{\bullet}
,(-1.5,0); (1.5,0) **\crv{(-.5,1)&(.1,1.4)&(1.5,1.2)}
?(.06)="a"
?(.12)="b"
?(.2)="c"
?(.29)="d"
?(.4)="e"
?(.5)="f"
?(.6)="g"
?(.7)="h"
?(.83)="i"
,(-1.5,0); (1.5,0) **\crv{(-.5,-1)&(.1,-1.4)&(1.5,-1.2)}
?(.06)="j"
?(.12)="k"
?(.2)="l"
?(.29)="m"
?(.4)="n"
?(.5)="o"
?(.6)="p"
?(.7)="q"
?(.83)="r"
,"a";"j"**@{.}
,"b";"k"**@{.}
,"c";"l"**@{.}
,"d";"m"**@{.}
,"e";"n"**@{.}
,"f";"o"**@{.}
,"g";"p"**@{.}
,"h";"q"**@{.}
,"i";"r"**@{.}
\ar (-1.5,0);(3,0)
\ar (-1.5,0);(-3,0)
\ar (-1.5,0);(-1.5,2)
\ar (-1.5,0);(-1.5,-2)
\end{xy}
\caption{The teardrop $T$}
\label{cc3fig2}
\end{figure}
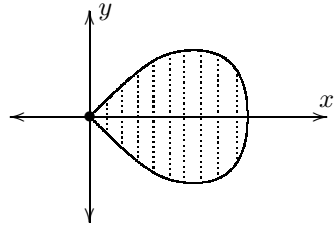
\end{ex}

\begin{rem}
\label{cc3rem3}
Let $X$ be a manifold with (g-)corners. Then smooth functions $f:X\ra\R$ are local objects: we can combine them using partitions of unity. In contrast, exterior functions $g:X\ra[0,\iy)$ as in Definition \ref{cc3def3} cannot be combined using partitions of unity, and have some non-local, global features.

To see this, observe that if $g:X\ra[0,\iy)$ is an exterior function, we may define a {\it multiplicity function\/} $\mu_g:\pd X\ra\N\cup\{\iy\}$ giving the order of vanishing of $g$ on the boundary faces of $X$. If $x\in X$ and $(x_1,\ldots,x_n)\in\R^n_k$ are local coordinates on $X$ near $x$ with $x=(0,\ldots,0)$ then either $g=x_1^{a_1}\cdots x_k^{a_k}F(x_1,\ldots,x_n)$ near $x$ for $F>0$ smooth, and we set $\mu_g(x,\{x_i=0\})=a_i$ for $i=1,\ldots,k$, or $g=0$ near $x$, when we write~$\mu_g(x,\{x_i=0\})=\iy$.

Then $\mu_g:\pd X\ra\N\cup\{\iy\}$ is locally constant, and hence constant on each connected component of $\pd X$. Also $g$ interior implies $\mu_g$ maps to $\N\subset\N\cup\{\iy\}$. If two points $x_1,x_2\in X$ lie in the image of the same connected component of $\pd X$ then the behaviour of $g$ near $x_1,x_2$ is linked even if $x_1,x_2$ are far away in $X$. Interior functions $\{g_i:i\in I\}$ can be combined with a partition of unity $\{\eta_i:i\in I\}$ to give interior $\sum_{i\in I}\eta_ig_i$ if the $\mu_{g_i}$ for $i\in I$ are all equal.
\end{rem}

The boundary $\pd X$ and corners $C_k(X)$ are not functorial under smooth or interior maps: smooth $f:X\ra Y$ generally do not lift to smooth $\pd f:\pd X\ra\pd Y$. However, as in the second author \cite[\S 4]{Joyc1} and \cite[\S 2.2 \& \S 3.4]{Joyc6}, the corners $C(X)=\coprod_{k=0}^{\dim X} C_k(X)$ do behave functorially.

\begin{dfn}
\label{cc3def9}
Let $X$ be a manifold with corners, or with g-corners. 
The {\it corners} of $X$ is the manifold with (g-)corners of mixed dimension
\begin{equation*}
C(X)=\ts\coprod_{k=0}^{\dim X} C_k(X),
\end{equation*}
as an object of $\cManc$ or $\cMangc$ in Definitions \ref{cc3def2} and \ref{cc3def6}. By \eq{cc3eq2} we have
\begin{equation*}
C(X)=\bigl\{(x,\ga):\text{$x\in X$, $\ga$ is a local $k$-corner 
component of $X$ at $x$, $k\ge 0$}\bigr\}.
\end{equation*}
Define morphisms $\Pi_X:C(X)\ra X$ and $\io_X:X\ra C(X)$ in $\cManc$ or $\cMangc$ by $\Pi_X:(x,\ga)\mapsto x$ and $\io_X:x\mapsto (x,X^\ci)$, that is, $\io_X:X\,{\buildrel\cong\over\longra}\,C^0(X)\subset C(X)$. Here $\io_X$ is interior, but $\Pi_X$ generally is not.

Now let $f:X\ra Y$ be a morphism in $\Manc$ or $\Mangc$. Suppose $\ga$ is a local $k$-corner component of $X$ at $x\in X$. For each small open neighbourhood $V$ of $x$ in $X$, $\ga$ gives a connected component $W$ of $V\cap S^k(X)$ with $x\in\overline W$. As $f$ preserves the stratifications $X=\coprod_{k\ge 0}S^k(X)$, $Y=\coprod_{l\ge 0}S^l(Y)$, we have $f(W)\subseteq S^l(Y)$ for some $l\ge 0$. Since $f$ is continuous, $f(W)$ is connected, and $f(x)\in\ov{f(W)}$. Thus there is a unique $l$-corner component $f_*(\ga)$ of $Y$ at $f(x)$, such that if $\ti V$ is a sufficiently small open neighbourhood of $f(x)$ in $Y$, then the connected component $\ti W$ of $\ti V\cap S^l(Y)$ given by $f_*(\ga)$ has $\ti W\cap f(W)\ne\es$. This $f_*(\ga)$ is independent of the choice of sufficiently small $V,\ti V$, so is well-defined.

Define a map $C(f):C(X)\ra C(Y)$ by $C(f):(x,\ga)\mapsto (f(x),f_*(\ga))$. 
As in \cite[\S 2.2 \& \S 3.4]{Joyc6}, this $C(f)$ is smooth, and interior, and so is a morphism in $\cMangcin$. Also $C(g\ci f)=C(g)\ci C(f)$ and $C(\id_X)=\id_{C(X)}$. Thus we have defined functors $C:\Manc\ra\cMancin$ and $C:\Mangc\ra\cMangcin$, which we call the {\it corner functors}. We extend them to $C:\cManc\ra\cMancin$ and $C:\cMangc\ra\cMangcin$ by $C(\coprod_{m\ge 0}X_m)=\coprod_{m\ge 0}C(X_m)$. 

Consider the inclusions of subcategories $\inc:\cMancin\hookra\cManc$ and $\inc:\cMangcin\hookra\cMangc$. The morphisms $\Pi_X:C(X)\ra X$ give a natural transformation $\Pi:\inc\ci C\Ra\Id$ in $\cManc$ or $\cMangc$. The morphisms $\io_X:X\ra C(X)$ give a natural transformation $\io:\Id\Ra C\ci\inc$ in $\cMancin$ or $\cMangcin$.
\end{dfn}

\begin{rem}
\label{cc3rem4}
If $X$ is an object in $\cManc$ or $\cMangc$, and $Y$ an object in $\cMancin$ or $\cMangcin$, it is easy to check that we have inverse 1-1 correspondences
\begin{equation*}
\xymatrix@C=80pt{ \Hom_{\begin{subarray}{l} \cManc \text{ or}\\ \cMangc  \end{subarray}}(\inc X,Y) \ar@<.5ex>[r]^{f\longmapsto C(f)\ci \io_X} & \Hom_{\begin{subarray}{l} \cMancin \text{ or}\\ \cMangcin  \end{subarray}}(X,C(Y)). \ar@<.5ex>[l]^{g\longmapsto \Pi_Y\ci g} }
\end{equation*}
Therefore $C:\cManc\ra\cMancin$ and $C:\cMangc\ra\cMangcin$ are {\it right adjoint\/} to $\inc:\cMancin\hookra\cManc$ and $\inc:\cMangcin\hookra\cMangc$, and $\Pi,\io$ are the units of the adjunction. (This is a new observation, not in~\cite{Joyc1,Joyc6}.)

This implies that the corner functors $C$, and hence the boundary $\pd X$ and corners $C_k(X)$ of a manifold with (g-)corners $X$, are not arbitrary constructions, but are determined up to natural isomorphism by the inclusions of subcategories $\Mancin\subset\Manc$ and~$\Mangcin\subset\Mangc$.
\end{rem}

\subsection{Tangent bundles and b-tangent bundles}
\label{cc35}

Finally we briefly discuss (co)tangent bundles of manifolds with (g-)corners. For a manifold with corners $X$ there are two kinds: the ({\it ordinary\/}) {\it tangent bundle\/} $TX$, the obvious generalization of tangent bundles of manifolds, and the {\it b-tangent bundle\/} ${}^bTX$ introduced by Melrose \cite[\S 2.2]{Melr2}, \cite[\S I.10]{Melr3}, \cite[\S 2]{Melr1}. The duals of the tangent bundle and b-tangent bundle are the {\it cotangent bundle\/} $T^*X$ and {\it b-cotangent bundle\/} ${}^bT^*X$. If $X$ is a manifold with g-corners, then ${}^bTX$ is well defined, but in general $TX$ is not. We follow~\cite[\S 2.3 \& \S 3.5]{Joyc6}. 

\begin{rem}
\label{cc3rem5}
{\bf(a)} Let $X$ be a manifold with corners of dimension $n$. The {\it tangent bundle\/} $TX\ra X$, {\it cotangent bundle\/} $T^*X\ra X$, {\it b-tangent bundle\/} ${}^bTX\ra X$, and {\it b-cotangent bundle\/} ${}^bT^*X\ra X$, are natural rank $n$ vector bundles on $X$, defined in detail in \cite[\S 2.3]{Joyc6}. In local coordinates, if $(x_1,\ldots,x_k,x_{k+1},\ldots,x_n)\in\R^n_k$ are local coordinates on an open set $\phi(U)\subset X$, from a chart $(U,\phi)$ with $U\subset\R^n_k$ open, then $TX,\ldots,{}^bT^*X$ are given on $\phi(U)$ by the bases of sections:
\begin{align*}
TX\vert_{\phi(U)}&=\Bigl\langle\frac{\pd}{\pd x_1},\ldots,\frac{\pd}{\pd x_k},\frac{\pd}{\pd x_{k+1}},\ldots,\frac{\pd}{\pd x_n}\Bigr\rangle_\R,\\
T^*X\vert_{\phi(U)}&=\bigl\langle \d x_1,\ldots,\d x_k,\d x_{k+1},\ldots,\d x_n\bigr\rangle_\R,\\
{}^bTX\vert_{\phi(U)}&=\Bigl\langle x_1\frac{\pd}{\pd x_1},\ldots,x_k\frac{\pd}{\pd x_k},\frac{\pd}{\pd x_{k+1}},\ldots,\frac{\pd}{\pd x_n}\Bigr\rangle_\R,\\
{}^bT^*X\vert_{\phi(U)}&=\bigl\langle x_1^{-1}\d x_1,\ldots,x_k^{-1}\d x_k,\d x_{k+1},\ldots,\d x_n\bigr\rangle_\R.
\end{align*}
Note that ``$x_i\frac{\pd}{\pd x_i}$'' is a nonzero section of ${}^bTX$  even where $x_i=0$ for $i=1,\ldots,k$, and ``$x_i^{-1}\d x_i$'' is a well defined section of ${}^bT^*X$ even where $x_i=0$. These are formal symbols, but are useful as they determine the transition functions for ${}^bTX,{}^bT^*X$ under change of coordinates $(x_1,\ldots,x_m)\rightsquigarrow(\ti x_1,\ldots,\ti x_m)$, etc.

There is a natural morphism of vector bundles $I_X:{}^bTX\ra TX$ mapping $x_i\frac{\pd}{\pd x_i}\mapsto x_i\cdot \frac{\pd}{\pd x_i}$ for $i=1,\ldots,k$ and $\frac{\pd}{\pd x_i}\mapsto \frac{\pd}{\pd x_i}$ for $i=k+1,\ldots,n$. It is an isomorphism over $X^\ci$, and has kernel of rank $k$ over~$S^k(X)$.
\smallskip

\noindent{\bf(b)} Tangent bundles $TX$ in $\Manc$ are functorial under smooth maps. That is, for any morphism $f:X\ra Y$ in $\Manc$ there is a corresponding morphism $Tf:TX\ra TY$ defined in \cite[Def.~2.14]{Joyc6}, functorial in $f$ and linear on the fibres of $TX,TY$, in a commuting diagram
\begin{equation*}
\xymatrix@C=80pt@R=15pt{ *+[r]{TX} \ar[d]^\pi \ar[r]_{Tf} &
*+[l]{TY} \ar[d]_\pi \\ *+[r]{X} \ar[r]^f & *+[l]{Y.\!} }
\end{equation*}
Equivalently, we may write $Tf$ as a vector bundle morphism $\d f:TX\ra f^*(TY)$ on $X$. It has a dual morphism $\d f^*:f^*(T^*Y)\ra T^*X$ on cotangent bundles.

Similarly, b-tangent bundles ${}^bTX$ in $\Manc$ are functorial under {\it interior\/} maps. That is, for any morphism $f:X\ra Y$ in $\Mancin$ there is a corresponding morphism ${}^bTf:{}^bTX\ra {}^bTY$ defined in \cite[Def.~2.15]{Joyc6}, functorial in $f$ and linear on the fibres of ${}^bTX,{}^bTY$, in a commuting diagram
\begin{equation*}
\xymatrix@C=30pt@R=15pt{ {{}^bTX} \ar[ddr]_(0.6)\pi \ar[dr]^{I_X}
\ar[rrr]_{{}^bTf} &&&
{{}^bTY} \ar[dr]^{I_Y} \ar[ddr]_(0.6)\pi \\
& {TX} \ar[d]^\pi \ar[rrr]^{Tf} &&&
{TY} \ar[d]^\pi \\
& {X} \ar[rrr]^f &&& {Y.\!} }
\end{equation*}
Equivalently, we have a vector bundle morphism ${}^b\d f:{}^bTX\ra f^*({}^bTY)$ on $X$, with a dual morphism ${}^b\d f^*:f^*({}^bT^*Y)\ra {}^bT^*X$ on b-cotangent bundles.
\smallskip

\noindent{\bf(c)} As in \cite[\S 3.5]{Joyc6}, tangent bundles $TX\ra X$ do not extend to manifolds with g-corners $X$. That is, for each $x\in X$ we can define a tangent space $T_xX$ with the usual functorial properties of tangent spaces, but $\dim T_xX$ need not be locally constant on $X$, so $T_xX$, $x\in X$ are not the fibres of a vector bundle.

However, b-(co)tangent bundles ${}^bTX\ra X$, ${}^bT^*X\ra X$ do extend nicely to manifolds with g-corners. If $X$ is locally modelled on $X_P$ for a weakly toric monoid $P$, as in \S\ref{cc33}, then ${}^bTX$ is locally trivial with fibre $\Hom_\Mon(P,\R)$, and ${}^bT^*X$ is locally trivial with fibre $P\ot_\N\R$. B-tangent bundles are functorial under interior maps, as in {\bf(b)}. So if $f:X\ra Y$ is a morphism in $\Mangcin$, as in \cite[Def.~3.43]{Joyc6} we have a morphism ${}^bTf:{}^bTX\ra {}^bTY$, and vector bundle morphisms ${}^b\d f:{}^bTX\ra f^*({}^bTY)$ and ${}^b\d f^*:f^*({}^bT^*Y)\ra {}^bT^*X$ on~$X$.
\smallskip

\noindent{\bf(d)} Let $X$ be a manifold with (g-)corners. Then \cite[\S 3.6]{Joyc6} constructs a natural exact sequence of vector bundles of mixed rank on $C(X)$:
\e
\smash{\xymatrix@C=18pt{ 0 \ar[r] & {}^bN_{C(X)} \ar[rr]^(0.45){{}^bi_T} && \Pi_X^*({}^bTX) \ar[rr]^(0.5){{}^b\pi_T} && {}^bT(C(X)) \ar[r] & 0. }}
\label{cc3eq3}
\e
Here ${}^bN_{C(X)}$ is called the {\it b-normal bundle\/} of $C(X)$ in $X$. It has rank $k$ on $C_k(X)$ for $k\ge 0$. Note that ${}^b\pi_T$ is {\it not\/} ${}^b\d\Pi_X$ from {\bf(c)} for $\Pi_X:C(X)\ra X$; ${}^b\d\Pi_X$ is undefined as $\Pi_X$ is not interior, and would go the opposite way to~${}^b\pi_T$.

Here is the dual complex to \eq{cc3eq3}, with ${}^bN^*_{C(X)}$ the {\it b-conormal bundle:}
\e
\smash{\xymatrix@C=18pt{ 0 \ar[r] & {}^bT^*(C(X)) \ar[rr]^(0.45){{}^b\pi^*_T} && \Pi_X^*({}^bT^*X) \ar[rr]^(0.5){{}^bi^*_T} && {}^bN^*_{C(X)} \ar[r] & 0. }}
\label{cc3eq4}
\e
In \cite[\S 3.6]{Joyc6} we define the {\it monoid bundle\/} $M_{C(X)}\ra C(X)$, which is a locally constant bundle of toric monoids with an embedding $M_{C(X)}\subset{}^bN_{C(X)}$ such that ${}^bN_{C(X)}\cong M_{C(X)}\ot_\N\R$. The monoid fibres have rank $k$ on $C_k(X)$. If $X$ is a manifold with corners then $M_{C(X)}$ has fibre $\N^k$ on $C_k(X)$, and the embedding $M_{C(X)}\subset{}^bN_{C(X)}$ is locally modelled on $\N^k\subset\R^k$. We also have a dual {\it comonoid bundle\/} $M_{C(X)}^\vee\ra C(X)$, with an embedding $M_{C(X)}^\vee\subset{}^bN^*_{C(X)}$.

If $(x,\ga)\in C(X)$, then $X$ near $x$ is locally diffeomorphic to $X_P\t Y$ for $P$ a toric monoid and $Y$ a manifold with g-corners, such that $\ga$ is identified with $\{\de_0\}\t Y$ for $\de_0$ the vertex in $X_P$. Then near $(x,\ga)$ in $C(X)$, the fibres of ${}^bN_{C(X)},{}^bN^*_{C(X)},M_{C(X)},M_{C(X)}^\vee$ are $\Hom_\Mon(P,\R),P\ot_\N\R,P^\vee,P$, respectively.
\end{rem}

\begin{ex}
\label{cc3ex5}
Let $X$ be a manifold with corners, and $(x,\ga)\in C_k(X)^\ci\subset C(X)$. We will explain the ideas of Remark \ref{cc3rem5}(d) near $(x,\ga)$. We can choose local coordinates $(x_1,\ldots,x_n)\in \R^n_k$ near $x$ in $X$, such that the local boundary component $\ga$ of $X$ at $x$ is $x_1=\cdots=x_k=0$ in coordinates. Then $(x_{k+1},\ldots,x_n)$ in $\R^{n-k}$ are local coordinates on $C_k(X)^\ci$ near $(x,\ga$), and \eq{cc3eq3}--\eq{cc3eq4} become
\begin{gather*}
\text{\begin{footnotesize}$\displaystyle
\xymatrix@!0@C=18pt@R=25pt{ *+[r]{0} \ar[rrr] &&& {}^bN_{C(X)} \ar@{=}[d] \ar[rrrrrr]_(0.45){{}^bi_T} &&&&&& \Pi_X^*({}^bTX) \ar@{=}[d] \ar[rrrrrr]_(0.5){{}^b\pi_T} &&&&&& {}^bT(C(X)) \ar@{=}[d] \ar[rrr] &&& *+[l]{0}  \\
*+[r]{\bigl\langle x_1\frac{\pd}{\pd x_1},\ldots,x_k\frac{\pd}{\pd x_k}\bigr\rangle_{\sst\R}} \ar[rrrrrrrrr] &&& {\vphantom{\bigl\langle}} &&&&&& \bigl\langle x_1\frac{\pd}{\pd x_1},\ldots,x_k\frac{\pd}{\pd x_k},\frac{\pd}{\pd x_{k+1}},\ldots,\frac{\pd}{\pd x_n}\bigr\rangle_{\sst\R} \ar[rrrrrrrrr] &&&&&& {\vphantom{\bigl\langle}} &&& *+[l]{\bigl\langle\frac{\pd}{\pd x_{k+1}},\ldots,\frac{\pd}{\pd x_n}\bigr\rangle_{\sst\R},\!\!}  }$\end{footnotesize}}\\
\text{\begin{footnotesize}$\displaystyle
\xymatrix@!0@C=9pt@R=25pt{ *+[r]{0} \ar[rrrrrr] &&&&&& {}^bT^*(C(X)) \ar@{=}[d]\ar[rrrrrrrrrrrr]_(0.45){{}^b\pi^*_T} &&&&&&&&&&&& \Pi_X^*({}^bT^*X) \ar@{=}[d]\ar[rrrrrrrrrrrr]_(0.5){{}^bi^*_T} &&&&&&&&&&&& {}^bN^*_{C(X)} \ar@{=}[d] \ar[rrrrrr] &&&&&& *+[l]{0}  \\
*+[r]{\bigl\langle \d x_{k+1},\ldots,\d x_n\bigr\rangle_{\sst\R}\!\!} \ar[rrrrrrrrrrrrrrrrr] &&&&&& {\vphantom{\bigl\langle}} &&&&&&&&&&& {\bigl\langle x_1^{-1}\d x_1,\ldots,x_k^{-1}\d x_k,\d x_{k+1},\ldots,\d x_n\bigr\rangle_{\sst\R}\!\!} \ar[rrrrrrrrrrrrrrrrrrr] & {\vphantom{\bigl\langle}} &&&&&&&&&&&& {\vphantom{\bigl\langle}} &&&&&& *+[l]{\bigl\langle x_1^{-1}\d x_1,\ldots,x_k^{-1}\d x_k\bigr\rangle_{\sst\R}.\!\!}  }$\end{footnotesize}}
\end{gather*}
Under these identifications we have
\begin{equation*}
M_{C(X)}=\ts\bigl\langle x_1\frac{\pd}{\pd x_1},\ldots,x_k\frac{\pd}{\pd x_k}\bigr\rangle_{\N},\quad
M_{C(X)}^\vee=\bigl\langle x_1^{-1}\d x_1,\ldots,x_k^{-1}\d x_k\bigr\rangle_{\N}.
\end{equation*}
\end{ex}

\section{\texorpdfstring{(Pre) $C^\iy$-rings with corners}{(Pre) C∞-rings with corners}}
\label{cc4}

Sections \ref{cc41}--\ref{cc43} study the most obvious generalization of (categorical) $C^\iy$-rings to manifolds with corners, which we will call ({\it categorical\/}) {\it pre $C^\iy$-rings with corners}. These were introduced by Kalashnikov \cite{Kala}, who called them $C^\iy$-rings with corners. But we define $C^\iy$-{\it rings with corners\/} in \S\ref{cc44} as the full subcategory of pre $C^\iy$-rings with corners $\bfC=(\fC,\fC_\rex)$ satisfying an extra condition, that invertible functions in $\fC_\rex$ should have logs in $\fC$. This will make $C^\iy$-rings and $C^\iy$-schemes with corners better behaved. Sections \ref{cc45}--\ref{cc47} develop the theory of $C^\iy$-rings with corners.

\subsection{\texorpdfstring{Categorical pre $C^\iy$-rings with corners}{Categorical pre C∞-rings with corners}}
\label{cc41}

Here is the obvious generalization of Definition~\ref{cc2def1}.

\begin{dfn} 
\label{cc4def1}
Write $\Eucc,\Euccin$ for the full subcategories of $\Manc,\Mancin$ with objects the Euclidean spaces with corners $\R^n_k\!=\![0,\iy)^k\!\t\!\R^{n-k}$ for $0\!\le\! k\!\le\! n$. We have inclusions of subcategories
\e
\Euc\subset\Euccin\subset\Eucc.
\label{cc4eq1}
\e

Define a {\it categorical pre $C^\iy$-ring with corners\/} to be a product-preserving functor $F:\Eucc\ra\Sets$. Here $F$ should also preserve the empty product, i.e.\ it maps $\R^0\t[0,\iy)^0=\{\es\}$ in $\Eucc$ to the terminal object in $\Sets$, the point~$*$.

If $F,G:\Eucc\ra\Sets$ are categorical pre $C^\iy$-rings with corners, a {\it morphism\/} $\eta:F\ra G$ is a natural transformation $\eta:F\Ra G$. Such natural transformations are automatically product-preserving. We write $\CPCRingsc$ for the category of categorical pre $C^\iy$-ring with corners. 

Define a {\it categorical interior pre $C^\iy$-ring with corners\/} to be a product-preserving functor $F:\Euccin\ra\Sets$. These form a category $\CPCRingscin$, with morphisms natural transformations. 

Define a commutative triangle of functors
\e
\begin{gathered}
\xymatrix@C=50pt@R=12pt{ \CPCRingsc \ar[rr]_{\Pi_\CPCRingsc^\CCRings} \ar[dr]_(0.4){\Pi_\CPCRingsc^\CPCRingscin\,\,\,\,\,} && \CCRings \\
& \CPCRingscin \ar[ur]_(0.6){\,\,\,\,\,\Pi_\CPCRingscin^\CCRings} }\!\!\!\!\!\!\!\!\!\!\!\!{}
\end{gathered}
\label{cc4eq2}
\e
by restriction to subcategories in \eq{cc4eq1}, so that for example $\Pi_\CPCRingsc^\CPCRingscin$ maps $F:\Eucc\ra\Sets$ to $F\vert_\Euccin:\Euccin\ra\Sets$. 
\end{dfn}

Here is the motivating example:

\begin{ex}
\label{cc4ex1}
{\bf(a)} Let $X$ be a manifold with corners. Define a categorical pre $C^\iy$-ring with corners $F:\Eucc\ra\Sets$ by $F=\Hom_\Manc(X,-)$. That is, for objects $\R^m\t[0,\iy)^n$ in $\Eucc\subset\Manc$ we have
\begin{equation*}
F\bigl(\R^m\t[0,\iy)^n\bigr)=\Hom_\Manc\bigl(X,\R^m\t[0,\iy)^n\bigr),
\end{equation*} 
and for morphisms $g:\R^m\t[0,\iy)^n\ra \R^{m'}\t[0,\iy)^{n'}$ in $\Eucc$ we have 
\begin{equation*}
F(g)=g\ci:\Hom_\Manc\bigl(X,\R^m\t[0,\iy)^n\bigr)\longra \Hom_\Manc\bigl(X,\R^{m'}\t[0,\iy)^{n'}\bigr)
\end{equation*}
mapping $F(g):h\mapsto g\ci h$. Let $f:X\ra Y$ be a smooth map of manifolds with corners, and $F,G:\Eucc\ra\Sets$ the functors corresponding to $X,Y$. Define a natural transformation $\eta:G\Ra F$ by
\begin{equation*}
\eta\bigl(\R^m\t[0,\iy)^n\bigr)=\ci f:\Hom\bigl(Y,\R^m\t[0,\iy)^n\bigr)\longra \Hom\bigl(X,\R^m\t[0,\iy)^n\bigr)
\end{equation*}
mapping $\eta:h\mapsto h\ci f$. 

Define a functor $F_\Manc^\CPCRingsc:\Manc\ra(\CPCRingsc)^{\bf op}$ to map $X\mapsto F$ on objects, and $f\mapsto\eta$ on morphisms, for $X,Y,F,G,f,\eta$ as above. 

All of this also works if $X,Y$ are manifolds with g-corners, as in \S\ref{cc33}, giving a functor~$F_\Mangc^\CPCRingsc:\Mangc\ra(\CPCRingsc)^{\bf op}$.
\smallskip

\noindent{\bf(b)} Similarly, if $X$ is a manifold with (g-)corners, define a categorical interior pre $C^\iy$-ring with corners $F:\Euccin\ra\Sets$ by $F=\Hom_\Mancin(X,-)$. This gives functors $F_\Mancin^\CPCRingscin:\Mancin\ra(\CPCRingscin)^{\bf op}$ and $F_\Mangcin^\CPCRingscin:\Mangcin\ra(\CPCRingscin)^{\bf op}$.
\end{ex}

In the language of Algebraic Theories, as in Ad\'amek, Rosick\'y and Vitale \cite{ARV}, $\Euc, \Eucc,\ab\Euccin$ are examples of {\it algebraic theories\/} (that is, small categories with finite products), and $\CCRings,\CPCRingsc,\CPCRingscin$ are the corresponding {\it categories of algebras}. Also the inclusions of subcategories \eq{cc4eq1} are {\it morphisms of algebraic theories}, and the functors \eq{cc4eq2} the corresponding morphisms. So, as for Proposition \ref{cc2prop2}, Ad\'amek et al.\ \cite[Prop.s 1.21, 2.5, 9.3 \& Th.~4.5]{ARV} give important results on their categorical properties:

\begin{thm}
\label{cc4thm1}
{\bf(a)} All small limits and directed colimits exist in the categories $\CPCRingsc,\CPCRingscin,$ and they may be computed objectwise in $\Eucc,\ab\Euccin$ by taking the corresponding small limits/directed colimits in $\Sets$.
\smallskip

\noindent{\bf(b)} All small colimits exist in $\CPCRingsc,\CPCRingscin$ though in general they are not computed objectwise in $\Eucc,\Euccin$ as colimits in $\Sets$.

\smallskip

\noindent{\bf(c)} The functors $\Pi_\CPCRingsc^\CCRings,\ldots,\Pi_\CPCRingsc^\CPCRingscin$ in \eq{cc4eq2} have left adjoints, so they preserve limits.
\end{thm}

\subsection{\texorpdfstring{Pre $C^\iy$-rings with corners}{Pre C∞-rings with corners}}
\label{cc42}

The next definition corresponds to Definition \ref{cc4def1} as Definitions \ref{cc2def1} and \ref{cc2def2} do.

\begin{dfn}
\label{cc4def2}
A {\it pre $C^\iy$-ring with corners\/} $\bfC$ assigns the data:
\begin{itemize}
\setlength{\itemsep}{0pt}
\setlength{\parsep}{0pt}
\item[(a)] Two sets $\fC$ and $\fC_\rex$.
\item[(b)] Operations $\Phi_f:\fC^m\t\fC_\rex^n\ra\fC$ for all smooth maps $f:\R^m\t[0,\iy)^n\ra\R$.
\item[(c)] Operations $\Psi_g:\fC^m\t\fC_\rex^n\ra\fC_\rex$ for all exterior $g:\R^m\t[0,\iy)^n\ra[0,\iy)$.
\end{itemize}
Here we allow one or both of $m,n$ to be zero, and consider $S^0$ to be the single point $\{\es\}$ for any set $S$. These operations must satisfy the following relations: 
\begin{itemize}
\setlength{\itemsep}{0pt}
\setlength{\parsep}{0pt}
\item[(i)] Suppose $k,l,m,n\ge 0$, and $e_i:\R^k\t[0,\iy)^l\ra\R$ is smooth for $i=1,\ldots,m$, and $f_j:\R^k\t[0,\iy)^l\ra[0,\iy)$ is exterior for $j=1,\ldots,n$, and $g:\R^m\t[0,\iy)^n\ra\R$ is smooth. Define smooth $h:\R^k\t[0,\iy)^l\ra\R$ by
\e
\begin{split}
h(x_1,\ldots,x_k,y_1,\ldots,y_l)=g\bigl(e_1(x_1,\ldots,y_l),\ldots,
e_m(x_1,\ldots,y_l)&,\\
f_1(x_1,\ldots,y_l),\ldots,f_n(x_1,\ldots,y_l)&\bigr).
\end{split}
\label{cc4eq3}
\e
Then for all $(c_1,\ldots,c_k,c'_1,\ldots,c'_l)\in\fC^k\t\fC_\rex^l$ we have
\begin{align*}
\Phi_h(c_1,\ldots,c_k,c'_1,\ldots,c'_l)=\Phi_g\bigl(
\Phi_{e_1}(c_1,\ldots,c'_l),\ldots,\Phi_{e_m}(c_1,\ldots,c'_l)&,\\
\Psi_{f_1}(c_1,\ldots,c'_l),\ldots,\Psi_{f_n}(c_1,\ldots,c'_l)&\bigr).
\end{align*}
\item[(ii)] Suppose $k,l,m,n\ge 0$, and $e_i:\R^k\t[0,\iy)^l\ra\R$ is smooth for $i=1,\ldots,m$, and $f_j:\R^k\t[0,\iy)^l\ra[0,\iy)$ is exterior for $j=1,\ldots,n$, and $g:\R^m\t[0,\iy)^n\ra[0,\iy)$ is exterior. Define exterior $h:\R^k\t[0,\iy)^l\ra[0,\iy)$ by \eq{cc4eq3}. Then for all $(c_1,\ldots,c_k,c'_1,\ldots,c'_l)\in\fC^k\t\fC_\rex^l$ we have
\begin{align*}
\Psi_h(c_1,\ldots,c_k,c'_1,\ldots,c'_l)=\Psi_g\bigl(
\Phi_{e_1}(c_1,\ldots,c'_l),\ldots,\Phi_{e_m}(c_1,\ldots,c'_l)&,\\
\Psi_{f_1}(c_1,\ldots,c'_l),\ldots,\Psi_{f_n}(c_1,\ldots,c'_l)&\bigr).
\end{align*}
\item[(iii)] Write $\pi_i:\R^m\t[0,\iy)^n\ra\R$ for projection to the $i^{\rm th}$ coordinate of $\R^m$ for $i=1,\ldots,m$, and $\pi'_j:\R^m\t[0,\iy)^n\ra[0,\iy)$ for projection to the $j^{\rm th}$ coordinate of $[0,\iy)^n$ for $j=1,\ldots,n$. Then for all $(c_1,\ldots,c_m,c'_1,\ldots,c'_n)$ in $\fC^m\t\fC_\rex^n$ and all $i=1,\ldots,m$, $j=1,\ldots,n$ we have
\end{itemize}
\begin{equation*}
\Phi_{\pi_i}(c_1,\ldots,c_m,c'_1,\ldots,c'_n)=c_i,\;\>
\Psi_{\pi'_j}(c_1,\ldots,c_m,c'_1,\ldots,c'_n)=c'_j.
\end{equation*}

We will refer to the operations $\Phi_f, \Psi_g$ as the \/{\it $C^\iy$-operations}\/, and we often write a pre $C^\iy$-ring with corners as a pair $\bfC=(\fC,\fC_\rex)$, leaving the $C^\iy$ operations implicit.

Let $\bfC=(\fC,\fC_\rex)$ and $\bfD=(\fD,\fD_\rex)$ be pre $C^\iy$-rings with corners. A {\it morphism\/} $\bs\phi:\bfC\ra\bfD$ is a pair $\bs\phi=(\phi,\phi_\rex)$ of maps $\phi:\fC\ra\fD$ and $\phi_\rex:\fC_\rex\ra\fD_\rex$, which commute with all the operations $\Phi_f,\Psi_g$ on $\bfC,\bfD$. Write $\PCRingsc$ for the category of pre $C^\iy$-rings with corners. 
\end{dfn}

As for Proposition \ref{cc2prop1}, we have:

\begin{prop}
\label{cc4prop1}
There is an equivalence of categories from $\CPCRingsc$ to\/ $\PCRingsc,$ which identifies $F:\Eucc\ra\Sets$ in $\CPCRingsc$ with\/ $\bfC=(\fC,\fC_\rex)$ in $\PCRingsc$ such that\/ $F\bigl(\R^m\t[0,\iy)^n\bigr)=\fC^m\t\fC_\rex^n$ for\/~$m,n\ge 0$.
\end{prop}

Under this equivalence, for a smooth function $f:\R^n_k\ra \R$, we identify $F(f)$ with $\Phi_f$, and for an exterior function $g:\R^n_k\ra[0,\infty)$, we identify $F(g)$ with $\Psi_g$. The proof of the proposition then follows from $F$ being a product preserving functor, and the definition of pre $C^\iy$-ring with corners.

We could define `interior pre $C^\iy$-rings with corners' following Definition \ref{cc4def2}, but replacing exterior maps by interior maps throughout. Instead we will do something more complicated, and define interior pre $C^\iy$-rings with corners as special examples of pre $C^\iy$-rings with corners, and interior morphisms as special morphisms between (interior) pre $C^\iy$-rings with corners. The advantage of this is that we can work with both interior and non-interior pre $C^\iy$-rings with corners and their morphisms in a single theory. 

\begin{dfn}
\label{cc4def3}
Let $\bfC=(\fC,\fC_\rex)$ be a pre $C^\iy$-ring with corners. Then $\fC_\rex$ is a monoid and $0_{\fC_\rex}\in\fC_\rex$ with $c'\cdot 0_{\fC_\rex}=0_{\fC_\rex}$ for all $c'\in\fC_\rex$.

We call $\bfC$ an {\it interior\/} pre $C^\iy$-ring with corners if $0_{\fC_\rex}\ne 1_{\fC_\rex}$, and there do not exist $c',c''\in\fC_\rex$ with $c'\ne 0_{\fC_\rex}\ne c''$ and $c'\cdot c''=0_{\fC_\rex}$. That is, $\fC_\rex$ should have no zero divisors. Write $\fC_\rin=\fC_\rex\sm\{0_{\fC_\rex}\}$. Then $\fC_\rex=\fC_\rin\amalg\{0_{\fC_\rex}\}$, where $\amalg$ is the disjoint union. Since $\fC_\rex$ has no zero divisors, $\fC_\rin$ is closed under multiplication, and $1_{\fC_\rex}\in\fC_\rin$ as $0_{\fC_\rex}\ne 1_{\fC_\rex}$. Thus $\fC_\rin$ is a submonoid of $\fC_\rex$. We write $1_{\fC_\rin}=1_{\fC_\rex}$.

Let $\bfC,\bfD$ be interior pre $C^\iy$-rings with corners, and $\bs\phi=(\phi,\phi_\rex):\bfC\ra\bfD$ be a morphism in $\PCRingsc$. We call $\bs\phi$ {\it interior\/} if $\phi_\rex(\fC_\rin)\subseteq\fD_\rin$. Then we write $\phi_\rin=\phi_\rex\vert_{\fC_\rin}:\fC_\rin\ra\fD_\rin$.
Interior morphisms are closed under composition and include the identity morphisms. 

Write $\PCRingscin$ for the (non-full) subcategory of $\PCRingsc$ with objects interior pre $C^\iy$-rings with corners, and morphisms interior morphisms.
\end{dfn}

\begin{lem}
\label{cc4lem1}
Let\/ $\bfC=(\fC,\fC_\rex)$ be an interior pre $C^\iy$-ring with corners. Then $\fC_\rex^\t\subseteq\fC_\rin$. If\/ $g:\R^m\t[0,\iy)^n\ra[0,\iy)$ is interior, then $\Psi_g:\fC^m\t\fC_\rex^n\ra\fC_\rex$ maps\/~$\fC^m\t\fC_\rin^n\ra\fC_\rin$.
\end{lem}

\begin{proof} Clearly $0_{\fC_\rex}$ is not invertible, so $0_{\fC_\rex}\notin\fC_\rex^\t$, and $\fC_\rex^\t\subseteq\fC_\rex\sm\{0_{\fC_\rex}\}=\fC_\rin$. As $g$ is interior we may write
\e
g(x_1,\ldots,x_m,y_1,\ldots,y_n)=y_1^{a_1}\cdots y_n^{a_n}\cdot \exp\ci h(x_1,\ldots,x_m,y_1,\ldots,y_n),
\label{cc4eq4}
\e
for $a_1,\ldots,a_n\in\N$ and $h:\R^m\t[0,\iy)^n\ra\R$ smooth. Then for $c_1,\ldots,c_m\in\fC$ and $c_1',\ldots,c_n'\in\fC_\rin$ we have
\begin{align*}
\Psi_g\bigl(c_1,\ldots,c_m,c_1',\ldots,c_n'\bigr)=c_1'^{a_1}\cdots c_n'^{a_n}\cdot\Psi_{\exp}\bigl[\Phi_h\bigl(c_1,\ldots,c_m,c_1',\ldots,c_n'\bigr)\bigr].
\end{align*}
Here $c_1'^{a_1}\cdots c_n'^{a_n}\in\fC_\rin$ as $\fC_\rin$ is a submonoid of $\fC_\rex$, and $\Psi_{\exp}[\cdots]\in\fC_\rin$ as $\Psi_{\exp}$ maps to $\fC_\rex^\t\subseteq\fC_\rin\subseteq\fC_\rex$. Thus $\Psi_g\bigl(c_1,\ldots,c_m,c_1',\ldots,c_n'\bigr)\in\fC_\rin$.
\end{proof}

Here is the analogue of Proposition \ref{cc4prop1}:

\begin{prop}
\label{cc4prop2}
There is an equivalence $\CPCRingscin\!\cong\PCRingscin,$ which identifies $F:\Euccin\ra\Sets$ in $\CPCRingscin$ with\/ $\bfC=(\fC,\fC_\rex)$ in $\PCRingscin$ such that\/ $F\bigl(\R^m\t[0,\iy)^n\bigr)=\fC^m\t\fC_\rin^n$ for\/~$m,n\ge 0$.
\end{prop}

\begin{proof} Let $F:\Euccin\ra\Sets$ be a categorical interior pre $C^\iy$-ring with corners. Define sets $\fC=F(\R)$, $\fC_\rin=F([0,\iy))$, and $\fC_\rex=\fC_\rin\amalg\{0_{\fC_\rex}\}$, where $\amalg$ is the disjoint union. Then $F\bigl(\R^m\t[0,\iy)^n\bigr)=\fC^m\t\fC_\rin^n$, as $F$ is product-preserving. Let $f:\R^m\t[0,\iy)^n\ra\R$ and $g:\R^m\t[0,\iy)^n\ra[0,\iy)$ be smooth. We must define maps $\Phi_f:\fC^m\t\fC_\rex^n\ra\fC$ and $\Psi_g:\fC^m\t\fC_\rex^n\ra\fC_\rex$.

Let $c_1,\ldots,c_m\in\fC$ and $c_1',\ldots,c_n'\in\fC_\rex$. Then some of $c_1',\ldots,c_n'$ lie in $\fC_\rin$ and the rest in $\{0_{\fC_\rex}\}$. For simplicity suppose that $c_1',\ldots,c_k'\in\fC_\rin$ and $c_{k+1}'=\cdots=c_n'=0_{\fC_\rex}$ for $0\le k\le n$. Define smooth $d:\R^m\t[0,\iy)^k\ra\R$, $e:\R^m\t[0,\iy)^k\ra[0,\iy)$ by
\e
\begin{split}
d(x_1,\ldots,x_m,y_1,\ldots,y_k)=f(x_1,\ldots,x_m,y_1,\ldots,y_k,0,\ldots,0),\\
e(x_1,\ldots,x_m,y_1,\ldots,y_k)=g(x_1,\ldots,x_m,y_1,\ldots,y_k,0,\ldots,0).
\end{split}
\label{cc4eq5}
\e
Then $F(d)$ maps $\fC^m\t\fC_\rin^k\ra\fC$. Set
\begin{equation*}
\Phi_f(c_1,\ldots,c_m,c_1',\ldots,c_k',0_{\fC_\rex},\ldots,0_{\fC_\rex})=F(d)(c_1,\ldots,c_m,c_1',\ldots,c_k').
\end{equation*}

Either $e:\R^m\t[0,\iy)^k\ra[0,\iy)$ is interior, or $e=0$. If $e$ is interior define 
\begin{equation*}
\Psi_g(c_1,\ldots,c_m,c_1',\ldots,c_k',0_{\fC_\rex},\ldots,0_{\fC_\rex})=F(e)(c_1,\ldots,c_m,c_1',\ldots,c_k').
\end{equation*}
If $e=0$ set $\Psi_g(c_1,\ldots,c_m,c_1',\ldots,c_k',0_{\fC_\rex},\ldots,0_{\fC_\rex})=0_{\fC_\rex}$. This defines $\Phi_f,\Psi_g$, and makes $\bfC=(\fC,\fC_\rex)$ into an interior pre $C^\iy$-ring with corners.

Conversely, let $\bfC=(\fC,\fC_\rex)$ be an interior pre $C^\iy$-ring with corners. Then $\fC_\rex=\fC_\rin\amalg\{0\}$ and we define a product-preserving functor $F:\Euccin\ra\Sets$ with $F\bigl(\R^m\t[0,\iy)^n\bigr)=\fC^m\t\fC_\rin^n$, using the fact from Lemma \ref{cc4lem1} that $\Psi_g:\fC^m\t\fC_\rex^n\ra\fC_\rex$ maps $\fC^m\t\fC_\rin^n\ra\fC_\rin$ for $g$ interior. The rest of the proof follows that of Proposition~\ref{cc4prop1}.
\end{proof}

It is often helpful to work with a small subset of $C^\iy$-operations. The next definition explains this small subset. Monoids were discussed in~\S\ref{cc32}.

\begin{dfn}
\label{cc4def4}
Let $\bfC=(\fC,\fC_\rex)$ be a pre $C^\iy$-ring with corners. Then (as in Definition \ref{cc4def5} below) $\fC$ is a $C^\iy$-ring, and thus a commutative $\R$-algebra. The $\R$-algebra structure makes $\fC$ into a monoid in two ways: under multiplication `$\cdot$' with identity 1, and under addition `$+$' with identity~0.

Define $g:[0,\iy)^2\ra[0,\iy)$ by $g(x,y)=xy$. Then $g$ induces $\Psi_g:\fC_\rex\t\fC_\rex\ra\fC_\rex$. Define multiplication $\cdot:\fC_\rex\t\fC_\rex\ra\fC_\rex$ by $c'\cdot c''=\Psi_g(c',c'')$. The map $1:\R^0\ra[0,\iy)$ gives an operation $\Psi_1:\{\es\}\ra\fC_\rex$. The {\it identity\/} in $\fC_\rex$ is $1_{\fC_\rex}=\Psi_1(\es)$. Then $(\fC_\rex,\cdot,1_{\fC_\rex})$ is a monoid.

The map $0:\R^0\ra[0,\iy)$ gives an operation $\Psi_0:\{\es\}\ra\fC_\rex$. This gives a distinguished element $0_{\fC_\rex}=\Psi_0(\es)$ in $\fC_\rex$, which satisfies $c'\cdot 0_{\fC_\rex}=0_{\fC_\rex}$ for all $c'\in\fC_\rex$, that is, $0_{\fC_\rex}$ is a {\it zero element\/} in the monoid $(\fC_\rex,\cdot,1_{\fC_\rex})$. We have:
\begin{itemize}
\setlength{\itemsep}{0pt}
\setlength{\parsep}{0pt}
\item[(a)] $\fC$ is a commutative $\R$-algebra.
\item[(b)] $\fC_\rex$ is a monoid, in multiplicative notation, with zero element~$0_{\fC_\rex}\in\fC_\rex$.

We write $\fC_\rex^\t$ for the group of invertible elements in~$\fC_\rex$.
\item[(c)] $\Phi_i:\fC_\rex\ra\fC$ is a monoid morphism, for $i:[0,\iy)\hookra\R$ the inclusion and $\fC$ a monoid under multiplication.
\item[(d)] $\Psi_{\exp}:\fC\ra\fC_\rex$ is a monoid morphism, for $\exp:\R\ra[0,\iy)$ and $\fC$ a monoid under addition.
\item[(e)] $\Phi_{\exp}=\Phi_i\ci\Psi_{\exp}:\fC\ra\fC$ for $\exp:\R\ra\R$ and $\Psi_{\exp}$ in (d).
\end{itemize}
Many of our definitions will use only the structures (a)--(e). When we write $\Phi_i,\Psi_{\exp},\Phi_{\exp}$ without further explanation, we mean those in (c)--(e). 
\end{dfn}

\begin{prop}
\label{cc4prop3}
Let\/ $\bfC=(\fC,\fC_\rex)$ be a pre\/ $C^\iy$-ring with corners, and suppose\/ $c'$ lies in the group\/ $\fC_\rex^\t$ of invertible elements in the monoid\/ $\fC_\rex$. Then there exists a unique\/ $c\in\fC$ such that\/ $\Phi_{\exp}(c)=\Phi_i(c')$ in\/ $\fC$.
\end{prop}

\begin{proof} Since $c'\in\fC_\rex^\t$ we have a unique inverse $c^{\prime -1}\in\fC_\rex^\t$. As in the proof of Proposition \ref{cc2prop6}(a), let $e:\R\ra\R$ be the inverse of $t\mapsto \exp(t)-\exp(-t)$. Define smooth $g:[0,\iy)^2\ra\R$ by $g(x,y)=e(x-y)$. Observe that if $(x,y)\in[0,\iy)^2$ with $xy=1$ then $x=\exp t$, $y=\exp(-t)$ for $t=\log x$, and so 
\begin{equation*}
\exp\ci g(x,y)=\exp\ci g(\exp t,\exp(-t))=\exp\ci e(\exp t-\exp(-t))=\exp(t)=x.
\end{equation*}
Therefore there is a unique smooth function $h:[0,\iy)^2\ra\R$ with
\e
\exp\ci g(x,y)-x=h(x,y)(xy-1).
\label{cc4eq6}
\e
We have operations $\Phi_g,\Phi_h:\fC_\rex^2\ra\fC$. Define $c=\Phi_g(c',c^{\prime -1})$. Then 
\begin{align*}
\Phi_{\exp}(c)-\Phi_i(c')&=\Phi_{\exp\ci g(x,y)-x}(c',c^{\prime -1})=\Phi_{h(x,y)(xy-1)}(c',c^{\prime -1})\\
&=\Phi_h(c',c^{\prime -1})\cdot(\Phi_i(c'\cdot c^{\prime -1})-1_\fC)=0,
\end{align*}
using Definition \ref{cc4def1}(i) in the first and third steps, and \eq{cc4eq6} in the second. Hence $\Phi_{\exp}(c)=\Phi_i(c')$. Uniqueness of $c$ follows from Proposition~\ref{cc2prop6}(a).
\end{proof}

\subsection{\texorpdfstring{Adjoints and (co)limits for pre $C^\iy$-rings with corners}{Adjoints and (co)limits for pre C∞-rings with corners}}
\label{cc43}

The equivalences of categories in Propositions \ref{cc2prop1}, \ref{cc4prop1} and \ref{cc4prop2} mean that facts about $\CPCRingsc,\CPCRingscin,\CCRings$ in \S\ref{cc41} correspond to fact about $\PCRingsc,\PCRingscin,\CRings$. Hence \eq{cc4eq2} should correspond to a commutative triangle of functors
\e
\begin{gathered}
\xymatrix@C=50pt@R=12pt{ \PCRingsc \ar[rr]_{\Pi_\PCRingsc^\CRings} \ar[dr]_(0.4){\Pi_\PCRingsc^\PCRingscin\,\,\,\,\,} && \CRings. \\
& \PCRingscin \ar[ur]_(0.6){\,\,\,\,\,\Pi_\PCRingscin^\CRings} }\!\!\!\!\!\!\!\!\!\!\!\!{}
\end{gathered}
\label{cc4eq7}
\e
The explicit definitions, which we give next, follow from the correspondences.

\begin{dfn}
\label{cc4def5}
The functor $\Pi_\PCRingsc^\CRings:\PCRingsc\ra\CRings$ in \eq{cc4eq7} acts on objects by $\bfC=(\fC,\fC_\rex)\mapsto\fC$, where the $C^\iy$-ring $\fC$ has $C^\iy$-operations $\Phi_f:\fC^m\ra\fC$ from smooth $f:\R^m\ra\R$ in Definition \ref{cc4def2}(b) with $n=0$, and on morphisms by $\bs\phi=(\phi,\phi_\rex)\mapsto\phi$. 

Define $\Pi_\PCRingscin^\CRings:\PCRingscin\ra\CRings$ in \eq{cc4eq7} to be the restriction of $\Pi_\PCRingsc^\CRings$ to~$\PCRingscin$.

To define $\Pi_\PCRingsc^\PCRingscin$ in \eq{cc4eq7}, let $\bfC=(\fC,\fC_\rex)$ be a pre $C^\iy$-ring with corners. We will define an interior pre $C^\iy$-ring with corners $\bs{\ti\fC}=(\fC,\ti\fC_\rex)$ where $\ti\fC_\rex=\fC_\rex\amalg\{0_{\ti\fC_\rex}\}$, and set $\Pi_\PCRingsc^\PCRingscin(\bfC)=\bs{\ti\fC}$. Here $\fC_\rex$ already contains a zero element $0_{\fC_\rex}$, but we are adding an extra $0_{\ti\fC_\rex}$ with $0_{\fC_\rex}\ne 0_{\ti\fC_\rex}$. 

Let $f:\R^m\t[0,\iy)^n\ra\R$ and $g:\R^m\t[0,\iy)^n\ra[0,\iy)$ be smooth, and write $\Phi_f,\Psi_g$ for the operations in $\bfC$. We must define maps $\ti\Phi_f:\fC^m\t\ti\fC{}_\rex^n\ra\fC$ and $\ti\Psi_g:\fC^m\t\ti\fC{}_\rex^n\ra\ti\fC_\rex$. Let $c_1,\ldots,c_m\in\fC$ and $c_1',\ldots,c_n'\in\ti\fC_\rex$. Then some of $c_1',\ldots,c_n'$ lie in $\fC_\rex$ and the rest in $\{0_{\ti\fC_\rex}\}$. For simplicity suppose that $c_1',\ldots,c_k'\in\fC_\rex$ and $c_{k+1}'=\cdots=c_n'=0_{\ti\fC_\rex}$ for $0\le k\le n$. Define smooth $d:\R^m\t[0,\iy)^k\ra\R$, $e:\R^m\t[0,\iy)^k\ra[0,\iy)$ by \eq{cc4eq5}. Set
\begin{equation*}
\ti\Phi_f(c_1,\ldots,c_m,c_1',\ldots,c_k',0_{\ti\fC_\rex},\ldots,0_{\ti\fC_\rex})=\Phi_d(c_1,\ldots,c_m,c_1',\ldots,c_k').
\end{equation*}
Either $e:\R^m\t[0,\iy)^k\ra[0,\iy)$ is interior, or $e=0$. If $e$ is interior define 
\begin{equation*}
\ti\Psi_g(c_1,\ldots,c_m,c_1',\ldots,c_k',0_{\ti\fC_\rex},\ldots,0_{\ti\fC_\rex})=\Psi_e(c_1,\ldots,c_m,c_1',\ldots,c_k').
\end{equation*}
If $e=0$ define $\ti\Psi_g(c_1,\ldots,c_m,c_1',\ldots,c_k',0_{\ti\fC_\rex},\ldots,0_{\ti\fC_\rex})=0_{\ti\fC_\rex}$. This defines the maps $\ti\Phi_f,\ti\Psi_g$. It is easy to check that these make $\bs{\ti\fC}=(\fC,\ti\fC_\rex)$ into an interior pre $C^\iy$-ring with corners.

Now let $\bs\phi:\bfC\ra\bfD$ be a morphism in $\PCRingsc$, and define $\bs{\ti\fC},\bs{\ti\fD}$ as above. Define $\ti\phi_\rex:\ti\fC_\rex\ra\ti\fD_\rex$ by $\ti\phi_\rex\vert_{\fC_\rex}=\phi_\rex$ and $\ti\phi_\rex(0_{\ti\fC_\rex})=0_{\ti\fD_\rex}$. Then $\bs{\ti\phi}=(\phi,\ti\phi_\rex):\bs{\ti\fC}\ra\bs{\ti\fD}$ is a morphism in~$\PCRingscin$.

Define $\Pi_\PCRingsc^\PCRingscin:\PCRingsc\hookra\PCRingscin$ by $\Pi_\PCRingsc^\PCRingscin:\bfC\mapsto\bs{\ti\fC}$ and $\Pi_\PCRingsc^\PCRingscin:\bs\phi\mapsto\bs{\ti\phi}$, for $\bfC,\ldots,\bs{\ti\phi}$ as above. 

We will show that $\Pi_\PCRingsc^\PCRingscin$ is right adjoint to $\inc:\PCRingscin\hookra\PCRingsc$. Suppose that $\bfC,\bfD$ are pre $C^\iy$-rings with corners with $\bfC$ interior. Then we can define a 1-1 correspondence
\begin{align*}
\Hom_\PCRingsc\bigl(\inc(\bfC),\bfD\bigr)&=\Hom_\PCRingsc(\bfC,\bfD)\\
&\cong\Hom_\PCRingscin\bigl(\bfC,\Pi_\PCRingsc^\PCRingscin(\bfD)),
\end{align*}
identifying $\bs\phi:\inc(\bfC)\ra\bfD$ with $\bs{\hat\phi}:\bfC\ra\Pi_\PCRingsc^\PCRingscin(\bfD)$, where $\bs\phi=(\phi,\phi_\rex)$ and $\bs{\hat\phi}=(\phi,\hat\phi_\rex)$ with $\hat\phi_\rex\vert_{\fC_\rin}=\phi_\rex\vert_{\fC_\rin}$ and $\hat\phi_\rex(0_{\fC_\rex})=0_{\ti\fD_\rex}$. This is functorial in $\bfC,\bfD$, and so proves that $\Pi_\PCRingsc^\PCRingscin$ is right adjoint to~$\inc$.

Also define functors $\Pi_\rsm,\Pi_\rex:\PCRingsc\ra\Sets$ by $\Pi_\rsm:\bfC\mapsto\fC$, $\Pi_\rex:\bfC\mapsto\fC_\rex$ on objects, and $\Pi_\rsm:\bs\phi\mapsto\phi$, $\Pi_\rex:\bs\phi\mapsto\phi_\rex$ on morphisms, where `$\rsm$', `$\rex$' are short for `smooth' and `exterior'. Define functors $\Pi_\rsm,\Pi_\rin:\PCRingscin\ra\Sets$ by $\Pi_\rsm:\bfC\mapsto\fC$, $\Pi_\rin:\bfC\mapsto\fC_\rin$ on objects, and $\Pi_\rsm:\bs\phi\mapsto\phi$, $\Pi_\rin:\bs\phi\mapsto\phi_\rin$ on morphisms, where `$\rin$' is short for `interior'.
\end{dfn}

Combining the equivalences in Propositions \ref{cc2prop1}, \ref{cc4prop1} and \ref{cc4prop2} with Theorem \ref{cc4thm1} and Definition \ref{cc4def5} yields:

\begin{thm}
\label{cc4thm2}
{\bf(a)} All small limits and directed colimits exist in the categories $\PCRingsc, \PCRingscin$. The functors\/ $\Pi_\rsm,\Pi_\rex:\PCRingsc\ra\Sets$ and\/ $\Pi_\rsm,\Pi_\rin:\PCRingscin\ra\Sets$ preserve limits and directed colimits.
\smallskip

\noindent{\bf(b)} All small colimits exist in $\PCRingsc,\PCRingscin,$ though in general they are  not preserved by $\Pi_\rsm,\Pi_\rex$ and\/~$\Pi_\rsm,\Pi_\rin$.
\smallskip

\noindent{\bf(c)} The functors $\Pi_\PCRingsc^\CRings,\Pi_\PCRingscin^\CRings,\Pi_\PCRingsc^\PCRingscin$ in \eq{cc4eq7} have left adjoints, so they preserve limits. The left adjoint of\/ $\Pi_\PCRingsc^\PCRingscin$ is the inclusion $\inc:\PCRingscin\hookra\PCRingsc,$ so $\inc$ preserves colimits.
\end{thm}

\begin{ex}
\label{cc4ex2}
The inclusion $\inc:\PCRingscin\hookra\PCRingsc$ in general does not preserve limits, and therefore cannot have a left adjoint. 

For example, suppose $\bfC,\bfD$ are interior pre $C^\iy$-rings with corners, and write $\bfE=\bfC\t\bfD$ and $\bfF=\bfC\t_\rin\bfD$ for the products in $\PCRingsc$ and $\PCRingscin$. Then Theorem \ref{cc4thm2}(a) implies that $\fE=\fC\t\fD$, $\fE_\rex=\fC_\rex\t\fD_\rex$, $\fF=\fC\t\fD$, $\fF_\rin=\fC_\rin\t\fD_\rin$. Since $\fC_\rex=\fC_\rin\amalg\{0_{\fC_\rex}\}$, etc., this gives
\begin{align*}
\fE_\rex&=(\fC_\rin\t\fD_\rin)\amalg (\fC_\rin\t\{0_{\fD_\rex}\})\amalg (\{0_{\fC_\rex}\}\t\fD_\rin)\amalg (\{0_{\fC_\rex}\}\t\{0_{\fD_\rex}\}),\\
\fF_\rex&=(\fC_\rin\t\fD_\rin)\amalg(\{0_{\fC_\rex}\}\t\{0_{\fD_\rex}\}).
\end{align*}
Thus $\bfE\not\cong\bfF$. Moreover in $\fE_\rex$ we have $(1_{\fC_\rex},0_{\fD_\rex})\cdot (0_{\fC_\rex},1_{\fD_\rex})=(0_{\fC_\rex},0_{\fD_\rex})$, so $\fE_\rex$ has zero divisors, and $\bfE$ is not an object in~$\PCRingscin$.
\end{ex}

\begin{prop}
\label{cc4prop4}
The functors $\Pi_\PCRingsc^\CRings,\Pi_\PCRingscin^\CRings$ in \eq{cc4eq7} have right adjoints, so they preserve colimits. 
\end{prop}

\begin{proof} We will define a functor $F_{\ge 0}:\CRings\ra\PCRingsc$, and show that it is right adjoint to~$\Pi_\PCRingsc^\CRings$. Let $\fC$ be a $C^\iy$-ring, and define $\fC_{\ge 0}$ to be the subset of elements $c\in\fC$ that satisfy the following condition: for all smooth $f:\R^n\ra\R$ such that $f\vert_{[0,\iy)\t \R^{n-1}}=0$ and for $d_1,\ldots, d_{n-1}\in \fC$, then $\Phi_f(c,d_1,\ldots, d_{n-1})=0\in \fC$. Note that $\fC_{\ge 0}$ is non-empty, as it contains $\Phi_{\exp}(\fC)\amalg\{0\}$. Also, any pre $C^\iy$-ring with corners $(\fD,\fD_\rex)$ has~$\Phi_i(\fD_\rex)\subseteq \fD_{\ge 0}$. 

We will make $(\fC,\fC_{\ge 0})$ into a pre $C^\iy$-ring with corners. Take a smooth map $f:\R^n_k\ra [0,\iy)$. Then there is a (non-unique) smooth extension $g:\R^n\ra\R$ such that $g\vert_{\R^n_k}=i\ci f$ for $i:[0,\iy)\ra \R$ the inclusion map. We define
\begin{equation*}
\Psi_f(c_1',\ldots, c'_k,c_{k+1},\ldots, c_{n})=\Phi_g(c_1',\ldots, c'_k,c_{k+1},\ldots, c_{n})
\end{equation*}
for $c_i'\in \fC_{\ge 0}$ and $c_i\in \fC$. To show this is well defined, say $h$ is another such extension, then $(g-h)\vert_{\R^n_k}=0$. We must prove $\Phi_{g-h}(c_1',\ldots, c'_k,c_{k+1},\ldots, c_{n})=0$ for all $c_i'\in \fC_{\ge 0}$ and~$c_i\in \fC$. 

As $g-h$ satisfies the hypothesis of Lemma \ref{cc4lem2} below, we can assume $g-h=f_1+\cdots +f_k$ such that $f_i\vert_{\R^{i-1}\t [0,\iy)\t \R^{n-i-1}}=0$. Then 
\begin{align*}
&\Phi_{g-h}(c_1',\ldots, c'_k,c_{k+1},\ldots, c_{n})=\Phi_{f_1+\cdots +f_k}(c_1',\ldots, c'_k,c_{k+1},\ldots, c_{n})\\
&\quad=\Phi_{f_1}(c_1',\ldots, c'_k,c_{k+1},\ldots, c_{n})+\cdots + \Phi_{f_k}(c_1',\ldots, c'_k,c_{k+1},\ldots, c_{n})=0,
\end{align*}
as $c_i'$ are in $\fC_{\ge 0}$. Therefore $\Psi_f$ is independent of the choice of extension $g$. 

To show that $\Psi_f$ maps to $\fC_{\ge 0}\subseteq\fC$, suppose $e:\R^m\ra\R$ is smooth with $e\vert_{[0,\iy)\t \R^{m-1}}=0$, and let $d_1,\ldots,d_{m-1}\in\fC$. Then
\e
\begin{split}
\Phi_e(&\Psi_f(c_1',\ldots, c'_k,c_{k+1},\ldots, c_{n}),d_1,\ldots, d_{m-1})\\
&=\Phi_e(\Phi_g(c_1',\ldots, c'_k,c_{k+1},\ldots, c_{n}),d_1,\ldots, d_{m-1})\\
&=\Phi_{e\ci(g\t\id_{\R^{m-1}})}(c_1',\ldots, c'_k,c_{k+1},\ldots, c_{n},d_1,\ldots, d_{m-1}).
\end{split}
\label{cc4eq8}
\e
Here $e\ci(g\t\id_{\R^{m-1}}):\R^{n+m-1}\ra\R$ restricts to zero on $\R^{n+m-1}_k\subset\R^{n+m-1}$, since $(g\t\id_{\R^{m-1}})\vert_{\R^{n+m-1}_k}=f\t\id_{\R^{m-1}}$, which maps to $[0,\iy)\t\R^{m-1}$, and $e\vert_{[0,\iy)\t \R^{m-1}}=0$. Regarding $e\ci(g\t\id_{\R^{m-1}})$ as an extension of $0:\R^{n+m-1}_k\ra\R$, the argument above shows \eq{cc4eq8} is independent of extension, so we can replace $e\ci(g\t\id_{\R^{m-1}})$ by $0:\R^{n+m-1}\ra\R$, and \eq{cc4eq8} is zero. As this holds for all $e,m$, we have $\Psi_f(c_1',\ldots,c_{n})\in\fC_{\ge 0}$, as we want. 

Similarly, smooth functions $f:\R^n_k\ra \R$ give well defined $C^\iy$-operations $\Phi_f$. Hence $(\fC,\fC_{\ge 0})$ is a pre $C^\iy$-ring with corners. Set~$F_{\ge 0}(\fC)=(\fC,\fC_{\ge 0})$. 

On morphisms, $F_{\ge 0}$ sends $\phi:\fC\ra\fD$ to $(\phi,\phi_\rex):(\fC,\fC_{\ge 0})\ra (\fD,\fD_{\ge 0})$, where $\phi_\rex=\phi\vert_{\fC_{\ge 0}}$. As $\phi$ respects the $C^\iy$-operations, the image of $\phi(\fC_{\ge 0})\subseteq\fD_{\ge 0}\subseteq\fD$, so $\phi_\rex$ is well defined.

To show $F_{\ge 0}$ is right adjoint to $\Pi_\PCRingsc^\CRings$, we describe the unit and counit of the adjunction. The unit acts on $(\fC,\fC_\rex)\in\PCRingsc$ as $(\id_\fC,\Phi_i):(\fC,\fC_\rex)\ra(\fC,\fC_{\ge 0})$. The counit acts on $\fC\in\CRings$ as $\id_\fC:\fC\ra\fC$.
The compositions $F_{\ge 0}\Ra F_{\ge 0}\ci\Pi_\PCRingsc^\CRings\ci F_{\ge 0} \Ra F_{\ge 0}$ and $\Pi_\PCRingsc^\CRings\Ra \Pi _\PCRingsc^\CRings\ci F_{\ge 0}\ci \Pi_\PCRingsc^\CRings \Ra \Pi_\PCRingsc^\CRings$ are the identity.

Since $\Pi_\PCRingscin^\CRings=\Pi_\PCRingsc^\CRings\ci\inc$, it now follows from Definition \ref{cc4def5} that $\Pi_\PCRingsc^\PCRingscin\ci F_{\ge 0}$ is right adjoint to~$\Pi_\PCRingscin^\CRings$.
\end{proof}

The next lemma was used in the previous proof.

\begin{lem}
\label{cc4lem2}
If\/ $f:\R^n\ra \R$ is smooth such that\/ $f\vert_{\R^n_k}=0,$ then there are smooth\/ $f_i:\R^n\ra \R$ for $i=1,\ldots,k$ such that\/ $f_i\vert_{\R^{i-1}\t [0,\iy)\t \R^{n-i-1}}=0$ for $i=1,\ldots, k,$ and\/~$f=f_1+\ldots + f_k$.
\end{lem}

\begin{proof}
Let $f$ be as in the lemma. Consider the following open subset $U=\cS^{k-1}\setminus \{(x_1,\ldots, x_k):x_i\ge 0\}$ of the dimension $k-1$ unit sphere $\cS^{k-1}\subset \R^k$. Define an open cover $U_1,\ldots, U_k$ of $U$ by $U_i=\{(x_1,\ldots, x_k)\in U:x_i<0\}$. Take a partition of unity $\rho_i:U\ra[0,1]$ for $i=1,\ldots, k$, subordinate to $\{U_1,\ldots, U_n\}$, with $\rho_i$ having support on $U_i$ and~ $\sum_{i=1}^k\rho_i=1$.

Define the $f_i$ as follows
\begin{equation*}
f_i=\begin{cases} 
f(x_1,\ldots, x_n)\rho_i\left( \frac{(x_1,\dots, x_k)}{\vert(x_1,\ldots, x_k)\vert}\right ),&  \text{if } x_i<0\text{ for some } i=1,\ldots, k,\\
0, & \text{otherwise,}
\end{cases}
\end{equation*}
where $\vert(x_1,\ldots, x_k)\vert$ is the length of the vector $(x_1,\ldots, x_k)\in \R^k$. It is clear these are smooth where the $\rho_i$ are defined. The $\rho_i$ are not defined in the first quadrant of $\R^k$, where all $x_i\ge 0$, however approaching the boundary of this quadrant, the $\rho_i$ are all constant. As $f\vert_{\R^n_k}=0$, then all derivatives of $f$ are zero in this quadrant, so the $f_i$ are smooth and identically zero on $\R^n_k$. In addition, $f_i\vert_{\R^{i-1}\t [0,\iy)\t \R^{n-i-1}}=0$, as the $\rho_i$ are zero outside of $U_i$.
Finally, as $\sum{\rho_i}=1$ and $f\vert_{\R^n_k}=0$, then $f=f_1+\ldots + f_k$ as required. 
\end{proof}

\subsection{\texorpdfstring{$C^\iy$-rings with corners}{C∞-rings with corners}}
\label{cc44}

We can now define $C^\iy$-rings with corners.

\begin{dfn}
\label{cc4def6}
Let $\bfC=(\fC,\fC_\rex)$ be a pre $C^\iy$-ring with corners. We call $\bfC$ a $C^\iy$-{\it ring with corners\/} if any (hence all) of the following hold:
\begin{itemize}
\setlength{\itemsep}{0pt}
\setlength{\parsep}{0pt}
\item[(i)] $\Phi_i\vert_{\smash{\fC_\rex^\t}}:\fC_\rex^\t\ra\fC$ is injective.
\item[(ii)] $\Psi_{\exp}:\fC\ra\fC_\rex^\t$ is surjective.
\item[(iii)] $\Psi_{\exp}:\fC\ra\fC_\rex^\t$ is a bijection.
\end{itemize}
Here Proposition \ref{cc4prop3} implies that (i),(ii) are equivalent, and Definition \ref{cc4def4}(e) and Proposition \ref{cc2prop6}(a) imply that $\Psi_{\exp}:\fC\ra\fC_\rex^\t$ is injective, so (ii),(iii) are equivalent, and therefore (i)--(iii) are equivalent. Write $\CRingsc$ for the full subcategory of $C^\iy$-rings with corners in~$\PCRingsc$.

We call a $C^\iy$-ring with corners $\bfC$ {\it interior\/} if it is an interior pre $C^\iy$-ring with corners. Write $\CRingscin\subset\CRingsc$ for the subcategory of interior $C^\iy$-rings with corners, and interior morphisms between them.

Define a commutative triangle of functors 
\e
\begin{gathered}
\xymatrix@C=50pt@R=12pt{ \CRingsc \ar[rr]_{\Pi_\CRingsc^\CRings} \ar[dr]_(0.4){\Pi_\CRingsc^\CRingscin\,\,\,\,\,} && \CRings \\
& \CRingscin \ar[ur]_(0.6){\,\,\,\,\,\Pi_\CRingscin^\CRings} }\!\!\!\!\!\!\!\!\!\!\!\!{}
\end{gathered}
\label{cc4eq9}
\e
by restricting the functors in \eq{cc4eq7} to $\CRingsc,\CRingscin$. It is clear from Definition \ref{cc4def5} that the restriction of $\Pi_\PCRingsc^\PCRingscin$ to $\CRingsc$ does map to $\CRingscin$. Then $\Pi_\PCRingsc^\PCRingscin$ right adjoint to $\inc$ in Definition \ref{cc4def5} implies that $\Pi_\CRingsc^\CRingscin$ is right adjoint to $\inc:\CRingscin\hookra\CRingsc$.
\end{dfn}

\begin{rem}
\label{cc4rem1}
We can interpret the condition that $\bfC=(\fC,\fC_\rex)$ be a $C^\iy$-ring with corners as follows. Imagine there is some `space with corners' $X$, such that $\fC=\bigl\{$smooth maps $c:X\ra\R\bigr\}$, and $\fC_\rex=\bigl\{$exterior maps $c':X\ra[0,\iy)\bigr\}$. If $c'\in\fC_\rex$ is invertible then $c'$ should map $X\ra(0,\iy)$, and we require there to exist smooth $c=\log c':X\ra\R$ in $\fC$ with $c'=\exp c$. This makes $C^\iy$-rings and $C^\iy$-schemes with corners behave more like manifolds with corners.
\end{rem}

Here is the motivating example, following Example~\ref{cc4ex1}:

\begin{ex} 
\label{cc4ex3}
{\bf(a)} Let $X$ be a manifold with corners. Define a $C^\iy$-ring with corners $\bfC=(\fC,\fC_\rex)$ by $\fC=C^\iy(X)$ and $\fC_\rex=\Ex(X)$, as sets. If $f:\R^m\t[0,\iy)^n\ra\R$ is smooth, define the operation $\Phi_f:\fC^m\t\fC_\rex^n\ra\fC$ by
\e
\Phi_f(c_1,\ldots,c_m,c_1',\ldots,c_n'):x\longmapsto f\bigl(c_1(x),\ldots,c_m(x),c_1'(x),\ldots,c_n'(x)\bigr).
\label{cc4eq10}
\e
If $g:\R^m\t[0,\iy)^n\ra[0,\iy)$ is exterior, define $\Psi_g:\fC^m\t\fC_\rex^n\ra\fC_\rex$ by
\e
\Psi_g(c_1,\ldots,c_m,c_1',\ldots,c_n'):x\longmapsto g\bigl(c_1(x),\ldots,c_m(x),c_1'(x),\ldots,c_n'(x)\bigr).
\label{cc4eq11}
\e
It is easy to check that $\bfC$ is a $C^\iy$-ring with corners. We write~$\bs C^\iy(X)=\bfC$.

Suppose $f:X\ra Y$ is a smooth map of manifolds with corners, and let $\bfC,\bfD$ be the $C^\iy$-rings with corners corresponding to $X,Y$. Write $\bs\phi=(\phi,\phi_\rex)$, where $\phi:\fD\ra\fC$ maps $\phi(d)=d\ci f$ and $\phi_\rex:\fD_\rex\ra\fC_\rex$ maps $\phi(d')=d'\ci f$. Then $\bs\phi:\bfD\ra\bfC$ is a morphism in $\CRingsc$.

Define a functor $F_\Manc^\CRingsc:\Manc\ra(\CRingsc)^{\bf op}$ to map $X\mapsto\bfC$ on objects, and $f\mapsto\bs\phi$ on morphisms, for $X,Y,\bfC,\bfD,f,\bs\phi$ as above. 

All the above also works for manifolds with g-corners, as in \S\ref{cc33}, giving a functor~$F_\Mangc^\CRingsc:\Mangc\ra(\CRingsc)^{\bf op}$.
\smallskip

\noindent{\bf(b)} Similarly, if $X$ is a manifold with (g-)corners, define an interior $C^\iy$-ring with corners $\bfC=(\fC,\fC_\rex)$ by $\fC=C^\iy(X)$ and $\fC_\rex=\In(X)\amalg\{0\}$ as sets, where $0:X\ra[0,\iy)$ is the zero function, with $C^\iy$-operations as in \eq{cc4eq10}--\eq{cc4eq11}. We write~$\bs C^\iy_\rin(X)=\bfC$.

Suppose $f:X\ra Y$ is an interior map of manifolds with (g-)corners, and let $\bfC,\bfD$ be the interior $C^\iy$-rings with corners corresponding to $X,Y$. Write $\bs\phi=(\phi,\phi_\rex)$, where $\phi:\fD\ra\fC$ maps $\phi(d)=d\ci f$ and $\phi_\rex:\fD_\rex\ra\fC_\rex$ maps $\phi(d')=d'\ci f$. Then $\bs\phi:\bfD\ra\bfC$ is a morphism in~$\CRingscin$. 

Define $F_\Mancin^\CRingscin:\Mancin\ra(\CRingscin)^{\bf op}$ and $F_\Mangcin^\CRingscin:\Mangcin\ab\ra(\CRingscin)^{\bf op}$ to map $X\mapsto\bfC$ on objects, and $f\mapsto\bs\phi$ on morphisms, for $X,Y,\bfC,\bfD,f,\bs\phi$ as above. 
\end{ex}

Recall that if $\cD$ is a category and $\cC\subset\cD$ a full subcategory, then $\cC$ is a {\it reflective subcategory\/} if the inclusion $\inc:\cC\hookra\cD$ has a left adjoint $\Pi:\cD\ra\cC$, which is called a {\it reflection functor}.

\begin{prop} 
\label{cc4prop5}
The inclusion\/ $\inc:\CRingsc\hookra\PCRingsc$ has a left adjoint reflection functor\/ $\Pi_\PCRingsc^\CRingsc$. Its restriction to $\PCRingscin$ is a left adjoint $\Pi_\PCRingscin^\CRingscin$ for\/~$\inc:\CRingscin\hookra\PCRingscin$.
\end{prop}

\begin{proof} Let $\bfC=(\fC,\fC_\rex)$ be a pre $C^\iy$-ring with corners. We will define a $C^\iy$-ring with corners $\bs{\hat\fC}=(\fC,\hat\fC_\rex)$. As a set, define $\hat\fC_\rex=\fC_\rex/\simc$, where $\sim$ is the equivalence relation on $\fC_\rex$ given by $c'\sim c''$ if there exists $c'''\in\fC_\rex^\t$ with $\Phi_i(c''')=1$ and $c'=c''\cdot c'''$. That is, $\hat\fC_\rex$ is the quotient of $\fC_\rex$ by the group $\Ker\bigl(\Phi_i\vert_{\smash{\fC_\rex^\t}}\bigr)\subseteq\fC_\rex^\t$. There is a natural surjective projection $\hat\pi:\fC_\rex\ra\hat\fC_\rex$.

If $f:\R^m\t[0,\iy)^n\ra\R$ is smooth and $g:\R^m\t[0,\iy)^n\ra[0,\iy)$ is exterior, it is not difficult to show there exist unique maps $\hat\Phi_f:\fC^m\t\hat\fC{}_\rex^n\ra\fC$ and $\hat\Psi_g:\fC^m\t\hat\fC{}_\rex^n\ra\hat\fC_\rex$ making the following diagrams commute:
\e
\begin{split}
\xymatrix@C=80pt@R=15pt{ *+[r]{\fC^m\t\fC_\rex^n} \ar[r]_(0.6){\Phi_f} \ar[d]^{\id^m\t \hat\pi^n} & *+[l]{\fC} \ar[d]_\id \\
*+[r]{\fC^m\t\hat\fC{}_\rex^n} \ar@{.>}[r]^(0.6){\hat\Phi_f}  & *+[l]{\fC,\!} }\quad
\xymatrix@C=80pt@R=15pt{ *+[r]{\fC^m\t\fC_\rex^n} \ar[r]_(0.6){\Psi_g} \ar[d]^{\id^m\t \hat\pi^n} & *+[l]{\fC_\rex} \ar[d]_{\hat\pi} \\
*+[r]{\fC^m\t\hat\fC{}_\rex^n} \ar@{.>}[r]^(0.6){\hat\Psi_g}  & *+[l]{\hat\fC_\rex,\!} }
\end{split}
\label{cc4eq12}
\e
and these make $\bs{\hat\fC}$ into a pre $C^\iy$-ring with corners. Clearly $\bs{\hat\fC}$ satisfies Definition \ref{cc4def6}(i). Therefore $\bs{\hat\fC}$ is a $C^\iy$-ring with corners.

Suppose $\bs\phi=(\phi,\phi_\rex):\bfC\ra\bfD$ is a morphism of pre $C^\iy$-rings with corners, and define $\bs{\hat\fC},\bs{\hat\fD}$ as above. Then by a similar argument to \eq{cc4eq12}, we find that there is a unique map $\hat\phi_\rex$ such that the following commutes:
\begin{equation*}
\xymatrix@C=80pt@R=15pt{ *+[r]{\fC_\rex} \ar[r]_{\phi_\rex} \ar[d]^{\hat\pi} & *+[l]{\fD_\rex} \ar[d]_{\hat\pi} \\
*+[r]{\hat\fC_\rex} \ar@{.>}[r]^{\hat\phi_\rex}  & *+[l]{\hat\fD_\rex,\!} }
\end{equation*}
and then $\bs{\hat\phi}=(\phi,\hat\phi_\rex):\bs{\hat\fC}\ra\bs{\hat\fD}$ is a morphism of $C^\iy$-rings with corners. Define a functor $\Pi_\PCRingsc^\CRingsc:\PCRingsc\ra\CRingsc$ to map $\bfC\mapsto\bs{\hat\fC}$ on objects and $\bs\phi\mapsto\bs{\hat\phi}$ on morphisms.

Now let $\bfC$ be a pre $C^\iy$-ring with corners, $\bfD$ a $C^\iy$-ring with corners, and $\bs\phi:\bfC\ra\bfD$ a morphism. Then we have a morphism $\bs{\hat\phi}:\bs{\hat\fC}\ra\bfD$, as $\bs{\hat\fD}=\bfD$, and $\bs\phi\leftrightarrow\bs{\hat\phi}$ gives a 1-1 correspondence
\begin{align*}
\Hom_\PCRingsc\bigl(\bfC,\inc(\bfD)\bigr)&=\Hom_\PCRingsc(\bfC,\bfD)\\
&\cong\Hom_\CRingsc\bigl(\Pi_\PCRingsc^\CRingsc(\bfC),\bfD\bigr),
\end{align*}
which is functorial in $\bfC,\bfD$. Hence $\Pi_\PCRingsc^\CRingsc$ is left adjoint to~$\inc$.

If $\bfC,\bs\phi$ above are interior then $\bs{\hat\fC},\bs{\hat\phi}$ are interior, so $\Pi_\PCRingsc^\CRingsc$ restricts to a functor $\Pi_\PCRingscin^\CRingscin:\PCRingscin\ra\CRingscin$, which is left adjoint to~$\inc:\CRingscin\hookra\PCRingscin$.
\end{proof}

We extend Theorem \ref{cc4thm2}, Example \ref{cc4ex2} and Proposition \ref{cc4prop4} to~$\CRingsc$.

\begin{thm}
\label{cc4thm3}
{\bf(a)} All small limits exist in\/ $\CRingsc,\CRingscin$. The functors\/ $\Pi_\rsm,\Pi_\rex:\CRingsc\ra\Sets$ and\/ $\Pi_\rsm,\Pi_\rin:\CRingscin\ra\Sets$ preserve limits.
\smallskip

\noindent{\bf(b)} All small colimits exist in\/ $\CRingsc,\CRingscin,$ though in general they are  not preserved by $\Pi_\rsm,\Pi_\rex$ and\/~$\Pi_\rsm,\Pi_\rin$.
\smallskip

\noindent{\bf(c)} The functors $\Pi_\CRingsc^\CRings,\Pi_\CRingscin^\CRings$ in \eq{cc4eq9} have left and right adjoints, so they preserve limits and colimits. 
\smallskip

\noindent{\bf(d)} The left adjoint of\/ $\Pi_\CRingsc^\CRingscin$ is the inclusion $\inc:\CRingscin\hookra\CRingsc,$ so $\inc$ preserves colimits. However, $\inc$ does not preserve limits, and has no left adjoint.
\end{thm}

\begin{proof} 
For (a), we claim that $\CRingsc,\CRingscin$ are closed under limits in $\PCRingsc,\PCRingscin$. To see this, note that $\Pi_\rsm,\Pi_\rex:\PCRingsc\ra\Sets$ preserve limits by Theorem \ref{cc4thm2}(a), and the conditions Definition \ref{cc4def6}(i)--(iii) are closed under limits of sets $\fC=\Pi_\rsm(\bfC)$ and $\fC_\rex=\Pi_\rex(\bfC)$. Thus part (a) follows from Theorem~\ref{cc4thm2}(a).

For (b), let $J$ be a small category and $\bs F:J\ra\CRingsc$ a functor. Then a colimit $\bfC=\varinjlim\bs F$ exists in $\PCRingsc$ by Theorem \ref{cc4thm2}(b). Since $\Pi_\PCRingsc^\CRingsc$ is a left adjoint it takes colimits to colimits, so $\Pi_\PCRingsc^\CRingsc(\bfC)=\varinjlim\Pi_\PCRingsc^\CRingsc\ci\bs F$ in $\CRingsc$. But $\Pi_\PCRingsc^\CRingsc$ acts as the identity on $\CRingsc\subset\PCRingsc$, so $\Pi_\PCRingsc^\CRingsc\ci\bs F\cong\bs F$, and $\Pi_\PCRingsc^\CRingsc(\bfC)=\varinjlim\bs F$ in $\CRingsc$. The same argument works for~$\CRingscin$.

For (c), $\Pi_\CRingsc^\CRings$ has a left adjoint by Theorem \ref{cc4thm2}(c) and Proposition \ref{cc4prop5}, as $\Pi_\CRingsc^\CRings=\Pi_\PCRingsc^\CRings\ci \inc$. It has a right adjoint as $F_{\ge 0}$ in the proof of Proposition \ref{cc4prop4} maps to $\CRingsc\subset\PCRingsc$. The same works for $\Pi_\CRingscin^\CRingscin$. For (d), $\Pi_\CRingsc^\CRingscin$ has left adjoint $\inc$ by Definition \ref{cc4def6}, and $\inc$ does not preserve limits by the analogue of Example~\ref{cc4ex2}.
\end{proof}

\begin{rem}
\label{cc4rem2}
Although $\inc:\CRingscin\hookra\CRingsc$ has no left adjoint, as in Theorem \ref{cc4thm3}(d), we will show in \S\ref{cc62} that $\inc:\CSchcin\hookra\CSchc$ {\it does\/} have a right adjoint, which is the analogous statement for spaces.
\end{rem}

\subsection{Free objects, generators, relations, and localization}
\label{cc45}

We generalize Definition~\ref{cc2def3}:

\begin{dfn} 
\label{cc4def7}
If $A,A_\rex$ are sets then by \cite[Rem.~14.12]{ARV} we can define the {\it free $C^\iy$-ring with corners $\bfF^{A,A_\rex}=(\fF^{A,A_\rex},\fF_\rex^{A,A_\rex})$ generated by\/} $(A,A_\rex)$. We may think of $\bfF^{A,A_\rex}$ as $\bs C^\iy\bigl(\R^A\t[0,\iy)^{A_\rex}\bigr)$, where $\R^A=\bigl\{(x_a)_{a\in A}:x_a\in\R\bigr\}$ and $[0,\iy)^{A_\rex}=\bigl\{(y_{a'})_{{a'}\in A_\rex}:y_{a'}\in[0,\iy)\bigr\}$. Explicitly, we define $\fF^{A,A_\rex}$ to be the set of maps $c:\R^A\t[0,\iy)^{A_\rex}\ra\R$ which depend smoothly on only finitely many variables $x_a,y_{a'}$, and $\fF^{A,A_\rex}_\rex$ to be the set of maps $c':\R^A\t[0,\iy)^{A_\rex}\ra[0,\iy)$ which depend smoothly on only finitely many variables $x_a,y_{a'}$, and operations $\Phi_f,\Psi_g$ are defined as in \eq{cc4eq10}--\eq{cc4eq11}. Regarding $x_a:\R^A\ra\R$ and $y_{a'}:[0,\iy)^{A_\rex}\ra[0,\iy)$ as functions for $a\in A$, $a'\in A_\rex$, we have $x_a\in\fF^{A,A_\rex}$ and $y_{a'}\in\fF^{A,A_\rex}_\rex$, and we call $x_a,y_{a'}$ the {\it generators\/} of~$\bfF^{A,A_\rex}$.

Then $\bfF^{A,A_\rex}$ has the property that if $\bfC=(\fC,\fC_\rex)$ is any (pre) $C^\iy$-ring with corners then a choice of maps $\al:A\ra\fC$ and $\al_\rex:A_\rex\ra\fC_\rex$ uniquely determine a morphism $\bs\phi:\bfF^{A,A_\rex}\ra\bfC$ with $\phi(x_a)=\al(a)$ for $a\in A$ and $\phi_\rex(y_{a'})=\al_\rex(a')$ for $a'\in A_\rex$. We write $\bfF^{A,A_\rex}=\bfF^{m,n}$ when $A=\{1,\ldots,m\}$ and $A_\rex=\{1,\ldots,n\}$, and then $\bfF^{m,n}\cong\bs C^\iy(\R^m\t[0,\iy)^n)=\bs C^\iy(\R^{m+n}_n)$.

The analogue of all this also holds in $\CRingscin$, with the same objects $\bfF^{m,n}$ and $\bfF^{A,A_\rin}$, which are interior $C^\iy$-rings with corners, and the difference that (interior) morphisms $\bfF^{m,n}\ra\bfC$ in $\CRingscin$ or $\PCRingscin$ are uniquely determined by elements $c_1,\ldots,c_m\in\fC$ and $c_1',\ldots,c_n'\in\fC_\rin$ (rather than $c_1',\ldots,c_n'\in\fC_\rex$), and similarly for $\bfF^{A,A_\rin}$ with~$\al_\rin:A_\rin\ra\fC_\rin$.
\end{dfn}

As for Proposition \ref{cc2prop3}, every object in $\CRingsc,\CRingscin$ can be built out of free $C^\iy$-rings with corners, in a certain sense.

\begin{prop}
\label{cc4prop6}
{\bf(a)} Every object\/ $\bfC$ in $\CRingsc$ admits a surjective morphism $\bs\phi:\bfF^{A,A_\rex}\ra\bfC$ from some free $C^\iy$-ring with corners $\bfF^{A,A_\rex}$.
\smallskip

\noindent{\bf(b)} Every object\/ $\bfC$ in $\CRingsc$ fits into a \begin{bfseries}coequalizer diagram\end{bfseries}
\e
\xymatrix@C=50pt{  \bfF^{B,B_\rex} \ar@<.1ex>@/^.3pc/[r]^{\bs\al} \ar@<-.1ex>@/_.3pc/[r]_{\bs\be} & \bfF^{A,A_\rex} \ar[r]^(0.54){\bs\phi} & \bfC, }
\label{cc4eq13}
\e
that is, $\bfC$ is the colimit of the diagram $\bfF^{B,B_\rex}\rightrightarrows\bfF^{A,A_\rex}$ in $\CRingsc,$ where $\bs\phi:\bfF^{A,A_\rex}\ra\bfC$ is automatically surjective.

The analogues of\/ {\bf(a)} and\/ {\bf(b)} also hold in $\CRingscin$.
\end{prop}

\begin{proof} The analogue in $\PCRingsc,\PCRingscin$ holds by Ad\'amek et al.\ \cite[Prop.s 11.26, 11.28, 11.30, Cor. 11.33, \& Rem.~14.14]{ARV}, by general facts about algebraic theories. The result for $\CRingsc,\CRingscin$ follows by applying the reflection functors $\Pi_\PCRingsc^\CRingsc,\Pi_\PCRingscin^\CRingscin$ of Proposition \ref{cc4prop5}.
\end{proof}

\begin{dfn}
\label{cc4def8}
Let $\bfC$ be a $C^\iy$-ring with corners, and $A,A_\rex$ be sets. We will write $\bfC(x_a:a\in A)[y_{a'}:a'\in A_\rex]$ for the $C^\iy$-ring with corners obtained by adding extra generators $x_a$ for $a\in A$ of type $\R$ and $y_{a'}$ for $a'\in A_\rex$ of type $[0,\iy)$ to $\bfC$. That is, by definition
\e
\bfC(x_a:a\in A)[y_{a'}:a'\in A_\rex]:=\bfC\ot_\iy \bfF^{A,A_\rex},
\label{cc4eq14}
\e
where $\bfF^{A,A_\rex}$ is the free $C^\iy$-ring with corners from Definition \ref{cc4def7}, with generators $x_a$ for $a\in A$ of type $\R$ and $y_{a'}$ for $a'\in A_\rex$ of type $[0,\iy)$, and $\ot_\iy$ is the coproduct in $\CRingsc$. As coproducts are a type of colimit, Theorem \ref{cc4thm3}(b) implies that $\bfC(x_a:a\in A)[y_{a'}:a'\in A_\rex]$ is well defined. Since $\bfF^{A,A_\rex}$ is interior, if $\bfC$ is interior then \eq{cc4eq14} is a coproduct in both $\CRingsc$ and $\CRingscin$, so $\bfC(x_a:a\in A)[y_{a'}:a'\in A_\rex]$ is interior by Theorem~\ref{cc4thm3}(d).

By properties of coproducts and free $C^\iy$-rings with corners, morphisms $\bs\phi:\bfC(x_a:a\in A)[y_{a'}:a'\in A_\rex]\ra\bfD$ in $\CRingsc$ are uniquely determined by a morphism $\bs\psi:\bfC\ra\bfD$ and maps $\al:A\ra\fD$, $\al_\rex:A_\rex\ra\fD_\rex$. If $\bfC,\bfD,\bs\psi$ are interior and $\al_\rex(A_\rex)\subseteq\fD_\rin$ then $\bs\phi$ is interior.

Next suppose $B,B_\rex$ are sets and $f_b\in\fC$ for $b\in B$, $g_{b'},h_{b'}\in\fC_\rex$ for $b'\in B_\rex$. We will write $\bfC/(f_b=0:b\in B)[g_{b'}=h_{b'}:b'\in B_\rex]$ for the $C^\iy$-ring with corners obtained by imposing relations $f_b=0$, $b\in B$ in $\fC$ of type $\R$, and $g_{b'}=h_{b'}$, $b'\in B_\rex$ in $\fC_\rex$ of type $[0,\iy)$. That is, we have a coequalizer diagram
\e
\xymatrix@C=35pt{  \bfF^{B,B_\rex} \ar@<.1ex>@/^.3pc/[r]^{\bs\al} \ar@<-.1ex>@/_.3pc/[r]_{\bs\be} & \bfC \ar[r]^(0.2){\bs\pi} & \bfC/(f_b=0:b\in B)[g_{b'}=h_{b'}:b'\in B_\rex], }
\label{cc4eq15}
\e
where $\bs\al,\bs\be$ are determined uniquely by $\al(x_b)=f_b$, $\al_\rex(y_{b'})=g_{b'}$, $\be(x_b)=0$, $\be_\rex(y_{b'})=h_{b'}$ for all $b\in B$ and $b'\in B_\rex$. As coequalizers are a type of colimit, Theorem \ref{cc4thm3}(b) shows that $\bfC/(f_b=0:b\in B)[g_{b'}=h_{b'}:b'\in B_\rex]$ is well defined. If $\bfC$ is interior and $g_{b'},h_{b'}\in\fC_\rin$ for all $b'\in B_\rex$ (that is, $g_{b'},h_{b'}\ne 0_{\fC_\rex}$) then \eq{cc4eq14} is also a coequalizer in $\CRingscin$, so Theorem \ref{cc4thm3}(d) implies that $\bfC/(f_b=0:b\in B)[g_{b'}=h_{b'}:b'\in B_\rex]$ and $\bs\pi$ are interior.

Note that round brackets $(\cdots)$ denote generators or relations of type $\R$, and square brackets $[\cdots]$ generators or relations of type $[0,\iy)$. If we add generators or relations of only one type, we use only these brackets. 
\end{dfn}

We construct two explicit examples of quotients in~$\CRingsc$:

\begin{ex}
\label{cc4ex4}
{\bf (a)} Say we wish to quotient a $C^\iy$-ring with corners $(\fC,\fC_\rex)$ by an ideal $I$ in $\fC$. Quotienting the $C^\iy$-ring by the ideal will result in additional relations on the monoid. While this quotient, $(\fD,\fD_\rex)$ is the coequalizer of a diagram such as \eq{cc4eq15}, it is also equivalent to the following construction:

The quotient is a $C^\iy$-ring with corners $(\fD,\fD_\rex)$ with a morphism $\bs{\pi}=(\pi,\pi_\rex):(\fC,\fC_\rex)\ra (\fD,\fD_\rex)$ such that $I$ is contained in the kernel of $\pi$, and is universal with respect to this property. That is, if $(\fE,\fE_\rex)$ is another $C^\iy$-ring with corners with morphism $\bs{\pi'}=(\pi',\pi'_\rex):(\fC,\fC_\rex)\ra (\fE,\fE_\rex)$ with $I$ contained in the kernel of $\pi'$, then there is a unique morphism $\bs{p}:(\fD,\fD_\rex)\ra (\fE,\fE_\rex)$ such that $\bs{p}\ci \bs{\pi}=\bs{\pi'}$. 

As a coequalizer is a colimit, by Theorem \ref{cc4thm3}(c) we have $\fD\cong \fC/I$, the quotient in $C^\iy$-rings. For the monoid, we require that smooth $f:\R\ra [0,\iy)$ give well defined operations $\Psi_f:\fD\ra\fD_\rex$. This means we require that if $a-b\in I$, then $\Psi_f(a)\sim \Psi_f(b)\in \fC_\rex$, and this needs to generate a monoid equivalence relation on $\fC_\rex$, so that a quotient by this relation is well defined. If $f:\R\ra[0,\iy)$ is identically zero, this follows. If $f:\R\ra [0,\iy)$ is non-zero and smooth this means that $f$ is positive, and hence that $\log f$ is well defined with $f=\exp \ci \log f$. By Hadamard's Lemma, if $a-b\in I$, then $\Phi_g(a)-\Phi_g(b)\in I$, for all $g:\R\ra\R$ smooth, and therefore $\Phi_{\log f}(a)-\Phi_{\log f}(b)\in I$. Hence in $\fC_\rex$ we only require that if $a-b\in I$, then $\Psi_{\exp}(a)\sim \Psi_{\exp}(b)$ in $\fC_\rex$. The monoid equivalence relation that this generates is equivalent to $c_1'\simc_I c_2'\in \fC_\rex$ if there exists $d\in I$ such that $c_1'=\Psi_{\exp}(d)\cdot c_2'$. 

We claim that $(\fC/I, \fC_\rex/\simc_I)$ is the required $C^\iy$-ring with corners. If $f:[0,\iy)\ra\R$ is smooth, and $c_1'\simc_I c_2'\in \fC_\rex$, then $\Phi_f(c_1')=\Phi_f(\Psi_{\exp}(d)c_2')$ in $\fC$ for some $d\in I$. Applying Hadamard's Lemma twice, we have 
\begin{align*}
\Phi_f(c_1')-\Phi_f(c_2')&=\Phi_{(x-y)g(x,y)}(\Psi_{\exp}(d)c_2',c_2')\\
&=\Phi_i(c_2')(\Phi_{\exp}(d)-1)\Phi_{g}(\Psi_{\exp}(d)c_2',c_2')\\
&=\Phi_i(c_2')(d-0)\Phi_{h(x,y)}(d,0)\Phi_{g}(\Psi_{\exp}(d)c_2',c_2'),
\end{align*} 
for smooth maps $g,h:\R^2\ra\R$. As $d\in I$, then $\Phi_f(c_1')-\Phi_f(c_2')\in I$, and $\Phi_f$ is well defined. A similar proof shows all the $C^\iy$-operations are well defined, and so $(\fC/I,\fC_\rex/\simc_I)$ is a pre $C^\iy$-ring with corners. 

We must show that $\Psi_{\exp}:\fC/I\ra  \fC_\rex/\simc_I $ has image equal to (not just contained in) the invertible elements $(\fC_\rex/\simc_I)^\t$. Say $[c_1']\in (\fC_\rex/\simc_I)^\t$, then there is $[c_2']\in  (\fC_\rex/\simc_I)^\t$ such that $[c_1'][c_2']=[c'd']=[1]$. So there is $d\in I$ such that $c_1'c_2'=\Psi_{\exp}(d)$. However, $\Psi_{\exp}(d)$ is invertible, so each of $c_1',c_2'$ must be invertible in $\fC_\rex$, and using that $\Psi_{\exp}$ is surjective onto invertible elements in $\fC_\rex^\t$ gives the result. Thus $(\fC/I,\fC_\rex/\simc_I)$ is a $C^\iy$-ring with corners. The quotient morphisms $\fC\ra \fC/I$ and $\fC_\rex\ra\fC_\rex/\simc_I$ give the required map $\bs{\pi}$. 

To show that this satisfies the required universal property of either \eq{cc4eq15} or the above, let $(\fE,\fE_\rex)$ be another $C^\iy$-ring with corners with a morphism $\bs{\pi'}=(\pi',\pi'_\rex):(\fC,\fC_\rex)\ra (\fE,\fE_\rex)$ such that $I\subset\Ker\pi'$. Then the unique morphism $\bs{p}:(\fD,\fD_\rex)\ra (\fE,\fE_\rex)$ is defined by $\bs{p}([c],[c'])=[\pi'(c),\pi'(c')]$. The requirement that $(\fE,\fE_\rex)$ factors through each diagram shows that this morphism is well defined and unique, giving the result. 

\noindent{\bf (b)} Say we wish to quotient a $C^\iy$-ring with corners $(\fC,\fC_\rex)$ by an ideal $P$ in the monoid $\fC_\rex$. By this we mean quotient $\fC_\rex$ by the equivalence relation $c_1'\sim c_2'$ if $c_1'=c_2'$ or $c_1',c_2'\in P$. This is known as a Rees quotient of semigroups, see Rees \cite[p.~389]{Rees}. Quotienting the monoid by this ideal will result in additional relations on both the monoid and the $C^\iy$-ring, which we will now show. While this quotient, $(\fD,\fD_\rex)$ is the coequalizer of a diagram such as \eq{cc4eq15}, it is also equivalent to the following construction:

The quotient is a $C^\iy$-ring with corners $(\fD,\fD_\rex)$ with a morphism $\bs{\pi}=(\pi,\pi_\rex):(\fC,\fC_\rex)\ra (\fD,\fD_\rex)$ such that $P$ is contained in the kernel of $\pi_\rex$, and is universal with respect to this property. That is, if $(\fE,\fE_\rex)$ is another $C^\iy$-ring with corners with morphism $\bs{\pi'}=(\pi',\pi'_\rex):(\fC,\fC_\rex)\ra (\fE,\fE_\rex)$ with $P$ contained in the kernel of $\pi'_\rex$, then there is a unique morphism $\bs{p}:(\fD,\fD_\rex)\ra (\fE,\fE_\rex)$ such that $\bs{p}\ci \bs{\pi}=\bs{\pi'}$. 

Similarly to part (a), we begin by quotienting $\fC_\rex$ by $P$, and then require that the $C^\iy$-operations are well defined. As all smooth $f:[0,\iy)\ra \R$ are equal to $\hat f\ci i:[0,\iy)\ra \R$ for a smooth function $\hat f:\R\ra\R$, we need only require that if $c_1'\sim c_2'$, then $\Phi_i(c_1')\sim \Phi_i(c_2')$. This generates a $C^\iy$-ring equivalence relation on the $C^\iy$-ring $\fC$; such a $C^\iy$-ring equivalence relation is the same data as giving an ideal $I\subset \fC$ such that $c_1\sim c_2\in \fC$ whenever $c_1-c_2\in I$. Here, this equivalence relation will be given by the ideal $\langle \Phi_i(P)\rangle$, that is, the ideal generated by the image of $P$ under $\Phi_i$. Quotienting $\fC$ by this ideal generates a further condition on the monoid $\fC_\rex$, as in part (a), that is $c_1'\sim c_2'$ if there is $d\in \langle \Phi_i(P)\rangle$ such that $c_1'=\Psi_{\exp}(d)c_2'$. 

The claim then is that we may take $(\fD,\fD_\rex)$ equal to $\bigl(\fC/\langle \Phi_i(P)\rangle,\fC_\rex/\simc_P\bigr)$ where $c_1'\simc_P c_2'$ if either $c_1',c_2'\in P$ or there is $d\in \langle \Phi_i(P)\rangle$ such that $c_1'=\Psi_{\exp}(d)(c_2')$. Similar applications of Hadamard's Lemma as in (a) show that $(\fD,\fD_\rex)$ is a pre $C^\iy$-ring with corners, and similar discussions show it is a $C^\iy$-ring with corners, and is isomorphic to the quotient. We will use the notation
\begin{equation*}
(\fD,\fD_\rex)=(\fC,\fC_\rex)/\simc_P=(\fC/\langle \Phi_i(P)\rangle,\fC_\rex/\simc_P)=(\fC/\simc_P,\fC_\rex/\simc_P)
\end{equation*}
to refer to this quotient later.
\end{ex}

\begin{dfn}
\label{cc4def9}
Let $\bfC=(\fC,\fC_\rex)$ be a $C^\iy$-ring with corners, and $A\subseteq\fC$, $A_\rex\subseteq\fC_\rex$ be subsets. A {\it localization\/} $\bfC(a^{-1}:a\in A)[a^{\prime -1}:a'\in A_\rex]$ of $\bfC$ at $(A,A_\rex)$ is a $C^\iy$-ring with corners $\bfD=\bfC(a^{-1}:a\in A)[a^{\prime -1}:a'\in A_\rex]$ and a morphism $\bs\pi:\bfC\ra\bfD$ such that $\pi(a)$ is invertible in $\fD$ for all $a\in A$ and $\pi_\rex(a')$ is invertible in $\fD_\rex$ for all $a'\in A_\rex$, with the universal property that if $\bfE=(\fE,\fE_\rex)$ is a $C^\iy$-ring with corners and $\bs\phi:\bfC\ra\bfE$ a morphism with $\phi(a)$ invertible in $\fE$ for all $a\in A$ and $\phi_\rex(a')$ invertible in $\fE_\rex$ for all $a'\in A_\rex$, then there is a unique morphism $\bs\psi:\bfD\ra\bfE$ with~$\bs\phi=\bs\psi\ci\bs\pi$.

Localizations $\bfC(a^{-1}:a\in A)[a^{\prime -1}:a'\in A_\rex]$ always exist, and are unique up to canonical isomorphism. In the notation of Definition \ref{cc4def8} we may write
\ea
&\bfC(a^{-1}:a\in A)[a^{\prime -1}:a'\in A_\rex]=
\label{cc4eq16}\\
&\bigl(\bfC(x_a:a\in A)[y_{a'}:a'\in A_\rex]\bigr)\big/\bigl(a\cdot x_a=1:a\in A\bigr)\bigl[a'\cdot y_{a'}=1:a'\in A_\rex\bigr].
\nonumber
\ea
That is, we add an extra generator $x_a$ of type $\R$ and an extra relation $a\cdot x_a=1$ of type $\R$ for each $a\in A$, so that $x_a=a^{-1}$, and similarly for each~$a'\in A_\rex$.

If $\bfC$ is interior and $A_\rex\subseteq\fC_\rin$ then $\bfC(a^{-1}:a\in A)[a^{\prime -1}:a'\in A_\rex]$ makes sense and exists in $\CRingscin$ as well as in $\CRingsc$, and Theorem \ref{cc4thm3}(d) implies that the two localizations are the same.
\end{dfn}

The next lemma follows from Theorem~\ref{cc4thm3}(c).

\begin{lem}
\label{cc4lem3}
Let\/ $\bfC=(\fC,\fC_\rex)$ be a $C^\iy$-ring with corners and\/ $c\in\fC,$ and write\/ $\bfC(c^{-1})=(\fD,\fD_\rex)$. Then $\fD\cong \fC(c^{-1}),$ the localization of the $C^\iy$-ring.
\end{lem}

\subsection{\texorpdfstring{Local $C^\iy$-rings with corners}{Local C∞-rings with corners}}
\label{cc46}

We now define local $C^\iy$-rings with corners. 

\begin{dfn}
\label{cc4def10}
Let $\bfC=(\fC,\fC_\rex)$ be a $C^\iy$-ring with corners. We call $\bfC$ {\it local\/} if there exists a $C^\iy$-ring morphism $\pi:\fC\ra\R$ such that each $c\in\fC$ is invertible in $\fC$ if and only if $\pi(c)\ne 0$ in $\R$, and each $c'\in\fC_\rex$ is invertible in $\fC_\rex$ if and only if $\pi\ci\Phi_i(c')\ne 0$ in $\R$. Note that if $\bfC$ is local then $\pi:\fC\ra\R$ is determined uniquely by $\Ker\pi=\{c\in\fC:c$ is not invertible$\}$, and $\fC$ is a local $C^\iy$-ring.

Write $\CRingsclo\subset\CRingsc$ and $\CRingscinlo\subset \CRingscin$ for the full subcategories of (interior) local $C^\iy$-rings with corners.
\end{dfn}
 
\begin{prop}
\label{cc4prop7}
$\CRingsclo$ is closed under colimits in $\CRingsc$. Thus all small colimits exist in $\CRingsclo$ and\/ $\CRingscinlo$.
\end{prop}

\begin{proof} Let $J$ be a category and $\bs F=(F,F_\rex):J\ra\CRingsclo$ a functor, and suppose a colimit $\bfC=(\fC,\fC_\rex)=\varinjlim\bs F$ exists in $\CRingsc$, with projections $\bs\phi_j:\bs F(j)\ra\bfC$. Then $\fC=\varinjlim F$ in $\CRings$ by Theorem \ref{cc4thm3}(c). For each $j\in J$ we have a unique $C^\iy$-ring morphism $\pi_j:F(j)\ra\R$ as $F(j)$ is local, with $\pi_j=\pi_k\ci F(\al)$ for all morphisms $\al:j\ra k$ in $J$. Hence by properties of colimits there is a unique morphism $\pi:\fC\ra\R$ with $\pi_j=\pi\ci\phi_j$ for all~$j\in J$.

Suppose $c\in\fC$ with $\pi(c)\ne 0$. Then $c=\phi_j(f)$ for some $j\in J$ and $f\in F(j)$, with $\pi_j(f)=\pi\ci\phi_j(f)=\pi(c)\ne 0$. Hence as $\bs F(j)$ is local, we have an inverse $f^{-1}\in F(j)$, and $\phi_j(f^{-1})=c^{-1}$ is the inverse of $c$ in $\fC$. The same argument shows that if $c'\in\fC_\rex$ with $\pi\ci\Phi_i(c')\ne 0$ then $c'$ is invertible in $\fC_\rex$. So $\bfC$ is local, and $\CRingsclo$ is closed under colimits in $\CRingsc$. The last part holds by Theorem~\ref{cc4thm3}(b),(d).
\end{proof}

\begin{lem} 
\label{cc4lem4}
If\/ $\bfC=(\fC,\fC_\rex)$ is a (pre) $C^\iy$-ring with corners and\/ $x:\fC\ra\R$ a $C^\iy$-ring morphism, then we have a morphism of (pre) $C^\iy$-rings with corners $(x,x_\rex):(\fC,\fC_\rex)\ra (\R,[0,\iy)),$ where $x_\rex(c')=x\ci \Phi_i(c')$ for\/~$c'\in\fC_\rex$.
\end{lem}

\begin{proof}
Let $\bfC=(\fC,\fC_\rex)$ and $x:\fC\ra\R$ be as in the statement. Take $c'\in \fC_\rex$. To show $x_\rex=x\ci \Phi_i:\fC_\rex\ra[0,\iy)$ is well defined, assume for a contradiction that $x\ci \Phi_i(c')=\ep<0\in \R$. Let $f:\R\ra\R$ be a smooth function such that $f$ is the identity on $[0,\iy)$ and it is zero on $(-\infty,\ep/2)$. Then $f\ci i=i$ for $i:[0,\iy)\ra \R$ the inclusion. So we have $0>\ep=x\ci \Phi_i(c')=x\ci \Phi_f\ci\Phi_i(c')=f(x\ci \Phi_i(c'))=f(\ep)=0$, and $x_\rex=x\ci \Phi_i:\fC_\rex\ra[0,\iy)$ is well defined.

For $(x,x_\rex)$ to be a morphism of pre $C^\iy$-rings with corners, it must respect the $C^\iy$-operations. For example, let $f:[0,\iy)\ra [0,\iy)$ be smooth, then there is a smooth function $g:\R\ra\R$ that extends $f$, so that $g\ci i=i\ci f$. Then $x_\rex(\Psi_f(c'))=x\ci \Phi_i(\Psi_f(c'))=\Phi_g(x\ci \Phi_i(c'))=\Phi_g(x_\rex(c'))$ as required. A similar proof holds for the other $C^\iy$-operations. 
\end{proof}

\begin{dfn}
\label{cc4def11}
Let $\bfC=(\fC,\fC_\rex)$ be a $C^\iy$-ring with corners. An $\R$-{\it point\/} $x$ of $\bfC$ is a $C^\iy$-ring morphism (or equivalently, an $\R$-algebra morphism) $x:\fC\ra\R$. Define $\bfC_x$ to be
the localization
\e
\bfC_x=\bfC\bigl(c^{-1}:c\in\fC,\; x(c)\ne 0\bigr)\bigl[c^{\prime -1}:c'\in\fC_\rex,\; x\ci\Phi_i(c')\ne 0\bigr],
\label{cc4eq17}
\e
with projection $\bs\pi_x:\bfC\ra\bfC_x$. If $\bfC$ is interior then $\bfC_x$ is interior by Definition \ref{cc4def9}. Theorem \ref{cc4thm4} below shows $\bfC_x$ is local. Theorem \ref{cc4thm4}(c) is the analogue of Proposition \ref{cc2prop5}. The point of the proof is to give an alternative construction of $\bfC_x$ from $\bfC$ by imposing relations, but adding no new generators.
\end{dfn}

\begin{thm}
\label{cc4thm4}
Let\/ $\bs\pi_x:\bfC\ra\bfC_x$ be as in Definition\/ {\rm\ref{cc4def11}}. Then:
\begin{itemize}
\setlength{\itemsep}{0pt}
\setlength{\parsep}{0pt}
\item[{\bf(a)}] $\bfC_x$ is a local\/ $C^\iy$-ring with corners.
\item[{\bf(b)}] $\bfC_x=(\fC_x,\fC_{x,\rex})$ and\/ $\bs\pi_x=(\pi_x,\pi_{x,\rex}),$ where $\pi_x:\fC\ra\fC_x$ is the local\/ $C^\iy$-ring associated to $x:\fC\ra\R$ in Definition\/~{\rm\ref{cc2def5}}.
\item[{\bf(c)}] $\pi_x:\fC\ra\fC_x$ and\/ $\pi_{x,\rex}:\fC_\rex\ra\fC_{x,\rex}$ are surjective.
\end{itemize}
\end{thm}

\begin{proof} Proposition \ref{cc2prop5} says that the local $C^\iy$-ring $\fC_x$ is $\fC/I$, for $I\subset\fC$ the ideal defined in \eq{cc2eq3}, with $\pi_x:\fC\ra\fC_x$ the projection $\fC\ra\fC/I$. Define
\begin{equation*}
\fD=\fC_x=\fC/I\quad\text{and}\quad \fD_\rex=\fC_\rex/\simc,
\end{equation*}
where $\sim$ is the equivalence relation on $\fC_\rex$ given by $c'\sim c''$ if there exists $i\in I$ with $c''=\Psi_{\exp}(i)\cdot c'$, and $\phi:\fC\ra\fD=\fC/I$, $\phi_\rex:\fC_\rex\ra\fD_\rex=\fC_\rex/\simc$ to be the natural surjective projections. Let $f:\R^m\t[0,\iy)^n\ra\R$ be smooth and $g:\R^m\t[0,\iy)^n\ra[0,\iy)$ be exterior, and write $\Phi_f,\Psi_g$ for the operations in $\bfC$. Then  it is not difficult to show there exist unique maps $\Phi'_f,\Psi'_g$ making the following diagrams commute:
\e
\begin{gathered}
\xymatrix@C=80pt@R=15pt{ *+[r]{\fC^m\t\fC_\rex^n} \ar[r]_(0.6){\Phi_f} \ar[d]^{\phi^m\t \phi_\rex^n} & *+[l]{\fC} \ar[d]_\phi \\
*+[r]{\fD^m\t\fD{}_\rex^n} \ar@{.>}[r]^(0.6){\Phi'_f}  & *+[l]{\fD,\!} }\quad
\xymatrix@C=80pt@R=15pt{ *+[r]{\fC^m\t\fC_\rex^n} \ar[r]_(0.6){\Psi_g} \ar[d]^{\phi^m\t \phi_\rex^n} & *+[l]{\fC_\rex} \ar[d]_{\phi_\rex} \\
*+[r]{\fD^m\t\fD{}_\rex^n} \ar@{.>}[r]^(0.6){\Psi'_g}  & *+[l]{\fD_\rex,\!} }
\end{gathered}
\label{cc4eq18}
\e
and these $\Phi'_f,\Psi'_g$ make $\bfD=(\fD,\fD_\rex)$ into a $C^\iy$-ring with corners, and $\bs\phi=(\phi,\phi_\rex):\bfC\ra\bfD$ into a surjective morphism.

Suppose that $\bfF$ is a $C^\iy$-ring with corners and $\bs\chi=(\chi,\chi_\rex):\bfC\ra\bfF$ a morphism such that $\chi(c)$ is invertible in $\fF$ for all $c\in\fC$ with $x(c)\ne 0$. The definition $\fD=\fC_x=\fC(c^{-1}:c\in\fC$, $x(c)\ne 0)$ in $\CRings$ in Definition \ref{cc2def5} implies that $\chi:\fC\ra\fF$ factorizes uniquely as $\chi=\xi\ci\phi$ for $\xi:\fD\ra\fF$ a morphism in $\CRings$. Hence $\chi(i)=0$ in $\fF$ for all $i\in I$, so $\chi_\rex(\Psi_{\exp}(i))=1_{\fF_\rex}$ for all $i\in I$. Thus if $c',c''\in\fC_\rex$ with $c''=\Psi_{\exp}(i)\cdot c'$ for $i\in I$ then $\chi_\rex(c')=\chi_\rex(c'')$. Hence $\chi_\rex$ factorizes uniquely as $\chi_\rex=\xi_\rex\ci\phi_\rex$ for~$\xi_\rex:\fD_\rex\ra\fF_\rex$. 

As $\bs\chi,\bs\phi$ are morphisms in $\CRingsc$ with $\bs\phi$ surjective we see that $\bs\xi=(\xi,\xi_\rex):\bfD\ra\bfF$ is a morphism. Therefore $\bs\chi:\bfC\ra\bfF$ factorizes uniquely as $\bs\chi=\bs\xi\ci\bs\phi$. Also $\phi(c)$ is invertible in $\fD=\fC_x$ for all $c\in\fC$ with $x(c)\ne 0$, by definition of $\fC_x$ in Definition \ref{cc2def5}. Therefore we have a canonical isomorphism
\begin{equation*}
\bfD\cong\bfC\bigl(c^{-1}:c\in\fC,\; x(c)\ne 0\bigr)
\end{equation*}
identifying $\bs\phi:\bfC\ra\bfD$ with the projection $\bfC\ra\bfC\bigl(c^{-1}:c\in\fC,\; x(c)\ne 0\bigr)$. Note that $x:\fC\ra\R$ factorizes as $x=\ti x\ci\phi$ for a unique morphism~$\ti x:\fD\ra\R$.

Next define
\begin{equation*}
\fE=\fD=\fC_x\quad\text{and}\quad \fE_\rex=\fD_\rex/\approx,
\end{equation*}
where $\approx$ is the monoidal equivalence relation on $\fD_\rex$ generated by the conditions that $d'\approx d''$ whenever $d',d''\in\fD_\rex$ with $\Phi'_i(d')=\Phi'_i(d'')$ in $\fD$ and $\ti x\ci\Phi'_i(d')\ne 0$. Write $\psi=\id:\fD\ra\fE$, and let $\psi_\rex:\fD_\rex\ra\fE_\rex$ be the natural surjective projection. Suppose $f:\R^m\t[0,\iy)^n\ra\R$ is smooth and $g:\R^m\t[0,\iy)^n\ra[0,\iy)$ is exterior. Then as for \eq{cc4eq18}, we claim there are unique maps $\Phi''_f,\Psi''_g$ making the following diagrams commute:
\e
\begin{gathered}
\xymatrix@C=80pt@R=15pt{ *+[r]{\fD^m\t\fD_\rex^n} \ar[r]_(0.6){\Phi'_f} \ar[d]^{\psi^m\t \psi_\rex^n} & *+[l]{\fD} \ar[d]_\psi \\
*+[r]{\fE^m\t\fE{}_\rex^n} \ar@{.>}[r]^(0.6){\Phi''_f}  & *+[l]{\fE,\!} }\quad
\xymatrix@C=80pt@R=15pt{ *+[r]{\fD^m\t\fD_\rex^n} \ar[r]_(0.6){\Psi'_g} \ar[d]^{\psi^m\t \psi_\rex^n} & *+[l]{\fD_\rex} \ar[d]_{\psi_\rex} \\
*+[r]{\fE^m\t\fE{}_\rex^n} \ar@{.>}[r]^(0.6){\Psi''_g}  & *+[l]{\fE_\rex.\!} }
\end{gathered}
\label{cc4eq19}
\e 

To see that $\Phi''_f$ in \eq{cc4eq19} is well defined, note that as $\Phi_i':\fD_\rex\ra\fD$ is a monoid morphism and $\approx$ is a monoidal equivalence relation generated by $d'\approx d''$ when $\Phi'_i(d')=\Phi'_i(d'')$, we have a factorization $\Phi_i'=\ti\Phi_i'\ci\psi_\rex$. We may extend $f$ to smooth $\ti f:\R^{m+n}\ra\R$, and then $\Phi_f'$ factorizes as
\begin{equation*}
\xymatrix@C=60pt@R=10pt{ \fD^m\t\fD_\rex^n \ar[rr]_{\Phi'_f} \ar[dr]_{\id_\fD^m\t(\Phi'_i)^n\,\,\,} && \fD. \\
& \fD^{m+n} \ar[ur]_{\Phi'_{\ti f}} }
\end{equation*}
Using $\Phi_i'=\ti\Phi_i'\ci\psi_\rex$, we see that $\Phi_f''$ in \eq{cc4eq19} exists and is unique.

For $\Psi_g''$, if $g=0$ then $\Psi''_g=0_{\fE_\rex}=[0_{\fD_\rex}]$ in \eq{cc4eq19}. Otherwise we may write $g$ using $a_1,\ldots,a_n$ and $h:\R^m\t[0,\iy)^n\ra\R$ as in \eq{cc4eq4}, and then
\begin{equation*}
\Psi_{g'}(d_1,\ldots,d_m,d_1',\ldots,d_n')=(d_1')^{a_1}\cdots (d_n')^{a_n}\cdot \Psi'_{\exp}\bigl[\Phi'_h(d_1,\ldots,d_m,d_1',\ldots,d_n')\bigr].
\end{equation*}
Since $\psi_\rex:\fD_\rex\ra\fE_\rex$ is a monoid morphism, as it is a quotient by a monoidal equivalence relation, we see from this and the previous argument applied to $\Phi_h(d_1,\ldots,d_m,d_1',\ldots,d_n')$ that $\Psi_g''$ in \eq{cc4eq19} exists and is unique. These $\Phi''_f,\Psi''_g$ make $\bfE=(\fE,\fE_\rex)$ into a $C^\iy$-ring with corners, and $\bs\psi=(\psi,\psi_\rex):\bfD\ra\bfE$ into a surjective morphism.

We will show that there is a canonical isomorphism $\bfE\cong\bfC_x$ which identifies $\bs\psi\ci\bs\phi:\bfC\ra\bfE$ with $\bs\pi_x:\bfC\ra\bfC_x$. Firstly, suppose $c\in\fC$ with $x(c)\ne 0$. Then $\psi\ci\phi(c)$ is invertible in $\fE=\fC_x$ by definition of $\fC_x$ in Definition \ref{cc2def5}. Secondly, suppose $c'\in\fC_\rex$ with $x\ci\Phi_i(c')\ne 0$. Set $d'=\phi_\rex(c')$. Then $\Phi_i'(d')=\phi\ci\Phi_i(c')$ is invertible in $\fD$. Now in the proof of Proposition \ref{cc4prop3}, we do not actually need $c'$ to be invertible in $\fC_\rex$, it is enough that $\Phi_i(c')$ is invertible in $\fC$. Thus this proof shows that there exists a unique $d\in\fD$ with $\Phi'_i(d')=\Phi'_{\exp}(d)=\Phi'_i\ci\Psi_{\exp}'(d)$. But then $d'\approx\Psi_{\exp}'(d)$, so $\psi_\rex(d')=\psi_\rex\ci\Psi_{\exp}'(d)=\Psi_{\exp}''(\psi(d))$. Hence $\psi_\rex\ci\phi_\rex(c')=\psi_\rex(d')$ is invertible in $\fE_\rex$, with inverse~$\Psi_{\exp}''(-\psi(d))$.

Thirdly, suppose that $\bfG$ is a $C^\iy$-ring with corners and $\bs\ze=(\ze,\ze_\rex):\bfC\ra\bfG$ a morphism such that $\ze(c)$ is invertible in $\fG$ for all $c\in\fC$ with $x(c)\ne 0$ and $\ze_\rex(c')$ is invertible in $\fG_\rex$ for all $c'\in\fC_\rex$ with $x\ci\Phi_i(c')\ne 0$. Then $\bs\ze=\bs\eta\ci\bs\phi$ for a unique $\bs\eta:\bfD\ra\bfG$, by the universal property of $\bfD$. Since $\phi_\rex:\fC_\rex\ra\fD_\rex$ is surjective and $x=\ti x\ci\phi$ we see that $\eta_\rex(d')$ is invertible in $\fG_\rex$ for all $d'\in\fD_\rex$ with $\ti x\ci\Phi'_i(d')\ne 0$.

Let $d',d''\in\fD_\rex$ with $\Phi'_i(d')=\Phi'_i(d'')$ in $\fD$ and $\ti x\ci\Phi'_i(d')\ne 0$, so that $d'\approx d''$. Then $\eta_\rex(d'),\eta_\rex(d'')$ are invertible in $\fG_\rex$ with $\Phi_i\ci\eta_\rex(d')=\Phi_i\ci\eta_\rex(d'')$ in $\fG$, so Definition \ref{cc4def6}(i) for $\bfG$ implies that $\eta_\rex(d')=\eta_\rex(d'')$. Since $\eta_\rex:\fD_\rex\ra\fG_\rex$ is a monoid morphism, and $\approx$ is a monoidal equivalence relation, and $\eta_\rex(d')=\eta_\rex(d'')$ for the generating relations $d'\approx d''$, we see that $\eta_\rex$ factorizes via $\fD_\rex/\approx$. Thus there exists unique $\th_\rex:\fE_\rex\ra\fF_\rex$ with $\eta_\rex=\th_\rex\ci\psi_\rex$. Set $\th=\eta:\fE=\fD\ra\fG$. Then $\eta=\th\ci\psi$ as $\psi=\id_\fD$. As $\bs\psi$ is surjective we see that $\bs\th=(\th,\th_\rex):\bfE\ra\bfF$ is a morphism in $\CRingsc$, with $\bs\eta=\bs\th\ci\bs\psi$, so that~$\bs\ze=\bs\th\ci\bs\psi\ci\bs\phi$.

This proves that $\bs\psi\ci\bs\phi:\bfC\ra\bfE$ satisfies the universal property of $\bs\pi_x:\bfC\ra\bfC_x$ from the localization \eq{cc4eq17}, so $\bfE\cong\bfC_x$ as we claimed. Parts (b),(c) of the theorem are now immediate, as $\fE=\fC_x$ and $\bs\phi,\bs\psi$ are surjective. For (a), observe that $x:\fC\ra\R$ factorizes as $\pi\ci\pi_x$ for $\pi:\fC_x\ra\R$ a morphism. If $\bar c\in\fC_x$ with $\pi(\bar c)\ne 0$ then as $\pi_x:\fC\ra\fC_x$ is surjective by (c) we have $\bar c=\pi_x(c)$ with $x(c)\ne 0$, so $\bar c=\pi_x(c)$ is invertible in $\fC_x$ by \eq{cc4eq17}. Similarly, if $\bar c'\in\fC_{x,\rex}$ with $\pi\ci\Phi_i(\bar c')\ne 0$ then as $\pi_{x,\rex}:\fC_\rex\ra\fC_{x,\rex}$ is surjective we find that $\bar c'$ is invertible in $\fC_{x,\rex}$. Hence $\bfC_x$ is a local $C^\iy$-ring with corners.
\end{proof}

We can characterize the equivalence relations that define $\fC_{x,\rex}=\fE_\rex$ in the above proof using the following proposition.

\begin{prop}
\label{cc4prop8}
Let\/ $\bfC=(\fC,\fC_\rex)$ be a $C^\iy$-ring with corners and\/ $x:\fC\ra \R$ an $\R$-point of\/ $\fC$. Let\/ $\bs{\pi}_x:\bfC\ra\bfC_x$ be as in Definition {\rm\ref{cc4def11},} and\/ $I$ the ideal defined in \eq{cc2eq3}. For any $c_1',c_2'\in\fC_\rex,$ then $\pi_{x,\rex}(c_1')=\pi_{x,\rex}(c_2')$ if and only if there are elements $a',b'\in \fC_\rex$ such that\/ $\Phi_i(a')-\Phi_i(b')\in I,$ $x\ci\Phi_i(a')\ne 0,$ and\/ $a'c_1'=b'c_2'$. Hence $\bfC_x=(\fC/I,\fC_\rex/\simc)$ where $c_1'\sim c_2'\in \fC_\rex$ if and only if there are elements $a',b'\in \fC_\rex$ such that\/ $\Phi_i(a')-\Phi_i(b')\in I,$ $x\ci\Phi_i(a')\ne 0,$ and\/~$a'c_1'=b'c_2'$.
\end{prop}

\begin{proof}
We first show that if $\pi_{x,\rex}(c_1')=\pi_{x,\rex}(c_2')$, then there are $a',b'$ satisfying the conditions. In Theorem \ref{cc4thm4}, we constructed $C^\iy$-rings with corners $\bfD=(\fD,\fD_\rex)$ and $\bfE=(\fE,\fE_\rex)$ and surjective morphisms $\bs\phi=(\phi,\phi_\rex):\bfC\ra\bfD$, $\bs\psi=(\psi,\psi_\rex):\bfD\ra\bfE$, where
\begin{equation*}
\fE=\fD=\fC_x=\fC/I,\quad \fD_\rex=\fC_\rex/\simc\quad\text{and}\quad \fE_\rex=\fD_\rex/\approx,
\end{equation*}
and $I\subset\fC$ is the ideal defined in \eq{cc2eq3}, and $\sim,\approx$ are explicit equivalence relations. Then we showed that there is a unique isomorphism $\bfE\cong\bfC_x$ identifying $\bs\psi\ci\bs\phi:\bfC\ra\bfE$ with $\bs\pi_x:\bfC\ra\bfC_x$. As $\pi_{x,\rex}(c_1')=\pi_{x,\rex}(c_2')$ we have $\psi_\rex\ci\phi_\rex(c_1')=\psi_\rex\ci\phi_\rex(c_2')$ in $\fE_\rex$. Thus~$\phi_\rex(c_1')\approx\phi_\rex(c_2')$. 

By definition $\approx$ is the monoidal equivalence relation on $\fD_\rex$ generated by the condition that $d'\approx d''$ whenever $d',d''\in\fD_\rex$ with $\Phi'_i(d')=\Phi'_i(d'')$ in $\fD$ and $\ti x\ci\Phi'_i(d')\ne 0$, where $\Phi'_i:\fD_\rex\ra\fD$ is the $C^\iy$-operation from the inclusion $i:[0,\iy)\hookra\R$, and $x:\fC\ra\R$ factorizes as $x=\ti x\ci\phi$ for a unique morphism $\ti x:\fD\ra\R$. Hence $\phi_\rex(c_1')\approx\phi_\rex(c_2')$ means that there is a finite sequence $\phi_\rex(c_1')=d_0',d_1', d_2', \ldots, d_{n-1}',d_n'=\phi_\rex(c_2')$ in $\fD_\rex$, and elements $e'_i,f'_i,g'_i\in \fD_\rex$ such that $\Phi_i'(e'_i)=\Phi_i'(f'_i)$, $\ti x\ci \Phi_i'(e'_i)\ne 0$, and $d'_{i-1}=e'_ig'_i,$ $d'_i=f'_ig'_i$ in $\fD_\rex$ for~$i=1,\ldots,n$.

As $\phi_\rex:\fC_\rex\ra\fD_\rex$ is surjective we can choose $e_i,f_i,g_i\in\fC_\rex$ with $e'_i=\phi_\rex(e_i)$, $f'_i=\phi_\rex(f_i)$, $g'_i=\phi_\rex(g_i)$ for $i=1,\ldots,n$. Then the conditions become
\begin{gather*}
\Phi_i(e_i)-\Phi_i(f_i)\in I,\quad x\ci \Phi_i(e_i)\ne 0\in\R,\quad i=1,\ldots,n,\\
c_1'\sim e_1g_1, \quad e_{i+1}g_{i+1}\sim f_ig_i, \quad i=1,\ldots,n-1,\quad c_2'\sim f_ng_n,
\end{gather*}
since equality in $\fD_\rex$ lifts to $\sim$-equivalence in $\fC_\rex$. By definition of $\sim$, this means that there exist elements $h_0,h_1,\ldots,h_n$ in the ideal $I\subset\fC$ in \eq{cc2eq3} such that
\e
\begin{gathered}
c_1'=\Psi_{\exp}(h_0)e_1g_1,\quad c_2'=\Psi_{\exp}(h_n)f_ng_n, \\ 
\text{and}\quad e_{i+1}g_{i+1}=\Psi_{\exp}(h_i)f_ig_i, \quad i=1,\ldots,n-1.
\end{gathered}
\label{cc4eq20}
\e
It is not hard to show that for any element $h\in I$, then the conditions $\Phi_i(e_i)-\Phi_i(f_i)\in I$ and $x\ci \Phi_i(e_i)\ne 0\in\R$ hold if and only if $\Phi_i(\Psi_{\exp}(h)e_i)-\Phi_i(f_i)\in I$ and $x\ci \Phi_i(\Psi_{\exp}(h)e_i)\ne 0\in\R$ hold. So we can remove the $h_i$ in \eq{cc4eq20}. We have that $\pi_{x,\rex}(c_1')=\pi_{x,\rex}(c_2')$ if and only if there are $e_i,f_i,g_i\in \fC_\rex$ such that
\e
\begin{gathered}
c_1'=e_1g_1,\quad c_2'= f_ng_n, 
\quad e_{i+1}g_{i+1}= f_ig_i, \quad i=1,\ldots,n-1,\\
\text{and} \quad x\ci \Phi_i(e_i)\ne 0\in\R, \quad i=1,\dots, n.
\end{gathered}
\label{cc4eq21}
\e

We define $a'=f_1f_2\cdots f_n$ and $b'=e_1e_2\cdots e_n$. Then using \eq{cc4eq21}, we see that $a'c_1'=b'c_2'$, $\Phi_i(a')-\Phi_i(b')\in I$ and $x\ci \Phi_i(a')\ne 0$ as required. 

For the reverse argument, say we have $a',b'\in \fC_\rex$ with $\Phi_i(a')-\Phi_i(b')\in I$ and $x\ci \Phi_i(a')\ne 0$. Let $n=1$, $e_1=a'$, $g_1=0$ and $f_1=b'$ in \eq{cc4eq20} then we see that $\pi_{x,\rex}(a')=\pi_{x,\rex}(b')$. As $x\ci \Phi_i(a')\ne 0$, then $\pi_{x,\rex}(a')$ is invertible in $\fC_\rex$. If we also have that $a'c_1'=b'c_2'$, then as $\pi_{x,\rex}$ is a morphism, $\pi_{x,\rex}(c_1')=\pi_{x,\rex}(c_2')$ and the result follows. 
\end{proof}

\begin{rem}
\label{cc4rem3}
Suppose $0\ne c'\in\fC_\rex$ with $\pi_{x,\rex}(c')=0$ for some $\R$-point $x:\fC\ra\R$. Then Proposition \ref{cc4prop8} gives $a'\in \fC_\rex$ with $x\ci\Phi_i(a')\ne 0$ and $a'c'=0$. Thus $a',c'$ are zero divisors in $\fC_\rex$. Conversely, if $\bfC$ is interior then $\fC_\rex$ has no zero divisors, so $\pi_{x,\rex}(c')\ne 0$ for all~$0\ne c'\in\fC_\rex$.
\end{rem}

\begin{ex}
\label{cc4ex5}
Let $X$ be a manifold with corners, or a manifold with g-corners, and $x\in X$. Define a $C^\iy$-ring with corners $\bfC_x=(\fC,\fC_\rex)$ such that $\fC$ is the set of germs at $x$ of smooth functions $c:X\ra\R$, and $\fC_\rex$ is the set of germs at $x$ of exterior functions~$c':X\ra[0,\iy)$. 

That is, elements of $\fC$ are $\sim$-equivalence classes $[U,c]$ of pairs $(U,c)$, where $U$ is an open neighbourhood of $x$ in $X$ and $c:U\ra\R$ is smooth, and $(U,c)\sim(\ti U,\ti c)$ if there exists an open neighbourhood $\hat U$ of $x$ in $U\cap\ti U$ with $c\vert_{\hat U}=\ti c\vert_{\hat U}$. Similarly, elements of $\fC_\rex$ are equivalence classes $[U,c']$, where $U$ is an open neighbourhood of $x$ in $X$ and $c':U\ra[0,\iy)$ is exterior. The $C^\iy$-operations $\Phi_f,\Psi_g$ are defined as in \eq{cc4eq10}--\eq{cc4eq11}, but for germs.

As the set of germs only depends on the local behaviour, the set of germs at $x$ of exterior functions is equal to the set of germs at $x$ of interior functions and the zero function. Hence $\bfC_x$ is an interior $C^\iy$-ring with corners.

There is a morphism $\pi:\fC\ra\R$ mapping $\pi:[U,c]\mapsto c(x)$. If $\pi([U,c])\ne 0$ then $[U,c]$ is invertible in $\fC$, and if $\pi\ci\Phi_i([U,c'])=c'(x)\ne 0$ then $[U,c']$ is invertible in $\fC_\rex$. Thus $\bfC_x$ is a local $C^\iy$-ring with corners. Write~$\bs C_x^\iy(X)=\bfC_x$. 

Example \ref{cc4ex3}(a) defines a $C^\iy$-ring with corners $\bs C^\iy(X)$. The localization $\bs\pi_{x_*}:(\bs C^\iy(X)){}_{x_*}$ of $\bs C^\iy(X)$ at $x_*:C^\iy(X)\ra\R$ from Definition \ref{cc4def11} is also a local $C^\iy$-ring with corners. There is a morphism $\bs\pi_x:\bs C^\iy(X)\ra\bs C^\iy_x(X)$ mapping $c\mapsto [X,c],$ $c'\mapsto [X,c']$. By the universal property of $(\bs C^\iy(X)){}_{x_*}$ we have $\bs\pi_x=\bs\la_x\ci \bs\pi_{x_*}$ for a unique morphism~$\bs\la_x:(\bs C^\iy(X)){}_{x_*}\ra\bs C_x^\iy(X)$.

The proof of \cite[Th.~4.41]{Joyc6} implies that the $C^\iy$-ring morphism $\la_x$ in $\bs\la_x=(\la_x,\la_{x,\rex})$ is always an isomorphism. But $\la_{x,\rex}$ need not be an isomorphism, as the next proposition shows.
\end{ex}

\begin{prop}
\label{cc4prop9}
In Example\/ {\rm\ref{cc4ex5},} suppose $X$ is a manifold with faces, as in Definition\/ {\rm\ref{cc3def7}}. Then $\bs\la_x$ is an isomorphism for all\/~$x\in X$. 

If\/ $X$ is a manifold with corners, but not a manifold with faces, then $\bs\la_x$ is not an isomorphism for some\/~$x\in X$.
\end{prop}

\begin{proof} Suppose $X$ is a manifold with faces, and $x\in X$, giving $\bs\la_x=(\la_x,\la_{x,\rex})$. Then $\la_x$ is an isomorphism as in Example \ref{cc4ex5}. We will show $\la_{x,\rex}$ is both surjective and injective, so that $\bs\la_x$ is an isomorphism.

For surjectivity, we need to show that for any $[U,c']\in C^\iy_{x,\rex}(X)$ there exists $c''\in C^\iy_\rex(X)$ with $c'=c''$ near $x$ in $X$, so that $c'=\pi_{x,\rex}(c'')=\la_{x,\rex}(\pi_{x_*,\rex}(c''))$. If $[U,c']=0$ we take $c'=0$, so suppose $[U,c']\ne 0$, and take $c''$ to be interior. As in Remark \ref{cc3rem3}, such $c''$ has a locally constant map $\mu_{c''}:\pd X\ra\N$ such that $c''$ vanishes to order $\mu_{c''}$ along $\pd X$ locally, and $c''>0$ on $X^\ci$. We choose $\mu_{c''}$ such that $\mu_{c''}=\mu_{c'}$ on any connected component of $\pd X$ which meets $x$, and $\mu_{c''}=0$ on other components. Then local choices of interior $c''$ on $X$ with prescribed vanishing order $\mu_{c''}$ on $\pd X$ can be combined using a partition of unity on $X$, with the local choice $c''=c'$ near $x$, to get $c''$ as required.

For injectivity, let $c_1',c_2'\in C^\iy_\rex(X)$ with $\pi_{x,\rex}(c_1')=\pi_{x,\rex}(c_2')$ in $C^\iy_{x,\rex}(X)$, which means that $c'_1=c'_2$ near $x$ in $X$. We must show that $\pi_{x_*,\rex}(c'_1)=\pi_{x_*,\rex}(c'_2)$ in $C^\iy_{x_*,\rex}(X)$. By Proposition \ref{cc4prop8}, this holds if there are $a',b'\in C^\iy_\rex(X)$ such that $a'=b'$ near $x$ in $X$, and $a'(x)\ne 0$, and~$a'c_1'=b'c_2'$. 

Let $\ti X$ be the connected component of $X$ containing $x$. If $c_1'=0$ near $x$ then $c_1'\vert_{\ti X}=c_2'\vert_{\ti X}$, and we can take $a'=b'=1$ on $\ti X$ and $a'=b'=0$ on $X\sm\ti X$. Otherwise $c_1',c_2'$ are interior on $\ti X$, so $\mu_{c'_1},\mu_{c_2'}$ map $\pd\ti X\ra\N$. Define $a',b'$ to be zero on $X\sm\ti X$. On $\ti X$, define locally constant $\mu_{a'},\mu_{b'}:\pd\ti X\ra\N$ by $\mu_{a'}=\max(\mu_{c_2'}-\mu_{c_1'},0)$ and $\mu_{a'}=\max(\mu_{c_1'}-\mu_{c_2'},0)$, so that $\mu_{a'}+\mu_{c_1'}=\mu_{b'}+\mu_{c_2'}$ on $\pd\ti X$, and $\mu_{a'},\mu_{b'}=0$ near $\Pi_{\ti X}^{-1}(x)$. Let $a'\vert_{\ti X}:\ti X\ra[0,\iy)$ be arbitrary interior with order of vanishing $\mu_{a'}$ on $\pd\ti X$. Then there is a unique interior $b'\vert_{\ti X}:\ti X\ra[0,\iy)$ with $a'c_1'=b'c_2'$, defined by $b'=a'c_1'(c_2')^{-1}$ on $\ti X^\ci$, and extended continuously over $\ti X$. These $a',b'$ show $\la_{x,\rex}$ is injective.

If $X$ is a manifold with corners, but not a manifold with faces, then there exist $x\in X$ and distinct $x'_1,x'_2\in\Pi_X^{-1}(x)$ which lie in the same connected component $\pd_iX$ of $\pd X$. Then any $c''\in C^\iy_\rex(X)$ has $\mu_{c''}(x'_1)=\mu_{c''}(x_2')$, but there exists $[U,c']\in C^\iy_{x,\rex}(X)$ with $\mu_{c'}(x'_1)\ne \mu_{c'}(x_2')$, so $[U,c']$ does not lie in the image of $\pi_{x,\rex}=\la_{x,\rex}\ci\pi_{x_*,\rex}$, and hence not in the image of $\la_{x,\rex}$, as $\pi_{x_*,\rex}$ is surjective. Thus $\la_{x,\rex}$ is not surjective, and $\bs\la_x$ not an isomorphism. 
\end{proof}

The proposition justifies the following definition in the g-corners case:

\begin{dfn}
\label{cc4def12}
A manifold with g-corners $X$ is called a {\it manifold with g-faces\/} if $\bs\la_x:(\bs C^\iy(X)){}_{x_*}\ra\bs C_x^\iy(X)$ in Example \ref{cc4ex5} is an isomorphism for all~$x\in X$.

\end{dfn}

In \S\ref{cc54} we will show that a manifold with g-corners defines an {\it affine\/} $C^\iy$-scheme with corners if and only if it is a manifold with g-faces.

\begin{ex}
\label{cc4ex6}
For any weakly toric monoid $P$, one can show using an argument similar to the proof of Proposition \ref{cc4prop9} that $X_P$ in Definition \ref{cc3def4} is a manifold with g-faces. Since every manifold with g-corners $X$ can be covered by open neighbourhoods diffeomorphic to $X_P$, it follows that every manifold with g-corners can be covered by open subsets which are manifolds with g-faces.
\end{ex}

\begin{rem}
\label{cc4rem4}
Let us compare Definitions \ref{cc3def8} and \ref{cc4def12}. For a manifold with g-corners $X$ to be a manifold with g-faces it is necessary, but not sufficient, that $\Pi_X\vert_{\pd_iX}:\pd_iX\ra X$ is injective for each connected component $\pd_iX$ of $\pd X$. This is because not every locally constant map $\mu_{c'}:\pd X\ra\N$ can be the vanishing multiplicity map of an interior morphism $c':X\ra[0,\iy)$, as g-corner strata of codimension $\ge 3$ can impose extra consistency conditions on $\mu_{c'}$. The first author \cite[Ex.~5.5.4]{Fran} gives an example of a manifold with g-corners with $\Pi_X\vert_{\pd_iX}$ injective for all components $\pd_iX\subset\pd X$, which is not a manifold with g-faces.
\end{rem}

\subsection{\texorpdfstring{Special classes of $C^\iy$-rings with corners}{Special classes of C∞-rings with corners}}
\label{cc47} 

We define and study several classes of $C^\iy$-rings with corners. Throughout, if $\bfC=(\fC,\fC_\rex)$ is an (interior) $C^\iy$-ring with corners, then we regard $\fC_\rex$ (and $\fC_\rin$) as monoids under multiplication, as in~\S\ref{cc32}.

\begin{dfn}
\label{cc4def13}
Let $\bfC$ be a $C^\iy$-ring with corners, and see Proposition \ref{cc4prop6}. We call $\bfC$ {\it finitely generated\/} if there exists a surjective morphism $\bs\phi:\bfF^{A,A_\rex}\ra\bfC$ from some free $C^\iy$-ring with corners $\bfF^{A,A_\rex}$ with $A,A_\rex$ finite sets. Equivalently, $\bfC$ is finitely generated if it fits into a coequalizer diagram \eq{cc4eq13} with $A,A_\rex$ finite. We call $\bfC$ {\it finitely presented\/} if it fits into a coequalizer diagram \eq{cc4eq13} with $A,A_\rex,B,B_\rex$ finite. Finitely presented implies finitely generated.

Write $\CRingscfp\subset\CRingscfg\subset\CRingsc$ for the full subcategories of finitely presented, and finitely generated, objects in $\CRingsc$.
\end{dfn}

\begin{prop}
\label{cc4prop10}
Suppose we are given a pushout diagram in\/ {\rm$\CRingsc$:}
\begin{equation*}
\xymatrix@C=60pt@R=14pt{ *+[r]{\bfC} \ar[r]_{\bs\be} \ar[d]^{\bs\al} & *+[l]{\bfE} \ar[d]_{\bs\de} \\
*+[r]{\bfD} \ar[r]^{\bs\ga} & *+[l]{\bfF,\!} }
\end{equation*}
so that\/ $\bfF=\bfD\amalg_\bfC\bfE,$ and\/ $\bfC,\bfD,\bfE$ fit into coequalizer diagrams in\/~{\rm$\CRingsc$:}
\ea
\xymatrix@C=50pt{  \bfF^{B,B_\rex} \ar@<.1ex>@/^.3pc/[r]^{\bs\ep} \ar@<-.1ex>@/_.3pc/[r]_{\bs\ze} & \bfF^{A,A_\rex} \ar[r]^{\bs\phi} & \bfC, }
\label{cc4eq22}\\
\xymatrix@C=50pt{  \bfF^{D,D_\rex} \ar@<.1ex>@/^.3pc/[r]^{\bs\eta} \ar@<-.1ex>@/_.3pc/[r]_{\bs\th} & \bfF^{C,C_\rex} \ar[r]^{\bs\chi} & \bfD, }
\label{cc4eq23}\\
\xymatrix@C=50pt{  \bfF^{F,F_\rex} \ar@<.1ex>@/^.3pc/[r]^{\bs\io} \ar@<-.1ex>@/_.3pc/[r]_{\bs\ka} & \bfF^{E,E_\rex} \ar[r]^{\bs\psi} & \bfE. }
\label{cc4eq24}
\ea
Then $\bfF$ fits into a coequalizer diagram in\/~{\rm$\CRingsc$:}
\e
\begin{gathered}
\xymatrix@C=50pt{  \bfF^{H,H_\rex} \ar@<.1ex>@/^.3pc/[r]^{\bs\la} \ar@<-.1ex>@/_.3pc/[r]_{\bs\mu} & \bfF^{G,G_\rex} \ar[r]^{\bs\om} & \bfF, }\qquad\text{with}\\
G\!=\!C\amalg E,\; G_\rex\!=\!C_\rex\amalg E_\rex,\; H\!=\!A\amalg D\amalg F,
\; H_\rex\!=\!A_\rex\amalg D_\rex\amalg F_\rex.
\end{gathered}
\label{cc4eq25}
\e
The analogue holds in $\CRingscin$ rather than $\CRingsc$.
\end{prop}

\begin{proof} As coproducts exist in $\CRingsc$, by general properties of colimits we may rewrite the pushout $\bfF=\bfD\amalg_\bfC\bfE$ as a coequalizer diagram
\e
\xymatrix@C=50pt{  \bfC \ar@<.5ex>[r]^(0.4){\bs\io_\bfD\ci\bs\al} \ar@<-.5ex>[r]_(0.4){\bs\io_\bfE\ci\bs\be} & \bfD\ot_\iy\bfE \ar[r]^(0.6){(\bs\ga,\bs\de)} & \bfF. }
\label{cc4eq26}
\e
Now in a diagram $\bfH\rra\bfG$, the coequalizer is the quotient of $\bfG$ by relations from each element of $\fH,\fH_\rex$. Given a surjective morphism $\bfI\twoheadrightarrow\bfH$, the coequalizers of $\bfH\rra\bfG$ and $\bfI\rra\bfG$ are the same, as we quotient $\bfG$ by the same set of relations. Hence using \eq{cc4eq22}, in \eq{cc4eq26} we can replace $\bfC$ by $\bfF^{A,A_\rex}$. Also using \eq{cc4eq23}--\eq{cc4eq24} gives a coequalizer diagram
\begin{equation*}
\xymatrix@C=20pt{  \bfF^{A,A_\rex} \ar@<.5ex>[r] \ar@<-.5ex>[r] & \mathop{\rm Coeq}\bigl(\bfF^{D,D_\rex}\rra\bfF^{C,C_\rex}\bigr)\ot_\iy\mathop{\rm Coeq}\bigl(\bfF^{F,F_\rex}\rra\bfF^{E,E_\rex}\bigr) \ar[r] & \bfF. }
\end{equation*}
As coproducts commute with coequalizers, this is equivalent to
\begin{equation*}
\xymatrix@C=20pt{  \bfF^{A,A_\rex} \ar@<.5ex>[r] \ar@<-.5ex>[r] & \mathop{\rm Coeq}\bigl((\bfF^{D,D_\rex}\ot_\iy\bfF^{F,F_\rex})\rra(\bfF^{C,C_\rex}\ot_\iy\bfF^{E,E_\rex})\bigr) \ar[r] & \bfF. }
\end{equation*}
Since $\bfF^{A,A_\rex}$ is free, the two morphisms $\bfF^{A,A_\rex}\ra\mathop{\rm Coeq}(\cdots)$ factor through the surjective morphism $\bfF^{C,C_\rex}\ot_\iy\bfF^{E,E_\rex}\ra\mathop{\rm Coeq}(\cdots)$. Thus we may combine the two coequalizers into one, giving a coequalizer diagram
\e
\xymatrix@C=20pt{ \bfF^{A,A_\rex}\ot_\iy\bfF^{D,D_\rex}\ot_\iy\bfF^{F,F_\rex} \ar@<.5ex>[r] \ar@<-.5ex>[r] & \bfF^{C,C_\rex}\ot_\iy\bfF^{E,E_\rex} \ar[r] & \bfF. }
\label{cc4eq27}
\e
But coproducts of free $C^\iy$-rings with corners are free over the disjoint unions of the generating sets. Thus \eq{cc4eq27} is equivalent to \eq{cc4eq25}. The same argument works in~$\CRingscin$.	
\end{proof}

\begin{prop}
\label{cc4prop11}
$\CRingscfg$ and\/ $\CRingscfp$ are closed under finite colimits in\/~$\CRingsc$.
\end{prop}

\begin{proof} $\bs C^\iy(*)=(\R,[0,\iy))$ is an initial object in $\CRingscfg,\CRingscfp$. As  $\CRingscfg,\CRingscfp$ have an initial object, finite colimits may be written as iterated pushouts, so it is enough to show $\CRingscfg,\CRingscfp$ are closed under pushouts. This follows from Proposition \ref{cc4prop10}, noting that if $A,A_\rex,C,C_\rex,E,E_\rex$ are finite then $G,G_\rex$ are finite, and if $A,\ldots,F_\rex$ are finite then $G,\ldots,H_\rex$ are finite.
\end{proof}

\begin{dfn}
\label{cc4def14}
We call a $C^\iy$-ring with corners $\bfC=(\fC,\fC_\rex)$ {\it firm}\/ if the sharpening $\fC_\rex^\sh$ is a finitely generated monoid. We denote by $\CRingscfi$ the full subcategory of $\CRingsc$ consisting of firm $C^\iy$-rings with corners.

If $\bfC$ is firm, then there are $c_1',\ldots,c_n'$ in $\fC_\rex$ whose images under the quotient $\fC_\rex\ra \fC_\rex^\sh$ generate $\fC_\rex^\sh$. This implies that each element in $\fC_\rex$ can be written as $\Psi_{\exp}(c)c_1^{\prime a_1}\cdots c_n^{\prime a_n}$ for some $c\in \fC$ and $a_i\ge 0$. If there exists a surjective morphism $\bs\phi:\bfF^{A,A_\rex}\ra\bfC$ with $A_\rex$ finite then $\bfC$ is firm, as $[\phi_\rex(y_{a'})]$ for $a'\in A_\rex$ generate $\fC_\rex^\sh$. Hence~$\CRingscfp\subset\CRingscfg\subset\CRingscfi$.
\end{dfn}

\begin{prop}
\label{cc4prop12}
$\CRingscfi$ is closed under finite colimits in\/~$\CRingsc$.
\end{prop}

\begin{proof} As in the proof of Proposition \ref{cc4prop11}, it is enough to show $\CRingscfi$ is closed under pushouts in $\CRingsc$. Take $\bfC,\bfD,\bfE\in \CRingscfi$ with morphisms $\bfC\ra\bfD$ and $\bfC\ra\bfE$, and consider the pushout $\bfD\amalg_{\bfC}\bfE$, with its morphisms $\bs\phi:\bfD\ra \bfD\amalg_{\bfC}\bfE$ and $\bs\psi:\bfE\ra \bfD\amalg_{\bfC}\bfE$. Then every element of $(\bfD\amalg_{\bfC}\bfE)_\rex$ is of the form 
\begin{align*}
\Psi_f(&\phi(d_1),\ldots,\phi(d_m),\psi(e_1),\ldots, \psi(e_n),\\
&\phi_\rex(d_1'),\ldots,\phi_\rex(d_k'), \psi_\rex(e_1'),\ldots, \psi_\rex(e_l')),
\end{align*} 
for smooth $f:\R^{m+n}\t[0,\iy)^{k+l}\ra [0,\iy)$, where $d_i\in \fD, d_i'\in \fD_\rex, e_i\in\fE,e_i'\in \fE_\rex$, and $d_i'$ are generators of the sharpening of $\fD_\rex$, and $e_i'$ generate the sharpening of $\fE_\rex$. If $f=0$ this gives 0, otherwise we may write 
\begin{equation*}
f(x_1,\ldots, x_{m+n},y_1,\ldots, y_{k+l})=y_1^{a_1}\ldots y_{k+l}^{a_{k+l}}e^{F(x_1,\ldots, x_{m+n},y_1,\ldots, y_{k+l})},
\end{equation*}
for $F:\R^{m+n}\!\t\![0,\iy)^{k+l}\!\ra\!\R$ smooth. In $(\bfD\!\amalg_{\bfC}\!\bfE)^\sh_\rex$, the above element maps~to 
\begin{equation*}
[\phi_\rex(d_1')]^{a_1}\cdots [\phi_\rex(d_k')]^{a_k}[\psi_\rex(e_1')]^{a_{k+1}}\cdots [\psi_\rex(e_l')]^{a_{k+l}}.
\end{equation*}
Hence $(\bfD\amalg_{\bfC}\bfE)_\rex^\sh$ is generated by the images of the generators of $\fD_\rex^\sh$ and $\fE_\rex^\sh$, and zero, and so $(\bfD\amalg_{\bfC}\bfE)_\rex^\sh$ is finitely generated. Thus $\CRingscfi$ is closed under pushouts.
\end{proof}

\begin{dfn}
\label{cc4def15}
Suppose $\bfC=(\fC,\fC_\rex)$ is an interior $C^\iy$-ring with corners, and let $\fC_\rin\subseteq\fC_\rex$ be the submonoid of \S\ref{cc42}, and $\fC_\rin^\sh$ its sharpening. Then:
\begin{itemize}
\setlength{\itemsep}{0pt}
\setlength{\parsep}{0pt}
\item[(i)] We call $\bfC$ {\it integral\/} if $\fC_\rin$ is an integral monoid.
\item[(ii)] We call $\bfC$ {\it torsion-free\/} if $\fC_\rin$ is an integral, torsion-free monoid.
\item[(iii)] We call $\bfC$ {\it saturated\/} if it is integral, and $\fC_\rin$ is a saturated monoid. Note that $\fC_\rin^\t\cong\fC$ as abelian groups since $\bfC$ is a $C^\iy$-ring with corners, so $\fC_\rin^\t$ is torsion-free. Therefore $\bfC$ saturated implies that $\bfC$ is torsion-free. 

\item[(iv)] We call $\bfC$ {\it toric\/} if it is saturated and firm. This implies that $\fC_\rin^\sh$ is a toric monoid. 
\item[(v)] We call $\bfC$ {\it simplicial\/} if it is toric and $\fC_\rin^\sh\cong\N^k$ for some $k\in\N$.
\end{itemize}
We will write $\CRingscDe\subset\CRingscto\subset \CRingscsa\subset\CRingsctf\subset \CRingscZ\subset\CRingscin$ for the full subcategories of simplicial, toric, saturated, torsion-free, and integral objects in $\CRingscin$.
\end{dfn}

\begin{ex}
\label{cc4ex7}
Let $X$ be a manifold with corners, and $\bs C^\iy_\rin(X)$ be the interior $C^\iy$-ring with corners from Example \ref{cc4ex3}(b). Let $S$ be the set of connected components of $\pd X$. For each $F\in S$, we choose an interior map $c_F:X\ra[0,\iy)$ which vanishes to order 1 on $F$, and to order zero on $\pd X\sm F$, such that $c_F=1$ outside a small neighbourhood $U_F$ of $\Pi_X(F)$ in $X$, where we choose $\{U_F:F\in S\}$ to be locally finite in $X$. Then every interior map $g:X\ra[0,\iy)$ may be written uniquely as $g=\exp(f)\cdot\prod_{F\in S}c_F^{a_F}$, for $f\in C^\iy(X)$ and $a_F\in\N$, $F\in S$. 

Hence as monoids we have $\In(X)\cong C^\iy(X)\t\N^S$. Therefore $\bs C^\iy_\rin(X)$ is integral, torsion-free, and saturated, and it is simplicial and toric if and only if $\pd X$ has finitely many connected components. A more complicated proof shows that if $X$ is a manifold with g-corners then $\bs C^\iy_\rin(X)$ is integral, torsion-free, and saturated, and it is toric if $\pd X$ has finitely many connected components.
\end{ex}

The functor $\Pi_\rin:\CRingscin\ra\Sets$ mapping $\bfC\mapsto\fC_\rin$ may be enhanced to a functor $\bar\Pi_\rin:\CRingscin\ra\Mon$ by regarding $\fC_\rin$ as a monoid, and then Theorem \ref{cc4thm3}(a) implies that $\bar\Pi_\rin$ preserves limits and directed colimits.

Write $\Mon_{\bf sa,tf,\Z}\subset\Mon_{\bf tf,\Z}\subset\Mon_\Z\subset\Mon$ for the full subcategories of saturated, torsion-free integral, and torsion-free integral, and integral monoids. They are closed under limits and directed colimits in $\Mon$. Thus we deduce:

\begin{prop} $\CRingscsa,\CRingsctf$ and\/ $\CRingscZ$ are closed under limits and under directed colimits in $\CRingscin$. Thus, all small limits and directed colimits exist in $\CRingscsa,\CRingsctf,\CRingscZ$.
\label{cc4prop13}
\end{prop}

In a similar way to Proposition \ref{cc4prop5} we prove:

\begin{thm}
\label{cc4thm5}
There are reflection functors $\Pi_\rin^\Z,\Pi_\Z^\tf,\Pi_\tf^\sa,\Pi_\rin^\sa$ in a diagram
\begin{equation*}
\xymatrix@C=22pt{ \CRingscsa \ar@<-.5ex>[r]_{\inc} &
\CRingsctf \ar@<-.5ex>[r]_{\inc} \ar@<-.5ex>[l]_{\Pi_\tf^\sa} & \CRingscZ \ar@<-.5ex>[r]_{\inc} \ar@<-.5ex>[l]_{\Pi_\Z^\tf} & \CRingscin, \ar@<-.5ex>[l]_{\Pi_\rin^\Z} \ar@<-3ex>@/_.5pc/[lll]_{\Pi_\rin^\sa} }
\end{equation*}
such that each of\/ $\Pi_\rin^\Z,\Pi_\Z^\tf,\Pi_\tf^\sa,\Pi_\rin^\sa$ is left adjoint to the corresponding inclusion functor\/~$\inc$.
\end{thm}

\begin{proof} Let $\bfC$ be an object in $\CRingscin$. We will construct an object $\bfD=\Pi_\rin^\Z(\bfC)$ in $\CRingscZ$ and a projection $\bs\pi:\bfC\ra\bfD$, with the property that if $\bs\phi:\bfC\ra\bfE$ is a morphism in $\CRingscin$ with $\bfE\in\CRingscZ$ then $\bs\phi=\bs\psi\ci\bs\pi$ for a unique morphism $\bs\psi:\bfD\ra\bfE$. Consider the diagram:
\e
\begin{gathered}
\xymatrix@C=45pt@R=15pt{ \bfC=\bfC^0 \ar@<2ex>@/^.6pc/[rrrr]^{\bs\pi=\bs\pi^0} \ar@/_.8pc/[drrrr]_(0.1){\bs\phi=\bs\phi^0} \ar[r]_(0.8){\bs\al^0} & \bfC^1 \ar[r]_(0.8){\bs\al^1} \ar@/_.6pc/[drrr]^(0.36){\bs\phi^1} \ar@<1ex>@/^.2pc/[rrr]^(0.2){\bs\pi^1} & \bfC^2 \ar[r]_(0.8){\bs\al^2} \ar@/_.4pc/[drr]^(0.5){\bs\phi^2} & \cdots \ar[r] & \bfD \ar[d]_{\bs\psi} \\
&&&&  \bfE. }
\end{gathered}
\label{cc4eq28}
\e

Define $\bfC^0=\bfC$ and $\bs\phi^0=\bs\phi$. By induction on $n=0,1,\ldots,$ if $\bfC^n,\bs\phi^n$ are defined, define an object $\bfC^{n+1}\in\CRingscin$ and morphisms $\bs\al^n:\bfC^n\ra\bfC^{n+1}$, $\bs\phi^{n+1}:\bfC^{n+1}\ra\bfE$ as follows. We have a monoid $\fC^n_\rin$, which as in \S\ref{cc32} has an abelian group $(\fC^n_\rin)^\gp$ with projection $\pi^\gp:\fC^n_\rin\ra(\fC^n_\rin)^\gp$, where $\bfC^n_\rin,\bfC^n$ are integral if $\pi^\gp$ is injective. Using the notation of Definition \ref{cc4def8}, we define
\e
\bfC^{n+1}=\bfC^n\big/\bigl[\text{$c'=c''$ if $c',c''\in \fC^n_\rin$ with $\pi^\gp(c')=\pi^\gp(c'')$}\bigr].
\label{cc4eq29}
\e
Write $\bs\al^n:\bfC^n\ra\bfC^{n+1}$ for the natural surjective projection. Then $\bfC^{n+1},\bs\al^n$ are both interior, since the relations $c'=c''$ in \eq{cc4eq29} are all interior.

We have a morphism $\bs\phi^n:\bfC^n\ra\bfE$ with $\bfE$ integral, so by considering the diagram with bottom morphism injective
\begin{equation*}
\xymatrix@C=90pt@R=15pt{ *+[r]{\fC_\rin^n} \ar[r]_(0.45){\pi^\gp} \ar[d]^{\phi^n_\rin} & *+[l]{(\fC^n_\rin)^\gp} \ar[d]_{(\phi^n_\rin)^\gp} \\
*+[r]{\fE_\rin} \ar@{^{(}->}[r]^(0.45){\pi^\gp} & *+[l]{(\fE_\rin)^\gp,\!} }
\end{equation*}
we see that if $c',c''\in \fC^n_\rin$ with $\pi^\gp(c')=\pi^\gp(c'')$ then $\phi^n_\rin(c')=\phi^n_\rin(c'')$. Thus by the universal property of \eq{cc4eq29}, there is a unique morphism $\bs\phi^{n+1}:\bfC^{n+1}\ra\bfE$ with $\bs\phi^n=\bs\phi^{n+1}\ci\bs\al^n$. This completes the inductive step, so we have defined $\bfC^n,\bs\al^n,\bs\phi^n$ for all $n=0,1,\ldots,$ where $\bfC^n,\bs\al^n$ are independent of~$\bfE,\bs\phi$.

Now define $\bfD$ to be the directed colimit $\bfD=\underrightarrow\lim_{n=0}^\iy\bfC^n$ in $\CRingscin$, using the morphisms $\bs\al^n:\bfC^n\ra\bfC^{n+1}$. This exists by Theorem \ref{cc4thm3}(a), and commutes with $\Pi_\rin:\CRingscin\ra\Sets$, where we can think of $\Pi_\rin$ as mapping to monoids. It has a natural projection $\bs\pi:\bfC\ra\bfD$, and also projections $\bs\pi^n:\bfC^n\ra\bfD$ for all $n$. By the universal property of colimits, there is a unique morphism $\bs\psi$ in $\CRingscin$ making \eq{cc4eq28} commute.

The purpose of the quotient \eq{cc4eq29} is to modify $\bfC^n$ to make it integral, since if $\bfC^n$ were integral then $\pi^\gp(c')=\pi^\gp(c'')$ implies $c'=c''$. It is not obvious that $\bfC^{n+1}$ in \eq{cc4eq29} is integral, as the quotient modifies $(\fC^n_\rin)^\gp$. However, the direct limit $\bfD$ is integral. To see this, suppose $d',d''\in\fD_\rin$ with $\pi^\gp(d')=\pi^\gp(d'')$ in $(\fD_\rin)^\gp$. Since $\fD_\rin=\underrightarrow\lim_{m=0}^\iy\fC^m_\rin$ in $\Mon$, for $m\gg 0$ we may write $d'=\pi^m_\rin(c')$, $d''=\pi^m_\rin(c'')$ for $c',c''\in\fC^m_\rin$. As $(\fD_\rin)^\gp=\underrightarrow\lim_{n=0}^\iy(\fC^n_\rin)^\gp$ and $\pi^\gp(d')=\pi^\gp(d'')$, for some $n\gg m$ we have
\begin{equation*}
\pi^\gp\ci\al^{n-1}_\rin\ci\cdots\ci\al^m_\rin(c')=\pi^\gp\ci\al^{n-1}_\rin\ci\cdots\ci\al^m_\rin(c'')\quad\text{in $(\fC^n_\rin)^\gp$.}
\end{equation*}
But then \eq{cc4eq29} implies that $\al^n_\rin\ci\cdots\ci\al^n_\rin(c')=\al^n_\rin\ci\cdots\ci\al^m_\rin(c'')$ in $\fC^{n+1}_\rin$, so $d'=d''$. Therefore $\pi^\gp:\fD_\rin\ra(\fD_\rin)^\gp$ is injective, and $\bfD$ is integral.

Set $\Pi_\rin^\Z(\bfC)=\bfD$. If $\bs\xi:\bfC\ra\bfC'$ is a morphism in $\CRingscin$, by taking $\bfE=\Pi_\rin^\Z(\bfC')$ and $\bs\phi=\bs\pi'\ci\bs\xi$ in \eq{cc4eq28} we see that there is a unique morphism $\Pi_\rin^\Z(\bs\xi)$ in $\CRingscZ$ making the following commute:
\begin{equation*}
\xymatrix@C=100pt@R=15pt{ *+[r]{\bfC} \ar[d]^{\bs\xi} \ar[r]_(0.45){\bs\pi} & *+[l]{\Pi_\rin^\Z(\bfC)} \ar[d]_{\Pi_\rin^\Z(\bs\xi)} \\  
*+[r]{\bfC'} \ar[r]^(0.45){\bs\pi'} & *+[l]{\Pi_\rin^\Z(\bfC').\!\!} }
\end{equation*}
This defines the functor $\Pi_\rin^\Z$. For any $\bfE\in\CRingscZ$, the correspondence between $\bs\phi$ and $\bs\psi$ in \eq{cc4eq28} implies that we have a natural bijection
\begin{equation*}
\Hom_\CRingscin\bigl(\bfC,\inc(\bfE)\bigr)\cong\Hom_\CRingscZ\bigl(\Pi_\rin^\Z(\bfC),\bfE\bigr).
\end{equation*}
This is functorial in $\bfC,\bfE$, and so $\Pi_\rin^\Z$ is left adjoint to $\inc:\CRingscZ\hookra\CRingscin$, as we have to prove.

The constructions of $\Pi_\Z^\tf,\Pi_\tf^\sa$ are very similar. For $\Pi_\Z^\tf$, if $\bfC$ is an object in $\CRingscZ$, the analogue of \eq{cc4eq29} is
\begin{equation*}
\bfC^{n+1}=\bfC^n\big/\bigl[\text{$c'=c''$ if $c',c''\in \fC^n_\rin$ with $\pi^{\rm tf}(c')=\pi^{\rm tf}(c'')$}\bigr],
\end{equation*}
where $\pi^{\rm tf}:\fC^n_\rin\ra(\fC^n_\rin)^\gp/\text{torsion}$ is the natural projection. For $\Pi_\tf^\sa$, if $\bfC$ is an object in $\CRingsctf$, the analogue of \eq{cc4eq29} is
\begin{align*}
\bfC^{n+1}=\bfC^n\bigl[&s_{c'}:\text{$c'\in \fC^n_\rin\subseteq(\fC^n_\rin)^\gp$ and there exists $c''\in (\fC^n_\rin)^\gp\sm \fC^n_\rin$}\\
&\text{with $c'=n_{c'}\cdot c''$, $n_c'=2,3,\ldots$}\bigr]\big/\bigl[\text{$n_{c'}\cdot s_{c'}=c'$, all $c',n_{c'},s_{c'}$}\bigr].
\end{align*}
Finally we set $\Pi_\rin^\sa=\Pi_\tf^\sa\ci\Pi_\Z^\tf\ci\Pi_\rin^\Z$. This completes the proof.
\end{proof}

As for Theorem \ref{cc4thm3}(b), we deduce:

\begin{cor}
\label{cc4cor1}
Small colimits exist in\/~$\CRingscsa,\!\CRingsctf,\!\CRingscZ$.
\end{cor}

In Definition \ref{cc4def8} we explained how to modify a $C^\iy$-ring with corners $\bfC$ by adding generators $\bfC(x_a:a\in A)[y_{a'}:a'\in A_\rex]$ and imposing relations $\bfC/(f_b=0:b\in B)[g_{b'}=h_{b'}:b'\in B_\rex]$. This is just notation for certain small colimits in $\CRingsc$ or $\CRingscin$. Corollary \ref{cc4cor1} implies that we can also add generators and relations in $\CRingscsa,\CRingsctf,\CRingscZ$, provided the relations $g_{b'}=h_{b'}$ are interior, that is,~$g_{b'},h_{b'}\ne 0_{\fC_\rex}$.

\begin{prop}
\label{cc4prop14}
$\CRingscto$ is closed under finite colimits in\/~$\CRingscsa$.
\end{prop}

\begin{proof}
Let $J$ be a finite category and $\bs F:J\ra\CRingscto$ a functor. Write $\bfC=\varinjlim\bs F$ for the colimit in $\CRingsc$. As $\CRingscto\subset\CRingscfi$, Proposition \ref{cc4prop12} says $\bfC$ is firm. Now $\Pi_\rin^\sa(\bfC)$ is the colimit $\varinjlim\bs F$ in $\CRingscsa$, as $\Pi_\rin^\sa$ is a reflection functor. By construction in the proof of Theorem \ref{cc4thm5} we see that $\Pi_\rin^\sa$ takes firm objects to firm objects, so $\Pi_\rin^\sa(\bfC)$ is firm, and hence toric, as toric means saturated and firm.
\end{proof}

\section{\texorpdfstring{$C^\iy$-schemes with corners}{C∞-schemes with corners}}
\label{cc5}

We extend the theory of $C^\iy$-schemes in Chapter \ref{cc2} to $C^\iy$-schemes with corners.

\subsection{\texorpdfstring{(Local) $C^\iy$-ringed spaces with corners}{(Local) C∞-ringed spaces with corners}}
\label{cc51}

\begin{dfn}
\label{cc5def1}
A {\it $C^\iy$-ringed space with corners\/} $\bX=(X,\bs{\O}_X)$ is a topological space $X$ with a sheaf $\bs{\O}_X$ of $C^\iy$-rings with corners on $X$. That is, for each open set $U\subseteq X$, then $\bs{\O}_X(U)=(\O_X(U),\O^{\rex}_{X}(U))$ is a $C^\iy$-ring with corners and $\bs{\O}_X$ satisfies the sheaf axioms in~\S\ref{cc23}.

A {\it morphism\/} $\bs f=(f,\bs f^\sh):(X,\bO_X)\ra (Y,\bO_Y)$ of $C^\iy$-ringed spaces with corners is a continuous map $f:X\ra Y$ and a morphism $\bs f^\sh=(f^\sh,f^\sh_\rex):f^{-1}(\bs{\O}_Y)\ra\bs{\O}_X$ of sheaves of $C^\iy$-rings with corners on $X$, for $f^{-1}(\bs{\O}_Y)=(f^{-1}(\O_Y),f^{-1}(\O_Y^\rex))$ as in Definition \ref{cc2def13}. Note that $\bs f^\sh$ is adjoint to a morphism $\bs f_\sh:\bO_Y\ra f_*(\bO_X)$ on $Y$ as in~\eq{cc2eq8}.

A {\it local\/ $C^\iy$-ringed space with corners\/} $\bX=(X,\bs{\O}_X)$ is a
$C^\iy$-ringed space for which the stalks $\bs{\O}_{X,x}=(\O_{X,x},\O^{\rex}_{X,x})$ of $\bs{\O}_X$ at $x$ are local $C^\iy$-rings with corners for all $x\in X$. We define morphisms of local $C^\iy$-ringed spaces with corners $(X,\bs{\O}_X),(Y,\bs{\O}_Y)$ to be morphisms of $C^\iy$-ringed spaces with corners, without any additional locality condition. 

Write $\CRSc$ for the category of $C^\iy$-ringed spaces with corners, and $\LCRSc$ for the full subcategory of local $C^\iy$-ringed spaces with corners. 

For brevity, we will use the notation that bold upper case letters $\bX,\bY,\bZ,\ldots$ represent $C^\iy$-ringed spaces with corners $(X,\bs{\O}_X)$, $(Y,\bs{\O}_Y)$, $(Z,\bs{\O}_Z)$, $\ldots,$ and bold lower case letters $\bs f,\bs g,\ldots$ represent morphisms of $C^\iy$-ringed spaces with corners $(f,\bs{f^\sh}),\ab(g,\bs{g^\sh}),\ab\ldots.$ When we write `$x\in\bX$' we mean that $\bX=(X,\bs{\O}_X)$ and $x\in X$. When we write `$\bs{U}$ {\it is open in\/} $\bX$' we mean that $\bs{U}=(U,\bs{\O}_U)$ and $\bX=(X,\bs{\O}_X)$ with $U\subseteq X$ an open set and $\bs{\O}_U=\bs{\O}_X\vert_U$.
\end{dfn}

Let $\bX=(X,\O_X,\O_X^\rex)\in \LCRSc$, and let $U$ be open in $X$. Take elements $s\in \O_X(U)$ and $s'\in \O_X^\rex(U)$. Then $s$ and $s'$ induce functions $s:U\ra \R$, $s':U\ra [0,\iy)$, that at each $x\in U$ are the compositions 
\begin{equation*}
\O_X(U)\xrightarrow{\rho_{X,x}} \O_{X,x}\xrightarrow{\pi_x} \R, \text{ and } \O_X^\rex(U)\xrightarrow{\rho^\rex_{X,x}} \O^\rex_{X,x}\xrightarrow{\pi^\rex_x} [0,\iy).
\end{equation*}
Here, $\rho_{X,x},\rho^\rex_{X,x}$ are the restriction morphism to the stalks, and $\pi_x,\pi^\rex_x$ are the unique morphisms that exist as $\O_{X,x}$ is local for each $x\in X$, as in Definition \ref{cc4def10} and Lemma \ref{cc4lem4}. We denote by $s(x)$ and $s'(x)$ the values of $s:U\ra \R$ and $s':U\ra[0,\iy)$ respectively at the point $x\in U$. We denote by $s_x\in \O_{X,x}$ and $s'_x\in \O_{X,x}^\rex$ the values of $s$ and $s'$ under the restriction morphisms to the stalks $\rho_{X,x}$ and $\rho^\rex_{X,x}$ respectively.

\begin{dfn}
\label{cc5def2}
Let $\bX=(X,\bO_X)$ be a $C^\iy$-ringed space with corners. We call $\bX$ an \/{\it interior}\/  $C^\iy$-ringed space with corners if one (hence all) of the following conditions hold, which are equivalent by Lemma~\ref{cc5lem1}:
\begin{itemize}
\setlength{\itemsep}{0pt}
\setlength{\parsep}{0pt}
\item[(a)] For all open $U\subseteq X$ and each $s'\in \O_X^\rex(U)$, then $U_{s'}=\{x\in U:s'_x\ne0\in \O_{X,x}^\rex\}$, which is always closed in $U$, is open in $U$, and the stalks $\bO_{X,x}$ are interior $C^\iy$-rings with corners.
\item[(b)] For all open $U\subseteq X$ and each $s'\in \O_X^\rex(U)$, then $U\setminus U_{s'}=\hat U_{s'}=\{x\in U:s'_x= 0\in \O_{X,x}^\rex\}$, which is always open in $U$, is closed in $U$, and the stalks $\bO_{X,x}$ are interior $C^\iy$-rings with corners.
\item[(c)] $\O_X^\rex$ is the sheafification of a presheaf of the form $\O_X^\rin \amalg \{0\}$, where $\O_X^\rin$ is a sheaf of monoids, such that $\bO_X^\rin(U)=(\O_X(U),\O^\rin_X(U)\amalg \{0\})$ is an interior $C^\iy$-ring with corners for each open~$U\subseteq X$. 
\end{itemize}
In each case, we can define a sheaf of monoids $\O_X^\rin$, such that $\O_X^\rin(U)=\bigl\{s'\in \O_X^\rex(U):s'_x\ne 0 \in \O_{X,x}^\rex$ for all $x\in U\bigr\}$. 

We call $\bX$ an \/{\it interior}\/  local $C^\iy$-ringed space with corners if $\bX$ is both a local $C^\iy$-ringed space with corners and an interior $C^\iy$-ringed space with corners.

If $\bX,\bY$ are interior (local) $C^\iy$-ringed spaces with corners, a morphism $\bs f:\bX\ra\bY$ is called \/{\it interior}\/ if the induced maps on stalks $\bs f^\sh_x:\bs{\O}_{Y,f(x)}\ra \bs{\O}_{X,x}$ are interior morphisms of interior $C^\iy$-rings with corners for all $x\in X$. This gives a morphism of sheaves $f^{-1}(\O_Y^\rin)\ra \O_X^\rin$. Write $\CRScin \subset \CRSc$ (and $\LCRScin \subset \LCRSc$) for the non-full subcategories of interior (local) $C^\iy$-ringed spaces with corners and interior morphisms. 
\end{dfn}

\begin{lem}
\label{cc5lem1}
Parts\/ {\rm(a)--(c)} in Definition\/ {\rm\ref{cc5def2}} are equivalent.
\end{lem}

\begin{proof}
Parts (a), (b) are equivalent by definition. The set $\hat U_{s'}$ is open, as the requirement that an element is zero in the stalk is a local requirement. That is, $s'_x=0$ if and only if $s'\vert_{V}=0\in \O_X^\rex(V)$ for some~$x\in V\subseteq U$. 

Suppose (a), (b) hold. Write $\O_X^\rin(U)=\{s'\in \O_X^\rex(U):s'_x\ne 0 \in \O_{X,x}^\rex$ for all $x\in U\}$. If $s'_1,s'_2\in \O_X^\rin(U)$ then $s'_{1,x},s'_{2,x}\ne 0$ for $x\in U$, and as the stalks are interior $s'_{1,x}\cdot s'_{2,x}\ne 0\in \O_{X,x}^\rex$. So $s'_1\cdot s'_2\in \O_X^\rin(U)$, and $\O_X^\rin$ is a monoid. Thus $(\O_X(U),\O_X^\rin(U)\amalg\{0\})$ is a pre-$C^\iy$-ring with corners, where the $C^\iy$-operations come from restriction from $\bO_X(U)$. As the invertible elements of the monoid and the $C^\iy$-rings of $(\O_X(U),\O_X^\rin(U)\amalg\{0\})$ are the same as those from $\bO_X(U)$, it is a $C^\iy$-ring with corners. Let $\hat\O_X^\rex$ be the sheafification of $\O_X^\rin\amalg\{0\}$, which is a subsheaf of $\O_X^\rex$. Note that $ \O_X^\rin(U)\amalg\{0\}$ already satisfies uniqueness, so the sheafification process means $\hat\O_X^\rex$ now satisfies gluing. Then $(\O_X,\hat\O_X^\rex)$ is a sheaf of $C^\iy$-rings with corners. 

There is a morphism $(\id, \id^\sh,\io^\sh_\rex):(X,\O_X,\O_X^\rex)\ra (X,\O_X,\hat\O_X^\rex)$. This is the identity on $X$, $\O_X$. On the sheaves of monoids, we have an inclusion $\io^\sh_\rex(U):\hat\O_X^\rex(U)\ra \O_X^\rex(U)$. On stalks, any non-zero element of $\O_{X,x}^\rex$ is an equivalence class represented by a section $s'\in \O_X^\rex(U)$. As (a) holds, we can choose $U\ni x$ so that $s'_y\ne 0$ for all $y\in U$. Then $s'\in\O_{X}^\rin(U)$, so there is $s''\in \hat\O_{X}^\rex(U)$ mapping to $s'$ under $\io^\sh_\rex(U)$. So $s''_x\mapsto s'_x$, and $\io^\sh_\rex$ is surjective on stalks as $0\mapsto 0$. As $\hat\O_X^\rex$ is a subsheaf of $\O_X^\rex$, $\io^\sh_\rex$ is injective on stalks. Hence $(\id, \id^\sh,\io^\sh_\rex)$ is an isomorphism, and (c) holds. Thus (a), (b) imply~(c).

Now suppose (c) holds. If $s'\in \O_X^\rex(U)$, where $\O_X^\rex$ is the sheafification of $\O_X^\rin\amalg\{0\}$, then if $s_{x}'\ne 0\in \O_{X,x}^\rex$ there is an open $x\in V\subseteq U$ and $s''\in \O_X^\rin(V)$ representing $s'\vert_V$, and therefore $s'_{x'}\ne 0\in \O_{X,x'}^\rex$ for all $x'\in V$, and the $U_{s'}$ defined in (a) is open. If $s_1',s_2'\in \O_X^\rex(U)$, and $s_{1,x}',s_{2,x}'\ne 0\in \O_{X,x}^\rex$, then there is an open set open $x\in V\subseteq U$ and $s_1'',s_2''\in \O_X^\rin(V)$ representing $s_1'\vert_V,s_2'\vert_V$. Then $s_{1,x}''\cdot s_{2,x}''\in \O_{X,x}^\rin\reflectbox{$\notin$} 0$, so the stalk $\bO_{X,x}=(\O_{X,x},\O_{X,x}^\rin\amalg\{0\})$ is interior, and (a) holds. Hence (c) implies~(a).
\end{proof}

\begin{rem}
\label{cc5rem1}
In \S\ref{cc41} we defined interior $C^\iy$-rings with corners $\bfC=(\fC,\fC_\rex)$ as special examples of $C^\iy$-rings of corners, with $\fC_\rex=\fC_\rin\amalg\{0\}$, so that $0\in\fC_\rex$ plays a special, somewhat artificial r\^ole in the interior theory; really one would prefer to write $\bfC=(\fC,\fC_\rin)$ and exclude 0 altogether.

One place this artificiality appears is in Definition \ref{cc5def2}(a)--(c). We want interior (local) $C^\iy$-ringed spaces with corners $\bX=(X,\bO_X)$ to be special examples of (local) $C^\iy$-ringed spaces with corners. At the same time, we na\"\i vely expect that if $\bX$ is interior then $\bO_X$ should be a sheaf of interior $C^\iy$-rings with corners. But this is false: $\bO_X$ is {\it never\/} a sheaf valued in $\CRingscin$ in the sense of Definition \ref{cc2def9}. For example, $\bO_X(\es)=(\{0\},\{0\})$, where $(\{0\},\{0\})$ is the final object in $\CRingsc$, and it is not equal to $(\{0\},\{0,1\})$, the final object in $\CRingscin$, contradicting Definition~\ref{cc2def9}(iii). 

A sheaf $\bO_X^\rin$ on $X$ valued in $\CRingscin$ is also a presheaf (but not a sheaf) valued in $\CRingsc$, so it has a sheafification $(\bO_X^\rin)^{\rm sh}$ valued in $\CRingsc$. Conversely, given a sheaf $\bO_X$ on $X$ valued in $\CRingsc$, we can ask whether it is the sheafification of a (necessarily unique) sheaf $\bO_X^\rin$ valued in $\CRingscin$. Definition \ref{cc5def2}(a)--(c) characterize when this holds, and this is the correct notion of `sheaf of interior $C^\iy$-rings with corners' in our context. The extra conditions $U_{s'}$ open and $\hat U_{s'}$ closed in $U$ in Definition \ref{cc5def2}(a),(b) are surprising, but they will be essential in the construction of the corner functor in~\S\ref{cc61}.
\end{rem}

The proof of the next theorem is essentially the same as showing ordinary ringed spaces have all small limits. An explicit proof can be found in the first author \cite[Th.~5.1.10]{Fran}. Existence of small limits in $\CRSc,\CRScin$ depends on Theorem \ref{cc4thm3}(b); and $\LCRSc,\LCRScin$ being closed under limits in $\CRSc,\CRScin$ depends on Proposition \ref{cc4prop7}; and $\LCRScin,\CRScin$ being closed under limits in $\LCRSc,\CRSc$ depends on Theorem~\ref{cc4thm3}(d).

\begin{thm}
\label{cc5thm1}
The categories $\CRSc,\LCRSc,\CRScin$ and\/ $\LCRScin$ have all small limits. Small limits commute with the inclusion and forgetful functors in the following diagram:
\begin{equation*}
\xymatrix@C=40pt@R=13pt{ {\LCRScin} \ar[d]^{} \ar[r]_(0.45){} & {\CRScin} \ar[d]_{} & \\  
{\LCRSc} \ar[r]^(0.45){} & {\CRSc} \ar[r]^(0.45){}  &  {\Top.} }
\end{equation*}
\end{thm}

In contrast to Theorem \ref{cc4thm3}(d), in Theorem \ref{cc6thm1} below we show that $\inc:\LCRScin\hookra\LCRSc$ has a right adjoint $C$. This implies:

\begin{cor}
\label{cc5cor1}
The inclusion $\inc:\LCRScin\hookra\LCRSc$ preserves colimits.
\end{cor}

\subsection{\texorpdfstring{Special classes of local $C^\iy$-ringed spaces with corners}{Special classes of local C∞-ringed spaces with corners}}
\label{cc52}

We extend parts of \S\ref{cc47} to local $C^\iy$-ringed spaces with corners.

\begin{dfn}
\label{cc5def3}
Let $\bX=(X,\bO_X)$ be an interior local $C^\iy$-ringed space with corners. We call $\bX$ {\it integral}, {\it torsion-free}, or {\it saturated}, if for each $x\in\bX$ the stalk $\bO_{X,x}$  is an integral, torsion-free, or saturated, interior local $C^\iy$-ring with corners, respectively, in the sense of Definition~\ref{cc4def15}.

For each open $U\subseteq X$, the monoid $\O_X^\rin(U)$ is a submonoid of $\prod_{x\in U}\O^\rin_{X,x}$. If $\bX$ is integral then each $\O^\rin_{X,x}$ is an integral monoid, so it is a submonoid of a group $(\O^\rin_{X,x})^\gp$, and $\O_X^\rin(U)$ is a submonoid of a group $\prod_{x\in U}(\O^\rin_{X,x})^\gp$, so it is integral, and $\bO^\rin_X(U)$ is an integral $C^\iy$-ring with corners. Conversely, if $\bO^\rin_X(U)$ is integral for all open $U\subseteq X$ one can show the stalks $\bO_{X,x}$ are integral, so $\bX$ is integral. Thus, an alternative definition is that $\bX$ is integral if $\bO^\rin_X(U)$ is integral for all open $U\subseteq X$. The analogues hold for torsion-free and saturated.

Write $\LCRScsa\subset\LCRSctf\subset\LCRScZ\subset\LCRScin$ for the full subcategories of saturated, torsion-free, and integral objects in $\LCRScin$. Also write $\LCRScsaex\!\subset\!\LCRSctfex\!\subset\!\LCRScZex\!\subset\!\LCRScinex\!\subset\!\LCRSc$ for the full subcategories of saturated, \ldots, interior objects in $\LCRSc$, but with exterior (i.e.\ all) rather than interior morphisms.	
\end{dfn}

Here is the analogue of Theorem \ref{cc4thm5}:

\begin{thm}
\label{cc5thm2}
There are coreflection functors $\Pi_\rin^\Z,\Pi_\Z^\tf,\Pi_\tf^\sa,\Pi_\rin^\sa$ in a diagram
\begin{equation*}
\xymatrix@C=22pt{ \LCRScsa \ar@<-.5ex>[r]_{\inc} &
\LCRSctf \ar@<-.5ex>[r]_{\inc} \ar@<-.5ex>[l]_{\Pi_\tf^\sa} & \LCRScZ \ar@<-.5ex>[r]_{\inc} \ar@<-.5ex>[l]_{\Pi_\Z^\tf} & \LCRScin, \ar@<-.5ex>[l]_{\Pi_\rin^\Z} \ar@<-3ex>@/_.5pc/[lll]_{\Pi_\rin^\sa} }
\end{equation*}
such that each of\/ $\Pi_\rin^\Z,\Pi_\Z^\tf,\Pi_\tf^\sa,\Pi_\rin^\sa$ is right adjoint to the corresponding inclusion functor\/~$\inc$.
\end{thm}

\begin{proof} We explain the functor $\Pi_\rin^\Z$; the arguments for $\Pi_\Z^\tf,\Pi_\tf^\sa,\Pi_\rin^\sa$ are the same.

Let $\bX=(X,\bO_X)$ be an object in $\LCRScin$. We will construct an object $\bs{\ti X}=(X,\bs{\ti\O}_X)$ in $\LCRScZ$, and set $\Pi_\rin^\Z(\bX)=\bs{\ti X}$. The topological spaces $X$ in $\bs{\ti X},\bX$ are the same. Define a presheaf $\cP\bs{\ti\O}_X$ of $C^\iy$-rings with corners on $X$ by $\cP\bs{\ti\O}_X(U)=\Pi_\rin^\Z(\bO_X^\rin(U))$ for each open $U\subseteq X$, where $\bO_X^\rin(U)\in\CRingscin$ is as in Definition \ref{cc5def2}(c), and $\Pi_\rin^\Z:\CRingscin\ra\CRingscZ$ is as in Theorem \ref{cc4thm5}. For open $V\subseteq U\subseteq X$, the restriction morphism $\rho_{UV}$ in $\cP\bs{\ti\O}_X$ is $\Pi_\rin^\Z$ applied to $\rho_{UV}:\bO_X^\rin(U)\ra\bO_X^\rin(V)$. Let $\bs{\ti\O}_X$ be the sheafification of~$\cP\bs{\ti\O}_X$.

Definition \ref{cc5def2}(c) implies that $\bs{\ti X}=(X,\bs{\ti\O}_X)$ is an object in $\LCRScin$. Since stalks are given by direct limits as in Definition \ref{cc2def10}, and $\Pi_\rin^\Z:\CRingscin\ra\CRingscZ$ preserves direct limits as it is a left adjoint, we have
\begin{equation*}
\bs{\ti\O}_{X,x}\cong\cP\bs{\ti\O}_{X,x}\cong \Pi_\rin^\Z(\bO_{X,x})\qquad\text{for all $x\in X$.}
\end{equation*}
Hence $\bs{\ti X}$ is an object of $\LCRScZ$ by Definition~\ref{cc5def3}.

If $\bs f=(f,\bs f^\sh):\bX\ra\bY$ is a morphism in $\LCRScin$ and $\bs{\ti X},\bs{\ti Y}$ are as above, in an obvious way we define a morphism $\bs{\ti f}=(f,\bs{\ti f}{}^\sh):\bs{\ti X}\ra\bs{\ti Y}$ in $\LCRScZ$ such that if $x\in X$ with $f(x)=y\in Y$ then $\bs{\ti f}{}^\sh$ acts on stalks by
\begin{equation*}
\bs{\ti f}{}_x^\sh\cong\Pi_\rin^\Z(\bs f_x^\sh):\bs{\ti\O}_{Y,y}\cong\Pi_\rin^\Z(\bO_{Y,y})\longra\bs{\ti\O}_{X,x}\cong\Pi_\rin^\Z(\bO_{X,x}).
\end{equation*}
We define $\Pi_\rin^\Z(\bs f)=\bs{\ti f}$. This gives a functor $\Pi_\rin^\Z:\LCRScin\ra\LCRScZ$. We can deduce that $\Pi_\rin^\Z$ is right adjoint to $\inc:\LCRScZ\hookra\LCRScin$ from the fact that $\Pi_\rin^\Z$ is left adjoint to $\inc$ in Theorem \ref{cc4thm5}, applied to stalks $\bO_{X,x},\bO_{Y,y}$ as above for all $x\in X$ with $f(x)=y$ in~$Y$.
\end{proof}

\begin{rem}
\label{cc5rem2}
The analogue of Theorem \ref{cc5thm2} does not work for the categories $\LCRScsaex\subset\cdots\subset\LCRScinex$. This is because the functors
$\Pi_\rin^\Z,\ldots,\Pi_\rin^\sa$ in Theorem \ref{cc4thm5} do not extend to exterior morphisms.
\end{rem}

Since right adjoints preserve limits, Theorems \ref{cc5thm1} and \ref{cc5thm2} imply:

\begin{cor}
\label{cc5cor2}
All small limits exist in\/ $\LCRScsa,\LCRSctf,\LCRScZ$.

\end{cor}

\subsection{The spectrum functor}
\label{cc53}

We now define a spectrum functor for $C^\iy$-rings with corners, in a similar way to Definition~\ref{cc2def16}. 

\begin{dfn}
\label{cc5def4}
Let $\bfC=(\fC,\fC_\rex)$ be a $C^\iy$-ring with corners, and use the notation of Definition \ref{cc4def11}. As in Definition \ref{cc2def15}, write $X_\fC$ for the set of $\R$-points of $\fC$ with topology ${\cal T}_\fC$. For each open $U\subseteq X_\fC$, define $\bs{\O}_{X_\fC}(U)=(\O_{X_\fC}(U),\O_{X_\fC}^\rex(U))$. Here $\O_{X_\fC}(U)$ is the set of functions $s:U\ra\coprod_{x\in U}\fC_x$ (where we write $s_x$ for its value at the point $x\in U$) such that $s_x\in\fC_x$ for all $x\in U$, and such that $U$ may be covered by open $W\subseteq U$ for which there exist $c\in\fC$ with $s_x=\pi_x(c)$ in $\fC_x$ for all $x\in W$. Similarly, $\O_{X_\fC}^\rex(U)$ is the set of $s':U\ra\coprod_{x\in U}\fC_{x,\rex}$ with $s'_x\in\fC_{x,\rex}$ for all $x\in U$, and such that $U$ may be covered by open $W\subseteq U$ for which there exist $c'\in\fC_\rex$ with $s'_x=\pi_{x,\rex}(c')$ in $\fC_{x,\rex}$ for all~$x\in W$. 

Define operations $\Phi_f$ and $\Psi_g$ on $\bs{\O}_{X_\fC}(U)$ pointwise in $x\in U$ using the operations $\Phi_f$ and $\Psi_g$ on $\bfC_x$. This makes $\bs{\O}_{X_\fC}(U)$ into a $C^\iy$-ring with corners. If $V\subseteq U\subseteq X_\fC$ are open, the restriction maps $\bs\rho_{UV}=(\rho_{UV},\rho_{UV,\rex}):\bs{\O}_{X_\fC}(U)\ra\bs{\O}_{X_\fC}(V)$ mapping $\rho_{UV}:s\mapsto s\vert_V$ and $\rho_{UV,\rex}:s'\mapsto s'\vert_V$ are morphisms of $C^\iy$-rings with corners. 

The local nature of the definition implies that $\bs{\O}_{X_\fC}=(\O_{X_{\fC}},\O_{X_{\fC}}^{\rex})$ is a sheaf of $C^\iy$-rings with corners on $X_\fC$. In fact, $\O_{X_\fC}$ is the sheaf of $C^\iy$-rings in Definition \ref{cc2def16}. By Proposition \ref{cc5prop1} below, the stalk $\bs{\O}_{X_\fC,x}$ at $x\in X_\fC$ is naturally isomorphic to $\bfC_x$, which is a local $C^\iy$-ring with corners by Theorem \ref{cc4thm4}(a). Hence $(X_\fC,\bs{\O}_{X_\fC})$ is a local $C^\iy$-ringed space with corners, which we call the {\it spectrum\/} of $\bfC$, and write as~$\Specc\bfC$.

Now let $\bs\phi=(\phi,\phi_\rex):\bfC\ra\bfD$ be a morphism of $C^\iy$-rings with corners. As in Definition \ref{cc2def16}, define the continuous function $f_{\phi}:X_\fD\ra X_\fC$ by $f_{\phi}(x)=x\ci\phi$. For open $U\subseteq X_\fC$ define $(\bs f_\phi)_\sh(U):\bs{\O}_{X_\fC}(U)\ra\bs{\O}_{X_\fD}(f_\phi^{-1}(U))$ to act by $\bs\phi_x:\bfC_{f_\phi(x)}\ra\bfD_x$ on stalks at each $x\in f_\phi^{-1}(U)$, where $\bs\phi_x$ is the induced morphism of local $C^\iy$-rings with corners. Then $(\bs f_\phi)_\sh:\bs{\O}_{X_{\fC}}\ra (f_\phi)_*(\bs{\O}_{X_\fD})$ is a morphism of sheaves of $C^\iy$-rings with corners on~$X_\fC$.

Let $\bs f_\phi^\sh:f_\phi^{-1}(\bs{\O}_{X_{\fC}})\ra\bs{\O}_{X_\fD}$ be the
corresponding morphism of sheaves of $C^\iy$-rings with corners on $X_\fD$ under \eq{cc2eq8}. The stalk map $\bs f_{\phi,x}^\sh:\bs{\O}_{X_{\fC},f_\phi(x)}\ra \bs{\O}_{X_\fD,x}$ of $\bs f_\phi^\sh$ at $x\in X_\fD$ is identified with $\bs\phi_x:\bfC_{f_\phi(x)}\ra\bfD_x$ under the isomorphisms $\bs{\O}_{X_{\fC},f_\phi(x)}\cong\bfC_{f_\phi(x)}$, $\bs{\O}_{X_\fD,x}\cong\bfD_x$ in Proposition \ref{cc5prop1}. Then $\bs f_\phi=(f_\phi,\bs f_\phi^\sh): (X_\fD,\bs{\O}_{X_\fD})\ra(X_\fC,\bs{\O}_{X_\fC})$ is a morphism of local $C^\iy$-ringed spaces with corners. Define $\Specc\bs\phi:\Specc\bfD\ra\Specc\bfC$ by $\Specc\bs\phi=\bs f_\phi$. Then $\Specc$ is a functor $(\CRingsc)^{\bf op}\ra\LCRSc$, the {\it spectrum functor}.
\end{dfn}

\begin{prop}
\label{cc5prop1}
In Definition\/ {\rm\ref{cc5def4},} the stalk\/ $\bs{\O}_{X_\fC,x}$ of\/ $\bs{\O}_{X_\fC}$ at\/ $x\in X_\fC$ is naturally isomorphic to\/~$\bfC_x$. 
\end{prop}

\begin{proof} We have $\bs{\O}_{X_\fC,x}=(\O_{X_\fC,x},\O_{X_\fC,x}^\rex)$ where elements $[U,s]\in\O_{X_\fC,x}$ and $[U,s']\in\O_{X_\fC,x}^\rex$ are $\sim$-equivalence classes of pairs $(U,s)$ and $(U,s')$, where $U$ is an open neighbourhood of $x$ in $X_\fC$ and $s\in\O_{X_\fC}(U)$, $s'\in\O_{X_\fC}^\rex(U)$, and $(U,s)\sim(V,t)$, $(U,s')\sim(V,t')$ if there exists open $x\in W\subseteq U\cap V$ with $s\vert_W=t\vert_W$ in $\O_{X_\fC}(W)$ and $s'\vert_W=t'\vert_W$ in $\O_{X_\fC}^\rex(W)$. Define a morphism of $C^\iy$-rings with corners $\bs\Pi=(\Pi,\Pi_\rex):\bs{\O}_{X_\fC,x}\ra\bfC_x$ by $\Pi:[U,s]\mapsto s_x\in\fC_x$ and $\Pi_\rex:[U,s']\mapsto s'_x\in\fC_{x,\rex}$.

Suppose $c_x\in\fC_x$ and $c_x'\in\fC_{x,\rex}$. Then $c_x=\pi_x(c)$ for $c\in\fC_x$ and $c_x'=\pi_{x,\rex}(c')$ for $c'\in\fC_{x,\rex}$ by Theorem \ref{cc4thm4}(c). Define $s:X_\fC\ra\coprod_{y\in X_\fC}\fC_{y}$ and $s':X_\fC\ra\coprod_{y\in X_\fC}\fC_{y,\rex}$ by $s_y=\pi_{y}(c)$ and $s_y'=\pi_{y,\rex}(c')$. Then $s\in\O_{X_\fC}(X_\fC)$, so that $[X_\fC,s]\in\O_{X_\fC,x}$ with $\Pi([X_\fC,s])=s_x=\pi_x(c)=c_x$, and similarly $s'\in\O_{X_\fC}^\rex(X_\fC)$ with $\Pi_\rex([X_\fC,s'])=c_x'$. Hence $\Pi:\O_{X_\fC,x}\ra\fC_x$ and $\Pi_\rex:\O_{X_\fC,x}^\rex\ra\fC_{x,\rex}$ are surjective.

Let $[U_1,s_1],[U_2,s_2]\in \O_{X_\fC,x}$ with $\Pi([U_1,s_1])=s_{1,x}=s_{2,x}=\Pi([U_2,s_2])$. Then by definition of $\bs{\O}_{X_\fC}(U_1),\bs{\O}_{X_\fC}(U_2)$ there exists an open neighbourhood $V$ of $x$ in $U_1\cap U_2$ and $c_1,c_2\in\fC$ with $s_{1,v}=\pi_v(c_1)$ and $s_{2,v}=\pi_v(c_2)$ for all $v\in V$. Thus $\pi_x(c_1)=\pi_x(c_2)$ as $s_{1,x}=s_{2,x}$. Hence $c_1-c_2$ lies in the ideal $I$ in \eq{cc2eq3} by Proposition \ref{cc2prop5}. Thus there exists $d\in\fC$ with $x(d)\ne 0\in \R$ and~$d\cdot(c_1-c_2)=0\in\fC$. 

Making $V$ smaller we can suppose that $v(d)\ne 0$ for all $v\in V$, as this is an open condition. Then $\pi_v(c_1)=\pi_v(c_2)\in\fC_v$ for $v\in V$, since $\pi_v(d)\cdot\pi_v(c_1)=\pi_v(d)\cdot\pi_v(c_2)$ as $d\cdot c_1=d\cdot c_2$ and $\pi_v(d)$ is invertible in $\fC_v$. Thus $s_{1,v}=\pi_v(c_1)=\pi_v(c_2)=s_{2,v}$ for $v\in V$, so $s_1\vert_V=s_2\vert_V$, and $[U_1,s_1]=[V,s_1\vert_V]=[V,s_2\vert_V]=[U_2,s_2]$. Therefore $\Pi:\O_{X_\fC,x}\ra\fC_x$ is injective, and an isomorphism.

Let $[U_1,s_1'],[U_2,s_2']\in \O_{X_\fC,x}^\rex$ with $\Pi_\rex([U_1,s_1'])=s_{1,x}'\ab=s_{2,x}'\ab=\Pi([U_2,s_2'])$. As above there exist an open neighbourhood $V$ of $x$ in $U_1\cap U_2$ and $c_1',c_2'\in\fC_\rex$ with $s_{1,v}'=\pi_{v,\rex}(c_1')$ and $s_{2,v}'=\pi_{v,\rex}(c_2')$ for all $v\in V$. At this point we can use Proposition \ref{cc4prop8}, which says that $\pi_{x,\rex}(x_2')=\pi_{x,\rex}(c_1')$ if and only if there are $a,b\in \fC_\rex$ such that $\Phi_i(a)-\Phi_i(b)\in I_x$, $x\ci\Phi_i(a)\ne0$ and $ac_1=bc_2$, where $I_x$ is the ideal in \eq{cc2eq3}. The third condition does not depend on $x$, whereas the first two conditions are open conditions in $x$, that is, if $\Phi_i(a)-\Phi_i(b)\in I_x$, $x\ci\Phi_i(a)\ne0$, then there is an open neighbourhood of $X$ such that $\Phi_i(a)-\Phi_i(b)\in I_v$, $v\ci\Phi_i(a)\ne0$ for all $v$ in that neighbourhood. 

Making $V$ above smaller if necessary, we can suppose that these conditions hold in $V$ and thus that $\pi_{v,\rex}(c_1')=\pi_{v,\rex}(c_2')$ for all $v\in V$. Hence $s_{1,v}'=s_{2,v}'$ for all $v\in V$, and $s_1'\vert_V=s_2'\vert_V$, so that $[U_1,s_1']=[V,s_1'\vert_V]=[V,s_2'\vert_V]=[U_2,s_2']$. Therefore $\Pi_\rex:\O_{X_\fC,x}^\rex\ra\fC_{x,\rex}$ is injective, and an isomorphism. So $\bs\Pi=(\Pi,\Pi_\rex):\bs{\O}_{X_\fC,x}\ra\bfC_x$ is an isomorphism, as we have to prove.
\end{proof}

\begin{dfn}
\label{cc5def5}
As $\CRingscin$ is a subcategory of $\CRingsc$ we can define the functor $\Speccin$ by restricting $\Specc$ to $(\CRingscin)^{\bf op}$. Let $\bfC=(\fC,\fC_\rex)$ be an interior $C^\iy$-ring with corners, and $\bX=\Speccin\bfC=(X,\O_X)$. Then Definition \ref{cc4def9} implies the localizations $\bfC_x$ are interior $C^\iy$-rings with corners, and $\bfC_x\cong\bO_{X,x}$ by Proposition~\ref{cc5prop1}. 

If $s'\in \O_{X}^\rex(U)$ with $s'_x\ne 0$ in $\O^\rex_{X,x}$ at $x\in X$, then $s'_{x'}=\pi_{x',\rex}(c')$ in $\fC_{x',\rex}\cong\O^\rex_{X,x'}$ for some $c'\in\fC_\rex$ and all $x'$ in an open neighbourhood $V$ of $x$ in $U$. Then $c'\ne 0$ in $\fC_\rex$, so $c'$ is non-zero in every stalk by Remark \ref{cc4rem3} as $\bfC$ is interior, and $s'$ is non-zero at every point in $V$. Therefore $\bX$ is an interior local $C^\iy$-ringed space with corners by Definition~\ref{cc5def2}(a). 

If $\bs\phi:\bfC\ra\bfD$ is a morphism of interior $C^\iy$-rings with corners, then $\Speccin\bs\phi=(f,\bs f^\sh)$ has stalk map $\bs f^\sh_x=\bs\phi_x:\bfC_{f_\phi(x)}\ra \bfD_x$. This map fits into the commutative diagram 
\begin{gather*}
\xymatrix@C=100pt@R=15pt{ *+[r]{\bfC} \ar[d]^{\bs{\pi}_{f_\phi(x)}} \ar[r]_{\bs\phi} & *+[l]{\bfD} \ar[d]_{\bs{\pi}_x} \\
*+[r]{\bfC_{f_\phi(x)}} \ar@{.>}[r]^{\bs\phi_x} & *+[l]{\bfD_x.\!}
}
\end{gather*}
As $\bs\phi$ is interior, and the maps $\bs{\pi}_{f_\phi(x)}, \bs{\pi}_{x}$ are interior and surjective, then $\bs f^\sh_x$ is interior. This implies that $\Speccin\bs\phi$ is an interior morphism of interior local $C^\iy$-ringed spaces with corners. Hence $\Speccin:(\CRingscin)^{\bf op}\ra \LCRScin$ is a well defined functor, which we call the {\it interior spectrum functor}.
\end{dfn}

\begin{dfn}
\label{cc5def6}
The {\it global sections functor\/} $\Gac: \LCRSc\ra (\CRingsc)^{\bf op}$ takes objects $(X,\bs{\O}_X)\in \LCRSc$ to $\bs{\O}_X(X)$ and takes morphisms $(f,\bs f^\sh):(X,\bs{\O}_X)\ra (Y,\bs{\O}_Y)$ to $\Gac : (f,\bs f^\sh)\mapsto \bs f_{\sh}(Y)$. Here $\bs f_{\sh}:\bs{\O}_Y\ra f_\ast (\bs{\O}_X)$ corresponds to $\bs f^\sh$ under~\eq{cc2eq8}.

The composition $\Gac\ci\Specc$ is a functor $(\CRingsc)^{\bf op}\ra(\CRingsc)^{\bf op}$, or equivalently a functor $\CRingsc\ra\CRingsc$. For each $C^\iy$-ring with corners $\bs{\fC}$ we define a morphism $\bs\Xi_\bfC=(\Xi,\Xi_\rex):\bfC\ra \Gac\ci\Specc\bfC$ in $\CRingsc$ by $\Xi(c):X_\fC\ra\coprod_{x\in X_\fC}\O_{X_\fC,x}$, $\Xi(c):x\mapsto \pi_x(c)$ in $\fC_x\cong\O_{X_\fC,x}$ for $c\in\fC$, and $\Xi_\rex(c'):X_\fC\ra\coprod_{x\in X_\fC}\O_{X_\fC,x}^\rex$, $\Xi_\rex(c'):x\mapsto \pi_{x,\rex}(c')$ in $\fC_{x,\rex}\cong\O^\rex_{X_\fC,x}$ for $c'\in\fC_\rex$. This is functorial in $\bfC$, so that the $\bs\Xi_\bfC$ for all $\bfC$ define a natural transformation $\bs\Xi:\Id_\CRingsc\Ra\Gac\ci\Specc$ of functors~$\Id_\CRingsc,\Gac\ci\Specc:\CRingsc\ra\CRingsc$.
\end{dfn}

Here is the analogue of Lemma~\ref{cc2lem1}.

\begin{lem}
\label{cc5lem2}
Let\/ $\bfC$ be a $C^\iy$-ring with corners, and\/ $\bX=\Specc \bfC$. For any\/ $c\in\fC,$ let\/ $U_c\subseteq X$ be as in Definition\/ {\rm\ref{cc2def15}}. Then $\bU_c\cong\Specc(\bfC(c^{-1}))$. If\/ $\bfC$ is firm, or interior, then so is $\bfC(c^{-1})$.
\end{lem}

\begin{proof}
Write $\bfC(c^{-1})=(\fD,\fD_\rex)$. By Lemma \ref{cc4lem3} we have $\fD\cong \fC(c^{-1})$. By Lemma \ref{cc2lem1}, we need only show there is an isomorphism of stalks $\fC_{x,\rex}\ra \fD_{\hat x,\rex}$. However, using the universal properties of $\bfC_x$, $\bfC(c^{-1})$ and $\bfC(c^{-1})_{\hat x}$ this follows by the same reasoning as Lemma \ref{cc2lem1}. If $\bfC$ is firm, then so is $\bfC(c^{-1})$, as $\fC(c^{-1})_\rex^\sh$ is generated by the image of $\fC_\rex^\sh$ under the morphism $\bfC\to\bfC(c^{-1})$. If $\bfC$ is interior, then $\bfC(c^{-1})$ is also interior, as otherwise zero divisors in $\fC(c^{-1})_\rex$ would have to come from zero divisors in~$\fC_\rex$.
\end{proof}

\begin{thm}
\label{cc5thm3}
The functor\/ $\Specc:(\CRingsc)^{\bf op}\ra\LCRSc$ is \begin{bfseries}right adjoint\end{bfseries} to $\Gac:\LCRSc\ra(\CRingsc)^{\bf op}$. This implies that for all\/ $\bfC$ in\/ $\CRingsc$ and all \/ $\bX$ in\/ $\LCRSc$ there are inverse bijections
\e
\xymatrix@C=150pt{ *+[r]{\Hom_\CRingsc(\bfC,\Gac(\bX))} \ar@<.5ex>[r]^{L_{\bfC,\bX}}
& *+[l]{\Hom_\LCRSc(\bX,\Specc\bfC).} \ar@<.5ex>[l]^{R_{\bfC,\bX}} }
\label{cc5eq1}
\e
If we let\/ $\bX=\Specc\bfC$ then $\bs\Xi_\bfC=R_{\bfC,\bX}(\bs\id_\bX),$ and\/ $\bs\Xi$ is the unit of the adjunction between $\Gac$ and\/~$\Specc$.
\end{thm}

\begin{proof} We follow the proof of \cite[Th.~4.20]{Joyc9}. Take $\bX\in \LCRSc$ and $\bfC\in\CRingsc$, and let $\bY=(Y,\bs{\O}_Y)=\Specc\bfC$. Define a functor $R_{\bfC,\bX}$ in \eq{cc5eq1} by taking $R_{\bfC,\bX}(\bs f):\bfC\ra\Gac(\bX)$ to be the composition
\begin{equation*}
\xymatrix@C=40pt{ \bfC \ar[r]^(0.3){\bs\Xi_\bfC} & \Gac\ci\Specc\bfC=\Gac(\bY) \ar[r]^(0.6){\Gac(\bs f)} & \Gac(\bX)}
\end{equation*}
for each morphism $\bs f:\bX\ra\bY$ in $\LCRSc$.
If $\bX=\Specc\bfC$ then we have $\bs\Xi_\bfC=R_{\bfC,\bX}(\bs\id_\bX)$. We see that $R_{\bfC,\bX}$ is an extension of the functor $R_{\fC,\bX}$ constructed in \cite[Th.~4.20]{Joyc9} for the adjunction between $\Spec$ and $\Ga$. This will also occur for $L_{\bfC,\bX}$.

Given a morphism $\bs\phi=(\phi,\phi_\rex):\bfC\ra\Gac(\bX)$ in $\CRingsc$ we define $L_{\bfC,\bX}(\bs\phi)=\bs g=(g,g^\sh,g^\sh_\rex)$ where $(g,g^\sh)=L_{\fC,\bX}(\phi)$ with $L_{\fC,\bX}$ constructed in \cite[Th.~4.20]{Joyc9}. Here, $g$ acts by $x\mapsto x_*\ci \phi$ where $x_*:\O_X(X)\ra\R$ is the composition of the $\si_x:\O_X(X)\ra\O_{X,x}$ with the unique morphism $\pi:\O_{X,x}\ra\R$, as $\bs{\O}_{X,x}$ is a local $C^\iy$-ring with corners. The morphisms $g^\sh,g_\rex^\sh$ are constructed as $g^\sh$ is in \cite[Th.~4.20]{Joyc9}, and we explain this explicitly now.

For $x\in X$ and $g(x)=y\in Y$, take the stalk map $\bs{\si}_x=(\si_x,\si_x^\rex):\bs{\O}_X(X)\ra\bs{\O}_{X,x}$. This gives the following diagram of $C^\iy$-rings with corners
\e
\begin{gathered}
\xymatrix@C=70pt@R=15pt{ *+[r]{\bfC} \ar[d]^{\bs{\pi}_y} \ar[r]_{\bs\phi} & *+[l]{\Gac(\bX)} \ar[d]_{\bs{\si}_x} & \\
*+[r]{\bfC_y\cong\bs{\O}_{Y,y}} \ar@{.>}[r]^{\bs\phi_x} & *+[l]{\bs{\O}_{X,x}} \ar[r]^{\bs\pi} &*+[r]{\R.}
}
\end{gathered}
\label{cc5eq2}
\e
We know $\bfC_y\cong\bs{\O}_{Y,y}$ by Proposition \ref{cc5prop1} and $\pi:\O_{X,x}\ra\R$ is the unique local morphism. If we have $(c,c')\in\bfC$ with $y(c)\ne 0$, and $y\circ \Phi_i(c')\ne 0$ then $\bs{\si}_x\ci\bs\phi(c,c')\in\bs{\O}_{X,x}$ with $\pi(\si_x\ci\phi(c))\ne 0$ and 
\begin{equation*}
\pi(\Phi_i\ci\si_x^\rex\ci\phi_\rex(c))=\pi(\si_{x}\ci\phi\ci\Phi_i(c))\ne 0.
\end{equation*}
As $\bs{\O}_{X,x}$ is a local $C^\iy$-ring with corners then $\bs{\si}_x\ci\bs\phi(c,c')$ is invertible in $\bs{\O}_{X,x}$. The universal property of $\bs{\pi}_y:\bfC\ra\bfC_y$ gives a unique morphism $\bs\phi_x:\bs{\O}_{Y,y}\ra \bs{\O}_{X,x}$ that makes \eq{cc5eq2} commute.

We define
\begin{equation*}
\bs g_\sh(V)=(g_\sh(V),g^{\rex}_\sh(V)):\bO_Y(V)\ra g_*(\bO_X)(V)=\bO_X(U)
\end{equation*} 
for each open $V\subseteq Y$ with $U=g^{-1}(V)\subseteq X$ to act by $\bs\phi_x=(\phi_x,\phi_{x,\rex})$ on stalks at each $x\in U$. We can identify elements $s\in\O_X(U)$, $s'\in\O_X^\rex(U)$ with maps $s:U\ra\coprod_{x\in U}\O_{X,x}$ and $s':U\ra\coprod_{x\in U}\O_{X,x}^\rex$, such that $s,s'$ are locally of the form $s:x\mapsto \pi_x(c)$ in $\fC_x\cong\O_{X,x}$ for $c\in\fC$ and $s':x\mapsto \pi_{x,\rex}(c')$ in $\fC_{x,\rex}\cong\O^\rex_{X,x}$ for $c'\in\fC_\rex$, and similarly for $t\in\O_Y(V)$, $t'\in\O_Y^\rex(V)$. If $t$ is locally of the form $t:y\mapsto \pi_y(d)$ in $\fD_y\cong\O_{Y,c}$ near $\ti y=g(\ti x)$ for $d\in\fD$, then $s=g_\sh(V)(t)$ is locally of the form $s:x\mapsto \pi_x(c)$ in $\fC_x\cong\O_{X,x}$ near $\ti x$ for $c=\phi(d)\in\fC$, so $g_\sh(V)$ does map $\O_Y(V)\ra\O_X(U)$, and similarly $g_\sh^\rex(V)$ maps $\O_Y^\rex(V)\ra\O_X^\rex(U)$. Hence $\bs g_\sh(V)$ is well defined.

This defines a morphism $\bs g_\sh:\bO_Y\ra g_*(\bO_X)$ of sheaves of $C^\iy$-rings with corners on $Y$, and we write $\bs g^\sh:g^{-1}(\bO_Y)\ra\bO_X$ for the corresponding morphism of sheaves of $C^\iy$-rings with corners on $X$ under \eq{cc2eq8}. At a point $x\in X$ with $g(x)=y\in Y$, the stalk map $\bs g^\sh_x:\bO_{Y,y}\ra \bO_{X,x}$ is $\bs\phi_x$. Then $\bs g=(g,\bs g^\sh)$ is a morphism in $\LCRSc$, and $L_{\bfC,\bX}(\bs\phi)=\bs g$. It remains to show that these define natural bijections, but this follows in a very similar way to~\cite[Th.~4.20]{Joyc9}.
\end{proof}

\begin{dfn}
\label{cc5def7}
We define the {\it interior global sections functor}\/ $\Gacin:\LCRScin\ab\ra (\CRingscin)^{\bf op}$ to act on objects $(X,\bs{\O}_X)$ by $\Gacin: (X,\bs{\O}_X) \mapsto (\fC,\fC_\rex)$ where $\fC=\O_X(X)$ and $\fC_\rex$ to be the set containing the zero element of $\O_X^{\rex}(X)$ and the elements of $\O_X^{\rex}(X)$ that are non-zero in every stalk. That is,
\e
\fC_\rex=\{c'\in \O_X^{\rex}(X) :  c'=0 \in \O_X^\rex(X), \text{ or } \si_x^\rex(c')\ne 0 \in\O_{X,x}^\rex \;\forall x\in X\},
\label{cc5eq3}
\e
where $\si_x^\rex$ is the stalk map $\si_x^\rex:\O_X^\rex\ra \O_{X,x}^\rex$. This is an interior $C^\iy$-ring with corners, where the $C^\iy$-ring with corners structure is given by restriction from $(\O_X(X),\O_X^{\rex}(X))$. We define $\Gacin$ to act on morphisms $(f,\bs f^\sh):(X,\bs{\O}_X)\ra (Y,\bs{\O}_Y)$ by $\Gacin : (f,\bs f^\sh)\mapsto \bs f_{\sh}(Y)\vert_{(\fC,\fC_\rex)}$ for $\bs f_{\sh}:\bs{\O}_Y\ra f_\ast (\bs{\O}_X)$ corresponding to $\bs f^\sh$ under~\eq{cc2eq8}.

As in Definition \ref{cc5def6}, for each interior $C^\iy$-ring with corners $\bfC$, we define a morphism $\bs\Xi^\rin_\bfC=(\Xi^\rin,\Xi^\rin_\rex):\bfC\ra \Gacin\ci\Speccin \bfC$ in $\CRingscin$ by $\Xi^\rin(c):X_\fC\ra\coprod_{x\in X_\fC}\O_{X_\fC,x}$, $\Xi^\rin(c):x\mapsto \pi_x(c)$ in $\fC_x\cong\O_{X_\fC,x}$ for $c\in\fC$, and $\Xi^\rin_\rex(c'):X_\fC\ra\coprod_{x\in X_\fC}\O_{X_\fC,x}^\rex$, $\Xi^\rin_\rex(c'):x\mapsto \pi_{x,\rex}(c')$ in $\fC_{x,\rex}\cong\O^\rex_{X_\fC,x}$ for $c'\in\fC_\rex$. We need to check $\Xi^\rin_\rex(c')$ lies in \eq{cc5eq3}, but this is immediate as $\bs{\pi}_x=(\pi_x,\pi_{x,\rex}):\bfC\ra\bfC_x$ is interior. The $\bs\Xi^\rin_\bfC$ for all $\bfC$ define a natural transformation $\bs\Xi^\rin:\Id_\CRingscin\Ra\Gacin \ci\Speccin$ of functors~$\Id_\CRingscin,\Gacin\ci\Speccin:\CRingscin\ra\CRingscin$.
\end{dfn}

\begin{thm}
\label{cc5thm4}
The functor\/ $\Speccin:(\CRingscin)^{\bf op}\ra\LCRScin$ is right adjoint to\/~$\Gacin:\LCRScin\ra(\CRingscin)^{\bf op}$. 
\end{thm}

\begin{proof}
This proof is identical to that of Theorem \ref{cc5thm3}. We need only check that the definition of $\Gacin(\bX)$, which may not be equal to $\bO_X(X)$, gives well defined maps $\bs{\si}^{\rin}_x:\Gacin(\bX)\ra \bO_{X,x}$. As $\Gacin(\bX)$ is a subobject of $\bO_X(X)$, these maps are the restriction of the stalk maps $\bs{\si}_x:\bO_X(X)\ra \bO_{X,x}$ to $\Gacin(\bX)$. The definition of $\Gacin(\bX)$ implies these maps are interior. 
\end{proof}

\subsection{\texorpdfstring{$C^\iy$-schemes with corners}{C∞-schemes with corners}}
\label{cc54}

\begin{dfn}
\label{cc5def8}
A local $C^\iy$-ringed space with corners that is isomorphic in $\LCRSc$ to $\Specc\bfC$ for some $C^\iy$-ring with corners $\bfC$ is called an \/{\it affine $C^\iy$-scheme with corners\/}. We define the category $\ACSchc$ to be the full sub-category of affine $C^\iy$-schemes with corners in~$\LCRSc$.

Let $\bX=(X,\bs{\O}_X)$ be a local $C^\iy$-ringed space with corners. We call $\bX$ a $C^\iy$-{\it scheme with corners\/} if $X$ can be covered by open sets $U\subseteq X$ such that $(U,\bs{\O}_X\vert_U)$ is a affine $C^\iy$-scheme with corners. We define the category $\CSchc$ of $C^\iy$-schemes with corners to be the full sub-category of $C^\iy$-schemes with corners in $\LCRSc$. Then $\ACSchc$ is a full subcategory of $\CSchc$.

A local $C^\iy$-ringed space with corners that is isomorphic in $\LCRScin$ to $\Speccin\bfC$ for some interior $C^\iy$-ring with corners $\bfC$ is called an \/{\it interior affine $C^\iy$-scheme with corners\/}. We define the category $\ACSchcin$ of interior affine $C^\iy$-schemes with corners to be the full sub-category of interior affine $C^\iy$-schemes with corners in $\LCRScin$, so $\ACSchcin$ is a non-full subcategory of $\ACSchc$.
We call an object $\bX\in\LCRScin$ an {\it interior $C^\iy$-scheme with corners\/} if it can be covered by open sets $U\subseteq X$ such that $(U,\bs{\O}_X\vert_U)$ is an interior affine $C^\iy$-scheme with corners. We define the category $\CSchcin$ of  interior $C^\iy$-schemes with corners to be the full sub-category of interior $C^\iy$-schemes with corners in $\LCRScin$. This implies that $\ACSchcin$ is a full subcategory of $\CSchcin$. Clearly~$\CSchcin\subset\CSchc$.

In Theorem \ref{cc5thm7}(b) below we show that an object $\bX$ in $\CSchc$ lies in $\CSchcin$ if and only if lies in $\LCRScin$. That is, the two natural definitions of `interior $C^\iy$-scheme with corners' turn out to be equivalent.
\end{dfn}

\begin{lem}
\label{cc5lem3}
Let\/ $\bX$ be an (interior) $C^\iy$-scheme with corners, and\/ $\bU\subseteq\bX$ be open. Then $\bU$ is also an (interior) $C^\iy$-scheme with corners.
\end{lem}

\begin{proof}
This follows from Lemma \ref{cc5lem2} and the fact that the topology on $\Specc\bfC$ is generated by subsets $U_c$ in Definition~\ref{cc2def15}.
\end{proof}

We explain the relation with manifolds with (g-)corners in Chapter~\ref{cc3}:  

\begin{dfn}
\label{cc5def9}
Define a functor $F^\CSchc_\Manc:\Manc\ra\CSchc$ that acts on objects $X\in\Manc$ by $F^\CSchc_\Manc(X)=(X,\bO_X)$, where $\bO_X(U)=\bs C^\iy(U)=(C^\iy(U),\Ex(U))$ from Example \ref{cc4ex3}(a) for each open subset $U\subseteq X$. If $V\subseteq U\subseteq X$ are open we define $\bs{\rho}_{UV}=(\rho_{UV},\rho_{UV}^{\rex}):\bs C^\iy(U)\ra \bs C^\iy(V)$ by~$\rho_{UV}:c\mapsto c\vert_V$ and $\rho_{UV}^{\rex}:c'\mapsto c'\vert_{V}$. 

It is easy to verify that $\bO_X$ is a sheaf of $C^\iy$-rings with corners on $X$, so $\bX=(X,\bO_X)$ is a $C^\iy$-ringed space with corners. We show in Theorem \ref{cc5thm5}(b) that $\bX$ is a $C^\iy$-scheme with corners, and it is also interior.

Let $f:X\ra Y$ be a morphism in $\Manc$. Writing $F^\CSchc_\Manc(X)=(X,\bs{\O}_X)$ and $F^\CSchc_\Manc(Y)=(Y,\bs{\O}_Y)$, for all open $U\subseteq Y$, we define 
\begin{equation*}
\bs f_\sh(U):\bO_Y(U)= \bs C^\iy(U)\longra f_*(\bO_x)(U)=\bO_X(f^{-1}(U))=\bs C^\iy(f^{-1}(U))
\end{equation*}
by $f_\sh(U):c\mapsto c\ci f$ for all $c\in C^\iy(U)$ and $f_{\sh}^{\rex}(U):c'\mapsto c'\ci f$ for all $c'\in \Ex(U)$. Then $\bs f_\sh(U)$ is a morphism of $C^\iy$ rings with corners, and $\bs f_\sh:\bs{\O}_Y\ra f_*(\bs{\O}_X)$ is a morphism of sheaves of $C^\iy$-rings with corners on $Y$. Let $\bs f^\sh:f^{-1}(\bs{\O}_Y)\ra\bs{\O}_X$ correspond to $\bs f_\sh$ under \eq{cc2eq8}. Then $F^\CSchc_\Manc(f)=\bs f=(f,\bs f^\sh): (X,\bs{\O}_X)\ra(Y,\bs{\O}_Y)$ is a morphism in $\CSchc$. Define a functor $F^\CSchcin_\Mancin:\Mancin\ra\CSchcin$ by restriction of $F^\CSchc_\Manc$ to~$\Mancin$.

All the above generalizes immediately from manifolds with corners to manifolds with g-corners in \S\ref{cc33}, giving functors $F^\CSchc_\Mangc:\Mangc\ra\CSchc$ and $F^\CSchcin_\Mangcin:\Mangcin\ra\CSchcin$. It also generalizes to manifolds with \hbox{(g-)corners} of mixed dimension $\cManc,\ldots,\cMangcin$.

These functors $F^\CSchc_\Manc,\ldots,F^\CSchcin_\Mangcin$ map to various subcategories of $\CSchc,\CSchcin$ defined in \S\ref{cc56}, including the subcategories $\CSchctoex,\ab\CSchcto$ of {\it toric\/} $C^\iy$-schemes with corners. We will write $F^\CSchcto_\Manc,\ldots$ for $F^\CSchc_\Manc,\ldots$ considered as mapping to these subcategories.

We say that a $C^\iy$-scheme with corners $\bX$ {\it is a manifold with corners\/} (or {\it is a manifold with g-corners\/}) if $\bX\cong F_\Manc^\CSchc(X)$ for some $X\in\Manc$ (or $\bX\cong F_\Mangc^\CSchc(X)$ for some~$X\in\Mangc$).
\end{dfn}

\begin{thm}
\label{cc5thm5}
Let\/ $X$ be a manifold with corners, or a manifold with g-corners, and\/ $\bX=F^\CSchc_\Manc(X)$ or\/ $\bX=F^\CSchc_\Mangc(X)$. Then:
\smallskip

\noindent{\bf (a)} If\/ $X$ is a manifold with faces or g-faces, then $\bX$ is an affine $C^\iy$-scheme with corners, and is isomorphic to $\Specc\bs C^\iy(X)$. It is also an interior affine $C^\iy$-scheme with corners, with\/~$\bX\cong\Speccin\bs C_\rin^\iy(X)$.

If\/ $X$ is a manifold with corners, but not with faces, then $\bX$ is not affine.
\smallskip

\noindent{\bf (b)} In general, $\bX$ is an interior $C^\iy$-scheme with corners. 

\smallskip
\noindent{\bf (c)} The functors $F^\CSchc_\Manc,F^\CSchcin_\Mancin,F^\CSchc_\Mangc,F^\CSchcin_\Mangcin$ are fully faithful.
\end{thm}

\begin{proof} For (a), let $X$ be a manifold with corners or a manifold with g-corners, and write $\bX=(X,\bs{\O}_X)$. Note that $\Gac(\bX)=\bO_X(X)=\bs C^\iy(X)$ by the definition of $\bO_X$. 
Consider the map $L_{\bfC,\bX}$ in \eq{cc5eq1}. In the notation of Theorem \ref{cc5thm3}, if we let $\bfC=\Gac(\bX)=\bs C^\iy(X)$, then $L_{\bs C^\iy(X),\bX}$ is a bijection 
\begin{align*}
&L_{\bs C^\iy(X),\bX}:\Hom_\CRingsc(\bs C^\iy(X),\bs C^\iy(X))\\
&\qquad \longra \Hom_\LCRSc(\bX,\Specc\bs C^\iy(X)).
\end{align*}
Write $\bY=\Specc\bs C^\iy(X)$, and define a morphism $\bs g=(g,\bs g^\sh):\bX\ra\bY$ in $\LCRSc$ by $\bs g=L_{\bs C^\iy(X),\bX}(\id_{\bs C^\iy(X)})$. 

The continuous map $g:X\ra Y$ is defined in the proof of Theorem \ref{cc5thm3} by $g(x)=x_*\ci \id_{C^\iy(X)}$, where $x_*$ is the evaluation map at the point $x\in X$. This is a homeomorphism of topological spaces as in the proof of \cite[Th.~4.41]{Joyc9}. On stalks at each $x\in X$ we have $\bs g_x^\sh=\bs\la_x:(\bs C^\iy(X))_{x_*}\ra\bs C^\iy_x(X)$, where $\bs\la_x$ is as in Example \ref{cc4ex5}. If $X$ has faces then $\bs\la_x$ is an isomorphism by Proposition \ref{cc4prop9}, and if $X$ has g-faces then $\bs\la_x$ is an isomorphism by Definition \ref{cc4def12}. Hence if $X$ has (g-)faces then $g,\bs g^\sh$ and $\bs g$ are isomorphisms, and~$\bX\cong \Specc\bs C^\iy(X)$.

Essentially the same proof for the interior case, using $\Speccin,\Gacin$ adjoint by Theorem \ref{cc5thm4}, shows that $\bX\cong\Speccin\bs C^\iy_\rin(X)$ if $X$ has (g-)faces.

If $X$ has corners, but not faces, then Proposition \ref{cc4prop9} gives $x\in X$ such that $\bs g^\sh_x=\bs\la_x$ is not an isomorphism. The proof shows that $\la_{x,\rex}$ is not surjective. If $\bX\cong\Specc\bfC$ for some $C^\iy$-ring with corners $\bfC$ then $\bfC_x\cong\bO_{X,x}=\bs C^\iy_x(X)$, and the localization morphism $\bfC\ra\bfC_x$ is surjective. But $\bfC\ra\bfC_x\cong\bs C^\iy_x(X)$ factors as $\bfC\ra\bs C^\iy(X)\,{\buildrel\bs\la_x\over\longra}\,\bs C^\iy_x(X)$, and $\la_{x,\rex}$ is not surjective, a contradiction. Hence $\bX$ is not affine, completing part~(a).

For (b), for any point $x\in X$, we can find an open neighbourhood $U$ of $x$ which is a manifold with (g-)faces, as in Example \ref{cc4ex6} in the g-corners case. Then $\bU\cong\Specc\bs C^\iy(U)\cong\Speccin\bs C^\iy_\rin(U)$ by (a). So $\bX$ can be covered by open $\bU\subseteq\bX$ which are (interior) affine $C^\iy$-schemes with corners, and $\bX$ is an (interior) $C^\iy$-scheme with corners.

For (c), if $f,g:X\ra Y$ are morphisms in $\Manc$ and $(f,\bs f^\sh)=F^\CSchc_\Manc(f)\ab =F^\CSchc_\Manc(g)=(g,\bs g^\sh)$, we see that $f=g$, so $F^\CSchc_\Manc$ is faithful. Suppose $\bs h=(h,\bs{h}^\sh):\bX\ra\bY$ is a morphism in $\CSchc$, and $x\in X$ with $h(y)=y$ in $Y$. We may choose open $y\in V\subseteq Y$ and $x\in U\subseteq h^{-1}(V)\subseteq X$ such that $\bU=(U,\bs\O_X\vert_U)$ and $\bs V=(V,\bs\O_Y\vert_V)$ are affine. Then $\bs h^\sh$ induces a morphism $\bs h_{UV}^\sh:\bs\O_Y(V)=\bs C^\iy(V)\ra\bs\O_X(U)=\bs C^\iy(U)$ in~$\CRingsc$. 

By considering restriction to points in $U,V$ we see that $\bs h_{UV}^\sh$ maps $C^\iy(V)\ra C^\iy(U)$ and $C^\iy_\rex(V)\ra C^\iy_\rex(U)$ by $c\mapsto c\ci h\vert_U$ for $c$ in $C^\iy(V)$ or $C^\iy_\rex(V)$. Thus, if $c:V\ra\R$ is smooth, or $c:V\ra[0,\iy)$ is exterior, then $c\ci h\vert_U:U\ra\R$ is smooth, or $c\ci h\vert_U:U\ra[0,\iy)$ is exterior. Since $V$ is affine, this implies that $h\vert_U:U\ra V$ is smooth, and $\bs h\vert_{\bU}=F_\Manc^\CSchc(h\vert_U)$. As we can cover $X,Y$ by such $U,V$, we see that $h$ is smooth and $\bs h=F_\Manc^\CSchc(h)$. Hence $F^\CSchc_\Manc$ is full. The argument for $F^\CSchcin_\Mancin,\ldots,F^\CSchcin_\Mangcin$ is essentially the same.
\end{proof}

Theorem \ref{cc5thm5} shows that we can regard $\Manc,\Mangc$ (or $\Mancin,\Mangcin$) as full subcategories of $\CSchc$ (or $\CSchcin$), and regard $C^\iy$-schemes with corners as generalizations of manifolds with (g-)corners.

\subsection{\texorpdfstring{Semi-complete $C^\iy$-rings with corners}{Semi-complete C∞-rings with corners}}
\label{cc55}

Proposition \ref{cc2prop8} says $\Spec\Xi_\fC:\Spec\ci\Ga\ci\Spec\fC\ra\Spec\fC$ is an isomorphism for any $C^\iy$-ring $\fC$. This was used in \S\ref{cc25} to define the category $\CRingsco$ of {\it complete\/} $C^\iy$-rings (those isomorphic to $\Ga\ci\Spec\fC$), such that $(\CRingsco)^{\bf op}$ is equivalent to the category of affine $C^\iy$-schemes $\ACSch$. This can be used to prove that fibre products and finite limits exist in~$\CSch$.

We now consider to what extent these results generalize to the corners case. Firstly, it turns out that $\Specc\bs\Xi_\bfC$ need not be an isomorphism for $C^\iy$-rings with corners $\bfC$. The next remark and example explain what can go wrong.

\begin{rem}
\label{cc5rem3}
Let $\bfC=(\fC,\fC_\rex)$ be a $C^\iy$-ring with corners. Write $\bX=(X,\bO_X)=\Specc\bfC$, an affine $C^\iy$-scheme with corners. The global sections $\Gac(\bX)=\bO_X(X)$ is a $C^\iy$-ring with corners, with a morphism $\bs\Xi_\bfC:\bfC\ra\bO_X(X)$, so we have a morphism of $C^\iy$-schemes with corners
\e
\Specc\bs\Xi_\bfC=\Specc(\Xi,\Xi_\rex):\Specc\ci\Gac\ci\Specc\bfC\longra\Specc\bfC.
\label{cc5eq4}
\e
We want to know how close \eq{cc5eq4} is to being an isomorphism.

First note that as the underlying $C^\iy$-scheme of $\bX$ is $\uX=\Spec\fC$, Proposition \ref{cc2prop8} shows \eq{cc5eq4} is an isomorphism on the level of $C^\iy$-schemes, and hence of topological spaces. Thus \eq{cc5eq4} is an isomorphism if $\Xi_\rex$ induces an isomorphism of sheaves of monoids. It is sufficient to check this on stalks. Therefore, \eq{cc5eq4} is an isomorphism if the following is an isomorphism for all~$x\in X$:
\begin{equation*}
\Xi_{x,\rex}:\fC_{x,\rex}\cong \O^\rex_{X,x}\longra(\O_X^\rex(X))_{x_*}.
\end{equation*}

Consider the diagram of monoids for $x\in X$:
\e
\begin{gathered}
\xymatrix@C=100pt@R=18pt{ *+[r]{\fC_\rex} \ar@{->>}[r]^(0.4){\pi_{x,\rex}} \ar[d]^{\Xi_\rex} & *+[l]{\fC_{x,\rex}\cong \O^\rex_{X,x_{\vphantom{(}}}} \ar@<-1ex>@{^{(}->}[d]_{\Xi_{x,\rex}} \\
*+[r]{\O_X^\rex(X)} \ar@{.>}[ru]^(0.4){\rho^\rex_{X,x}} \ar@{->>}[r]_(0.4){\hat\pi_{x,\rex}}& *+[l]{(\O_X^\rex(X))_{x_*}.} }
\end{gathered}
\label{cc5eq5}
\e
Here $\rho_{X,x}^\rex$ takes an element $s'\in \O_X^\rex(X)$ to its value in the stalk $\O_{X,x}^\rex$. It is not difficult to show that the upper left triangle and the outer rectangle in \eq{cc5eq5} commute, $\pi_{x,\rex},\hat\pi_{x,\rex}$ are surjective, and $\Xi_{x,\rex}$ is injective. However, as Example \ref{cc5ex1} shows, $\Xi_{x,\rex}$ need not be surjective (in which case \eq{cc5eq4} is not an isomorphism), and the bottom right triangle of \eq{cc5eq5} need not commute.

That is, if two elements of $\O_X^\rex(X)$ agree locally, then while they have the same image in the stalk $\O^\rex_{X,x}$ they do not necessarily have the same value in the localization of $\O_X^\rex(X)$ at $x_*$. This is because for $c',d'\in \fC_\rex$, equality in $\fC_{\rex,x}$ requires a global equality. That is, there need to be $a',b'\in \fC_\rex$ such that $a'c'=b'd'\in\fC_\rex$ with $a',b'$ satisfying additional conditions as in Proposition \ref{cc4prop8}. In the $C^\iy$-ring $\fC$, this equality is only a local equality, as the $a'$ and $b'$ can come from bump functions. However, in the monoid, bump functions do not necessarily exist, meaning this condition is stronger and harder to satisfy. 

In the following example, $\bfC$ is interior, and the above discussion also holds with $\Speccin,\Gacin$ in place of $\Specc,\Gac$. 
\end{rem}

\begin{ex}
\label{cc5ex1}
Define open subsets $U,V\subset\R^2$ by $U=(-1,1)\t(-1,\iy)$ and $V=(-1,1)\t(-\iy,1)$. The important properties for our purposes are that $\R^2\sm U$, $\R^2\sm V$ are connected, but $\R^2\sm(U\cup V)$ is disconnected, being divided into connected components $(-\iy,-1]\t\R$ and~$[1,\iy)\t\R$.

Choose smooth $\vp,\psi:\R^2\ra\R$ such that $\vp\vert_U>0$, $\vp\vert_{\R^2\sm U}=0$, $\psi\vert_V>0$, $\psi\vert_{\R^2\sm V}=0$, so $\vp,\psi$ map to $[0,\iy)$. Define a $C^\iy$-ring with corners $\bfC$ by
\begin{align*}
\bfC=\R(x_1,x_2)[y_1,\ldots &,y_6]\big/\bigl(\Phi_i(y_1)\!=\!\Phi_i(y_2)\!=\!0,\; \Phi_i(y_3)=\Phi_i(y_4)=\vp(x_1,x_2), \\
&\Phi_i(y_5)\!=\!\Phi_i(y_6)\!=\!\psi(x_1,x_2)\bigr)\bigl[y_1y_3\!=\!y_2y_4,\; y_1y_5\!=\!y_2y_6\bigr],
\end{align*}
using the notation of Definition \ref{cc4def8}. Let $\bX=(X,\bO_X)=\Specc\bfC$, an affine $C^\iy$-scheme with corners. As a topological space we write
\begin{align*}
X=\bigl\{(x_1,x_2,y_1,\ldots,y_6)\in\R^2\t[0,\iy)^6:y_1&=y_2=0,\; y_3=y_4=\vp(x_1,x_2),\\
y_5&=y_6=\psi(x_1,x_2)\bigr\}.
\end{align*}
The projection $X\ra\R^2$ mapping $(x_1,x_2,y_1,\ldots,y_6)\mapsto (x_1,x_2)$ is a homeomorphism, and we use this to identify $X\cong\R^2$. 

We will show that \eq{cc5eq4} is {\it not\/} an isomorphism, in contrast to Proposition \ref{cc2prop8} for $C^\iy$-rings and $C^\iy$-schemes.

For each $x\in X$ we have a monoid $\O_{X,x,\rex}\cong\fC_{x,\rex}$ by Proposition \ref{cc5prop1}, containing elements $\pi_{x,\rex}(y_j)$ for $j=1,\ldots,6$. If $x\in U$ then using Proposition \ref{cc4prop8} and the relations $\Phi_i(y_3)=\Phi_i(y_4)=\vp(x_1,x_2)$, $y_1y_3=y_2y_4$, and $\vp(x)>0$ we see that $\pi_{x,\rex}(y_1)=\pi_{x,\rex}(y_2)$. Similarly if $x\in V$ then  $\Phi_i(y_5)=\Phi_i(y_6)=\psi(x_1,x_2)$, $y_1y_5=y_2y_6$, and $\psi(x)>0$ give $\pi_{x,\rex}(y_1)=\pi_{x,\rex}(y_2)$. Thus as $U\cup V=(-1,1)\t\R$ we may define an element $z\in\O^\rex_X(X)$ by
\begin{equation*}
z(x_1,x_2)=\begin{cases} \pi_{x,\rex}(y_1), & x\le -1, \\
\pi_{x,\rex}(y_1)=\pi_{x,\rex}(y_2), & x\in(-1,1), \\
\pi_{x,\rex}(y_2), & x\ge 1.
\end{cases}
\end{equation*}

We will show that for $x=(-2,0)$ in $X$ there does not exist $c'\in\fC_\rex$ with $\Xi_{x,\rex}\ci\pi_{x,\rex}(c')=\hat\pi_{x,\rex}(z)$ in $(\O_X^\rex(X))_{x_*}$. Thus $\Xi_{x,\rex}:\O^\rex_{X,x}\ra (\O_X^\rex(X))_{x_*}$ is not surjective. This implies that the bottom right triangle of \eq{cc5eq5} does not commute at $z\in\O_X^\rex(X)$. Also $\Specc\bs\Xi_\bfC$ in \eq{cc5eq4} is not an isomorphism, as $\Xi_{x,\rex}$ is part of the action of $\Specc\bs\Xi_\bfC$ on stalks at $x$, and is not an isomorphism.

Suppose for a contradiction that such $c'$ does exist. Then $\hat\pi_{x,\rex}\ci\Xi_\rex(c')=\hat\pi_{x,\rex}(z)$ as the outer rectangle of \eq{cc5eq5} commutes, so by Proposition \ref{cc4prop8} there exist $a',b'\in \O_X^\rex(X)$ with $a'(x)=b'(x)\ne 0$ and $a'=b'$ near $x$, such that $a'z=b'\Xi_\rex(c')$. Locally on $X$, $a',b'$ are of the form $\exp(f(x_1,x_2,y_1,\ldots,y_6))y_1^{a_1}\ldots y_6^{a_6}$ or zero, where the transitions between different representations of this form (e.g.\ changing $a_1,\ldots,a_6$) satisfy strict rules. We will consider $a',b'$ in the regions $x_1\le -1$, $x_1\ge 1$, $U\sm V$, $V\sm U$. In each region, $a',b'$ must admit a global representation~$\exp(f)y_1^{a_1}\ldots y_6^{a_6}$.

Since $a'(x)=b'(x)\ne 0$, we see that $a',b'$ are of the form $\exp(f)$ in $x_1\le -1$. On crossing the wall $x=-1$, $y\le -1$ between $x\le -1$ and $U\sm V$, $y_3,y_4$ become invertible, so $a',b'$ may be of the form $\exp(f)y_3^{a_3}y_4^{a_4}$ in $U\sm V$. On crossing the wall $x=1$, $y\le -1$ between $U\sm V$ and $x\ge 1$, we see that $a',b'$ must be of the form $\exp(f)y_3^{a_3}y_4^{a_4}$ in $x_1\ge 1$. Similarly, on crossing the wall $x=-1$, $y\ge 1$ between $x\le -1$ and $V\sm U$, $y_5,y_6$ become invertible, so $a',b'$ may be of the form $\exp(f)y_5^{a_5}y_6^{a_6}$ in $V\sm U$. On crossing the wall $x=1$, $y\ge 1$ we see that $a',b'$ must be of the form $\exp(f)y_5^{a_5}y_6^{a_6}$ in $x_1\ge 1$. Comparing this with the previous statement gives $a_5=a_6=0$. Hence $a',b'$ are invertible at~$(2,0)$.

Now $c'$ is either zero, or of the form $\exp(f)y_1^{a_1}\ldots y_6^{a_6}$. Specializing $a'z=b'\Xi_\rex(c')$ at $(-2,0)$ and using $a',b'$ invertible there gives $c'\ne 0$ with $a_1=0$ and $a_i=0$ for $i\ne 1$. But specializing at $(2,0)$ and using $a',b'$ invertible there gives $a_2=1$ and $a_i=0$ for $i\ne 2$, a contradiction. So no such $c'$ exists.

One can prove that $\bfC$ above is toric, and hence also firm, saturated, torsion-free, integral, and interior, in the notation of \S\ref{cc47}. There are simpler examples of $\bfC$ lacking these properties for which  \eq{cc5eq4} is not an isomorphism and \eq{cc5eq5} does not commute. We chose this example to show that we cannot easily avoid the problem by restricting to nicer classes of $C^\iy$-rings with corners.
\end{ex}

Example \ref{cc5ex1} implies that we cannot generalize complete $C^\iy$-rings in \S\ref{cc25} to the corners case, such that the analogue of Theorem \ref{cc2thm3}(a) holds. But we will use the following proposition to define {\it semi-complete\/} $C^\iy$-rings with corners, which have some of the good properties of complete $C^\iy$-rings. Note that Example \ref{cc5ex1} gives an example where no choice of $\fD_\rex$ can make the canonical map $(\fD,\fD_\rex)\ra\Gac\ci\Specc(\fD,\fD_\rex)$ surjective on the monoids.

\begin{prop}
\label{cc5prop2}
Let\/ $(\fC,\fC_\rex)$ be a $C^\iy$-ring with corners and set\/ $\bX=\Specc(\fC,\fC_\rex)$. Then there is a $C^\iy$-ring with corners $(\fD,\fD_\rex)$ with\/ $\fD\cong \Ga\ci\Spec\fC$ a complete $C^\iy$-ring, such that\/ $\Specc(\fD,\fD_\rex)\cong \bX$ and the canonical map $(\fD,\fD_\rex)\ra \Gac\ci\Specc(\fD,\fD_\rex)$ is an isomorphism on $\fD,$ and injective on $\fD_\rex$. If\/ $(\fC,\fC_\rex)$ is firm, or interior, then $(\fD,\fD_\rex)$ is firm, or interior.
\end{prop}

\begin{proof}
We define $(\fD,\fD_\rex)$ such that $\fD=\Ga\ci\Spec\fC=\O_X(X)$, and let $\fD_\rex$ be the submonoid of $\O_X^\rex(X)$ generated by the invertible elements $\Psi_{\exp}(\fD)$ and the image $\phi_\rex(\fC_\rex)$. One can check that the $C^\iy$-operations from $(\O_X(X),\O_X^\rex(X))$ restrict to $C^\iy$-operations on $(\fD,\fD_\rex)$, and make $(\fD,\fD_\rex)$ into a $C^\iy$-ring with corners. Let $\bY=\Specc(\fD,\fD_\rex)$. If $(\fC,\fC_\rex)$ is firm, then $(\fD,\fD_\rex)$ is firm, as the sharpening $\fD_\rex^\sh$ is the image of $\fC_\rex^\sh$ under $\phi_\rex$, hence the image of the generators generates $\fD_\rex^\sh$. If $(\fC,\fC_\rex)$ is interior, then $(\fD,\fD_\rex)$ is interior, as elements in both $\Psi_{\exp}(\fD)$ and $\phi_\rex(\fC_\rex)$ have no zero divisors. 

Now the canonical morphism $(\phi,\phi_\rex):(\fC,\fC_\rex)\ra \Gac\ci\Specc(\fC,\fC_\rex)$ gives a morphism $(\psi,\psi_\rex):(\fC,\fC_\rex)\ra (\fD,\fD_\rex)$ where $\psi=\phi$, and $\psi_\rex=\phi_\rex$ with its image restricted to the submonoid $\fD_\rex$ of $\O_{X,\rex}(X)$. As $\fD$ is complete, then $\Spec(\fD)\cong\Spec(\fC)\cong (X,\O_X)$, and $\Specc(\psi,\psi_\rex):\bY\cong\Specc(\fD,\fD_\rex)\ra\Specc(\fC,\fC_\rex)\cong \bX$ is an isomorphism on the topological space and the sheaves of $C^\iy$-rings. To show that $\Specc(\psi,\psi_\rex)$ is an isomorphism on the sheaves of monoids, we show $\psi_\rex$ induces an isomorphism on the stalks $\O_{X,x}^\rex\cong\fC_{x,\rex}$ and $\O_{Y,x}^\rex\cong\fD_{x,\rex}$, for all $\R$-points $x\in X$. 

The stalk map corresponds to the morphism $\psi_{x,\rex}:\fC_{x,\rex}\ra\fD_{x,\rex}$, which is defined by $\psi_{x,\rex}(\pi_{x,\rex}(c'))=\hat\pi_{x,\rex}(\psi_\rex(c'))$, where $\pi_{x,\rex}:\fC_\rex\ra\fC_{x,\rex}$ and $\hat\pi_{x,\rex}:\fD_\rex\ra\fD_{x,\rex}$ are the localization morphisms. Now, as $\Specc(\psi,\psi_\rex)$ is an isomorphism on the sheaves of $C^\iy$-rings, we know that $\psi_{x}:\fC_x\ra\fD_x$ is an isomorphism, which implies we have an isomorphism $\psi_{x,\rex}\vert_{\fC_{x,\rex}^\t}:\fC_{x,\rex}^\t\ra\fD_{x,\rex}^\t$ of invertible elements in the monoids. This gives the following commutative diagram of monoids:
\begin{equation*}
\xymatrix@C=50pt@R=17pt{ &*+[r]{\fC_\rex} \ar[r]_{\psi_\rex} \ar@{>>}[d]^{\pi_{x,\rex}} & *+[r]{\fD_\rex\subset\O_X^\rex(X)}  \ar@{>>}[d]^{\hat\pi_{x,\rex}}& \\
*+[r]{\,\fC^{\t}_{x,\rex}\,} \ar@{^{(}->}[r] \ar@<-.8ex>@/_.9pc/[rrr]^(0.45){\cong} &*+[r]{\fC_{x,\rex}}  \ar[r]^{\psi_{x,\rex}}& *+[r]{\,\fD_{x,\rex}\,}& *+[l]{\,\fD^\t_{x,\rex}.}  \ar@{_{(}->}[l] }
\end{equation*}

To show $\psi_{x,\rex}$ is injective, suppose $a_x',b_x'\in\fC_{x,\rex}$ with $\psi_{x,\rex}(a_x')=\psi_{x,\rex}(b_x')$. As $\pi_{x,\rex}$ is surjective we have $a_x'=\pi_{x,\rex}(a')$, $b_x'=\pi_{x,\rex}(b')$ in $\fC_{x,\rex}$ for $a',b'\in \fC_\rex$. Then in $\fD_{x,\rex}$ we have 
\begin{align*}
\hat\pi_{x,\rex}(\psi_\rex(a'))&=\psi_{x,\rex}(\pi_{x,\rex}(a'))=\psi_{x,\rex}(a_x')\\
&=\psi_{x,\rex}(b_x')=\psi_{x,\rex}(\pi_{x,\rex}(b'))=\hat\pi_{x,\rex}(\psi_\rex(b')),
\end{align*}
so Proposition \ref{cc4prop8} gives $e',f'\in\fD_\rex$ such that $e'\psi_\rex(a')=f'\psi_\rex(b')$ in $\fD_\rex$, with $\Phi_i(e')-\Phi_i(f')\in I$ and $x\ci\Phi_i(e')\ne 0$. Hence
\e
\begin{split}
e'\vert_xa'_x&=e'\vert_x\pi_{x,\rex}(a')=e'\vert_x\psi_\rex(a')\vert_x=(e'\psi_\rex(a'))\vert_x\\
&=(f'\psi_\rex(b'))\vert_x=f'\vert_x\psi_\rex(b')\vert_x=f'\vert_x\pi_{x,\rex}(b')=f'\vert_xb'_x
\end{split}
\label{cc5eq6}
\e
in $\fC_{x,\rex}$. Under the maps $\Phi_{i,x}:\fC_{x,\rex}\ra\fC_x$ and $\pi:\fC_x\ra\R$ we have $\Phi_{i,x}(e'\vert_x)=\Phi_{i,x}(f'\vert_x)$ as $\Phi_i(e')-\Phi_i(f')\in I$, and $\pi\ci\Phi_{i,x}(e'\vert_x)=x\ci\Phi_i(e')\ne 0$. Hence $e'\vert_x$ is invertible in $\fC_{x,\rex}$, so $\Phi_{i,x}(e'\vert_x)=\Phi_{i,x}(f'\vert_x)$ and Definition \ref{cc4def6}(i) imply that $e'\vert_x=f'\vert_x$, and thus \eq{cc5eq6} shows that $a_x'=b_x'$, so $\psi_{x,\rex}$ is injective.

To show $\psi_{x,\rex}$ is surjective, let $d_x'\in \fD_{x,\rex}$, and write $d_x'=\hat\pi_{x,\rex}(d')$ for $d'\in \fD_\rex$. As $\fD_\rex$ is generated by $\psi_\rex(\fC_\rex)$ and $\Psi_{\exp}(\fD)$ we have $d'=\psi_\rex(c')\cdot e'$ for $c'\in \fC$ and $e'\in \Psi_{\exp}(\fD)$ invertible. 
Thus $d_x'=\psi_{x,\rex}\ci\pi_{x,\rex}(c')\cdot \hat\pi_{x,\rex}(e')$, where $\hat\pi_{x,\rex}(e')\in\fD_{x,\rex}^\t$. As $\psi_{x,\rex}\vert_{\fC_{x,\rex}^\t}:\fC_{x,\rex}^\t\ra\fD_{x,\rex}^\t$ is an isomorphism and $\pi_{x,\rex}$ is surjective, there exists $c''\in \fC_\rex$ with $\psi_{x,\rex}\ci\pi_{x,\rex}(c'')=\hat\pi_{x,\rex}(e')$. Then $d_x'=\psi_{x,\rex}[\pi_{x,\rex}(c'c'')]$, so $\psi_{x,\rex}$ is surjective. Hence $\psi_{x,\rex}$ is an isomorphism, and $\Specc(\psi,\psi_\rex):\Specc(\fD,\fD_\rex)\ra\Specc(\fC,\fC_\rex=\bX$ is an isomorphism.

Thus there is an isomorphism $\Gac\ci\Specc(\fD,\fD_\rex)\cong \Gac\ci\Specc(\fC,\fC_\rex)$, which identifies the canonical map $(\fD,\fD_\rex)\ra \Gac\ci\Specc(\fD,\fD_\rex)$ with the inclusion $(\fD,\fD_\rex)\hookra (\O_X(X),\O_X^\rex(X))=\Gac\ci\Specc(\fC,\fC_\rex)$. By definition this inclusion is an isomorphism on $\fD,$ and injective on $\fD_\rex$. This completes the proof.
\end{proof}

\begin{dfn}
\label{cc5def10}
We call a $C^\iy$-ring with corners $\bfD=(\fD,\fD_\rex)$ {\it semi-complete\/} if $\fD$ is complete, and $\bs\Xi_\bfD:(\fD,\fD_\rex)\ra \Gac\ci\Specc(\fD,\fD_\rex)$ is an isomorphism on $\fD$ (this is equivalent to $\fD$ complete) and injective on $\fD_\rex$. Write $\CRingscsc\subset\CRingsc$ and $\CRingscscin\subset\CRingscin$ for the full subcategories of (interior) semi-complete $C^\iy$-rings with corners.

Given a $C^\iy$-ring with corners $\bfC$, Proposition \ref{cc5prop2} constructs a semi-com\-pl\-ete $C^\iy$-ring with corners $\bfD$ with a morphism $\bs\psi:\bfC\ra\bfD$ such that $\Specc\bs\psi:\Specc\bfD\ra\Specc\bfC$ is an isomorphism. It is easy to show that this map $\bfC\mapsto\bfD$ is functorial, yielding a functor $\Pi_{\rm all}^{\rm sc}:\CRingsc\ra\CRingscsc$ mapping $\bfC\mapsto\bfD$ on objects, with $\Specc\cong\Specc\ci \Pi_{\rm all}^{\rm sc}$. Similarly we get $\Pi_{\rm in}^{\rm sc,in}:\CRingscin\ra\CRingscscin$ with~$\Speccin\cong\Speccin\ci \Pi_{\rm in}^{\rm sc,in}$.
\end{dfn}

Here is a partial analogue of Theorem \ref{cc2thm3}:

\begin{thm}
\label{cc5thm6}
{\bf(a)} $\Specc\vert_{\cdots}:(\CRingscsc)^{\bf op}\ra\ACSchc$ is faithful and essentially surjective, but not full.
\smallskip

\noindent{\bf(b)} Let\/ $\bX$ be an affine $C^\iy$-scheme with corners. Then $\bX\cong\Specc\bfD,$ where $\bfD$ is a semi-complete $C^\iy$-ring with corners.
\smallskip

\noindent{\bf(c)} The functor $\Pi_{\rm all}^{\rm sc}:\CRingsc\ra\CRingscsc$ is left adjoint to the inclusion\/ $\inc:\CRingscsc\hookra\CRingsc$. That is, $\Pi_{\rm all}^{\rm sc}$ is a \begin{bfseries}reflection functor\end{bfseries}.

\smallskip

\noindent{\bf(d)} All small colimits exist in $\CRingscsc,$ although they may not coincide with the corresponding small colimits in $\CRingsc$. 
\smallskip

\noindent{\bf(e)} $\Specc\vert_{\cdots}:(\CRingscsc)^{\bf op}\ra\LCRSc$ is right adjoint to $\Pi_{\rm all}^{\rm sc}\ci\Gac:\LCRSc\ra(\CRingscsc)^{\bf op}$. Thus $\Spec\vert_{\cdots}$ takes limits in $(\CRingscsc)^{\bf op}$ (equivalently, colimits in $\CRingscsc$) to limits in~$\LCRSc$.
\smallskip

The analogues of\/ {\bf(a)\rm--\bf(e)} hold in the interior case, replacing $\CRingscsc,\ab\ldots,\ab\LCRSc$ by\/~$\CRingscscin,\ldots,\LCRScin$.
\end{thm}

\begin{proof} For (a), if $\bs\phi:\bfD\ra\bfC$ is a morphism in $\CRingscsc$, we have a commutative diagram
\e
\begin{gathered}
\xymatrix@C=140pt@R=15pt{
*+[r]{\bfD\,} \ar[d]^{\bs\phi} \ar@{^{(}->}[r]_(0.4){\bs\Xi_\bfD} & *+[l]{\Gac\ci\Specc\bfD}\ar[d]_{\Gac\ci\Specc\bs\phi} \\
*+[r]{\bfC\,} \ar@{^{(}->}[r]^(0.4){\bs\Xi_\bfC} & *+[l]{\Gac\ci\Specc\bfC.\!} }	
\end{gathered}
\label{cc5eq7}
\e
As the rows are injective by semi-completeness, $\Gac\ci\Specc\bs\phi$ determines $\bs\phi$, so $\Specc\bs\phi$ determines $\bs\phi$. Thus $\Specc\vert_{\cdots}$ is injective on morphisms, that is, it is faithful. Essential surjectivity follows by Proposition \ref{cc5prop2}. To see that $\Specc\vert_{\cdots}$ is not full, let $\bfC$ be as in Example \ref{cc5ex1}, and set $\bfD=\bs C^\iy([0,\iy))$. Then
\begin{equation*}
\Spec^{\rm c}_{\bfC,\bfD}:\Hom_\CRingscsc(\bfD,\bfC)\longra \Hom_\ACSchc(\Specc\bfC,\Specc\bfD)
\end{equation*}
may be identified with the non-surjective map
\begin{equation*}
\Xi_\rex:\fC_\rex\longra\O_X^\rex(X).
\end{equation*}

Part (b) is immediate from Proposition \ref{cc5prop2}. Part (c) is easy to check, where the unit of the adjunction is $\bs\psi:\bfC\ra\bfD=\inc\ci \Pi_{\rm all}^{\rm sc}(\bfC)$, and the counit is $\bs\id_\bfD:\Pi_{\rm all}^{\rm sc}\ci\inc(\bfD)=\bfD\ra\bfD$. For (d), given a small colimit in $\CRingscsc$, the colimit exists in $\CRingsc$ by Theorem \ref{cc4thm3}(b), and applying the reflection functor $\Pi_{\rm all}^{\rm sc}$ gives the (possibly different) colimit in $\CRingscsc$. Part (e) holds as $\Pi_{\rm all}^{\rm sc},\Gac$ are left adjoint to $\inc,\Specc$ by (c) and Theorem \ref{cc5thm3}. The extension to the interior case is immediate.
\end{proof}

\begin{thm}
\label{cc5thm7}
{\bf(a)} Suppose\/ $\bfC\in\CRingsc$ and\/ $\bX=\Specc\bfC$ lies in\/ $\LCRScin\subset\LCRSc$. Then $\bX\cong\Speccin\bfD$ for some\/~$\bfD\in\CRingscin$.

\smallskip

\noindent{\bf(b)} An object\/ $\bX\!\in\!\CSchc$ lies in $\CSchcin$ if and only if it lies in\/~$\LCRScin$.
\end{thm}

\begin{proof}
For (a), by Theorem \ref{cc5thm6}(b) we may take $\bfC$ to be semi-complete, so that $\bs\Xi_\bfC:(\fC,\fC_\rex)\ra(\O_X(X),\O_X^\rex(X))$ is an isomorphism on $\fC$ and injective on $\fC_\rex$. As $\bX$ lies in $\LCRScin$, the stalks $\bO_{X,x}$ are interior for $x\in X$, and writing
\begin{equation*}
\O_X^\rin(X)=\bigl\{f'\in\O_X^\rex(X):\text{$f'\vert_x\ne 0$ in $\O^\rex_{X,x}$ for all $x\in X$}\bigr\}, 
\end{equation*}
then $\bO_X^\rin(X)=(\O_X(X),\O_X^\rin(X)\amalg\{0\})$ is an interior $C^\iy$-subring with corners of $\bO_X(X)=(\O_X(X),\O_X^\rex(X))$. Define $\bfE\!=\!(\fE,\fE_\rex)$, where $\fE\!=\!\O_X(X)$ and
\e
\begin{split}
\fE_\rex=\bigl\{e'&\in\O_X^\rex(X):\text{ there exists a decomposition $X=\ts\coprod_{i\in I}X_i$ with}\\
&\text{$X_i$ open and closed in $X$ and elements $c_i'\in\fC_\rex$ for $i\in I$} \\ 
&\text{with $\Xi_{\fC,\rex}(c_i')\vert_{X_i}=e'\vert_{X_i}$ for all $i\in I$}\bigr\}.
\end{split}
\label{cc5eq8}
\e

We claim $\fE,\fE_\rex$ are closed under the $C^\iy$-operations on $\O_X(X),\ab\O_X^\rex(X)$, so that $\bfE$ is a $C^\iy$-subring with corners of $\bO_X(X)$. To see this, let $g:\R^m\t[0,\iy)^n\ra[0,\iy)$ be exterior and $e_1,\ldots,e_m\in\fE$, $e_1',\ldots,e_n'\in\fE_\rex$. Then $e_j=\Xi_\bfC(c_j)$ for unique $c_j\in\fC$, $j=1,\ldots,m$ as $\Xi_\bfC$ is an isomorphism. Also for each $k=1,\ldots,n$ we get a decomposition $X=\ts\coprod_{i\in I}X_i$ in \eq{cc5eq8} for $e_k'$. By intersecting the subsets of the decompositions we can take a common decomposition $X=\ts\coprod_{i\in I}X_i$ for all $k=1,\ldots,n$, and elements $c_{i,k}'\in\fC_\rex$ with $\Xi_{\fC,\rex}(c_{i,k}')\vert_{X_i}=e_k'\vert_{X_i}$ for all $i\in I$ and $k=1,\ldots,n$. Then $e'=\Psi_g(e_1,\ldots,e_m,e_1',\ldots,e_n')\in \O_X^\rex(X)$ satisfies the conditions of \eq{cc5eq8} with the same decomposition $X=\ts\coprod_{i\in I}X_i$, and elements $c_i'=\Psi_g(c_1,\ldots,c_m,c_{i,1}',\ldots,c_{i,n}')$, so $e'\in\fE_\rex$, and~$\bfE\in\CRingsc$.

Observe that $\bs\Xi_\fC:\bfC\ra\bO_X(X)$ factors via the inclusion $\bfE\hookra\bO_X(X)$, by taking the decomposition in \eq{cc5eq8} to be into one set~$X_1=X$.
 
Next define $\fD=\fE$ and $\fD_\rin=\O_X^\rin(X)\cap\fE_\rex$, so that $\bfD=(\fD,\fD_\rin\amalg\{0\})$ is the intersection of the $C^\iy$-subrings $\bO_X^\rin(X)$ and $\bfE$ in $\bO_X(X)$, and thus lies in $\CRingscin$, as $\bO_X^\rin(X)$ does. We now have morphisms in $\CRingsc$
\begin{equation*}
\xymatrix@C=90pt{ \bfC \ar[r]^{\bs\Xi_\bfC} & \bfE & \bfD, \ar[l]_{\inc} }
\end{equation*}
giving morphisms in $\LCRSc$
\e
\xymatrix@C=55pt{ \bX=\Specc\bfC  & \Specc\bfE \ar[l]_(0.43){\Specc\bs\Xi_\bfC} \ar[r]^(0.38){\Specc\inc} & \Specc\bfD=\Speccin\bfD. }	
\label{cc5eq9}
\e

We will show that both morphisms in \eq{cc5eq9} are isomorphisms, so that $\bX\cong\Speccin\bfD$, as we have to prove. As $\fC\cong\fD=\fE$, equation \eq{cc5eq9} is isomorphic on the level of $C^\iy$-schemes. Thus it is sufficient to prove \eq{cc5eq9} induces isomorphisms on the stalks of the sheaves of monoids. That is, for each $x\in X$ we must show that the following maps are isomorphisms:
\begin{equation*}
\xymatrix@C=90pt{ \fC_{x,\rex} \ar[r]^{\Xi_{\fC,x,\rex}} & \fE_{x,\rex} & \fD_{x,\rex}. \ar[l]_{\inc_{x,\rex}} }
\end{equation*}

Suppose $e''\in\fE_{x,\rex}$. Then $e''=\pi_{\fE_{x,\rex}}(e')$ for $e'\in\fE_\rex$. Let $X_i,c_i'$, $i\in I$ be as in \eq{cc5eq8} for $e'$. Then $x\in X_i$ for unique $i\in I$. Define $a',b'\in\fE_\rex$ by $a'\vert_{X_i}=1$, $a'\vert_{X\sm X_i}=0$ and $b'=a'$. Then $a'e'=b'\Xi_{\fC,\rex}(c_i')$, so Proposition \ref{cc4prop8} gives $e''\!=\!\pi_{\fE_{x,\rex}}(e')\!=\!\pi_{\fE_{x,\rex}}\!\ci\!\Xi_{\fC,\rex}(c_i')\!=\!\Xi_{\fC,x,\rex}\!\ci\!\pi_{\fC_{x,\rex}}(c_i')$. Thus $\Xi_{\fC,x,\rex}$ is surjective.

As the composition $\fC_{x,\rex}\,{\buildrel\Xi_{\fC,x,\rex}\over\longra}\,\fE_{x,\rex}\ra\O_X^\rex(X)_x\ra \O_{X,x}^\rex$ is an isomorphism since $\bX=\Specc\bfC$, we see that $\Xi_{\fC,x,\rex}$ is injective, and thus an isomorphism.

Suppose $0\ne e''\in\fE_{x,\rex}$. Then $e''=\pi_{\fE_{x,\rex}}(e')$ for $e'\in\fE_\rex\subseteq\O_X^\rex(X)$. Define $Y=\bigl\{y\in X:e'\vert_y\ne 0$ in $\O_{X,y}^\rex\bigr\}$. Then $Y$ is open and closed in $X$ by Definition \ref{cc5def2}(a), as $\bX$ is interior, and $x\in Y$ as $e''\ne 0$. Define $d'\in\fD_\rin$ by $d'\vert_Y=e'\vert_Y$ and $d'\vert_{X\sm Y}=1$. Define $a',b'\in\fE_\rex$ by $a'\vert_Y=1$, $a'\vert_{X\sm Y}=0$ and $b'=a'$. Then $a'd'=b'e'$ with $x\ci\Phi_i(a')=1$, so Proposition \ref{cc4prop8} gives $e''=\pi_{\fE_{x,\rex}}(e')=\pi_{\fE_{x,\rex}}\ci\inc(d')=\inc_{x,\rex}\ci\pi_{\fD_{x,\rex}}(d')$, and $\inc_{x,\rex}$ is surjective.

Suppose $d_1'',d_2''\in\fD_{x,\rex}$ with $\inc_{x,\rex}(d_1'')=\inc_{x,\rex}(d_2'')$. Then $d''_i=\pi_{\fD_{x,\rex}}(d'_i)$ for $d'_i\in\fD_\rin\amalg\{0\}$, $i=1,2$. By Proposition \ref{cc4prop8}, $\inc_{x,\rex}(d_1'')=\inc_{x,\rex}(d_2'')$ means that there exist $a',b'\in\fE_\rex$ with $\Phi_i(a')=\Phi_i(b')$ near $x$, and $x\ci\Phi_i(a')\ne 0$, and $a'd_1'=b'd_2'$ in $\fE_\rex$. If $a'd_1'\vert_x=b'd_2'\vert_x=0$ in $\O_{X,x}^\rin$ then $d_1'=d_2'=0$, so $d_1''=d_2''=0$, as we want. Otherwise 
$d_1',d_2'\ne 0$, so $d_1',d_2'\in\fD_\rin$.

Define $Y=\bigl\{y\in X:a'\vert_y\ne 0$ in $\O_{X,y}^\rex\bigr\}$. Then $Y$ is open and closed in $X$ as above, and $Y=\bigl\{y\in X:b'\vert_y\ne 0$ in $\O_{X,y}^\rex\bigr\}$, since $a'd_1'=b'd_2'$ with $d_1',d_2'\in\fD_\rin$. Define $a'',b''\in\fD_\rin$ by $a''\vert_Y=a'\vert_Y$, $a''\vert_{X\sm Y}=d_2'\vert_{X\sm Y}$, $b''\vert_Y=b'\vert_Y$, $b''\vert_{X\sm Y}=d_1'\vert_{X\sm Y}$. Then $\Phi_i(a'')=\Phi_i(b'')$ near $x$ as $x\in Y$ and $\Phi_i(a')=\Phi_i(b')$ near $x$,  and $x\ci\Phi_i(a'')\ne 0$, and $a''d_1'=b''d_2'$, so Proposition \ref{cc4prop8} says that $d''_1=\pi_{\fD_{x,\rex}}(d'_1)=\pi_{\fD_{x,\rex}}(d'_2)=d''_2$. Hence $\inc_{x,\rex}$ is injective, and thus an isomorphism. This proves (a). Part (b) follows from (a) by covering $\bX$ with open affine $C^\iy$-subschemes.
\end{proof}

\subsection{\texorpdfstring{Special classes of $C^\iy$-schemes with corners}{Special classes of C∞-schemes with corners}}
\label{cc56}

\begin{dfn}
\label{cc5def11}
Let $\bX$ be a $C^\iy$-scheme with corners. We call $\bX$ {\it finitely presented}, or {\it finitely generated}, or {\it firm}, or {\it integral}, or {\it torsion-free}, or {\it saturated}, or {\it toric}, or {\it simplicial}, if we can cover $\bX$ by open $\bU\subseteq\bX$ with $\bU\cong\Specc\bfC$ for $\bfC$ a $C^\iy$-ring with corners which is firm, integral, torsion-free, saturated,  toric, or simplicial, respectively, in the sense of Definitions \ref{cc4def13}, \ref{cc4def14}, \ref{cc4def15}. We write
\begin{equation*}
\CSchcfp\subset\CSchcfg\subset\CSchcfi\subset\CSchc
\end{equation*}
for the full subcategories of finitely presented, finitely generated, and firm $C^\iy$-schemes with corners. We write
\begin{equation*}
\CSchcDe\subset\CSchcto\subset\CSchcsa\subset \CSchctf\subset\CSchcZ\subset\CSchcin
\end{equation*}
for the full subcategories of simplicial, toric, saturated, torsion-free, and integral objects in $\CSchcin$, respectively. We write
\e
\begin{split}
&\CSchcDeex\subset\CSchctoex\subset\CSchcsaex\subset \\
&\CSchctfex\subset\CSchcZex\subset\CSchcinex\subset\CSchc
\end{split}
\label{cc5eq10}
\e
for the full subcategories of simplicial, \ldots, interior objects in $\CSchc$, but with exterior (i.e.\ all) rather than interior morphisms. We write 
\begin{align*}
\CSchcfiin&\subset\CSchcin, & \CSchcfiZ&\subset\CSchcZ, \\ \CSchcfitf&\subset\CSchctf,&
\CSchcfiinex&\subset\CSchcinex, \\ \CSchcfiZex&\subset\CSchcZex, &\CSchcfitfex&\subset\CSchctfex,
\end{align*}
for the full subcategories of objects which are also firm.

These live in a diagram of subcategories
\ea
\label{cc5eq11}\\[-15pt]
\begin{gathered}
\!\!\!\text{\begin{footnotesize}$\xymatrix@!0@C=29.5pt@R=25pt{
& \CSchcDe\!\! \ar[rr] \ar[dl] &&
\!\CSchcto\!\! \ar[dd] \ar[dl] \ar[rr] &&
\!\CSchcfitf\!\! \ar[dd] \ar[dl] \ar[rr] && \!\CSchcfiZ\!\! \ar[dd] \ar[dl] \ar[rr] && \!\CSchcfiin\!\! \ar[dd] \ar[dl] \\
\!\CSchcDeex\!\!\!\!\!\!\! \ar[rr] && \!\CSchctoex\!\!\!\!\!\!\! \ar[dd] \ar[rr] &&
\!\CSchcfitfex\!\!\!\!\!\!\! \ar[dd] \ar[rr] && \!\CSchcfiZex\!\!\!\!\!\!\! \ar[dd] \ar[rr] && \!\CSchcfiinex\!\!\!\!\!\!\!\!\!\! \ar[dd] \ar[rr] && \!\CSchcfi \ar[dd] \\
&&& \!\CSchcsa\!\! \ar[rr] \ar[dl] &&
\!\CSchctf\!\! \ar[rr] \ar[dl] && \!\CSchcZ\!\! \ar[rr] \ar[dl] && \!\CSchcin\!\! \ar[dl] \\
&& \!\!\!\!\!\CSchcsaex\!\!\! \ar[rr] &&
\!\CSchctfex\!\!\! \ar[rr] && \!\CSchcZex\!\!\!  \ar[rr] && \!\CSchcinex\!\!\!  \ar[rr] && \!\CSchc.\! }$\end{footnotesize}}
\end{gathered}
\nonumber
\ea
\end{dfn}

\begin{rem}
\label{cc5rem4}
{\bf(i)} General $C^\iy$-schemes with corners $\bX\in\CSchc$ may have `corner behaviour' which is very complicated and pathological. The subcategories of $\CSchc$ in \eq{cc5eq11} have `corner behaviour'  which is progressively nicer, and more like ordinary manifolds with corners, as we impose more conditions. 
\smallskip

\noindent{\bf(ii)} It is easy to see that if $X$ is a manifold with corners then $\bX=F_\Manc^\CSchc(X)$ in Definition \ref{cc5def9} is an object of $\CSchcDe$, since $\bs C^\iy(\R^n_k)$ is a simplicial $C^\iy$-ring with corners, and similarly if $X$ is a manifold with g-corners then $\bX=F_\Mangc^\CSchc(X)$ is an object of $\CSchcto$.

We can think of $\CSchcDe,\CSchcDeex,\CSchcto,\CSchctoex$ as good analogues of $\Mancin,\Manc,\Mangcin,\Mangc$  in the world of $C^\iy$-schemes.
\smallskip

\noindent{\bf(iii)} For a $C^\iy$-scheme with corners $\bX$ to be {\it firm\/} basically means that locally $\bX$ has only finitely many boundary faces. This is a mild condition, which will hold automatically in most interesting applications. An example of a non-firm $C^\iy$-scheme with corners is $\bs{[0,\iy)}^I:=\prod_{i\in I}\bs{[0,\iy)}$ for $I$ an infinite set.

We will need to restrict to firm $C^\iy$-schemes with corners to define fibre products in \S\ref{cc57}, and for parts of the theory of corner functors in Chapter~\ref{cc6}.
\end{rem}

Parts (b),(c) of the next theorem are similar to Theorem~\ref{cc5thm7}(b).

\begin{thm}
\label{cc5thm8}
{\bf(a)} The functor $\Pi_\rin^\Z:\LCRScin\ra\LCRScZ$ of Theorem\/ {\rm\ref{cc5thm2}} maps $\CSchcin\ra\CSchcZ$. Restricting gives a functor $\Pi_\rin^\Z:\CSchcin\ra\CSchcZ$ which is right adjoint to the inclusion $\inc:\CSchcZ\hookra\CSchcin$.

The analogues also hold for the subcategories $\CSchcsa\subset\LCRScsa$ and\/ $\CSchctf\subset\LCRSctf$ and functors $\Pi_\Z^\tf,\Pi_\tf^\sa,\Pi_\rin^\sa,$ giving a diagram
\begin{equation*}
\xymatrix@C=22pt{ \CSchcsa \ar@<-.5ex>[r]_{\inc} &
\CSchctf \ar@<-.5ex>[r]_{\inc} \ar@<-.5ex>[l]_{\Pi_\tf^\sa} & \CSchcZ \ar@<-.5ex>[r]_{\inc} \ar@<-.5ex>[l]_{\Pi_\Z^\tf} & \CSchcin. \ar@<-.5ex>[l]_{\Pi_\rin^\Z} \ar@<-3ex>@/_.5pc/[lll]_{\Pi_\rin^\sa} }
\end{equation*}

\noindent{\bf(b)} An object\/ $\bX$ in $\CSchcin$ lies in $\CSchcsa,\CSchctf$ or $\CSchcZ$ if and only if it lies in $\LCRScsa,\LCRSctf$ or $\LCRScZ,$ respectively.
\smallskip

\noindent{\bf(c)} Let\/ $\bX$ be an object in $\CSchcfi$. Then the following are equivalent:
\begin{itemize}
\setlength{\itemsep}{0pt}
\setlength{\parsep}{0pt}
\item[{\bf(i)}] $\bX$ lies in $\CSchcfiin$.
\item[{\bf(ii)}] $\bX$ lies in $\LCRScin$.
\item[{\bf(iii)}] $\bX$ may be covered by open $\bU\subset\bX$ with\/ $\bU\cong\Speccin\bfC$ for $\bfC$ a firm, interior $C^\iy$-ring with corners.
\end{itemize}
Also $\bX$ lies in $\CSchcto$ if and only if it lies in $\LCRScsa$.
\end{thm}

\begin{proof} For (a), consider the diagram of functors
\e
\begin{gathered}
\xymatrix@C=180pt@R=15pt{
*+[r]{(\CRingscZ)^{\bf op}\,} \ar@<.5ex>[d]^{\Spec^{\rm c}_\Z} \ar@<.5ex>@{^{(}->}[r]^{\inc} & *+[l]{(\CRingscin)^{\bf op}} \ar@<.5ex>[l]^{\Pi_\rin^\Z} \ar@<.5ex>[d]^{\Speccin} \\
*+[r]{\LCRScZ\,} \ar@<.5ex>@{^{(}->}[r]^{\inc} \ar@<.5ex>[u]^{\Ga^{\rm c}_\Z} & *+[l]{\LCRScin.\!} \ar@<.5ex>[l]^{\Pi_\rin^\Z} \ar@<.5ex>[u]^{\Ga^{\rm c}_\rin} }
\end{gathered}
\label{cc5eq12}
\e
Here $\Spec^{\rm c}_\Z,\Ga^{\rm c}_\Z$ are the restrictions of $\Speccin,\Gacin$ to $(\CRingscZ)^{\bf op},
\LCRScZ$, and do map to $\LCRScZ,(\CRingscZ)^{\bf op}$. The functors $\Pi_\rin^\Z$ are from Theorems \ref{cc4thm5} and \ref{cc5thm2}, and are right adjoint to the inclusions $\inc$, noting that taking opposite categories converts left to right adjoints. Also $\Speccin$ is right adjoint to $\Gacin$ by Theorem \ref{cc5thm4}, which implies that $\Spec^{\rm c}_\Z$ is right adjoint to~$\Ga^{\rm c}_\Z$.

Clearly $\inc\ci\Ga^{\rm c}_\Z=\Gacin\ci\inc$. As $\Spec^{\rm c}_\Z\ci\Pi_\rin^\Z$, $\Pi_\rin^\Z\ci\Speccin$ are the right adjoints of these we have a natural isomorphism $\Spec^{\rm c}_\Z\ci\Pi_\rin^\Z\cong\Pi_\rin^\Z\ci\Speccin$. Thus, for any $\bfC\in\CRingscin$ we have $\Spec^{\rm c}_\Z\ci\Pi_\rin^\Z(\bfC)\cong\Pi_\rin^\Z\ci\Speccin\bfC$ in~$\LCRScZ$.

Let $\bX\in\CSchcin$. Then $\bX$ is covered by open $\bU\subset\bX$ with $\bU\cong\Speccin\bfC$ for $\bfC\in\CRingscin$. Hence $\Pi_\rin^\Z(\bX)$ is covered by open $\Pi_\rin^\Z(\bU)\subset\Pi_\rin^\Z(\bX)$ with $\Pi_\rin^\Z(\bU)\cong\Pi_\rin^\Z\ci\Speccin\bfC\cong\Spec^{\rm c}_\Z\ci\Pi_\rin^\Z(\bfC)$ for $\Pi_\rin^\Z(\bfC)\in\CRingscZ$. Thus $\Pi_\rin^\Z(\bX)\in\CSchcZ$, and $\Pi_\rin^\Z$ maps $\CSchcin\ra\CSchcZ$. That the restriction $\Pi_\rin^\Z:\CSchcin\ra\CSchcZ$ is right adjoint to $\inc$ follows from Theorem~\ref{cc5thm2}. This proves the first part of (a). The rest of (a) follows by the same argument.

For (b), the `only if' part is obvious. For the `if' part, suppose $\bX$ lies in $\CSchcin$ and $\LCRScsa$. Then $\Pi_\rin^\sa(\bX)$ lies in $\CSchcsa$ by (a). But $\Pi_\rin^\sa$ is isomorphic to the identity on $\LCRScsa\subset\LCRScin$, so $\Pi_\rin^\sa(\bX)\cong\bX$, and $\bX$ lies in $\CSchcsa$. The arguments for $\CSchctf,\CSchcZ$ are the same.

For (c), clearly (iii) $\Ra$ (i) $\Ra$ (ii), so it is enough to show that (ii) $\Ra$ (iii). Suppose $\bX$ lies in $\CSchcfi$ and $\LCRScin$, and $x\in\bX$. Then $x$ has an open neighbourhood $\bV\subset\bX$ with $\bV\cong\Specc\bfD$ for $\bfD$ firm. Let $d_1',\ldots,d_n'\in\fD_\rex$ such that $[d_1'],\ldots,[d_n']$ are generators for $\fD_\rex^\sh$. Choose an open neighbourhood $\bU$ of $x$ in $\bV$ such that for each $i=1,\ldots,n$, either $\pi_{u,\rex}(d_i')=0$ in $\fD_{u,\rex}\cong\O_{X,x}^\rex$ for all $u\in U$, or $\pi_{e,\rex}(d_i')\ne 0$ in $\fD_{u,\rex}\cong\O_{X,x}^\rex$ for all $u\in U$. As the topology of $\bV\cong \Specc\bfD$ is generated by open sets $U_d$ as in Definition \ref{cc2def15}, making $U$ smaller if necessary we can suppose that $U=U_d$ for $d\in\fD$. Then Lemma \ref{cc5lem2} says that~$\bU\cong\Specc(\bfD(d^{-1}))$. 

Define $\bfC$ to be the semi-complete $C^\iy$-ring with corners constructed from $\bfD(d^{-1})$ in Proposition \ref{cc5prop2}, so that we have a morphism $\bs\psi:\bfD(d^{-1})\ra\bfC$ with $\Specc\bs\psi$ an isomorphism, implying that $\bU\cong\Specc\bfC$. We will show that $\bfC$ is firm and interior. Write $d_1'',\ldots,d_n''$ for the images of $d_1',\ldots,d_n'$ in $\fD(d^{-1})_\rex$. Since $[d_1'],\ldots,[d_n']$ generate $\fD_\rex^\sh$, $[d_1''],\ldots,[d_n'']$ generate $\fD(d^{-1})_\rex^\sh$. As by construction $\fC_\rex$ is generated by $\psi_\rex(\fD(d^{-1})_\rex)$ and invertibles $\Psi_{\exp}(\fC)$, so $[\psi_\rex(d_1'')],\ldots,[\psi_\rex(d_n'')]$ generate $\fC_\rex^\sh$, and $\bfC$ is firm.

By definition the natural map $\Xi_\rex:\fC_\rex\ra\O_X^\rex(U)$ is injective. For each $i=1,\ldots,n$, either $\Xi_\rex\ci\psi_\rex(d_i'')\vert_u=0$ in $\O_{X,u}^\rex$ for all $u\in U$, so that $\psi_\rex(d_i'')=0$, or $\Xi_\rex\ci\psi_\rex(d_i'')\vert_u=0$ in $\O_{X,u}^\rex$ for all $u\in U$, so that $\Xi_\rex\ci\psi_\rex(d_i'')$ is not a zero divisor in $\O_X^\rex(U)$, as the $\O_{X,u}^\rex$ are interior, and thus $\psi_\rex(d_i'')$ is not a zero divisor in $\fC_\rex$. Since $[\psi_\rex(d_1'')],\ldots,[\psi_\rex(d_n'')]$ generate $\fC_\rex^\sh$, and each $\psi_\rex(d_i'')$ is either zero or not a zero divisor, we see that $\fC_\rex$ has no zero divisors. Also $0_{\fC_\rex}\ne 1_{\fC_\rex}$ as $U\ne\es$. So $\bfC$ is interior. Thus (ii) implies (iii), and the proof is complete.
\end{proof}

Let $\bs f:\bX\ra\bY$ be a morphism in $\CSchc$, and $x\in\bX$ with $\bs f(x)=y$ in $\bY$. Then by definition, $\bX,\bY$ are locally isomorphic to $\Specc\bfC$, $\Specc\bfD$ near $x,y$ for $C^\iy$-rings with corners $\bfC,\bfD$. However, it is {\it not\/} clear that $\bs f$ must be locally isomorphic to $\Specc\bs\phi$ for some morphism $\bs\phi:\bfD\ra\bfC$ in $\CRingsc$, and the authors expect this is false in some cases. This is because $\Specc:(\CRingsc)^{\bf op}\ra\ACSchc$ is not full, as in Theorem \ref{cc5thm6}(a), so morphisms $\Specc\bfC\ra\Specc\bfD$ in $\CSchc$ need not be of the form $\Specc\bs\phi$. We now show that if $\bY$ is {\it firm}, then $\bs f$ is locally isomorphic to $\Specc\bs\phi$. We use this in Theorem \ref{cc5thm9} to give a criterion for existence of fibre products in~$\CSchc$.

\begin{prop}
\label{cc5prop3}
Suppose\/ $\bs f:\bX\ra\bY$ is a morphism in $\CSchc$ with $\bY$ firm. Then for all\/ $x\in\bX$ with\/ $\bs f(x)=y\in\bY,$ there exists a commutative diagram in $\CSchc\!:$
\e
\begin{gathered}
\xymatrix@C=80pt{ *+[r]{\Specc\bfC} \ar[r]_\cong \ar[d]^{\Specc\bs\phi} & \,\bU\, \ar[d]^{\bs f\vert_\bU} \ar@{^{(}->}[r] & *+[l]{\bX} \ar[d]_{\bs f} \\
*+[r]{\Specc\bfD} \ar[r]^\cong  & \,\bV\, \ar@{^{(}->}[r] & *+[l]{\bY.\!} }
\end{gathered}
\label{cc5eq13}
\e
Here $\bU\subseteq\bX,$ $\bV\subseteq\bY$ are affine open neighbourhoods of\/ $x,y$ with\/ $\bU\subseteq\bs f^{-1}(\bV),$ and\/ $\bs\phi:\bfD\ra\bfC$ is a morphism in $\CRingscsc$ with\/ $\bfD$ firm. The analogue also holds with\/ $\CSchc,\CRingscsc$ replaced by $\CSchcin,\CRingscscin$.
\end{prop}

\begin{proof} As $\bY$ is a firm $C^\iy$-scheme with corners there exists an open neighbourhood $\bV$ of $y$ in $\bY$ and an isomorphism $\bV\cong\Specc\bfD$ for $\bfD$ a firm $C^\iy$-ring with corners. By Proposition \ref{cc5prop2} we may take $\bfD$ to be semi-complete.

As $\bX$ is a $C^\iy$-scheme with corners there exists an affine open neighbourhood $\bs{\dot U}$ of $x$ in $\bX$ with $\bs{\dot U}\cong\Specc\bs{\dot{\mathfrak C}}$. Since subsets $\dot U_c=\bigl\{x\in\dot U:x(c)\ne 0\bigr\}$ for $c\in\dot{\mathfrak C}$ are a basis for the topology of $\dot U$ by Definition \ref{cc2def16}, there exists $c\in\dot{\mathfrak C}$ with $\ddot U:=\dot U_c$ an open neighbourhood of $x$ in $\dot U\cap f^{-1}(V)$. Set $\bs{\ddot{\mathfrak C}}=\bs{\dot{\mathfrak C}}(c^{-1})$. Then $\bs{\ddot U}\cong\Specc\bs{\ddot{\mathfrak C}}$ is an open neighbourhood of $x$ in $\bs f^{-1}(\bV)\subseteq\bX$ by Lemma~\ref{cc5lem2}.

As $\bfD$ is firm, the sharpening $\fD_\rex^\sh$ is finitely generated. Choose $d'_1,\ldots,d'_n$ in $\fD_\rex$ whose images generate $\fD_\rex^\sh$. Then $\Xi_\fD^\rex(d'_i)$ lies in $(\Gac\ci\Specc\bfD)_\rex\cong \O_Y^\rex(V)$ for $i=1,\ldots,n$, so $f^\sh_\rex(\ddot U)\ci\Xi_\fD^\rex(d'_i)$ lies in $(\Gac\ci\Specc\bs{\ddot{\mathfrak C}})_\rex\cong \O_X^\rex(\ddot U)$ for $i=1,\ldots,n$. By definition of the sheaf $\O_X^\rex$, any section of $\O_X^\rex$ is locally modelled on $\Xi_{\bs{\ddot{\mathfrak C}}}^\rex(c')$ for some $c'\in \ddot{\mathfrak C}_\rex$. Thus, we may choose an open neighbourhood $U$ of $x$ in $\ddot U$ and elements $c_1',\ldots,c_n'\in \ddot{\mathfrak C}_\rex$ such that $f^\sh_\rex(\ddot U)\ci\Xi_\fD^\rex(d'_i)\vert_U=\Xi_{\bs{\ddot{\mathfrak C}}}^\rex(c'_i)\vert_U$ for $i=1,\ldots,n$. Making $U$ smaller if necessary, as above we may take $U$ of the form $\ddot U_c$ for $c\in\ddot{\mathfrak C}$. Set $\bfC=\Pi_{\rm all}^{\rm sc}(\bs{\ddot{\mathfrak C}}(c^{-1}))$. Then $\bU\cong\Specc\bfC$ by Lemma \ref{cc5lem2} and Definition \ref{cc5def10}, for $\bU$ an affine open neighbourhood of $x$ in $\bX$ with $\bU\subseteq\bs f^{-1}(\bV)$, and $\bfC$ a semi-complete $C^\iy$-ring.

We will show there is a unique morphism $\bs\phi=(\phi,\phi_\rex):\bfD\ra\bfC$ such that \eq{cc5eq13} commutes. To see this, note that there is a unique $C^\iy$-ring morphism $\phi:\fC\ra\fD$ making the $C^\iy$-scheme part of \eq{cc5eq13} commute by Theorem \ref{cc2thm3}(a), as $\fC,\fD$ are complete. Let $d'\in\fD_\rex$. Since $d_1',\ldots,d_n'$ generate $\fD_\rex^\sh$ we may write $d'=\Psi_{\exp}(d)(d_1')^{a_1}\cdots(d_n')^{a_n}$ for $d\in\fD$ and $a_1,\ldots,a_n>0$. 

We claim that defining $\phi_\rex(d')=\Psi_{\exp}(\phi(d))(c_1')^{a_1}\cdots(c_n')^{a_n}$ is independent of the presentation $d'=\Psi_{\exp}(d)(d_1')^{a_1}\cdots(d_n')^{a_n}$, and yields a monoid morphism $\phi_\rex:\fD_\rex\ra\fC_\rex$ with all the properties we need. To see this, note that
\begin{align*}
\Xi_\fC^\rex&(\phi_\rex(d'))=\Xi_\fC^\rex\bigl(\Psi_{\exp}(\phi(d))(c_1')^{a_1}\cdots(c_n')^{a_n}\bigr)\\
&=\bigl(\Xi_\fC^\rex\ci\Psi_{\exp}\ci\phi(d)\bigr)
\bigl(\Xi_\fC^\rex(c_1')\bigr)^{a_1}\cdots(\Xi_\fC^\rex(c_n'))^{a_n}\\
&=\Psi_{\exp}\bigl(\Xi_\fC\ci\phi(d)\bigr)
\bigl(f^\sh_\rex(U)\ci\Xi_\fD^\rex(d'_1)\bigr)^{a_1}\cdots(f^\sh_\rex(U)\ci\Xi_\fD^\rex(d'_n))^{a_n}\\
&=\Psi_{\exp}\bigl(f^\sh(U)\ci\Xi_\fD(d)\bigr)
\bigl(f^\sh_\rex(U)\ci\Xi_\fD^\rex(d'_1)\bigr)^{a_1}\cdots(f^\sh_\rex(U)\ci\Xi_\fD^\rex(d'_n))^{a_n}\\
&=f^\sh_\rex(U)\ci\Xi_\fD^\rex\bigl(\Psi_{\exp}(d)(d_1')^{a_1}\cdots(d_n')^{a_n}\bigr)=f^\sh_\rex(U)\ci\Xi_\fD^\rex(d').
\end{align*}
Thus $\Xi_\fC^\rex(\phi_\rex(d'))$ is independent of the presentation, and $\Xi_\fC^\rex$ is injective as $\fC$ is semi-complete, so $\phi_\rex(d')$ is well defined. By identifying $\bfC,\bfD$ with $C^\iy$-subrings with corners of $\bO_\bX(U),\bO_\bY(V)$, and $\bs\phi=(\phi,\phi_\rex)$ with the restriction of $\bs f_\sh(U):\bO_\bY(V)\ra \bO_\bX(U)$ to these $C^\iy$-subrings, in a similar way to \eq{cc5eq7}, we see that $\bs\phi$ is a morphism in $\CRingsc$, proving the first part. The same arguments work in the interior case.
\end{proof}

\subsection{\texorpdfstring{Fibre products of $C^\iy$-schemes with corners}{Fibre products of C∞-schemes with corners}}
\label{cc57}

We use Proposition \ref{cc5prop3} to construct fibre products of $C^\iy$-schemes with corners.

\begin{thm}
\label{cc5thm9}
Suppose\/ $\bs g:\bX\ra\bZ,$ $\bs h:\bY\ra \bZ$ are morphisms in $\CSchc,$ and\/ $\bZ$ is firm. Then the fibre product\/ $\bX\t_{\bs g,\bZ,\bs h}\bY$ exists in $\CSchc,$ and is equal to the fibre product in\/ $\LCRSc$. Similarly, if\/ $\bs g:\bX\ra\bZ,$ $\bs h:\bY\ra \bZ$ are morphisms in $\CSchcin,$ and\/ $\bZ$ is firm, then the fibre product\/ $\bX\t_{\bs g,\bZ,\bs h}\bY$ exists in $\CSchcin,$ and is equal to the fibre product in\/~$\LCRScin$.

The inclusion $\CSchcin\hookra\CSchc$ preserves fibre products.
\end{thm}

\begin{proof}
Let $\bs g:\bX\ra \bZ$, $\bs h:\bY\ra \bZ$ be morphisms in $\CSchc\subset\LCRSc$. By Theorem \ref{cc5thm1} there exists a fibre product $\bW=\bX\t_{\bs g,\bZ,\bs h}\bY$ in $\LCRSc$, in a Cartesian square in $\LCRSc$:
\e
\begin{gathered}
\xymatrix@C=80pt@R=15pt{ *+[r]{\bW} \ar[d]^{\bs e} \ar[r]_{\bs f} & *+[l]{\bY} \ar[d]_{\bs h} \\ *+[r]{\bX} \ar[r]^{\bs g} & *+[l]{\bZ.\!} }	
\end{gathered}
\label{cc5eq14}
\e
We will prove that $\bW$ is a $C^\iy$-scheme with corners. Then \eq{cc5eq14} is Cartesian in $\CSchc\subset\LCRSc$, so $\bW$ is a fibre product $\bX\t_{\bs g,\bZ,\bs h}\bY$ in~$\CSchc$.

Let $w\in\bW$ with $\bs e(w)=x\in\bX$, $\bs f(w)=y\in\bY$ and $\bs g(x)=\bs h(y)=z\in\bZ$. As $\bZ$ is firm we may pick an affine open neighbourhood $\bV$ of $z$ in $\bZ$ with $\bV\cong\Specc\bfF$ for a firm $C^\iy$-ring with corners $\bfF$. By Proposition \ref{cc5prop2} we may take $\bfF$ to be semi-complete.

By Proposition \ref{cc5prop3} we may choose affine open neighbourhoods $\bT$ of $x$ in $\bs g^{-1}(\bV)\subset\bX$ and $\bU$ of $y$ in $\bs h^{-1}(\bV)\subset\bY$, and semi-complete $C^\iy$-rings with corners $\bfD,\bfE$ with isomorphisms $\bT\cong\Specc\bfD$, $\bU\cong\Specc\bfE$, and morphisms $\bs\phi:\bfF\ra\bfD$, $\bs\psi:\bfF\ra\bfE$ in $\CRingscsc$ such that $\Specc\bs\phi$, $\Specc\bs\psi$ are identified with $\bs g\vert_\bT:\bT\ra\bV$, $\bs h\vert_\bU:\bU\ra\bV$ by the isomorphisms $\bT\cong\Specc\bfD$, $\bU\cong\Specc\bfE$ and~$\bV\cong\Specc\bfF$.

Set $\bS=\bs e^{-1}(\bT)\cap\bs f^{-1}(\bU)$. Then $\bS$ is an open neighbourhood of $w$ in $\bW$. Write $\bfC=\bfD\amalg_{\bs\phi,\bfF,\bs\psi}\bfE$ for the pushout in $\CRingscsc$, which exists by Theorem \ref{cc5thm6}(d), with projections $\bs\eta:\bfD\ra\bfC$, $\bs\th:\bfE\ra\bfC$. Consider the diagrams
\begin{equation*}
\xymatrix@C=80pt@R=15pt{ *+[r]{\bS} \ar[d]^{\bs e\vert_\bS} \ar[r]_{\bs f\vert_\bS} & *+[l]{\bU} \ar[d]_{\bs h\vert_\bU} \\ *+[r]{\bT} \ar[r]^{\bs g\vert_\bT} & *+[l]{\bV,\!} }	\qquad
\xymatrix@C=80pt@R=15pt{ *+[r]{\Specc\bfC} \ar[d]^{\Specc\bs\eta} \ar[r]_{\Specc\bs\th} & *+[l]{\Specc\bfE} \ar[d]_{\Specc\bs\psi} \\ *+[r]{\Specc\bfD} \ar[r]^{\Specc\bs\phi} & *+[l]{\Specc\bfF,\!} }
\end{equation*}
in $\LCRSc$. The left hand square is Cartesian as \eq{cc5eq14} is, by local properties of fibre products. The right hand square is Cartesian by Theorem \ref{cc5thm6}(e), as $\bfC=\bfD\amalg_{\bs\phi,\bfF,\bs\psi}\bfE$. The isomorphisms $\bT\cong\Specc\bfD$, $\bU\cong\Specc\bfE$, $\bV\cong\Specc\bfF$ identify the right and bottom sides of both squares. Hence the two squares are isomorphic. Thus $\bS\cong\Specc\bfC$. So each point $w$ in $\bW$ has an affine open neighbourhood, and $\bW$ is a $C^\iy$-scheme with corners, as we have to prove.

The proof for the interior case is essentially the same. Theorem \ref{cc5thm8}(c) allows us to write $\bV\cong\Speccin\bfF$ for $\bfF\in\CSchcfiin$, and Theorem \ref{cc5thm6} and Propositions \ref{cc5prop2} and \ref{cc5prop3} also work in the interior case. The last part then follows from Theorem~\ref{cc5thm1}.
\end{proof}

\begin{thm}
\label{cc5thm10}
Consider the family of full subcategories:
\e
\begin{aligned}
\CSchcfp&\subset \LCRSc, & \CSchcfg&\subset \LCRSc, \\
\CSchcfi&\subset \LCRSc, & \CSchcfiin&\subset \LCRScin, \\
\CSchcfiZ&\subset \LCRScZ, & \CSchcfitf&\subset \LCRSctf, \\
\CSchcto&\subset \LCRScsa.
\end{aligned}
\label{cc5eq15}
\e
In each case, the subcategory is closed under finite limits in the larger category. Hence, fibre products and all finite limits exist in $\CSchcfp,\ldots,\CSchcto$.
\end{thm}

\begin{proof}
In the proof of Theorem \ref{cc5thm9}, if $\bX,\bY,\bZ$ are finitely presented we may take $\bfD,\bfE,\bfF$ to be finitely presented, and then $\bfC$ is finitely presented by Proposition \ref{cc4prop11}, so $\bW=\bX\t_{\bs g,\bZ,\bs h}\bY$ is finitely presented. Hence all fibre products exist in $\CSchcfp$, and agree with fibre products in~$\LCRSc$.

As $\CSchcfp$ has a final object $\Specc(\R,[0,\iy))$, and all fibre products in a category with a final object are (iterated) fibre products, then $\CSchcfp$ is closed under finite limits in $\LCRSc$, and all such finite limits exist in $\CSchcfp$ by Theorem \ref{cc5thm1}. The same argument works for $\CSchcfg,\ab\CSchcfi,\ab\CSchcfiin$, using Proposition \ref{cc4prop12} for the latter two.

For $\CSchcfiZ,\CSchcfitf,\CSchcto$ we apply the coreflection functors $\Pi_\rin^\Z,\Pi_\Z^\tf,\Pi_\tf^\sa$ of Theorem \ref{cc5thm2}, noting that the restrictions of these map
\begin{equation*}
\xymatrix@C=35pt{ \CSchcfiin \ar[r]^{\Pi_\rin^\Z} & \CSchcfiZ \ar[r]^{\Pi_\Z^\tf} & \CSchcfitf \ar[r]^{\Pi_\tf^\sa} & \CSchcto, }	
\end{equation*}
by Theorem \ref{cc5thm8} and the fact (clear from the proof of Theorem \ref{cc5thm8}) that $\Pi_\rin^\Z,\Pi_\Z^\tf,\Pi_\tf^\sa$ preserve firmness. 

Let $J$ be a finite category and $\bs F:J\ra\CSchcfiZ$ a functor. Then a limit $\bX=\varprojlim\bs F$ exists in $\CSchcfiin$ from above, which is also a limit in $\LCRScin$. Hence $\Pi_\rin^\Z(\bX)$ is a limit $\varprojlim\Pi_\rin^\Z\ci\bs F$ in $\LCRScZ$, since $\Pi_\rin^\Z:\LCRScin\ra\LCRScZ$ preserves limits as it is a right adjoint. But $\Pi_\rin^\Z$ acts as the identity on $\LCRScZ\subset\LCRScin$, so $\Pi_\rin^\Z(\bX)$ is a limit $\varprojlim\bs F$ in $\LCRScZ$. As $\bX\in\CSchcfiin$ we have $\Pi_\rin^\Z(\bX)\in\CSchcfiZ$, so $\CSchcfiZ$ is closed under finite limits in $\LCRScZ$, and all finite limits exist in $\CSchcfiZ$ by Corollary \ref{cc5cor2}. The arguments for $\CSchcfitf,\CSchcto$ are the same.
\end{proof}

We will return to the subject of fibre products in \S\ref{cc66}--\S\ref{cc67}, using `corner functors' from Chapter \ref{cc6} as a tool to study them.

\section{Boundaries, corners, and the corner functor}
\label{cc6}

In \S\ref{cc34}, for manifolds with (g-)corners of mixed dimension, we defined {\it corner functors\/} $C:\cManc\ra\cMancin$ and $C:\cMangc\ra\cMangcin$ which are right adjoint to the inclusions $\inc:\cMancin\hookra\cManc$ and $\inc:\cMangcin\hookra\cMangc$. These encode the boundary $\pd X$ and $k$-corners $C_k(X)$ of a manifold with (g-)corners $X$ in a functorial way.

This chapter will show that analogous corner functors $C:\LCRSc\ra\LCRScin$ and $C:\CSchc\ra\CSchcin$ may be defined for local $C^\iy$-ringed spaces and $C^\iy$-schemes with corners $\bX$. This allows us to define the boundary $\pd\bX$ and $k$-corners $C_k(\bX)$ for firm $C^\iy$-schemes with corners $\bX$. We use corner functors to study existence of fibre products in categories such as~$\CSchctoex$.

\subsection{\texorpdfstring{The corner functor for $C^\iy$-ringed spaces with corners}{The corner functor for C∞-ringed spaces with corners}}
\label{cc61}

\begin{dfn}
\label{cc6def1}
Let $\bX=(X,\bO_X)$ be a local $C^\iy$-ringed space with corners. We will define the {\it corners\/} $C(\bX)$ in $\LCRScin$. As a set, we define 
\e
C(X)=\bigl\{(x,P):\text{$x\in X$, $P$ is a prime ideal in }\O_{X,x}^\rex \bigr\},
\label{cc6eq1}
\e
where prime ideals in monoids are as in Definition \ref{cc3def4}. Define $\Pi_X: C(X)\ra X$ as a map of sets by~$\Pi_X:(x,P)\mapsto x$. 

We define a topology on $C(X)$ to be the weakest topology such that $\Pi_X$ is continuous, so $\Pi_X^{-1}(U)$ is open for all $U\subset X$, and such that for all open $U\subset X$, for all elements $s'\in \O_X^\rex(U)$, then 
\e
\dot U_{s'}=\bigl\{(x,P):x\in U,\; s'_x\notin  P \bigr\}\subseteq \Pi_X^{-1}(U)\quad\text{and}\quad \ddot U_{s'}=\Pi_X^{-1}(U)\setminus\dot U_{s'}
\label{cc6eq2}
\e
are both open and closed in $\Pi_X^{-1}(U)$. Then $\Pi_X^{-1}(U)=\dot U_1=\ddot U_0$, and $\es=\ddot U_{1}=\dot U_0$ for $0,1\in \O_X^\rex(U)$. The collection $\bigl\{\dot U_{s'},\ddot U_{s'}:\text{open } U\subset X$, $s'\in \O_X^\rex(U)\bigr\}$ is a subbase for the topology.

Pull back the sheaves $\O_X,\O_X^\rex$ by $\Pi_X$ to get sheaves $\Pi_X^{-1}(\O_X),\Pi_X^{-1}(\O_X^\rex)$ on $C(X)$. For each $x\in X$ and each prime $P$ in $\O_{X,x}^\rex$, we have ideals 
\begin{equation*}
P\subset \O_{X,x}^\rex\cong(\Pi_X^{-1}(\O^\rex_{X}))_{(x,P)},\quad \langle\Phi_i(P) \rangle\subset \O_{X,x}\cong(\Pi_X^{-1}(\O_{X}))_{(x,P)}.
\end{equation*} 
Here $\langle \Phi_i(P)\rangle$ is the ideal generated by the image of $P$ under $\Phi_i:\O^\rex_{X,x}\ra\O_{X,x}$. 

We define the sheaf of $C^\iy$-rings $\O_{C(X)}$ on $C(X)$ to be the sheafification of the presheaf of $C^\iy$-rings $\Pi_X^{-1}(\O_X)/I$ on $C(X)$, where for open $U\subset C(X)$
\begin{equation*}
I(U)=\bigl\{s\in \Pi_X^{-1}(\O_X)(U):s_{(x,P)}\in \langle \Phi_i(P)\rangle \text{ for all } (x,P)\in U\bigr\}.
\end{equation*}
Let $\O^\rex_{C(X)}$ be the sheafification of the presheaf $U\!\mapsto\! \Pi_X^{-1}(\O^\rex_X)(U)/\simc$ where the equivalence relation $\simc$ is generated by the relations, for $s_1', s_2'\in  \Pi_X^{-1}(\O^\rex_X)(U)$, that $s'_1\sim s_2'$ if for each $(x,P)\in U$ either $s'_{1,x},s'_{2,x}\in P$, or there is $p\in \langle \Phi_i(P)\rangle$ such that $s'_{1,x}=\Psi_{\exp}(p)s'_{2,x}$. This is similar to quotienting the $C^\iy$-ring with corners $(\Pi_X^{-1}(\O_X)(U), \Pi_X^{-1}(\O_X^\rex)(U))$ by a prime ideal in $\Pi_X^{-1}(\O_X^\rex)(U)$, which we described in Example \ref{cc4ex4}(b), and creates a sheaf of $C^\iy$-rings with corners $\bO_{C(X)}=(\O_{C(X)},\O^\rex_{C(X)})$ on~$C(X)$. 

Lemma \ref{cc6lem1} shows $C(\bX)=(C(X),\bO_{C(X)})$ is an interior local $C^\iy$-ringed space with corners, and the stalk of $\bO_{C(X)}$ at $(x,P)$ is an interior local $C^\iy$-ring with corners isomorphic to $\bO_{X,x}/\simc_P$, using the notation of Example~\ref{cc4ex4}(b).

Define sheaf morphisms $\Pi_X^\sh:\Pi_X^{-1}(\O_X)\ra \O_{C(X)}$ and $\Pi_{X,\rex}^\sh:\Pi_X^{-1}(\O_X^\rex)\ra \O_{C(X)}^\rex$ to be the sheafifications of the projections $\Pi_X^{-1}(\O_X)\ra\Pi_X^{-1}(\O_X)/I$ and $\Pi_X^{-1}(\O_X^\rex)\ra\Pi_X^{-1}(\O_X^\rex)/\simc$. Then $\bs\Pi_\bX^\sh=(\Pi_X^\sh,\Pi_{X,\rex}^\sh):\Pi_X^{-1}(\bO_X)\ra \bO_{C(X)}$ is a morphism of sheaves of $C^\iy$-rings with corners, and $\bs\Pi_\bX=(\Pi_X,\bs\Pi_\bX^\sh)$ is a morphism $\bs\Pi_\bX:C(\bX)\ra\bX$ in~$\LCRSc$.

Let $\bs f=(f,f^\sh,f^\sh_\rex):\bX\ra\bY$ be a morphism in $\LCRSc$. We will define a morphism $C(\bs f):C(\bX)\ra C(\bY)$. On topological spaces, we define $C(f):C(X)\ra C(Y)$ by $(x,P)\mapsto (f(x),(f^{\sh}_{\rex,x})^{-1}(P))$ where $f^{\sh}_{\rex,x}:\O^\rex_{Y,f(x)}\ra\O^\rex_{X,x}$ is the stalk map of $f^\sh_\rex$. This is continuous as if $t\in \O_Y^\rex(U)$ for some open $U\subset Y$ then $C(f)^{-1}(\dot U_{t})=\dot{(f^{-1}(U))}_{f^\rex_\sh(U)(t)}$ is open, and similarly $C(f)^{-1}(\ddot U_{t})=\ddot{(f^{-1}(U))}_{f^\rex_\sh(U)(t)}$ is open. Also~$\Pi_Y\ci C(f)=f\ci \Pi_X$. 

To define the sheaf morphism $C(\bs f)^\sh:C(f)^{-1}(\bO_{C(Y)})\ra\bO_{C(X)}$, consider the diagram of presheaves on $C(X)$:
\begin{equation*}
\xymatrix@C=170pt@R=15pt{
*+[r]{C(f)^{-1}\ci\Pi_Y^{-1}(\O_Y)=\Pi_X^{-1}\ci f^{-1}(\O_Y)} \ar[r]_(0.7){\Pi_X^{-1}(f^\sh)} \ar[d] & *+[l]{\Pi_X^{-1}(\O_X)} \ar[d] \\
*+[r]{C(f)^{-1}(\Pi_Y^{-1}(\O_Y)/I_Y)} \ar[d]^{\text{sheafify}} \ar@{.>}[r]^{\Pi_X^{-1}(f^\sh)_*} & *+[l]{\Pi_X^{-1}(\O_X)/I_X} \ar[d]_{\text{sheafify}} \\
*+[r]{C(f)^{-1}(\O_{C(Y)}) } \ar@{.>}[r]^{C(f)^\sh} & *+[l]{\O_{C(X)}.\!}}
\end{equation*}
Here there is a unique morphism $\Pi_X^{-1}(f^\sh)_*$ making the upper rectangle commute as $\Pi_X^{-1}(f^\sh)$ maps $C(f)^{-1}(I_Y)\ra I_X$ by definition of $I_Y,I_X,C(f)$. Hence there is a unique morphism $C(f)^\sh$ making the lower rectangle commute by properties of sheafification. Similarly we define $C(f)^\sh_\rex:C(f)^{-1}(\O_{C(Y)}^\rex)\ra\O_{C(X)}$, and $C(\bs f)^\sh=(C(f)^\sh,C(f)^\sh_\rex):C(f)^{-1}(\bO_{C(Y)})\ra\bO_{C(X)}$ is a morphism of sheaves of $C^\iy$-rings with corners on $C(X)$, so $C(\bs f)=(C(f),C(\bs f)^\sh):C(\bX)\ra C(\bY)$ is a morphism  in $\LCRSc$. We see that~$\bs{\Pi}_\bY\ci C(\bs f)=\bs f\ci\bs{\Pi}_\bX$.

On the stalks at $(x,P)$ in $C(X)$, we have
\begin{equation*}
C(\bs f)^\sh_{(x,P)}=(\bs f^\sh_x)_*:\bO_{Y,f(x)}/\simc_{(f^\sh_{\rex,x})^{-1}(P)}\longra \bO_{X,x}/\simc_P.
\end{equation*}
For the monoid sheaf, if $s'\mapsto 0$ in the stalk, then $s'=[s'']$ for some $s''\in \O_{Y,f(x)}^\rex$, and $f^\sh_{\rex,(x,P)}(s'')_x\in P$. Then $s''_{f(x)}\in (f^\sh_{\rex,x})^{-1}(P)$, so $s''\simc_{(f^\sh_{\rex,x})^{-1}(P)} 0$, giving $s'=0$. Therefore $C(\bs f)$ is an interior morphism. One can check that $C(\bs f\ci \bs g)=C(\bs f)\ci C(\bs g)$ and $C(\bs\id_\bX)=\bs\id_{C(\bX)}$, and thus $C:\LCRSc\ra\LCRScin$ is a well defined functor. 

If we now assume $\bX$ is interior, then $\{(x,\{0\}):x\in X\}$ is contained in $C(X)$, and there is an inclusion of sets $\io_X:X\hookrightarrow C(X)$, $\io_X:x\mapsto (x,\{0\})$. The image of $X$ under $\io_X:X\hookrightarrow C(X)$ is closed in $C(X)$, since for all open $U\subset X$, 
\begin{equation*}
\io_X(U)=\bigcap\nolimits_{s'\in \O_X^\rin(U)}\dot U_{s'}
\end{equation*}
is closed in $\Pi_X^{-1}(U)$. Also, $\io_X$ is continuous, as the definition of interior implies $\io_X^{-1}(\dot U_{s'})=\bigl\{x\in U:s'_x\ne 0\in \O_{X,x}^\rex\bigr\}$ is open, and $\io_X^{-1}(\ddot U_{s'})=\bigl\{x\in U:s'_x= 0$ in $\O_{X,x}^\rex\bigr\}$ is open for all~$s'\in \O_X^\rex(U)$.

Now any $s\in \O_X(U)$ gives equivalence classes in $\Pi_X^{-1}(\O_X)(\Pi_X^{-1}(U))$, and $\Pi_X^{-1}(\O_X)(\Pi_X^{-1}(U))/\simc$, and $\O_{C(X)}(\Pi^{-1}(U))$, and $\io_X^{-1}(\O_{C(X)})(U)$. So there is a map $\O_X(U)\ra \io_X^{-1}(\O_{C(X)})(U)$, and a similar map $\O_X^\rex(U)\ra \io_X^{-1}(\O_{C(X)}^\rex)(U)$. These maps respect restriction and form morphisms of sheaves of $C^\iy$-rings with corners. The stalks of $\io_X^{-1}(\bO_{C(X)})$ are isomorphic to the stalks of $\bO_{C(X)}$, which are isomorphic to $\bO_{X,x}/\simc_{\{0\}}\cong \bO_{X,x}$ by Lemma \ref{cc6lem1}. So there is a canonical isomorphism $\bs \io_\bX^\sh:\io_X^{-1}(\bO_{C(X)})\ra\bO_X$, and $\bs \io_\bX=(\io_X,\bs \io_\bX^\sh):\bX\ra C(\bX)$ is a morphism, which is interior, as the stalk maps are isomorphisms.

Thus, for each interior $\bX$ we have defined a morphism $\bs \io_\bX:\bX\ra C(\bX)$ in $\LCRScin$, which is an inclusion of $\bX$ as an open and closed $C^\iy$-subscheme with corners in $C(\bX)$. One can show that $\bs{\Pi}_\bX\ci\bs\io_\bX=\bs\id_\bX$. Also, if $\bs f:\bX\ra\bY$ is a morphism in $\LCRScin$ one can show that~$\bs\io_\bY\ci\bs f=C(\bs f)\ci\bs \io_\bX$.
\end{dfn}

\begin{lem}
\label{cc6lem1}
In Definition\/ {\rm\ref{cc6def1},} $C(\bX)=(C(X),\bO_{C(X)})$ is an interior local $C^\iy$-ringed space with corners in the sense of Definition\/ {\rm\ref{cc5def2}}. The stalk\/ $\bO_{C(X),(x,P)}$ of\/ $\bO_{C(X)}$ at\/ $(x,P)$ in $C(X)$ is an interior local\/ $C^\iy$-ring with corners isomorphic to $\bO_{X,x}/\simc_P,$ in the notation of Example\/~{\rm\ref{cc4ex4}(b)}.
\end{lem}

\begin{proof}
We first verify the openness conditions for $C(\bX)$ to be an {\it interior\/} $C^\iy$-ringed space with corners in Definition \ref{cc5def2}(a),(b). Let $U\subset C(X)$ be open and $s'\in \O_{C(X)}^\rex(U)$, and as in Definition \ref{cc5def2}(a),(b) write
\begin{equation*}
U_{s'}=\bigl\{(x,P)\in U:s'_{(x,P)}\ne 0\in \O^\rex_{C(X),(x,P)}\bigr\}\quad\text{and}\quad \hat U_{s'}=U\sm U_{s'}.\end{equation*}
Then there is an open cover $\{U_i\}_{i\in I}$ of $U$ such that $s'\vert_{U_i}=[s_i']$ for $s_i'\in\Pi^{-1}(\O^\rex_X)(U_i)$, where $s'_i\vert_{(x,P)}\notin P$ if $(x,P)\in U_{s'}\cap U_i$ and $s_i'\vert_{(x,P)}\in P$ if $(x,P)\in U_{s'}\cap\hat U_i$. The definition of the inverse image sheaf implies there is an open cover $\{W_{j}^i\}_{j\in J_i}$ of $U_i$ such that $s_i'\vert_{W_{i,j}}=[s'_{i,j}]$ for some $s'_{i,j}\in \O^\rex_X(V_{i,j})$ for open $V_{i,j}\subset X$ such that $V_{i,j}\supset \Pi(W_{i,j})$. Then $s'_{i,j}\vert_x\notin  P$ for all $(x,P)\in U_{s'}\cap W_{i,j}$, and $s'_{i,j}\vert_x\in  P$ for all $(x,P)\in\hat U_{s'}\cap W_{i,j}$. Thus
\begin{equation*}
U_{s'}\cap W_{i,j}=W_{i,j}\cap \dot{(V_{i,j})}_{s'_{i,j}}
\quad
\hat U_{s'}\cap W_{i,j}=W_{i,j}\cap \ddot{( V_{i,j})}_{s'_{i,j}} \end{equation*}
are both open in $C(X)$ by the definition of the topology on $C(X)$ involving $\dot U_{s'},\ddot U_{s'}$ in \eq{cc6eq2}. Taking the union over $i\in I$, $j\in J_i$, and using that $\bigcup_{i\in I,\;j\in J_i} W_{i,j}=U$, we see that $U_{s'},\hat U_{s'}$ are open in $U$, as we have to prove.

The stalks of $\O_{C(X)},\O_{C(X)}^\rex$ at a point $(x,P)\in C(X)$ are isomorphic to $\O_{X,x}/\langle \Phi_i(P)\rangle$ and $\O_{X,x}^\rex/\simc_P$, where $s_1'\simc_P s_2'$ in $\O_{X,x}^\rex$ if $s_1',s_2'\in P$ or there is a $p\in \langle \Phi_i(P)\rangle$ such that $\Psi_{\exp}(p)s_1'=s_2'$, as in Example \ref{cc4ex4}(b). To see this, consider that the definitions give us the following diagram, where we will show the arrow $\bs{t}$ exists and is an isomorphism
\e
\begin{gathered}
\xymatrix@C=20pt@R=20pt{*+[l]{\bO_{X,x}} \ar@{->}[r]^(0.35){\sim}_(0.35){\hat{\bs{\Pi}}^\sh_x} \ar@{>>}[d]^{}& *+[c]{\Pi^{-1}(\bO_X)_{(x,P)}}\ar@{>>}[r]_{}&*+[r]{\bO_{C(X),(x,P)}.} \\
\bO_{X,x}/\simc_P\ar@{-->>}@<-.8ex>@/^-.5pc/[rru]_{\bs{t}}^\cong }
\end{gathered}
\label{cc6eq3}
\e

To see that $\bs{t}$ exists, we use the universal property of $\bO_{X,x}/\simc_P$ as in Example \ref{cc4ex4}(b). Let $U$ be an open set in $X$ and take $s'\in \O_X^\rex(U)$ such that $s'_x\in P$. Then $s'$ maps to an equivalence class in $\O^\rex_{C(X)}(\Pi^{-1}(U))$. Consider the open set $\ddot U_{s'}=\bigl\{(x,P)\in C(X): x\in U$, $s'_x\in P\bigr\}$, then restricting to this open set, we see that our $(x,P)$ is in $\ddot U_{s'}\subset \Pi^{-1}(U)$ and so we can restrict the equivalence class of $s'$ to $\ddot U_{s'}$. In this open set however, $s'_x\in P$ for all $(x,P)\in \ddot U_{s'}$ so $s'\simc 0$, and $s'$ lies in the kernel of the composition of the top row of \eq{cc6eq3}. Then the universal property of $\bO_{X,x}/\simc_P$ says that $\bs{t}$ must exist and commute with the diagram. Also, $\bs{t}$ must be surjective as the top line is surjective. 

To see that $\bs{t}$ is injective is straightforward. For example, in the monoid case, if $[s'_{1,x}],[s'_{2,x}]\in \O_{X,x}^\rex/\simc_P$ with representatives $s'_1,s'_2\in \O^\rex_X(U)$, and if $t_\rex([s'_{1,x}])=t_\rex([s'_{2,x}])$ then $s'_{1,x}\simc s'_{2,x}$ for all $(x,P)\in V$ for some open $V\subset \Pi^{-1}(U)$. This means at every $(x,P)\in V$ then $s'_{1,x}\simc_Ps'_{2,x}$, so this must be true at our $(x,P)$, so that $[s'_{1,x}]=[s'_{2,x}]\in \O_{X,x}^\rex/\simc_P$ as required.

Now $\bO_{X,x}/\simc_P$ is interior as the complement of $P \in \O_{X,x}^\rex$ has no zero divisors. It is also local, as the unique morphism $\O_{X,x}\ra \R$ must have $P$ in its kernel so it factors through the morphism $\O_{X,x}\ra\O_{X,x}^\rex/\simc_P$ giving a unique morphism $\O_{X,x}^\rex/\simc_P\ra \R$ with the correct properties to be local. Thus $\bO_{C(X),(x,P)}$ is an interior local $C^\iy$-ring with corners for all $(x,P)\in C(X)$, so $C(\bX)$ is an interior local $C^\iy$-ringed space with corners by Definition~\ref{cc5def2}.
\end{proof}

\begin{thm}
\label{cc6thm1}
The corner functor $C:\LCRSc\ra\LCRScin$ is right adjoint to the inclusion $\inc:\LCRScin\hookra\LCRSc$. Thus we have natural, functorial isomorphisms $\Hom_\LCRScin(C(\bX),\bY)\cong\Hom_\LCRSc(\bX,\inc(\bY))$.

As $C$ is a right adjoint, it preserves limits in $\LCRSc,\LCRScin$.
\end{thm}

\begin{proof} We describe the unit $\bs{\eta}:\Id\Ra C\ci \inc$ and counit $\bs{\ep}:\inc\ci C\Ra \Id$ of the adjunction. Here $\bs{\ep}_\bX=\bs{\Pi}_{\bX}$ for $\bX$ a $C^\iy$-scheme with corners, and $\bs{\eta}_\bX=\bs\io_\bX$ for $\bX$ an interior $C^\iy$-scheme with corners. That $\bs{\ep}$ is a natural transformation follows from $\bs{\Pi}_\bY\ci C(\bs f)=\bs f\ci\bs{\Pi}_\bX$ for morphisms $\bs f:\bX\ra\bY$ in $\LCRSc$, and that $\bs\eta$ is a natural transformation follows from $\bs\io_\bY\ci\bs f=C(\bs f)\ci\bs \io_\bX$ for morphisms $\bs f:\bX\ra\bY$ in $\LCRScin$. Finally, to prove the adjunction, we must show the following compositions are the identity natural transformations:
\begin{equation*}
\xymatrix@C=35pt{C \ar@{=>}[r]^(0.35){\bs\eta*\id_C} & C\ci \inc\ci C \ar@{=>}[r]^(0.6){\id_C*\bs\ep} & C,  }\quad
\xymatrix@C=35pt{\inc \ar@{=>}[r]^(0.33){\id_{\inc}*\bs\eta} & \inc\ci C\ci \inc \ar@{=>}[r]^(0.6){\bs\eta*\id_{\inc}} & \inc.  }
\end{equation*}
Both of these follow as~$\bs\Pi_\bX\ci\bs\io_\bX=\bs\id_\bX$.
\end{proof}

Definition \ref{cc5def3} defined subcategories $\LCRScsa\subset\cdots\subset\LCRScin$ and $\LCRScsaex\subset\cdots\subset\LCRScinex\subset\LCRSc$, where the objects $\bX$ have stalks $\bO_{X,x}$ which are saturated, torsion-free, integral, or interior. From the definition of $C(\bX)$ in Definition \ref{cc6def1}, we see that if the stalks $\bO_{X,x}$ are saturated, \ldots, interior then the stalks $\bO_{C(\bX),(x,P)}\cong \bO_{X,x}/\simc_P$ are also saturated, \ldots, interior, as these properties are preserved by quotients by prime ideals, so $C(\bX)$ is saturated, \ldots, interior, respectively. Hence from Theorem \ref{cc6thm1} we deduce:

\begin{thm}
\label{cc6thm2}
Restricting $C$ in Theorem\/ {\rm\ref{cc6thm1}} gives corner functors
\begin{align*}
C&:\LCRScsaex\longra 	\LCRScsa,&
C&:\LCRSctfex\longra 	\LCRSctf,\\
C&:\LCRScZex\longra 	\LCRScZ,&
C&:\LCRScinex\longra 	\LCRScin,
\end{align*}
which are right adjoint to the corresponding inclusions.
\end{thm}
 
\subsection{\texorpdfstring{The corner functor for $C^\iy$-schemes with corners}{The corner functor for C∞-schemes with corners}}
\label{cc62}

Motivated by Theorem \ref{cc6thm1}, na\"\i vely we might expect that there exists $\ti C:\CRingsc\ra\CRingscin$ left adjoint to $\inc:\CRingscin\hookra\CRingsc$, fitting into a diagram of adjoint functors similar to~\eq{cc5eq12}:
\e
\begin{gathered}
\xymatrix@C=180pt@R=15pt{
*+[r]{(\CRingscin)^{\bf op}\,} \ar@<.5ex>[d]^{\Speccin} \ar@<.5ex>@{^{(}->}[r]^{\inc} & *+[l]{(\CRingsc)^{\bf op}} \ar@<.5ex>[l]^{\ti C} \ar@<.5ex>[d]^{\Specc} \\
*+[r]{\LCRScin\,} \ar@<.5ex>@{^{(}->}[r]^{\inc} \ar@<.5ex>[u]^{\Gacin} & *+[l]{\LCRSc.\!} \ar@<.5ex>[l]^{C} \ar@<.5ex>[u]^{\Gac} }
\end{gathered}
\label{cc6eq4}
\e
As in the proof of Theorem \ref{cc5thm8}(a) this would imply $\Speccin\ci\ti C\cong C\ci\Specc$, allowing us to deduce that $C$ maps $\CSchc\ra\CSchcin$. However, no such left adjoint $\ti C$ for $\inc$ exists by Theorem~\ref{cc4thm3}(d).

Surprisingly, we can define a `corner functor' $\ti C:\CRingsc\ra\CRingscsc$ such that $\Specc\ci\ti C$ maps to $\LCRScin\subset\LCRSc$, with $\Specc\ci\ti C\cong C\ci\Specc$, as we prove in the next definition and theorem.

\begin{dfn}
\label{cc6def2}
Let $\bfC=(\fC,\fC_\rex)$ be a $C^\iy$-ring with corners. Using the notation of Definition \ref{cc4def8}, define a $C^\iy$-ring with corners $\bs{\ti\fC}$ by
\ea
\bs{\ti\fC}=\bfC&\bigl[\al_{c'},\be_{c'}:c'\in\fC_\rex\bigr]\big/\bigl(\Phi_i(\al_{c'})+\Phi_i(\be_{c'})=1\;\> \forall c'\in\fC_\rex\bigr)
\label{cc6eq5}\\
&\bigl[\al_{1_{\fC_\rex}}\!=\!1_{\fC_\rex},\; \al_{c'}\al_{c''}\!=\!\al_{c'c''}\; \forall c',c''\!\in\!\fC_\rex,\;  c'\be_{c'}\!=\!\al_{c'}\be_{c'}\!=\!0\;  \forall c'\!\in\!\fC_\rex\bigr].
\nonumber
\ea
That is, we add two $[0,\iy)$-type generators $\al_{c'},\be_{c'}$ to $\bfC$ for each $c'\in\fC_\rex$, and impose $\R$-type relations $(\cdots)$ and $[0,\iy)$-type relations $[\cdots]$ as shown.

Let $\bs\phi:\bfC\ra\bfD$ be a morphism in $\CRingsc$, and define $\bs{\ti\fC},\bs{\ti\fD}$ from $\bfC,\bfD$ as above, writing the additional generators in $\fD_\rex$ as $\bar\al_{d'},\bar\be_{d'}$. Define a morphism $\bs{\ti\phi}:\bs{\ti\fC}\ra\bs{\ti\fD}$ by the commutative diagram
\begin{equation*}
\xymatrix@C=180pt@R=15pt{
*+[r]{\bfC\bigl[\al_{c'},\be_{c'}:c'\in\fC_\rex\bigr]} \ar@{->>}[r]_(0.7){\text{project}} \ar[d]^{\bs\phi[\al_{c'}\mapsto\bar\al_{\phi_\rex(c')},\; \be_{c'}\mapsto\bar\be_{\phi_\rex(c')},\; c'\in\fC_\rex]} & *+[l]{\bs{\ti\fC}} \ar@{..>}[d]_{\bs{\ti\phi}} \\
*+[r]{\bfD\bigl[\bar\al_{d'},\bar\be_{d'}:d'\in\fD_\rex\bigr]} \ar@{->>}[r]^(0.7){\text{project}}  & *+[l]{\bs{\ti\fD}.\!} }	
\end{equation*}
That is, we begin with $\bs\phi:\bfC\ra\bfD$, and act on the extra generators by mapping $\al_{c'}\mapsto \bar\al_{\phi_\rex(c')}$, $\be_{c'}\mapsto\bar\be_{\phi_\rex(c')}$ for all $c'\in\fC_\rex$. These map the relations in $\bs{\ti\fC}$ to the relations in $\bs{\ti\fD}$, and so descend to a unique morphism $\bs{\ti\phi}$ as shown.

Clearly mapping $\bs\phi\mapsto\bs{\ti\phi}$ takes compositions and identities to compositions and identities, so mapping $\bfC\mapsto\bs{\ti\fC}$ on objects and $\bs\phi\mapsto\bs{\ti\phi}$ on morphisms defines a functor $\CRingsc\ra\CRingsc$. For reasons explained in \S\ref{cc63}, we define $\ti C:\CRingsc\ra\CRingscsc$ to be the composition of this functor with $\Pi_{\rm all}^{\rm sc}:\CRingsc\ra\CRingscsc$ in Definition \ref{cc5def10}. That is, we define the {\it corner functor\/} $\ti C:\CRingsc\ra\CRingscsc$ to map $\bfC\mapsto \Pi_{\rm all}^{\rm sc}(\bs{\ti\fC})$ on objects, and $\bs\phi\mapsto \Pi_{\rm all}^{\rm sc}(\bs{\ti\phi})$ on morphisms, for $\bfC,\bfD,\bs\phi,\bs{\ti\fC},\bs{\ti\fD},\bs{\ti\phi}$ as above. 

For each $\bfC\in\CRingsc$, define $\bs{\ti\Pi}_\bfC:\bfC\ra\ti C(\bfC)$ to be the composition 
\e
\xymatrix@C=33pt{ \bfC\,\, \ar@{^{(}->}[r]^(0.3){\inc} & \bfC\bigl[\al_{c'},\be_{c'}:c'\in\fC_\rex\bigr] \ar[r]^(0.7){\text{project}} &  \bs{\ti\bfC} \ar[r]^(0.3){\bs\psi_{\bs{\ti\bfC}}}  & \Pi_{\rm all}^{\rm sc}(\bs{\ti\fC})=\ti C(\bfC), }
\label{cc6eq6}
\e
for $\bs\psi_{\bs{\ti\bfC}}$ as in Definition \ref{cc5def10}. This is functorial in $\bfC$, and so defines a natural transformation $\bs{\ti\Pi}:\Id\Ra\ti C$ of functors~$\CRingsc\ra\CRingsc$.

We will also write $\al_{c'},\be_{c'}$ for the images of $\al_{c'},\be_{c'}$ in $\ti C(\fC)_\rex$, and write
\e
\Prc_\bfC=\{P \subset \fC_\rex:P \text{ is a prime ideal}\}.
\label{cc6eq7}
\e\end{dfn}

\begin{thm}
\label{cc6thm3}
In Definition\/ {\rm\ref{cc6def2},} with\/ $C:\LCRSc\ra\LCRScin$ as in Definition\/ {\rm\ref{cc6def1},} for each\/ $\bfC\in\CRingsc$ there is a natural isomorphism 
\e
\bs\eta_\bfC:\Specc\ti C(\bfC)\longra C(\Specc\bfC)
\label{cc6eq8}
\e
in $\LCRSc$. So $\Specc\ti C(\bfC)$ lies in $\LCRScin\subset\LCRSc$. The following diagram commutes in {\rm$\LCRSc$:}
\e
\begin{gathered}
\xymatrix@C=80pt@R=15pt{ 
*+[r]{\Specc\ti C(\bfC)} \ar[rr]_{\bs\eta_\bfC} \ar[dr]_{\Specc\bs{\ti\Pi}_\bfC} && *+[l]{C(\Specc\bfC)} \ar[dl]^{\,\,\,\,\,\bs\Pi_{\Specc\bfC}} \\
& \Specc\bfC. }	
\end{gathered}
\label{cc6eq9}
\e

For each morphism $\bs\phi:\bfC\ra\bfD$ in $\CRingsc$ the following commutes:
\e
\begin{gathered}
\xymatrix@C=200pt@R=15pt{ 
*+[r]{\Specc\ti C(\bfD)} \ar[d]^{\Specc\ti C(\bs\phi)} \ar[r]^{\bs\eta_\bfD}_\cong & *+[l]{C(\Specc\bfD)} \ar[d]_{C(\Specc\bs\phi)} \\
*+[r]{\Specc\ti C(\bfC)} \ar[r]^{\bs\eta_\bfC}_\cong & *+[l]{C(\Specc\bfC),\!} }	
\end{gathered}
\label{cc6eq10}
\e
so $\Specc\ti C(\bs\phi)$ is a morphism in $\LCRScin\!\subset\!\LCRSc$. Thus $\bs\eta:\Specc\ci\ti C\Ra C\ci\Specc$ is a natural isomorphism of functors\/~$(\CRingsc)^{\bf op}\!\ra\!\LCRScin$.
\end{thm}

\begin{proof}
Let $\bfC\in\CRingsc$, and define $\bs{\ti\fC}$ as in \eq{cc6eq4}, so that $\ti C(\bfC)=\Pi_{\rm all}^{\rm sc}(\bs{\ti\fC})$. Write $\bX=\Specc\bfC$ and $\bs{\ti X}=\Specc\bs{\ti\fC}$. As in \eq{cc6eq6}, there is a functorial morphism $\bs\psi_{\bs{\ti\fC}}:\bs{\ti\fC}\ra\ti C(\bfC)$ such that $\Specc\bs\psi_{\bs{\ti\fC}}:\Specc\ti C(\bfC)\ra\Specc\bs{\ti\fC}$ is an isomorphism. Write $\bs{\check\Pi}_\bfC:\bfC\ra\bs{\ti\fC}$ for the composition of the first two morphisms in \eq{cc6eq6}, so that $\bs{\ti\Pi}_\bfC=\bs\psi_{\bs{\ti\fC}}\ci\bs{\check\Pi}_\bfC$, and set $\bs{\check\Pi}_\bX=\Specc\bs{\check\Pi}_\bfC:\bs{\ti X}\ra\bX$. We will construct an isomorphism 
\e
\bs\ze_\bfC:\bs{\ti X}=\Specc\bs{\ti\fC}\longra C(\Specc\bfC)=C(\bX)
\label{cc6eq11}
\e
in $\LCRSc$, and define $\bs\eta_\bfC$ in \eq{cc6eq8} by $\bs\eta_\bfC=\bs\ze_\bfC\ci\Specc\bs\psi_{\bs{\ti\fC}}$.

First we define the isomorphism \eq{cc6eq11} at the level of points. A point of $\bs{\ti X}$ is an $\R$-algebra morphism $\ti x:\ti\fC\ra\R$. Composing with the projection $\fC[\al_{c'},\be_{c'}:c'\in\fC_\rex]\ra\ti\fC$ gives a morphism $\hat x:\fC[\al_{c'},\be_{c'}:c'\in\fC_\rex]\ra\R$. Such morphisms are in 1-1 correspondence with data $(x,a_{c'},b_{c'}:c'\in\fC_\rex)$, where $x:\fC\ra\R$ is an $\R$-algebra morphism (i.e. $x\in X$) and $a_{c'}=\hat x(\al_{c'})$, $b_{c'}=\hat x(\be_{c'})$ lie in $[0,\iy)\subseteq\R$ for all $c'\in\fC_\rex$. A morphism $\hat x:\fC[\al_{c'},\be_{c'}:c'\in\fC_\rex]\ra\R$ descends to a (unique) morphism $\ti x:\ti\fC\ra\R$ if and only if $\hat x$ maps the relations in \eq{cc6eq5} to relations in $\R$. Hence we may identify
\e
\begin{split}
\ti X\cong\bigl\{&(x,a_{c'},b_{c'}:c'\in\fC_\rex):x\in X,\;\> a_{c'},b_{c'}\in[0,\iy) \;\> \forall c'\in\fC_\rex,\\ 
&a_{c'}+b_{c'}=1\;\> \forall c'\in\fC_\rex,\;\>
a_{1_{\fC_\rex}}=1,\; a_{c'}a_{c''}=a_{c'c''}\;\> \forall c',c''\in\fC_\rex,\\
& x\ci\Phi_i(c')b_{c'}=a_{c'}b_{c'}=0\;\>  \forall c'\in\fC_\rex\bigr\}.
\end{split}
\label{cc6eq12}
\e

We can simplify \eq{cc6eq12}. Let $(x,a_{c'},b_{c'}:c'\in\fC_\rex)$ be a point in the right hand side. The equations $a_{c'}+b_{c'}=1$ and $a_{c'}b_{c'}=0$ imply that $(a_{c'},b_{c'})$ is $(1,0)$ or $(0,1)$ for each $c'$. Define $P=\{c'\in\fC_\rex:a_{c'}=0\}$. Then $a_{1_{\fC_\rex}}=1$ implies that $1_{\fC_\rex}\notin P$, and $a_{c'}a_{c''}=a_{c'c''}$ implies that if $c'\in\fC_\rex$ and $c''\in P$ then $c'c''\in P$, so $P$ is an ideal. Also $a_{c'}a_{c''}=a_{c'c''}$ implies that if $c'c''\in P$ then $c'\in P$ or $c''\in P$, so $P$ is a prime ideal. The condition $x\ci\Phi_i(c')b_{c'}=0$ then becomes $x\ci\Phi_i(c')=0$ if $c'\in P$. 

Conversely, if $P$ is a prime ideal and $x\ci\Phi_i(c')=0$ for $c'\in P$, then setting $a_{c'}=0$, $b_{c'}=1$ if $c'\in P$ and $a_{c'}=1$, $b_{c'}=0$ if $c'\in \fC_\rex\sm P$ then all the equations of \eq{cc6eq12} hold. Thus, mapping $\bigl(x,a_{c'},b_{c'}:c'\in\fC_\rex\bigr)\mapsto \bigl(x,P=\{c'\in\fC_\rex:a_{c'}=0\}\bigr)$ turns \eq{cc6eq12} into a bijection
\e
\begin{split}
\ti X\cong\bigl\{&(x,P):\text{$x\in X$, $P\in\Prc_\bfC$, $x\ci\Phi_i(c')=0$ if $c'\in P$}\bigr\}.
\end{split}
\label{cc6eq13}
\e

\begin{lem}
\label{cc6lem2}
Let\/ $x\in\bX,$ so that\/ $x:\fC\ra\R$ is an $\R$-point, with localization $\bs\pi_x:\bfC\ra\bfC_x\cong\bO_{X,x}$ giving a morphism $\pi_{x,\rex}:\fC_\rex\ra\fC_{x,\rex}\cong\O_{X,x}^\rex$. Then we have inverse bijective maps
\e
\xymatrix@C=75pt{\bigl\{P\in\Prc_\bfC:x\ci\Phi_i(c')=0\; \forall c'\in P\bigr\} \ar@<.5ex>[r]^(0.61){P\longmapsto P_x=\pi_{x,\rex}(P)} &
\Prc_{\bfC_x}\!\cong\Prc_{\bO_{X,x}}. \ar@<.5ex>[l]^(0.38){P_x\longmapsto P=\pi_{x,\rex}^{-1}(P_x)} }\!\!\!
\label{cc6eq14}
\e
\end{lem}

\begin{proof}
Let $P$ lie in the left hand side of \eq{cc6eq14}, and set $P_x=\pi_{x,\rex}(P)$. Then $1\notin P_x$ as $x\ci\Phi_i(c')=0$ for all $c'\in P$. If $c_x'\in \fC_{x,\rex}$ and $c_x''\in P_x$ then $c_x'c_x''\in P_x$ as $P$ is an ideal and $\pi_{x,\rex}$ is surjective. Thus $P_x$ is an ideal. To show $P_x$ is prime, suppose $c_x',d_x'\in \fC_{\rex,x}$ with $c_x'd_x'\in P_x$. As $\pi_{x,\rex}$ is surjective we can choose $c', d'\in \fC_\rex$ and $p'\in P$ with $\pi_{x,\rex}(c')=c_x'$, $\pi_{x,\rex}(d')=d_x'$ and $\pi_{x,\rex}(c'd')=\pi_{x,\rex}(p)$. Hence Proposition \ref{cc4prop8} gives $a',b'\in \fC_\rex$ with $a'c'd'=b'p'$ and $x\ci\Phi_i(a')\ne 0$. But $b'p'\in P$ as $P$ is an ideal, so $a'c'd'\in P$. Since $P$ is prime, one of $a',c',d'$ must lie in $P$, and $a'\notin P$ as $x\ci\Phi_i(a')\ne 0$, so one of $c',d'$ must lie in $P$. Thus one of $c'_x,d'_x$ lie in $P_x$, and $P_x$ is prime. Hence the top morphism in \eq{cc6eq14} is well defined.

Next let $P_x\in\Prc_{\bfC_x}$, and set $P=\pi_{x,\rex}^{-1}(P_x)$. It is easy to check $P$ lies in the left hand side of \eq{cc6eq14}. So the bottom morphism in \eq{cc6eq14} is well defined.

Suppose $P$ lies in the left hand side of \eq{cc6eq14}. Clearly $P\subseteq  \pi_{x,\rex}^{-1}(\pi_{x,\rex}(P))$. If $c'\in \pi_{x,\rex}^{-1}(\pi_{x,\rex}(P))$ then there exists $p'\in P$ with $\pi_{x,\rex}(c')=\pi_{x,\rex}(p')$. Proposition \ref{cc4prop8} gives $a',b'\in \fC_\rex$ with $a'c'=b'p'$ and $x\ci\Phi_i(a')\ne 0$, and $p'\in P$ implies $a'c'=b'p'\in P$, which implies $c'\in P$ as $\pi_{x,\rex}(a')\ne 0$ so $a'\notin P$. Hence $P=\pi_{x,\rex}^{-1}(\pi_{x,\rex}(P))$. If $P_x\in\Prc_{\bfC_x}$ then $P_x=\pi_{x,\rex}(\pi_{x,\rex}^{-1}(P_x))$ as $\pi_{x,\rex}$ is surjective. Thus the maps in \eq{cc6eq14} are inverse.
\end{proof}

Combining equations \eq{cc6eq1} and \eq{cc6eq13} and Lemma \ref{cc6lem2} gives a bijection of sets $\ze_\fC:\ti X\ra C(X)$, defined explicitly by, if $\ti x:\ti\fC\ra\R$ corresponds to $(x,a_{c'},b_{c'}:c'\in\fC_\rex)$ in \eq{cc6eq25}, and $P=\{c'\in\fC_\rex:a_{c'}=0\}$ is the associated prime ideal, and $P_x=\pi_{x,\rex}(P)$ in $\fC_{x,\rex}\cong\O_{X,x}^\rex$, then $\ze_\fC(\ti x)=(x,P_x)$ in $C(X)$ in \eq{cc6eq1}. Then in an analogue of \eq{cc6eq9}, on the level of sets we have
\e
\check\Pi_X=\Pi_X\ci\ze_\fC:\ti X\longra X,
\label{cc6eq15}
\e
as both sides map $\ti x\mapsto x$.

\begin{lem}
\label{cc6lem3} 
$\ze_\fC:\ti X\ra C(X)$ is a homeomorphism of topological spaces.
\end{lem}

\begin{proof} By definition of the topology on $\Spec\ti\fC$ in Definition \ref{cc2def15}, the topology on $\ti X$ is generated by open sets $\ti U_{\ti c}=\{\ti x\in\ti X:\ti x(\ti c)\ne 0\}$ for $\ti c\in\ti\fC$. As $\ti\fC$ is a quotient of $\fC[\al_{c'},\be_{c'}:c'\in\fC_\rex]$, the topology is generated by $\ti U_c$ for $c\in\fC$, and $\ti U_{\Phi_f(\al_{c'})}$, $\ti U_{\Phi_f(\be_{c'})}$ for all $c'\in\fC_\rex$ and smooth $f:[0,\iy)\ra\R$. But as $\al_{c'},\be_{c'}$ take only values 0,1 on $\ti X$ with $\al_{c'}+\be_{c'}=1$, the only possibilities for $\ti U_{\Phi_f(\al_{c'})}$, $\ti U_{\Phi_f(\be_{c'})}$ are $\dot V_{c'}=\{\ti x\in\ti X:\al_{c'}\vert_{\ti x}\ne 0\}$ and $\ddot V_{c'}=\{\ti x\in\ti X:\al_{c'}\vert_{\ti x}\ne 1\}$. So the topology on $\ti X$ is generated by $\ti U_c$ for all $c\in\fC$ and $\dot V_{c'},\ddot V_{c'}$ for all~$c'\in\fC_\rex$.

For $C(X)$ in Definition \ref{cc6def1}, the topology is generated by open sets $\Pi_X^{-1}(U_c)$ for $c\in\fC$ and subsets $\dot U_{s'},\ddot U_{s'}$ for open $U\subset X$ and $s'\in\O_X^\rex(U)$. As $\bX=\Specc\bfC$, rather than taking all such $\dot U_{s'},\ddot U_{s'}$, it is enough to take $U=X$ and $s'=\Xi_{\fC,\rex}(c')$ for all $c'\in\fC_\rex$, as all other $\dot U_{s'},\ddot U_{s'}$ can be generated from $\Pi_X^{-1}(U_c)$ and $\dot X_{\Xi_{\fC,\rex}(c')},\ddot X_{\Xi_{\fC,\rex}(c')}$. Thus, the topology on $C(X)$ is generated by $\Pi_X^{-1}(U_c)$ for all $c\in\fC$ and $\dot X_{\Xi_{\fC,\rex}(c')},\ddot X_{\Xi_{\fC,\rex}(c')}$ for all $c'\in\fC_\rex$. 

It is easy to check that $\ze_\fC$ maps $\ti U_c\mapsto\Pi_X^{-1}(U_c)$, $\dot V_{c'}\mapsto\dot X_{\Xi_{\fC,\rex}(c')}$ and $\ddot V_{c'}\mapsto\ddot X_{\Xi_{\fC,\rex}(c')}$. Therefore $\ze_\fC$ is a bijection and maps a basis for the topology of $\ti X$ to a basis for the topology of $C(X)$, so it is a homeomorphism.
\end{proof}

\begin{lem}
\label{cc6lem4} 
In sheaves of\/ $C^\iy$-rings with corners on $\ti X,$ there is a unique morphism $\bs\ze_\bfC^\sh$ making the following diagram commute:
\e
\begin{gathered}
\xymatrix@C=200pt@R=15pt{ 
*+[r]{\ze_\fC^{-1}\ci\Pi_X^{-1}(\bO_X)} \ar[d]^{\ze_\fC^{-1}(\bs\Pi_\bX^\sh)} \ar@{=}[r]_{\text{by \eq{cc6eq15}}} & *+[l]{\check\Pi_X^{-1}(\bO_X)} \ar[d]_{\bs{\check\Pi}_\bX^\sh} \\
*+[r]{\ze_\fC^{-1}(\bO_{C(X)})} \ar@{..>}[r]^{\bs\ze_\bfC^\sh} & *+[l]{\bO_{\ti X}.\!} }	
\end{gathered}
\label{cc6eq16}
\e
Furthermore, $\bs\ze_\bfC^\sh$ is an isomorphism.
\end{lem}

\begin{proof} By Definition \ref{cc6def1}, $\bs\Pi_\bX^\sh:\Pi_X^{-1}(\bO_X)\ra\bO_{C(X)}$ is the universal surjective quotient of $\Pi_X^{-1}(\bO_X)$ in sheaves of $C^\iy$-rings with corners, such that if $U\subseteq C(X)$ is open and $s'\in\Pi_X^{-1}(\O_X^\rex)(U)$ with $s'\vert_{(x,P_x)}\in P_x$ for each $(x,P_x)\in U$ then $\Pi_{X,\rex}^\sh(U)(s')=0$ in $\O_{C(X)}^\rex(U)$. 
As $\bX=\Specc\bfC$, locally on $X$ any section of $\O^\rex_\bX$ is of the form $\Xi_{\fC,\rex}(c')$ for $c'\in\fC_\rex$, so locally on $C(X)$ any section $s'$ of $\Pi_X^{-1}(\bO_X)$ is of the form $\Pi_{X,\rex}^{-1}\ci\Xi_{\fC,\rex}(c')$. Thus it is enough to verify the universal property above when $s'=\Pi_{X,\rex}^{-1}\ci\Xi_{\fC,\rex}(c')\vert_U$ for $c'\in\fC_\rex$, as it is local in $s'$ anyway. That is, $\bs\Pi_\bX^\sh:\Pi_X^{-1}(\bO_X)\ra\bO_{C(X)}$ is characterized by the property that if $c'\in\fC_\rex$ and $U\subset C(X)$ is open with $\pi_{x,\rex}(c')\in P_x\subset\fC_{x,\rex}\cong\O_{X,x}^\rex$ for all $(x,P_x)\in U$ then $\Pi_{X,\rex}^\sh(C(X))\ci\Xi_{\fC,\rex}(c')\vert_U=0$.

Since $\ze_\fC$ is a homeomorphism by Lemma \ref{cc6lem3}, the pullback $\ze_\fC^{-1}(\bs\Pi_\bX^\sh)$ in \eq{cc6eq16} has the analogous universal property. Let $U\subseteq C(X)$ and $c'\in\fC_\rex$ have the properties above, and set $\ti U=\ze_\fC^{-1}(U)$. Suppose $\ti x\in\ti U$ with $\ze_\fC(\ti x)=(x,P_x)$ in  $U\subseteq C(X)$, and let $\ti x$ correspond to $(x,a_{c''},b_{c''}:c''\in\fC_\rex)$ in \eq{cc6eq12}. By construction of $\ze_\fC$ we have $a_{c'}=0$, $b_{c'}=1$ as $\pi_{x,\rex}(c')\in P_x$. Thus the relation $c'\be_{c'}=0$ in $\ti\fC_\rex$ in \eq{cc6eq5} implies that $\ti\pi_{\ti x,\rex}(c')=0$ in $\ti\fC_{\ti x,\rex}\cong\O_{\ti X,\ti x}^\rex$. As this holds for all $\ti x\in\ti U$ we have $\check\Pi_{X,\rex}^\sh(\ti X)\ci\Xi_{\fC,\rex}(c')\vert_{\ti U}=0$. Hence the universal property of $\ze_\fC^{-1}(\bs\Pi_\bX^\sh)$ gives a unique morphism $\bs\ze_\bfC^\sh$ making \eq{cc6eq16} commute.

To show $\bs\ze_\bfC^\sh$ is an isomorphism, it is enough to show it is an isomorphism on stalks at each $\ti x\in\ti X$. Let $\ti x$ correspond to $(x,a_{c'},b_{c'}:c'\in\fC_\rex)$ in \eq{cc6eq12}, with $\ze_\fC(\ti x)=(x,P_x)$, so that $a_{c'}=0$, $b_{c'}=1$ if $\pi_{x,\rex}(c')\in P_x$ and $a_{c'}=1$, $b_{c'}=0$ otherwise. The stalk $\ze_\fC^{-1}(\bO_{C(X)})_{\ti x}$ is $\bO_{C(X),(x,P_x)}$, which is $\bfC_x/\simc_{P_x}$. The stalk $\bO_{\ti X,\ti x}$ is $\bs{\ti\fC}_{\ti x}$. By the values of $a_{c'},b_{c'}$, in \eq{cc6eq5} we know that $\al_{c'}=0$, $\be_{c'}=1$ near $\ti x$ if $\pi_{x,\rex}(c')\in P_x$, and $\al_{c'}=1$, $\be_{c'}=0$ near $\ti x$ otherwise. Hence the generators $\al_{c'},\be_{c'}$ add no new generators to $\bs{\ti\fC}_{\ti x}$, and $\bs{\ti\fC}_{\ti x}$ is a quotient of $\bfC_x$. The relations in \eq{cc6eq5} turn into relations in $\bs{\ti\fC}_{\ti x}$ according to the local values 0 or 1 of $\al_{c'},\be_{c'}$. Therefore we see that
\begin{equation*}
\bs{\ti\fC}_{\ti x}\cong\bfC_x\big/\bigl[\text{$\pi_{x,\rex}(c')=0$ if $c'\in\fC_\rex$ with $\pi_{x,\rex}(c')\in P_x$}\bigr]=\bfC_x/\simc_{P_x},
\end{equation*}
where the second step holds as $\pi_{x,\rex}$ is surjective. So we have an isomorphism of stalks $\ze_\fC^{-1}(\bO_{C(X)})_{\ti x}\cong\bO_{\ti X,\ti x}$, which is the natural map~$\bs\ze_{\bfC,\ti x}^\sh$.
\end{proof}

Lemmas \ref{cc6lem3} and \ref{cc6lem4} show $\bs\ze_\bfC=(\ze_\bfC,\bs\ze_\bfC^\sh)$ is an isomorphism $\bs{\ti X}\ra C(\bX)$, as in \eq{cc6eq11}. So $\bs\eta_\bfC=\bs\ze_\bfC\ci\Specc\bs\psi_{\bs{\ti\fC}}$ is an isomorphism in \eq{cc6eq8}. Equations \eq{cc6eq15}--\eq{cc6eq16} imply that \eq{cc6eq9} commutes. This proves the first part of Theorem~\ref{cc6thm3}.

For the second part, consider the diagram
\begin{equation*}
\xymatrix@C=110pt@R=18pt{ 
*+[r]{\Specc\ti C(\bfD)} \ar[d]^{\Specc\ti C(\bs\phi)=\Specc \Pi_{\rm all}^{\rm sc}(\bs{\ti\phi})} \ar@<.5ex>@/^.7pc/[rr]^{\bs\eta_\bfD} \ar[r]_{\Specc\bs\psi_{\bs{\ti\fD}}} & \Specc\bs{\ti\fD} \ar[r]_{\bs\ze_\bfD} \ar[d]^{\Specc\bs{\ti\phi}} & *+[l]{C(\Specc\bfD)} \ar[d]_{C(\Specc\bs\phi)} \\
*+[r]{\Specc\ti C(\bfC)} \ar@<-.5ex>@/_.7pc/[rr]_{\bs\eta_\bfC} \ar[r]^{\Specc\bs\psi_{\bs{\ti\fC}}} & \Specc\bs{\ti\fC} \ar[r]^{\bs\ze_\bfC} & *+[l]{C(\Specc\bfC).\!} }	
\end{equation*}
The left hand square commutes as $\bs\psi:\Id\Ra \Pi_{\rm all}^{\rm sc}$ is a natural transformation in Definition \ref{cc5def10}. The right hand square commutes by functoriality of $\bs\ze_\bfC$ in \eq{cc6eq11}. The semicircles commute by definition of $\bs\eta_\bfC,\bs\eta_\bfD$. Hence \eq{cc6eq10} commutes, and the theorem follows.
\end{proof}

The next proposition follows from the proof of Theorem \ref{cc6thm3}. Here $C(X)^P$ is closed as it may be defined by the equations $\al_{c'}=0$ for $c'\in P$, $\al_{c'}=1$ for $c'\in\fC_\rex\sm P$. Example \ref{cc6ex1} below shows the $C(X)^P$ need not be open.

\begin{prop}
\label{cc6prop1}
Let\/ $\bfC=(\fC,\fC_\rex)\in\CRingsc,$ and write $\bX=\Specc\bfC$. With\/ $C(X)$ as in {\rm\eq{cc6eq1},} for each\/ $P\in\Prc_\bfC,$ define
\ea
X^P&=\bigl\{x\in X:\text{$x\ci\Phi_i(c')=0$ for all\/ $c'\in P$}\bigr\},
\label{cc6eq17}\\
\begin{split}
C(X)^P&=\bigl\{(x,P_x)\in C(X):\text{$\pi_{x,\rex}(P)=P_x$ in $\fC_{x,\rex}\cong\O_{X,x}^\rex$}\bigr\}\\
&=\bigl\{(x,P_x)\in C(X):P=\pi_{x,\rex}^{-1}(P_x)\bigr\} \\
&=\bigl\{(x,P_x)\in C(X):\text{if\/ $c'\in\fC_\rex$ then $\pi_{x,\rex}(c')\in P_x$ }\\
&\qquad\qquad \text{in $\fC_{x,\rex}\cong\O_{X,x}^\rex$ if and only if\/ $c'\in P$}\bigr\}\subseteq C(X),
\end{split}
\label{cc6eq18}
\ea
where the three expressions in \eq{cc6eq18} are equivalent. Then $\Pi_X:C(X)\ra X$ restricts to a bijection $C(X)^P\ra X^P$. Also $C(X)^P$ is closed in $C(X),$ but not necessarily open, and just as sets we have
\e
C(X)=\coprod\nolimits_{P\in\Prc_\bfC}C(X)^P.
\label{cc6eq19}
\e
\end{prop}

As an easy consequence of Theorem \ref{cc6thm3}, we deduce:

\begin{thm}
\label{cc6thm4}
The corner functor $C:\LCRSc\ra\LCRScin$ of\/ {\rm\S\ref{cc61}} maps $\CSchc\ra\CSchcin$. We write $C:\CSchc\ra\CSchcin$ for the restriction, and call it a \begin{bfseries}corner functor\end{bfseries}. It is right adjoint to\/~$\inc:\CSchcin\hookra\CSchc$.

As $C$ is a right adjoint, it preserves limits in $\CSchc,\CSchcin$.
\end{thm}

\begin{proof}
Let $\bX$ lie in $\CSchc$. Then we can cover $X$ by open $\bU\subseteq\bX$ with $\bU\cong\Specc\bfC$ for $\bfC\in\CRingsc$. Since the definition of $\bs\Pi_\bX:C(\bX)\ra\bX$ is local over $\bX$, we have $\bs\Pi_\bX^{-1}(\bU)\cong C(\bU)\cong C(\Specc\bfC)$. Theorem \ref{cc6thm3} gives $C(\Specc\bfC)\cong\Specc \ti C(\bfC)$, with~$\ti C(\bfC)\in\CRingsc$. 

Since $\Specc \ti C(\bfC)$ lies in $\LCRScin\subset\LCRSc$, Theorem \ref{cc5thm7} shows that $\Specc \ti C(\bfC)\cong\Speccin\bfD$ for $\bfD\in\CRingscin$. Thus we can cover $C(\bX)$ by open $\bs\Pi_\bX^{-1}(\bU)\subseteq C(\bX)$ with $\bs\Pi_\bX^{-1}(\bU)\cong\Speccin\bfD$ for $\bfD\in\CRingscin$, so $C(\bX)$ is an interior $C^\iy$-scheme with corners. Therefore $C:\LCRSc\ra\LCRScin$ maps $\CSchc\ra\CSchcin$. The rest is immediate from Theorem~\ref{cc6thm1}.
\end{proof}

The next proposition is easy to prove by considering local models.

\begin{prop}
\label{cc6prop2}
The corner functors $C$ for $\Manc,\Mangc,\cManc,\cMangc$ in Definition\/ {\rm\ref{cc3def9},} and for $\CSchc$ in Theorem\/ {\rm\ref{cc6thm4},} commute with the functors $F_\Manc^\CSchc,\ldots,F_\cMangc^\CSchc$ in Definition\/ {\rm\ref{cc5def9}} up to natural isomorphism.
\end{prop}

Combining Theorems \ref{cc5thm8}(b), \ref{cc6thm2} and \ref{cc6thm4} yields:

\begin{thm}
\label{cc6thm5}
Restricting $C$ in Theorem\/ {\rm\ref{cc6thm4}} gives corner functors
\begin{align*}
C&:\CSchcsaex\longra 	\CSchcsa,&
C&:\CSchctfex\longra 	\CSchctf,\\
C&:\CSchcZex\longra 	\CSchcZ,&
C&:\CSchcinex\longra 	\CSchcin,
\end{align*}
in the notation of Definition\/ {\rm\ref{cc5def11},} which are right adjoint to the corresponding inclusions. As these corner functors are right adjoints, they preserve limits.
\end{thm}

\subsection{\texorpdfstring{The corner functor for firm $C^\iy$-schemes with corners}{The corner functor for firm C∞-schemes with corners}}
\label{cc63}

The following proposition is why we included $\Pi_{\rm all}^{\rm sc}$ in defining $\ti C$ in Definition \ref{cc6def2}. It forces relations on $\al_{c'},\be_{c'}$ in $\ti C(\bfC)=\Pi_{\rm all}^{\rm sc}(\bs{\ti\fC})$, which need not hold in~$\bs{\ti\fC}$.

\begin{prop}
\label{cc6prop3}
Let\/ $\bfC\in\CRingsc$ and write $\bfD=\ti C(\bfC)$ with morphism $\bs{\ti\Pi}_\bfC:\bfC\ra\bfD,$ so we have elements $\al_{c'},\be_{c'},\ti\Pi_{\fC,\rex}(c')$ in $\fD_\rex$ for all\/ $c'\in\fC_\rex$. Write $c'\mapsto[c'],$ $d'\mapsto[d']$ for the projections $\fC_\rex\mapsto\fC_\rex^\sh,$ $\fD_\rex\mapsto\fD_\rex^\sh$. Then:
\begin{itemize}
\setlength{\itemsep}{0pt}
\setlength{\parsep}{0pt}
\item[{\bf(a)}] If\/ $c_1',c_2'\in\fC_\rex$ with\/ $[c_1']=[c_2']$ in $\fC_\rex^\sh$ then $\al_{c_1'}=\al_{c_2'},$ $\be_{c_1'}=\be_{c_2'}$ in~$\fD_\rex$.
\item[{\bf(b)}] If\/ $c_1',\ldots,c_n'\in\fC_\rex$ and\/ $i_1,\ldots,i_n,j_1,\ldots,j_n$ are positive integers then $\al_{c_1^{\prime i_1}\cdots c_n^{\prime i_n}}=\al_{c_1^{\prime j_1}\cdots c_n^{\prime j_n}}$ and\/ $\be_{c_1^{\prime i_1}\cdots c_n^{\prime i_n}}=\be_{c_1^{\prime j_1}\cdots c_n^{\prime j_n}}$ in~$\fD_\rex$.
\item[{\bf(c)}] $\fD_\rex^\sh$ is generated by $[\al_{c'}],[\be_{c'}],[\ti\Pi_{\fC,\rex}(c')]$ for all\/ $c'\in\fC_\rex$.
\item[{\bf(d)}] If\/ $\bfC$ is firm then $\bfD$ is firm. Thus $\ti C$ maps $\CRingscfi\ra\CRingscfi$.
\end{itemize}
\end{prop}

\begin{proof} Let $\bs{\ti\fC}$ be as in \eq{cc6eq5} and $\bs{\ti X}=\Specc\bs{\ti\fC}$, so that $\bfD\subseteq\bO_{\ti X}(\ti X)$, and elements of $\fD_\rex$ are global sections of the sheaf $\O_{\ti X}^\rex$ on $\ti X$. Hence $d_1',d_2'\in\fD_\rex$ have $d_1'=d_2'$ if and only if $d_1'\vert_{\ti x}=d_2'\vert_{\ti x}$ in $\O^\rex_{\ti X,\ti x}\cong\ti\fC_{\ti x,\rex}$ for all~$\ti x\in\ti X$.

In the proof of Theorem \ref{cc6thm3} we saw that if $\ti x\in\ti X$ with $\ze_\fC(\ti x)=(x,P_x)$ in $C(X)$ and $c'\in\fC_\rex$, and $P$ corresponds to $P_x$ under \eq{cc6eq14}, then $\al_{c'}\vert_{\ti x}=0$, $\be_{c'}\vert_{\ti x}=1$ in $\O^\rex_{\ti X,\ti x}\cong\ti\fC_{\ti x,\rex}$ if $c'\in P$, and $\al_{c'}\vert_{\ti x}=1$, $\be_{c'}\vert_{\ti x}=0$ otherwise. Thus, if $c_1',c_2'\in\fC_\rex$ satisfy $c_1'\in P$ if and only if $c_2'\in P$ for any prime ideal $P\subset\fC_\rex$, then $\al_{c_1'}\vert_{\ti x}=\al_{c_2'}\vert_{\ti x}$, $\be_{c_1'}\vert_{\ti x}=\be_{c_2'}\vert_{\ti x}$ for all $\ti x\in\ti X$, so $\al_{c_1'}=\al_{c_2'},$ $\be_{c_1'}=\be_{c_2'}$ in~$\fD_\rex$.

Parts (a),(b) are consequences of this criterion. For (a), as prime ideals in $\fC_\rex$ are preimages of prime ideals in $\fC_\rex^\sh$, if $[c_1']=[c_2']$ in $\fC_\rex^\sh$ then $c_1'\in P$ if and only if $c_2'\in P$ for any prime $P\subset\fC_\rex$. For (b), if $i_1,\ldots,j_n>0$ the prime property of $P$ implies that $c_1^{\prime i_1}\cdots c_n^{\prime i_n}\in P$ if and only if $c_1^{\prime j_1}\cdots c_n^{\prime j_n}\in P$.

For (c), $\ti\fC{}_\rex^\sh$ is generated by $[\al_{c'}],[\be_{c'}],[c']$ for $c'\in\fC_\rex$ by \eq{cc6eq5}, and $\psi_{\ti\fC,\rex}^\sh:\ti\fC{}_\rex^\sh\ra\fD_\rex^\sh$ is surjective as in the construction of  $\bfD=\Pi_{\rm all}^{\rm sc}(\bs{\ti\fC})$ in Proposition \ref{cc5prop2}, $\fD_\rex$ is generated in $\O_{\ti X}^\rex(X)$ by invertibles and $\ti\fC_\rex$, so (c) follows.

For (d), let $\bfC$ be firm, so that $\fC_\rex^\sh$ is finitely generated, and choose $c_1',\ldots,c_n'$ in $\fC_\rex$ such that $[c_1'],\ldots,[c_n']$ generate $\fC_\rex^\sh$. Then (a),(b) imply that there are at most $2^n$ distinct $\al_{c'}$ and $2^n$ distinct $\be_{c'}$ in $\fD_\rex$, as every $\al_{c'}$ equals $\al_{\prod_{i\in S}c_i'}$ for some subset $S\subseteq\{1,\ldots,n\}$. Hence (c) shows that $\fD_\rex^\sh$ is generated by the $2^{n+1}+n$ elements $[\al_{\prod_{i\in S}c_i'}],[\be_{\prod_{i\in S}c_i'}]$ for $S\subseteq\{1,\ldots,n\}$ and $[\ti\Pi_{\fC,\rex}(c_i')]$ for $i=1,\ldots,m$. Therefore $\bfD$ is firm.
\end{proof}

As for Theorem \ref{cc6thm4}, Theorems \ref{cc5thm8}(c), \ref{cc6thm4} and Proposition \ref{cc6prop3}(d) imply:

\begin{thm}
\label{cc6thm6}
The corner functor $C:\CSchc\ra\CSchcin$ from\/ {\rm\S\ref{cc62}} maps $\CSchcfi\ab\ra\CSchcfiin$. We call the restriction $C:\CSchcfi\ra\CSchcfiin$ a \begin{bfseries}corner functor\end{bfseries}. It is right adjoint to\/~$\inc:\CSchcfiin\hookra\CSchcfi$.

As $C$ is a right adjoint, it preserves limits in $\CSchcfi,\CSchcfiin$.
\end{thm}

Combining Theorems \ref{cc5thm8}(c), \ref{cc6thm5}, and \ref{cc6thm6} yields:

\begin{thm}
\label{cc6thm7}
Restricting $C$ in Theorem\/ {\rm\ref{cc6thm6}} gives corner functors
\begin{align*}
C&:\CSchctoex\longra 	\CSchcto,&
C&:\CSchcfitfex\longra 	\CSchcfitf,\\
C&:\CSchcfiZex\longra 	\CSchcfiZ,&
C&:\CSchcfiinex\longra 	\CSchcfiin,
\end{align*}
in the notation of Definition\/ {\rm\ref{cc5def11},} which are right adjoint to the corresponding inclusions. As these corner functors are right adjoints, they preserve limits.
\end{thm}

For a firm $C^\iy$-ring with corners $\bfC$ with $\bX=\Specc\bfC$, we can give an alternative description of~$C(\bX)$:

\begin{thm}
\label{cc6thm8}
Let\/ $\bfC$ be a firm $C^\iy$-ring with corners and\/ $\bX=\Specc\bfC$. Then $\Prc_\bfC$ in \eq{cc6eq7} is finite. Hence in the decomposition {\rm\eq{cc6eq19},} the subsets $C(X)^P\subseteq C(X)$ are open as well as closed, and\/ \eq{cc6eq19} lifts to a disjoint union 
\e
C(\bX)=\coprod\nolimits_{P\in\Prc_\bfC}C(\bX)^P\qquad\text{in\/ $\CSchcfiin$.}
\label{cc6eq20}
\e

For each\/ $P\in\Prc_\bfC,$ Example\/ {\rm\ref{cc4ex4}(b)} defines a $C^\iy$-ring with corners $\bfC/\simc_P$ with projection $\bs\pi_P:\bfC\ra\bfC/\simc_P,$ where $\bfC/\simc_P$ is firm and interior. Write\/ $\bX{}^P=\Speccin(\bfC/\simc_P)$ and\/ $\bs\Pi^P=\Specc\bs\pi_P:\bX{}^P\ra\bX$. Then
\e
\ts C\bigl(\bs\Pi^P\bigr)\ci \bs\io_{\bX{}^P}:
\bX{}^P\longra C(\bX)
\label{cc6eq21}
\e
is an isomorphism with the open and closed\/ $C^\iy$-subscheme $C(\bX)^P\subseteq C(\bX)$. Thus~$\coprod_{P\in\Prc_\bfC}\bX{}^P\cong C(\bX)$.
\end{thm}

\begin{proof} First note that prime ideals $P\subset\fC_\rex$ are in 1-1 correspondence with prime ideals $P^\sh$ in $\fC_\rex^\sh$, which is finitely generated as $\bfC$ is firm. If $[c_1'],\ldots,[c_n']$ are generators of $\fC_\rex^\sh$ then $P^\sh=\bigl\langle\{[c_1'],\ldots,[c_n']\}\cap P^\sh\bigr\rangle$, so there are at most $2^n$ such prime ideals $P^\sh$. Hence $\Prc_\bfC$ is finite. Thus in \eq{cc6eq19}, $C(X)$ is the disjoint union of {\it finitely many\/} closed sets $C(X)^P$, so the $C(X)^P$ are also open. Equation \eq{cc6eq20} follows.

Write $\bfD=\ti C(\bfC)$ as in Proposition \ref{cc6prop3} and $\bY=\Specc\bfD$, so Theorem \ref{cc6thm3} gives an isomorphism $\bs\eta_\bfC:\bY\ra C(\bX)$. Fix $P\in\Prc_\bfC$, and write $\bY{}^P=\bs\eta_\bfC^{-1}(C(\bX)^P)$, so that $\bY{}^P$ is open and closed in $\bY$ with $\bs\eta_\bfC\vert_{\bY{}^P}:\bY{}^P\ra C(\bX)^P$ an isomorphism. The proof of Theorem \ref{cc6thm3} implies that $\bY{}^P$ may be defined as a $C^\iy$-subscheme of $\bY$ by the equations $\al_{c'}=0$, $\be_{c'}=1$ for $c'\in P$ and $\al_{c'}=1$, $\be_{c'}=0$ for $c'\in\fC_\rex\sm P$, where this is only finitely many equations, as there are only finitely many distinct $\al_{c'}$ in $\fD_\rex$ by the proof of Proposition \ref{cc6prop3}(d). Thus we see that
\ea
&\bY{}^P\cong\Specc\bigl(\bfD/[\text{$\al_{c'}=0$, $\be_{c'}=1$, $c'\in P$, $\al_{c''}=1$, $\be_{c''}=0$, $c''\in\fC_\rex\sm P$}]\bigr)
\nonumber\\
&=\Specc\bigl(\Pi_{\rm all}^{\rm sc}(\bs{\ti\fC})/[\text{$\al_{c'}=0$, $\be_{c'}=1$, $c'\in P$, $\al_{c''}=1$, $\be_{c''}=0$, $c''\in\fC_\rex\sm P$}]\bigr)
\nonumber\\
&\cong\Specc\ci \Pi_{\rm all}^{\rm sc}\bigl(\bs{\ti\fC}/[\text{$\al_{c'}=0$, $\be_{c'}=1$, $c'\in P$, $\al_{c''}=1$, $\be_{c''}=0$, $c''\in\fC_\rex\sm P$}]\bigr)
\nonumber\\
&=\Specc\bigl(
\bfC\bigl[\al_{c'},\be_{c'}:c'\in\fC_\rex\bigr]\big/\bigl(\Phi_i(\al_{c'})+\Phi_i(\be_{c'})=1\;\> \forall c'\in\fC_\rex\bigr)
\nonumber\\
&\quad\bigl[\al_{1_{\fC_\rex}}\!=\!1_{\fC_\rex},\; \al_{c'}\al_{c''}\!=\!\al_{c'c''}\; \forall c',c''\!\in\!\fC_\rex,\;  c'\be_{c'}\!=\!\al_{c'}\be_{c'}\!=\!0\;  \forall c'\!\in\!\fC_\rex 
\nonumber\\
&\quad\text{$\al_{c'}=0$, $\be_{c'}=1$, $c'\in P$, $\al_{c''}=1$, $\be_{c''}=0$, $c''\in\fC_\rex\sm P$}\bigr]\bigr)
\nonumber\\
&\cong\Specc\bigl(\bfC/[c'=0,\; c'\!\in\!P]\bigr)=\Specc(\bfC/\simc_P)=\bX{}^P,
\label{cc6eq22}
\ea
where in the third step we use that $\Pi_{\rm all}^{\rm sc}$ commutes with applying relations after composing with $\Specc$, in the fourth that $\Specc\ci \Pi_{\rm all}^{\rm sc}\cong\Specc$ and \eq{cc6eq5}, and in the fifth that adding generators $\al_{c'},\be_{c'}$ and setting them to 0 or 1 is equivalent to adding no generators, and the only remaining effective relations are $c'\be_{c'}=0$ when $c'\in P$, so that $\be_{c'}=1$, giving $c'=0$.

Composing the inverse of \eq{cc6eq22} with $\bs\eta_\bfC\vert_{\bY{}^P}$ gives a morphism $\bs\th_P:\bX{}^P\ra C(\bX)$ which is an isomorphism with $C(\bX)^P\subset C(\bX)$. As $\bfC/\simc_P$ is interior, $\bX{}^P$ is interior, so \eq{cc6eq21} is well defined. One can check, from the actions on points and on stalks $\bO_{X^P,x},\bO_{C(X),(x,P_x)}$, that \eq{cc6eq21} is $\bs\th_P$. The theorem follows.
\end{proof}

We explain by example how Theorem \ref{cc6thm8} goes wrong in the non-firm case.

\begin{ex}
\label{cc6ex1}
Write $\bs{[0,\iy)}=F_\Manc^\CSchc([0,\iy))=\Specc\bs C^\iy([0,\iy))$ as an affine $C^\iy$-scheme with corners. Proposition \ref{cc6prop2} gives $C(\bs{[0,\iy)})\cong\bs{\{*\}}\amalg\bs{[0,\iy)}$, writing $\pd[0,\iy)=\{*\}$, a single point.

Let $S$ be an infinite set. Consider $\bX=\prod_{s\in S}\bs{[0,\iy)}_s$, the product in $\CSchc$ of a copy $\bs{[0,\iy)}_s$ of $\bs{[0,\iy)}$ for each $s$ in $S$. Since $\Specc$ preserves limits as it is a right adjoint, we have $\bX=\Specc\bfC$ for~$\bfC=\bigot_{\iy}^{s\in S}\bs C^\iy([0,\iy))_s$.

As $C:\CSchc\ra\CSchcin$ preserves limits by Theorem \ref{cc6thm4} we have
\e
\ts C(\bX)=C\bigl(\prod_{s\in S}\bs{[0,\iy)}_s\bigr)\cong\prod_{s\in S}\bigl(\bs{\{*\}}_s\amalg\bs{[0,\iy)}_s\bigr).
\label{cc6eq23}
\e
Just as a set, we may decompose the right hand side of \eq{cc6eq23} as
\e
\ts C(X)\cong\coprod_{T\subseteq S}\bigl(\prod_{t\in T}\{*\}_t\t\prod_{s\in S\sm T}[0,\iy)_s\bigr)=:\coprod_{T\subseteq S}C(X)^T.
\label{cc6eq24}
\e
One can show that subsets $T\subseteq S$ are in natural 1-1 correspondence with prime ideals $P\subset\fC_\rex$, and \eq{cc6eq24} is the decomposition \eq{cc6eq19} for~$\bX=\Specc\bfC$.

Since $S$ is an infinite set, the subsets $C(X)^T\subset C(X)$ in \eq{cc6eq24} are closed, but not open, in \eq{cc6eq24}. This is because $C(X)$ has the product topology, as the forgetful functor $\CSchc\ra\Top$ preserves limits, and an infinite product of topological spaces $\prod_{s\in S}X_s$ has a weaker topology than one might expect: if $U_s\subseteq X_s$ is open for $s\in S$ then $\prod_{s\in S}U_s$ is generally {\it not\/} open in $\prod_{s\in S}X_s$, unless $U_s=X_s$ for all but finitely many~$s\in S$.
 
Although Theorem \ref{cc6thm8} makes sense when $\bfC$ is not firm and $\Prc_\bfC$ is infinite, it is false in general, as $C(\bX)\cong\coprod_P\bX{}^P$ fails already for topological spaces.
\end{ex}

\subsection{\texorpdfstring{The sheaves of monoids $\check M^\rex_{C(X)},\check M^\rin_{C(X)}$ on $C(\bX)$}{The sheaves of monoids Mᵉˣ,Mⁱⁿ on C(X)}}
\label{cc64}

As in Remark \ref{cc3rem5}(d), for a manifold with g-corners $X$, in \cite[\S 3.6]{Joyc6} we defined the {\it comonoid bundle\/} $M_{C(X)}^\vee\ra C(X)$, a local system of toric monoids on $C(X)$ with the property that if $(x,\ga)\in C(X)^\ci$ and $M_{C(X)}^\vee\vert_{(x,\ga)}=P$ then $X$ near $x$ is locally diffeomorphic to $X_P\t\R^{\dim X-\rank P}$ near $(\de_0,0)$. If $X$ is a manifold with corners then $M_{C(X)}^\vee$ has fibre $\N^k$ over $C_k(X)$. We now generalize these ideas to firm $C^\iy$-schemes with corners.

\begin{dfn}
\label{cc6def3}
Let $\bX=(X,\bO_X)$ be a local $C^\iy$-ringed space with corners, so that Definition \ref{cc6def1} defines the corners $C(\bX)=(C(X),\bO_{C(X)})$ in $\LCRScin$, with a morphism $\bs\Pi_\bX:C(\bX)\ra\bX$. We have a sheaf of monoids $\Pi_X^{-1}(\O_X^\rex)$ on $C(X)$, with stalks $(\Pi_X^{-1}(\O^\rex_{X}))_{(x,P)}=\O_{X,x}^\rex$ at $(x,P)$ in $C(X)$, and we have a natural prime ideal $P\subset (\Pi_X^{-1}(\O^\rex_{X}))_{(x,P)}$ in the stalk at each~$(x,P)\in C(X)$.

Define a presheaf of monoids $\cP \check M^\rex_{C(X)}$ on $C(X)$ by, for each open $U\subseteq C(X)$,
\begin{align*}
\cP \check M^\rex_{C(X)}(U)=\Pi_X^{-1}(\O_X^\rex)(U)\big/\bigl[s=1:s\in\Pi_X^{-1}(\O_X^\rex)(U),\; s\vert_{(x,P)}\notin P&\\ \text{for all $(x,P)\in U$}&\bigr].
\end{align*}
For open $V\subseteq U\subseteq C(X)$, the restriction map $\rho_{UV}:\cP \check M^\rex_{C(X)}(U)\!\ra\!\cP \check M^\rex_{C(X)}(V)$ is induced by $\rho_{UV}:\Pi_X^{-1}(\O_X^\rex)(U)\ra\Pi_X^{-1}(\O_X^\rex)(V)$. This gives a well defined presheaf $\cP \check M^\rex_{C(X)}$ valued in $\Mon$. Write $\check M^\rex_{C(X)}$ for its sheafification. 

The surjective projections $\Pi_X^{-1}(\O_X^\rex)(U)\ra \cP \check M^\rex_{C(X)}(U)$ for open $U\subseteq C(X)$ induce a surjective morphism $\Pi_{\check M^\rex_{C(X)}}:\Pi_X^{-1}(\O_X^\rex)\ra \check M^\rex_{C(X)}$. Local sections $s$ of $\Pi_X^{-1}(\O_X^\rex)$ lying in $\O_{X,x}^\rex\sm P$ at each $(x,P)$ have $\pi(s)=1$ in $\check M^\rex_{C(X)}$, and $\check M^\rex_{C(X)}$ is universal with this property. The stalk of $\check M^\rex_{C(X)}$ at $(x,P)$ is 
\e
(\check M^\rex_{C(X)})_{(x,P)}\cong \O_{X,x}^\rex\big/\bigl[s=1:s\in \O_{X,x}^\rex\sm P\bigr].
\label{cc6eq25}
\e

Now suppose that $\bX$ is interior. Then as $\bO_{X,x}$ is interior for $x\in X$, the monoid $\O_{X,x}^\rex$ has no zero divisors, and $\O_{X,x}^\rin=\O_{X,x}^\rex\sm\{0\}$ is also a monoid. Write $\check M^\rin_{C(X)}\subset\check M^\rex_{C(X)}$ for the subsheaf of sections of $\check M^\rex_{C(X)}$ which are nonzero at every $(x,P)\in C(X)$. Then $\check M^\rin_{C(X)}$ is also a sheaf of monoids on $C(X)$, with stalk $(\check M^\rin_{C(X)})_{(x,P)}=(\check M^\rex_{C(X)})_{(x,P)}\sm\{0\}$ at each~$(x,P)\in C(X)$.

Next let $\bs f:\bX\ra\bY$ be a morphism in $\LCRSc$, so that $C(\bs f):C(\bX)\ra C(\bY)$ is a morphism in $\LCRScin$. Consider the diagram:
\e
\begin{gathered}
\!\!\!\xymatrix@C=60pt@R=19pt{ *+[r]{C(f)^{-1}(\O_{C(Y)}^\rex)} \ar[d]^{C(f)^\sh_\rex} & {\begin{subarray}{l} \ts C(f)^{-1}\ci \Pi_Y^{-1}(\O_Y^\rex) = \\ \ts \quad \Pi_X^{-1}\!\ci\! f^{-1}\!\ci\! (\O_Y^\rex)\end{subarray}} \ar@{->>}[l]^(0.47){\raisebox{-13pt}{$\st C(f)^{-1}(\Pi_{Y,\rex}^\sh)$}} \ar@{->>}[r]_(0.47){\raisebox{-10pt}{$\st C(f)^{-1}(\Pi_{\check M^\rex_{C(Y)}})$}} \ar[d]^(0.55){\Pi_X^{-1}(f_\rex^\sh)} & *+[l]{C(f)^{-1}(\check M^\rex_{C(Y)})} \ar@{..>}[d]_{\check M^\rex_{C(f)}} \\ 
*+[r]{\O_{C(X)}^\rex}  & {\Pi_X^{-1}(\O_X^\rex)} \ar@{->>}[l]_{\Pi_{X,\rex}^\sh} \ar@{->>}[r]^{\Pi_{\check M^\rex_{C(X)}}}  & *+[l]{\check M^\rex_{C(X)}.\!} }\!\!\!
\end{gathered}
\label{cc6eq26}
\e
Here $C(f)^{-1}(\Pi_{\check M^\rex_{C(Y)}}):C(f)^{-1}\ci \Pi_Y^{-1}(\O_Y^\rex)\twoheadrightarrow C(f)^{-1}(\check M^\rex_{C(Y)})$ is the quotient sheaf of monoids which is universal for morphisms $C(f)^{-1}\ci \Pi_Y^{-1}(\O_Y^\rex)\ra M$, for $M$ a sheaf of monoids on $C(X)$, such that a local section $s$ of $C(f)^{-1}\ci \Pi_Y^{-1}(\O_Y^\rex)$ is mapped to $1$ in $M$ if $s\vert_{(x,P)}\in\O_{Y,y}^\rex\sm Q$ at each $(x,P)\in C(X)$ with $C(f)(x,P)=(y,Q)$ in $C(Y)$.

Now taking $M=\check M^\rex_{C(X)}$, the morphism $\Pi_{\check M^\rex_{C(X)}}\ci\Pi_X^{-1}(f_\rex^\sh):C(f)^{-1}\ci \Pi_Y^{-1}(\O_Y^\rex)\ra \check M^\rex_{C(X)}$ has the property required. This is because $C(\bs f)$ is interior, so $C(f)^\sh_\rex$ maps $\O_{C(Y),(y,Q)}^\rex\sm\{0\}\ra \O_{C(X),(x,P)}^\rex\sm\{0\}$, and thus $\Pi_X^{-1}(f_\rex^\sh)$ maps $\O_{Y,y}^\rex\sm Q\ra \O_{X,x}^\rex\sm P$, and $\Pi_{\check M^\rex_{C(X)}}$ maps $\O_{X,x}^\rex\sm P\ra 1$. Hence there is a unique morphism $\check M^\rex_{C(f)}:C(f)^{-1}\ci \Pi_Y^{-1}(\O_Y^\rex)\ra\check M^\rex_{C(X)}$ of sheaves of monoids on $C(X)$ making \eq{cc6eq26} commute.

If $\bs g:\bY\ra\bZ$ is another morphism in $\LCRSc$, then composing \eq{cc6eq26} for $\bs f,\bs g$ vertically we see that $\check M^\rex_{C(g\ci f)}=\check M^\rex_{C(f)}\ci C(f)^{-1}(\check M^\rex_{C(g)})$. So the $\check M^\rex_{C(X)},\ab\check M^\rex_{C(f)}$ are contravariantly functorial.
\end{dfn}

\begin{prop}
\label{cc6prop4}
Let\/ $\bX$ be a firm $C^\iy$-scheme with corners. Then $\check M^\rex_{C(X)}$ in Definition\/ {\rm\ref{cc6def3}} is a \begin{bfseries}locally constant\end{bfseries} sheaf of monoids on $C(X),$ and the stalks\/ $(\check M^\rex_{C(X)})_{(x,P)}$ for\/ $(x,P)\in C(X)$ are finitely generated sharp monoids, with zero elements not equal to the identity. If\/ $\bX$ is interior, the same holds for\/~$\check M^\rin_{C(X)},$ but without with zero elements.
\end{prop}

\begin{proof} As the proposition is local in $\bX$, it is sufficient to prove it in the case that $\bX=\Specc\bfC$ for $\bfC$ a firm $C^\iy$-ring with corners. Then Theorem \ref{cc6thm8} defines an isomorphism $\coprod_{P\in\Prc_\bfC}\bX{}^P\ra C(\bX)$. 

We will prove that for $P\in\Prc_\bfC$, the restriction of $\check M^\rex_{C(X)}$ to the image $C(\bX)^P$ of $\bX{}^P$ under \eq{cc6eq21} is the constant sheaf with fibre the monoid
\e
\fC_\rex\big/\bigl[c'=1:c'\in \fC_\rex\sm P\bigr]\cong \fC_\rex^\sh/\bigl[c''=1:c''\in \fC^\sh_\rex\sm P^\sh\bigr].
\label{cc6eq27}
\e
Here $\fC_\rex^\sh=\fC_\rex/\fC_\rex^\t$ and $P^\sh=P/\fC_\rex^\t$, and the two sides of \eq{cc6eq27} are isomorphic as $\fC_\rex^\t\subseteq\fC_\rex\sm P$, so quotienting by $\fC_\rex^\t$ commutes with setting all elements of $\fC_\rex\sm P$ or $\fC^\sh_\rex\sm P^\sh$ equal to 1. Equation \eq{cc6eq27} is finitely generated as $\fC^\sh_\rex$ is, since $\bfC$ is firm. Also \eq{cc6eq27} is the disjoint union of $\{1\}$, and the image of $P$, which is an ideal in \eq{cc6eq27}, and so contains no units. Hence \eq{cc6eq27} is sharp. The image of $0_{\fC_\rex}$ is a zero element not equal to the identity.

Suppose $P\in\Prc_\bfC$, and $x'\in\bX{}^P$ maps to $(x,P_x)\in C(\bX)$ under \eq{cc6eq21}, where $P_x\subset\O_{X,x}^\rex=\fC_{x,\rex}$ is a prime ideal, with $P_x=\pi_{x,\rex}(P)\subsetneq\fC_{x,\rex}$ and $P=\pi_{x,\rex}^{-1}(P_x)\subsetneq\fC_\rex$ by Lemma \ref{cc6lem2}. Consider the diagram of monoids:
\e
\begin{gathered}
\xymatrix@C=160pt@R=15pt{
*+[r]{\fC_\rex} \ar@{->>}[r]_(0.15){\text{project}} \ar@{->>}[d]^{\pi_{x,\rex}} & *+[l]{\fC_\rex\big/\bigl[c'=1:c'\in \fC_\rex\sm P\bigr]=\eq{cc6eq27}} \ar@{..>>}[d]_{(\pi_{x,\rex})_*} \\
*+[r]{\fC_{x,\rex}} \ar@{->>}[r]^(0.15){\text{project}} & *+[l]{\fC_{x,\rex}\big/\bigl[c''=1:c''\in \fC_{x,\rex}\sm P_x\bigr]=(\check M^\rex_{C(X)})_{(x,P_x)}.\!} }
\end{gathered}
\label{cc6eq28}
\e
As $P=\pi_{x,\rex}^{-1}(P_x)$, if $c'\in \fC_\rex\sm P$ then $c''=\pi_{x,\rex}(c')\in \fC_{x,\rex}\sm P_x$, so there is a unique, surjective morphism $(\pi_{x,\rex})_*$ making \eq{cc6eq28} commute. 

We claim that $(\pi_{x,\rex})_*$ is an isomorphism. To see this, note that as in \S\ref{cc46}, $\fC_{x,\rex}$ may be obtained from $\fC_\rex$ by inverting all elements $c'\in\fC_\rex$ with $x\ci\Phi_i(c')\ne 0$, and setting $c''=1$ for all $c''\in\Psi_{\exp}(I)$, where $I\subset\fC$ is the ideal vanishing near $x$. Here $c'\in\fC_\rex$ with $x\ci\Phi_i(c')\ne 0$ and $c''\in\Psi_{\exp}(I)$ both lie in the set $\fC_\rex\sm P$ of elements set to 1 in the right hand of \eq{cc6eq28}. Also first inverting $c'$ and then setting $c'=1$ is equivalent to just setting~$c'=1$. 

Thus $(\pi_{x,\rex})_*$ is an isomorphism. So the fibre of $(\check M^\rex_{C(X)})_{(x,P_x)}$ at each $(x,P_x)$ in the image of $\bX{}^P$ is naturally isomorphic to \eq{cc6eq27}, which is independent of $(x,P_x)$. It easily follows that $\check M^\rex_{C(X)}$ is constant with fibre \eq{cc6eq27} on the image $\bX{}^P$. The analogue for $\check M^\rin_{C(X)}$ when $\bX$ is interior is immediate.
\end{proof}

\begin{rem}
\label{cc6rem1}
If $\bX$ lies in $\LCRSc$ or $\CSchc$, but not in $\CSchcfi$, then $\check M^\rex_{C(X)}$ in Definition \ref{cc6def3} may not be locally constant.

For example, suppose $\bfC\in\CRingsc$ is not firm, and $\fC_\rex$ has infinitely many prime ideals $P$, and $\bX=\Specc\bfC$. Then \eq{cc6eq19} decomposes $C(X)$ into infinitely many disjoint closed subsets $C(X)^P$, which need not be open. The proof of Proposition \ref{cc6prop4} shows that $\check M^\rex_{C(X)}\vert_{C(X)^P}$ is constant on $C(X)^P$ with value \eq{cc6eq27} for each $P$. But if the $C(X)^P$ are not open, as happens in Example \ref{cc6ex1}, then in general $\check M^\rex_{C(X)}$ is not locally constant.
\end{rem}

Proposition \ref{cc6prop4} implies that if $\bX$ is a firm $C^\iy$-scheme with corners then
\begin{equation*}
C(\bX)=\coprod_{\begin{subarray}{l}\text{iso. classes $[M]$ of finitely generated}\\ \text{sharp monoids $M$ with zero elements}\end{subarray}}C_M(\bX),
\end{equation*}
where $C_M(\bX)\subseteq C(\bX)$ is the open and closed $C^\iy$-subscheme of points $(x,P_x)$ in $C(\bX)$ with~$(\check M^\rex_{C(X)})_{(x,P_x)}\cong M$.

\begin{ex}
\label{cc6ex3}
In Theorem \ref{cc6thm8}, let $\bfC=\bs C^\iy(\R^n_k)$, so that $\bX=\Specc\bfC=F_\Manc^\CSchc(\R^n_k)$. Take $P$ to be the prime ideal $P=\an{x_1,\ldots,x_k}$ in $\fC_\rex$, so that \eq{cc6eq21} maps $\bX{}^P$ to the corner stratum $\{x_1=\cdots=x_k=0\}$ of $\R^n_k$ in $C(\bX)$. The proof of Proposition \ref{cc6prop4} shows that $\check M^\rex_{C(X)}$ on $\bX{}^P$ is the constant sheaf with fibre \eq{cc6eq27}. In this case $\bfC$ is interior with
\e
\fC_\rin=\bigl\{x_1^{a_1}\cdots x_k^{a_k}\exp f:a_i\in\N,\; f\in C^\iy(\R^n_k)\bigr\},\;\> \fC_\rex=\fC_\rin\amalg\{0\}.
\label{cc6eq29}
\e
Since $\fC_\rex\sm P=\fC_\rex^\t=\bigl\{\exp f:f\in C^\iy(\R^n_k)\bigr\}$, equation \eq{cc6eq27} is $\fC_\rex/\fC_\rex^\t=\fC^\sh_\rex\cong\N^k\amalg\{0\}$. So $\check M^\rex_{C(X)}$ has fibre $\N^k\amalg\{0\}$, and $\check M^\rin_{C(X)}$ has fibre $\N^k$, on~$\bX{}^P$.

Now let $Y$ be a manifold with corners, and $\bY=F_\Manc^\CSchc(Y)$. Then Proposition \ref{cc6prop2} says that $C(\bY)\cong F_\cManc^\CSchc(C(Y))$. Let $(y,\ga)\in C_k(Y)^\ci$. Then $y\in S^k(Y)$, and $Y$ near $y$ is locally diffeomorphic to $\R^n_k$ near 0, such that the local $k$-corner component $\ga$ of $Y$ at $y$ is identified with $\{x_1=\cdots=x_k=0\}$ in $\R^n_k$. Hence $C(\bY)$ and $M^\rex_{C(Y)},M^\rin_{C(Y)}$ near $(y,\ga)$ are locally isomorphic to $C(\bX),\check M^\rex_{C(X)},\check M^\rin_{C(X)}$ near $\bigl(0,\{x_1=\cdots=x_k=0\}\bigr)$. Therefore $\check M^\rex_{C(Y)},\check M^\rin_{C(Y)}$ have fibres $\N^k\amalg\{0\},\N^k$ on $C_k(Y)^\ci$, and hence on $C_k(Y)$, as $C_k(Y)^\ci$ is dense in $C_k(Y)$ and $\check M^\rex_{C(Y)},\check M^\rin_{C(Y)}$ are locally constant by Proposition~\ref{cc6prop4}.

As in Remark \ref{cc3rem5}(d), in \cite[\S 3.6]{Joyc6} we define the {\it comonoid bundle\/} $M_{C(Y)}^\vee\ra C(Y)$, which has fibre $\N^k$ over $C_k(Y)$. It is easy to check that the sheaf of continuous sections of $M_{C(Y)}^\vee$ is canonically isomorphic to~$\check M^\rin_{C(Y)}$.
\end{ex}

\begin{ex}
\label{cc6ex4}
We can generalize Example \ref{cc6ex3} to manifolds with g-corners. Let $Q$ be a toric monoid and $l\ge 0$, and set $\bfC=\bs C^\iy(X_Q\t\R^l)$, so that $\bX=\Specc\bfC=F_\Mangc^\CSchc(X_Q\t\R^l)$. Take $P$ to be the prime ideal of functions in $\fC_\rex$ vanishing on $\{\de_0\}\t\R^l$, so that \eq{cc6eq21} maps $\bX{}^P$ to the corner stratum $\{\de_0\}\t\R^l$ in $C(\bX)$. The analogue of \eq{cc6eq29} is
\begin{equation*}
\fC_\rin=\bigl\{\la_q\exp f:q\in Q,\; f\in C^\iy(X_Q\t\R^l)\bigr\},\;\> \fC_\rex=\fC_\rin\amalg\{0\},	
\end{equation*}
for $\la_q:X_Q\ra[0,\iy)$ as in Definition \ref{cc3def5}, so $\fC_\rin^\sh\cong Q$ and $\fC_\rex^\sh\cong Q\amalg\{0\}$, and $\check M^\rex_{C(X)},\check M^\rin_{C(X)}$ have fibres $Q\amalg\{0\},Q$ on~$\bX{}^P$.

Now let $Y$ be a manifold with g-corners, and $\bY=F_\Manc^\CSchc(Y)$, so that $C(\bY)\cong F_\cManc^\CSchc(C(Y))$ by Proposition \ref{cc6prop2}. Let $(y,\ga)\in C(Y)^\ci$. Then $Y$ near $y$ is locally diffeomorphic to $X_Q\t\R^l$ near $(\de_0,0)$ for some toric monoid $Q$ and $l\ge 0$, such that the local corner component $\ga$ of $Y$ at $y$ is identified with $\{\de_0\}\t\R^l$ in $X_Q\t\R^l$. Hence $\check M^\rex_{C(Y)},\check M^\rin_{C(Y)}$ have fibres $Q\amalg\{0\},Q$ at~$(y,\ga)$.

As in Remark \ref{cc3rem5}(d), the comonoid bundle $M_{C(Y)}^\vee\ra C(Y)$ of \cite[\S 3.6]{Joyc6} also has fibre $Q$ at $(y,\ga)$. It is easy to check that the sheaf of continuous sections of $M_{C(Y)}^\vee$ is canonically isomorphic to~$\check M^\rin_{C(Y)}$.
\end{ex}

\subsection{\texorpdfstring{The boundary $\pd\bX$ and $k$-corners $C_k(\bX)$}{The boundary ∂X and k-corners Cᵏ(X)}}
\label{cc65}

If $X$ is a manifold with (g-)corners then the corners $C(X)$ from \S\ref{cc34} has a decomposition $C(X)=\coprod_{k=0}^{\dim X}C_k(X)$ with $C_0(X)\cong X$ and $C_1(X)=\pd X$. We generalize this to firm (interior) $C^\iy$-schemes with corners~$\bX$.

\begin{dfn}
\label{cc6def4}
The {\it dimension\/} $\dim M$ in $\N\amalg\{\iy\}$ of a monoid $M$ is the maximum length $d$ (or $\iy$ if there is no maximum) of a chain of prime ideals $\emptyset\subsetneq P_1\subsetneq P_2\subsetneq \cdots\subsetneq P_d\subsetneq M$. If $M$ is finitely generated then $\dim M<\iy$. If $M$ is toric then~$\dim M=\dim_\R (M\ot_\N\R)$.

Let $\bX$ be a firm $C^\iy$-scheme with corners. For each $k=0,1,\ldots,$ define the $k$-{\it corners\/} $C_k(\bX)\subseteq C(\bX)$ to be the $C^\iy$-subscheme of $(x,P)\in C(\bX)$ with $\dim(\check M^\rex_{C(X)})_{(x,P)}=k+1$. Here $(x,P)\mapsto\dim(\check M^\rex_{C(X)})_{(x,P)}$ is a locally constant function $C(X)\ra\N$ by Proposition \ref{cc6prop4}, so $C_k(\bX)$ is open and closed in $C(\bX)$. Also $\dim(\check M^\rex_{C(X)})_{(x,P)}\ge 1$ as $(\check M^\rex_{C(X)})_{(x,P)}$ has at least one prime ideal $(\check M^\rex_{C(X)})_{(x,P)}\sm\{1\}$, since $(\check M^\rex_{C(X)})_{(x,P)}$ is sharp and $(\check M^\rex_{C(X)})_{(x,P)}\sm\{1\}\ne\es$ as $0\ne 1$. Hence $C(\bX)=\coprod_{k\ge 0}C_k(\bX)$. 

We define the {\it boundary\/} $\pd\bX$ to be $\pd\bX=C_1(\bX)\subset C(\bX)$.

If $\bX$ is interior then $\dim(\check M^\rex_{C(X)})_{(x,P)}=\dim(\check M^\rin_{C(X)})_{(x,P)}+1$, as chains of ideals in $(\check M^\rin_{C(X)})_{(x,P)}$ lift to chains in $(\check M^\rex_{C(X)})_{(x,P)}$, plus $\{0\}$. Hence $C_k(\bX)$ is the $C^\iy$-subscheme of $(x,P)\in C(\bX)$ with $\dim(\check M^\rin_{C(X)})_{(x,P)}=k$.

For $\bX$ interior, if $(\check M^\rin_{C(X)})_{(x,P)}\ne\{1\}$ then $(\check M^\rin_{C(X)})_{(x,P)}\sm\{1\}$ is prime, so $\dim(\check M^\rin_{C(X)})_{(x,P)}>0$. But $(\check M^\rin_{C(X)})_{(x,P)}=\{1\}$ if and only if $P=\{0\}$, as $0\ne p\in P$ descends to $[p]\ne 1$ in $(\check M^\rin_{C(X)})_{(x,P)}$ since $\O_{X,x}^\rex$ has no zero divisors. Hence $C_0(\bX)=\{(x,\{0\}):x\in\bX\}\subseteq C(\bX)$. It is now easy to see that $\bs\io_\bX:\bX\ra C(\bX)$ in Definition \ref{cc6def1} is an isomorphism~$\bs\io_\bX:\bX\ra C_0(\bX)$.
\end{dfn}

\begin{ex}
\label{cc6ex5}
Let $X$ be a manifold with (g-)corners and $\bX=F_\Mangc^\CSchc(X)$, which is a firm interior $C^\iy$-scheme with corners. As in Proposition \ref{cc6prop2} there is a natural isomorphism $C(\bX)\cong F_\cMangc^\CSchc(C(X))$. It is easy to check that this identifies $C_k(\bX)\cong F_\Mangc^\CSchc(C_k(X))$ for each $k=0,\ldots,\dim X$.
\end{ex}

\begin{ex}
\label{cc6ex6}
If $\bX$ lies in $\CSchcfi$ but not $\CSchcfiin$ then $C_0(\bX)$ can contain $(x,P)$ with $P\ne\{0\}$. For instance, if $\bX=\Specc\bigl(\R[x,y]/[xy=0]\bigr)$ then $\bigl((0,0),\an{x}\bigr)$ and $\bigl((0,0),\an{y}\bigr)$ lie in $C_0(\bX)$. Also $\{0\}$ is not prime in $\O_{X,(0,0)}^\rex$, so $\bigl((0,0),\{0\}\bigr)\notin C(\bX)$. Clearly $\bX\not\cong C_0(\bX)$, as only one of the two is interior.
\end{ex}

\subsection{Fibre products and corner functors}
\label{cc66}

The authors' favourite categories of `nice' $C^\iy$-schemes with corners generalizing $\Mangcin,\Mangc$ are $\CSchcto,\CSchctoex$. Theorem \ref{cc5thm10} says that fibre products and finite limits exist in $\CSchcto$. Now we address the question of when fibre products exist in $\CSchctoex$, or more generally in $\CSchcfiinex$, or in categories like those in \eq{cc5eq10} in which the objects are interior, but the morphisms need not be interior. We use corner functors as a tool to do this.

\begin{thm}
\label{cc6thm9}
Let\/ $\bs g:\bX\ra\bZ,$ $\bs h:\bY\ra\bZ$ be morphisms in $\CSchctoex$. Consider the question of whether a fibre product\/ $\bW=\bX\t_{\bs g,\bZ,\bs h}\bY$ exists in\/ $\CSchctoex,$ with projections $\bs e:\bW\ra\bX,$ $\bs f:\bW\ra\bY$.
\begin{itemize}
\setlength{\itemsep}{0pt}
\setlength{\parsep}{0pt}
\item[{\bf(a)}] If\/ $\bW$ exists then the following is a bijection of sets:
\e
(e,f):W\longra\bigl\{(x,y)\in X\t Y:g(x)=h(y)\bigr\}.
\label{cc6eq30}
\e
Note that this need not hold for fibre products in\/~$\CSchcto$.
\item[{\bf(b)}] If\/ $\bW$ exists then\/ $\bW$ is isomorphic to an open and closed\/ $C^\iy$-subscheme $\bW'$ of\/ $C(\bX)\t_{C(\bs g),C(\bZ),C(\bs h)}C(\bY),$ where the fibre product is taken in $\CSchcto,$ using\/ $C(\bs g),C(\bs h)$ interior, and exists by Theorem\/~{\rm\ref{cc5thm10}}.
\item[{\bf(c)}] Writing\/ $C(X)\t_{C(Z)}C(Y)$ for the topological space of\/ $C(\bX)\t_{C(\bZ)}C(\bY),$ we have a natural continuous map
\e
\begin{split}
&\xymatrix@C=36pt{ C(X)\t_{C(Z)}C(Y) \ar[r] & \bigl\{(x,y)\in X\t Y:g(x)=h(y)\bigr\}, } \\[-5pt]
&\xymatrix@C=70pt{ w' \ar@{|->}[r]^(0.38){(\Pi_{C(X)},\Pi_{C(Y)})} & \bigl((x,P_x),(y,P_y)\bigr) \ar@{|->}[r]^(0.6){\Pi_X\t\Pi_Y} &  (x,y). }
\end{split}
\label{cc6eq31}
\e
If there does not exist an open and closed subset\/ $W'\subset C(X)\t_{C(Z)}C(Y)$ such that\/ \eq{cc6eq31} restricts to a bijection on $W',$ then no fibre product\/ $\bW=\bX\t_{\bs g,\bZ,\bs h}\bY$exists in $\CSchctoex$.
\item[{\bf(d)}] Suppose $\bs g,\bs h$ are interior, and a fibre product\/ $\bW=\bX\t_{\bs g,\bZ,\bs h}\bY,$ $\bs e,\bs f$ exists in\/ $\CSchcto$. Then $\bW$ is also a fibre product in $\CSchctoex$ if and only if the following diagram is Cartesian in {\rm$\CSchcto$:}
\e
\begin{gathered}
\xymatrix@C=100pt@R=15pt{ *+[r]{C(\bW)} \ar[d]^{C(\bs e)} \ar[r]_{C(\bs f)} & *+[l]{C(\bY)} \ar[d]_{C(\bs h)} \\ 
*+[r]{C(\bX)} \ar[r]^{C(\bs g)} & *+[l]{C(\bZ),} }	
\end{gathered}
\label{cc6eq32}
\e
that is, if\/ $C(\bW)\cong C(\bX)\t_{C(\bs g),C(\bZ),C(\bs h)}C(\bY)$ in $\CSchcto$.
\end{itemize}

The analogues hold for\/ $\CSchcfiinex,\CSchcfiZex,$ and\/ $\CSchcfitfex,$ except that\/ \eq{cc6eq30} is a bijection for fibre products in\/~$\CSchcfiin$.
\end{thm}

\begin{proof} For (a), suppose $\bW$ exists, and apply the universal property of fibre products for morphisms from the point $\bs{*}$ into $\bW,\bX,\bY,\bZ$. There is a 1-1 correspondence between morphisms $w:\bs *\ra\bW$ in $\CSchctoex$ and points $w'=w(*)\in\bW$. Using this, the universal property gives the bijection~\eq{cc6eq30}.

This argument does not work for fibre products in $\CSchcto$, as (necessarily interior) morphisms $w:\bs *\ra\bW$ need not correspond to points of $\bW$. For example, morphisms $w:\bs *\ra\bs{[0,\iy)}$ in $\CSchcto$ correspond to points of $(0,\iy)$ not $[0,\iy)$. Equation \eq{cc6eq30} is not a bijection for the fibre products in $\CSchcto$ in Examples \ref{cc6ex8}--\ref{cc6ex9}. This also holds for $\CSchcfiZex$ and~$\CSchcfitfex$.

The inclusion $\inc:\CSchcfiin\hookra\CSchcfiinex$ preserves fibre products by Theorems \ref{cc5thm9} and \ref{cc5thm10}. Thus a fibre product in $\CSchcfiin$ is also a fibre product in $\CSchcfiinex$, and \eq{cc6eq30} is a bijection. 

For (b), suppose $\bW$ exists. As $C:\CSchctoex\ra\CSchcto$ preserves limits by Theorem \ref{cc6thm7}, we have $C(\bW)\cong C(\bX)\t_{C(\bs g),C(\bZ),C(\bs h)}C(\bY)$. But $\bs\io_\bW:\bW\ra C(\bW)$ is an isomorphism with an open and closed $C^\iy$-subscheme $C_0(\bW)\subseteq C(\bW)$ by Definition \ref{cc6def4}. Part (c) follows from~(a),(b).

For (d), let $\bs g,\bs h$ be interior and $\bW$ be a fibre product in $\CSchcto$. If $\bW$ is a fibre product in $\CSchctoex$ then \eq{cc6eq32} is Cartesian as $C:\CSchctoex\ra\CSchcto$ preserves limits by Theorem \ref{cc6thm7}, proving the `only if' part. 

Suppose \eq{cc6eq32} is Cartesian. Let $\bs c:\bV\ra\bX$ and $\bs d:\bV\ra\bY$ be morphisms in $\CSchctoex$ with $\bs g\ci\bs c=\bs h\ci\bs d$. Consider the diagram:
\begin{equation*}
\xymatrix@!0@C=70pt@R=25pt{ 
\bV \ar@/_2.2pc/[ddddr]_(0.6){\bs c} \ar@/_1pc/[ddrrr]^(0.3){\bs d} \ar@/_2pc/[dddrr]_(0.4){C(\bs c)\ci\bs\io_\bV\!\!\!}_(0.25){\rin\!\!\!} \ar@/^.8pc/[drrrr]^(0.7){C(\bs d)\ci\bs\io_\bV}^(0.5)\rin \ar@{..>}@/^.2pc/[drr]^(0.8){\bs b}^(0.6)\rin \ar@{..>}@/_.2pc/[ddr]^(0.8){\bs a} \\ 
&& *+[r]{C(\bW)} \ar[dl]_(0.3){\bs\Pi_\bW} \ar[dd]^(0.28){C(\bs e)}^(0.65)\rin (0.7)\ \ar[rr]_{C(\bs f)}^\rin && *+[l]{C(\bY)} \ar[dl]^{\bs\Pi_\bY} \ar[dd]_{C(\bs h)}^\rin \\ 
& *+[r]{\bW} \ar[dd]^{\bs e} \ar[rr]_(0.25){\bs f} && *+[l]{\bY} \ar[dd]_(0.75){\bs h} \\
&& *+[r]{C(\bX)} \ar[dl]_{\bs\Pi_\bX} \ar[rr]^(0.7){C(\bs g)}^(0.3)\rin && *+[l]{C(\bZ)} \ar[dl]^{\bs\Pi_\bZ} \\
& *+[r]{\bX} \ar[rr]^{\bs g} && *+[l]{\bZ.\!\!} }	
\end{equation*}
Here morphisms labelled `$\rin$' are interior.

The morphisms $C(\bs c)\ci\bs\io_\bV:\bV\ra C(\bX)$, $C(\bs d)\ci\bs\io_\bV:\bV\ra C(\bY)$ are interior with $C(\bs g)\ci C(\bs c)\ci\bs\io_\bV=C(\bs h)\ci C(\bs d)\ci\bs\io_\bV$ as $\bs g\ci\bs c=\bs h\ci\bs d$. So the Cartesian property of \eq{cc6eq32} gives unique interior $\bs b:\bV\ra C(\bW)$ with $C(\bs e)\ci\bs b=C(\bs c)\ci\bs\io_\bV$ and $C(\bs f)\ci\bs b=C(\bs d)\ci\bs\io_\bV$. Set $\bs a=\bs\Pi_\bW\ci\bs b:\bV\ra\bW$. Then
\begin{equation*}
\bs e\ci\bs a=\bs e\ci\bs\Pi_\bW\ci\bs b=\bs\Pi_\bX\ci C(\bs e)\ci\bs b=\bs\Pi_\bX\ci C(\bs c)\ci\bs\io_\bV=\bs c\ci\bs\Pi_\bV\ci\bs\io_\bV=\bs c,
\end{equation*}
and similarly $\bs f\ci\bs a=\bs d$. We claim that $\bs a:\bV\ra\bW$ is unique with $\bs e\ci\bs a=\bs c$ and $\bs f\ci\bs a=\bs d$, which shows that $\bW=\bX\t_{\bs g,\bZ,\bs h}\bY$ in $\CSchctoex$. To see this, note that as $C$ is right adjoint to $\inc:\CSchcto\hookra\CSchctoex$, there is a 1-1 correspondence between morphisms $\bs a:\bV\ra\bW$ in $\CSchctoex$ and morphisms $\bs b:\bV\ra C(\bW)$ in $\CSchcto$ with $\bs a=\bs\Pi_\bW\ci\bs b$. But $\bs e\ci\bs a=\bs c$, $\bs f\ci\bs a=\bs d$ imply that $C(\bs e)\ci\bs b=C(\bs c)\ci\bs\io_\bV$, $C(\bs f)\ci\bs b=C(\bs d)\ci\bs\io_\bV$, and $\bs b$ is unique under these conditions as above. This proves the `if' part.
\end{proof}

Here are three examples:

\begin{ex}
\label{cc6ex7}
Define manifolds with corners $X=[0,\iy)^2$, $Y=*$ a point, and $Z=[0,\iy)$, and exterior maps $g:X\ra Z$, $h:Y\ra Z$ by $g(x_1,x_2)=x_1x_2$ and $h(*)=0$. Let $\bX,\bY,\bZ,\bs g,\bs h$ be the images of $X,\ldots,h$ under $F_\Manc^\CSchc$. As $\bX,\bY,\bZ$ are firm, a fibre product $\bs{\ti W}=\bX\t_{\bs g,\bZ,\bs h}\bY$ exists in $\CSchcfi$, with $\bs{\ti W}=\Specc\bfC$, where $\bfC=\R[x_1,x_2]/[x_1x_2=0]$. Here $\bfC,\bs{\ti W}$ are not interior, as the relation $x_1x_2=0$ is not of interior type. In Theorem \ref{cc6thm9} we have
\begin{align*}
C(X)\t_{C(Z)}C(Y)&\cong ([0,\iy)\t\{0\})\amalg (\{0\}\t[0,\iy)) \amalg \{(0,0)\},\\
\{(x,y)\in X\t Y:g(x)=h(y)\bigr\}&=\bigl\{(x_1,x_2)\in[0,\iy)^2:x_1x_2=0\bigr\},
\end{align*}
where the map \eq{cc6eq31} takes $(x_1,x_2)\mapsto(x_1,x_2)$. Considering a neighbourhood of $(0,0)$ we see that no such open and closed subset $W'$ exists in Theorem \ref{cc6thm9}(c). Hence no fibre product $\bW=\bX\t_\bZ\bY$ exists in $\CSchcfiinex$, or $\CSchctoex$, or in any category like those in \eq{cc5eq10} in which the objects are interior, but the morphisms need not be interior.
\end{ex}

The next two examples are `b-transverse', but not `c-transverse', in the sense of Definition \ref{cc6def5} below.

\begin{ex}
\label{cc6ex8}
Define manifolds with corners $X=Y=[0,\iy)$, $Z=[0,\iy)^2$, and interior maps $g:X\ra Z$, $h:Y\ra Z$ by $g(x)=(x,x)$ and $h(y)=(y,y^2)$. Let $\bX,\bY,\bZ,\bs g,\bs h$ be the images of $X,\ldots,h$ under $F_\Mancin^\CSchcin$. Then $\bs g:\bX\ra\bZ$, $\bs h:\bY\ra\bZ$ are morphisms in $\CSchcto$, so a fibre product $\bs{\ti W}=\bX\t_{\bs g,\bZ,\bs h}\bY$ exists in $\CSchcto$ by Theorem \ref{cc5thm10}. 

Using the proof of Theorem \ref{cc5thm10} we can show that $\bs{\ti W}=\Speccin\bfC$, where
\e
\bfC=\Pi_\rin^\sa\bigl(\R[x,y]/[x=y,\; x=y^2]\bigr)\cong\R[x,y]/[x=y=1]\cong\bs C^\iy(*),
\label{cc6eq33}
\e
so $\bs{\ti W}$ is a single ordinary point, which maps to $x=1$ in $\bX$ and $y=1$ in $\bY$. This is because on applying $\Pi_\rin^\Z$ in Theorem \ref{cc4thm5}, $x=y$, $x=y^2$ force~$x=y=1$.

In contrast, the fibre product $\bs{\hat W}=\bX\t_\bZ\bY$ in $\CSchcfiin$ is $\bs{\hat W}=\Speccin\bigl(\R[x,y]/[x=y$, $x=y^2]\bigr)$, which is two points, an ordinary point at $x=y=1$, and another at $x=y=0$ with a non-toric corner structure. Theorems \ref{cc5thm9}--\ref{cc5thm10} imply that $\bs{\hat W}$ is also the fibre product in~$\CSchcfi$.

For the fibre product $\bW=\bX\t_\bZ\bY$ in $\CSchctoex$, we have
\begin{equation*}
C(\bX)\t_{C(\bZ)}C(\bY) \cong\bs{\{(1,1)\}}\amalg\bs{\{(0,0)\}}\qquad\text{in $\CSchcto$,}
\end{equation*}
which is two ordinary points, $(1,1)$ from $C_0(\bX)\t_{C_0(\bZ)}C_0(\bY)\cong\bs{\ti W}$, and $(0,0)$ from $C_1(\bX)\t_{C_2(\bZ)}C_1(\bY)\cong \bs*\t_{\bs *}\bs *$. (If we had formed the fibre product in $\CSchcfiin$ there would have been a third, non-toric point over $(0,0)$ in $C_0(\bX)\t_{C_0(\bZ)}C_0(\bY)\cong\bs{\hat W}$.) Also
\begin{equation*}
\{(x,y)\in X\t Y:g(x)=h(y)\bigr\}=\bigl\{(1,1),(0,0)\bigr\}.
\end{equation*}
Thus $W'=C(X)\t_{C(Z)}C(Y)$ satisfies the condition in Theorem \ref{cc6thm9}(c), and in fact $\bW=C(\bX)\t_{C(\bZ)}C(\bY) \cong\bs{\{(1,1)\}}\amalg\bs{\{(0,0)\}}$ is a fibre product in $\CSchctoex$. Hence in this case fibre products $\bX\t_\bZ\bY$ exist in $\CSchcto$ and $\CSchctoex$, but are different, so $\inc:\CSchcto\hookra\CSchctoex$ does not preserve limits, in contrast to the last part of Theorem \ref{cc5thm9}. The same holds for $\inc:\CSchcfiZ\hookra\CSchcfiZex$ and $\inc:\CSchcfitf\hookra\CSchcfitfex$, as the fibre products $\bs{\ti W},\bW$ are the same in these categories.
\end{ex}

\begin{ex}
\label{cc6ex9}
Define manifolds with corners $X=[0,\iy)\t\R$, $Y=[0,\iy)$ and $Z=[0,\iy)^2$, and interior maps $g:X\ra Z$, $h:Y\ra Z$ by $g(x_1,x_2)=(x_1,x_1e^{x_2})$ and $h(y)=(y,y)$. Let $\bX,\bY,\bZ,\bs g,\bs h$ be the images of $X,\ldots,h$ under $F_\Manc^\CSchc$. Then $\bs g:\bX\ra\bZ$, $\bs h:\bY\ra\bZ$ are morphisms in $\CSchcto$, so a fibre product $\bs{\ti W}=\bX\t_{\bs g,\bZ,\bs h}\bY$ exists in $\CSchcto$ by Theorem \ref{cc5thm10}. As for \eq{cc6eq33} we find that $\bs{\ti W}=\Speccin\bfC$, where
\begin{align*}
\bfC&=\Pi_\rin^\sa\bigl(\R[x_1,x_2,y]/[x_1=y,\; x_1e^{x_2}=y]\bigr)\cong\R[x_1,x_2,y]/[x_1=y,\; x_2=0]
\nonumber\\
&\cong\R[y]\cong\bs C^\iy([0,\iy)),
\end{align*}
so $\bs{\ti W}\cong\bs{[0,\iy)}$. Similarly, we have a fibre product in $\CSchcto$
\begin{equation*}
C(\bX)\t_{C(\bZ)}C(\bY) \cong\bs{\bigl\{(y,0,y):y\in[0,\iy)\bigr\}}\amalg\bs{\bigl\{(0,x_2,0):x_2\in\R\bigr\}},
\end{equation*}
 where the first component is $C_0(\bX)\t_{C_0(\bZ)}C_0(\bY)\cong\bs{\ti W}$, and the second is $C_1(\bX)\t_{C_2(\bZ)}C_1(\bY)$. Also
\begin{align*}
\{(x,y)\in X\t Y:g(x)=h(y)\bigr\}=\bigl\{&(x_1,x_2,y)\in[0,\iy)\t\R\t[0,\iy):\\
&\text{either $x_1=y$, $x_2=0$ or $x_1=y=0$}\bigr\}.
\end{align*}
Considering a neighbourhood of $(0,0,0)$ we see from Theorem \ref{cc6thm9}(c) that no fibre product $\bW=\bX\t_\bZ\bY$ exists in $\CSchctoex$. Thus in this case a fibre product $\bX\t_\bZ\bY$ exists in $\CSchcto$, but not in~$\CSchctoex$.
 \end{ex}

\subsection{\texorpdfstring{B- and c-transverse fibre products in $\Mangcin,\CSchcto$}{B- and c-transverse fibre products in Manᵍᶜⁱⁿ, C∞Schᶜᵗᵒ}}
\label{cc67}

\subsubsection{Results on b- and c-transverse fibre products from \cite{Joyc6}}
\label{cc671}

In \cite[Def.~4.24]{Joyc6} we define transversality for manifolds with g-corners:

\begin{dfn}
\label{cc6def5}
Let $g:X\ra Z$ and $h:Y\ra Z$ be morphisms in $\Mangcin$ or $\cMangcin$. Then:
\begin{itemize}
\setlength{\itemsep}{0pt}
\setlength{\parsep}{0pt}
\item[(a)] We call $g,h$ {\it b-transverse\/} if ${}^bT_xg\op{}^bT_yh:{}^bT_xX\op{}^bT_yY\ra{}^bT_zZ$ is surjective for all $x\in X$ and $y\in Y$ with~$g(x)=h(y)=z\in Z$.
\item[(b)] We call $g,h$ {\it c-transverse\/} if they are b-transverse, and for all $x\in X$ and $y\in Y$ with $g(x)=h(y)=z\in Z$, the linear map ${}^b\ti N_xg\op{}^b\ti N_yh:{}^b\ti N_xX\op{}^b\ti N_yY\ra{}^b\ti N_zZ$ is surjective, and the submonoid
\end{itemize}
\begin{equation*}
\bigl\{(\la,\mu)\!\in\! \ti M_xX\!\t\! \ti M_yY:\text{$\ti M_xg(\la)\!=\!\ti M_yh(\mu)$ in $\ti M_zZ$}\bigr\}\!\subseteq\! \ti M_xX\!\t \!\ti M_yY
\end{equation*}
\begin{itemize}
\setlength{\itemsep}{0pt}
\setlength{\parsep}{0pt}
\item[]is not contained in any proper face $F\subsetneq\ti M_xX\t \ti M_yY$ of $\ti M_xX\t\ti M_yY$.
\end{itemize}
\end{dfn}

Here are the main results on b- and c-transversality~\cite[Th.s 4.26--4.28]{Joyc6}:

\begin{thm} Let\/ $g:X\ra Z$ and\/ $h:Y\ra Z$ be b-transverse (or c-transverse) morphisms in $\Mangcin$. Then $C(g):C(X)\ra C(Z)$ and\/ $C(h):C(Y)\ra C(Z)$ are also b-transverse (or c-transverse) morphisms in\/~$\cMangcin$.
\label{cc6thm10}
\end{thm}

\begin{thm} Let\/ $g:X\ra Z$ and\/ $h:Y\ra Z$ be b-transverse morphisms in $\Mangcin$. Then a fibre product\/ $W=X\t_{g,Z,h}Y$ exists in $\Mangcin,$ with\/ $\dim W=\dim X+\dim Y-\dim Z$. Explicitly, we may write
\begin{equation*}
W^\ci=\bigl\{(x,y)\in X^\ci\t Y^\ci:\text{$g(x)=h(y)$ in $Z^\ci$}\bigr\},
\end{equation*}
and take $W$ to be the closure $\ov{W^\ci}$ of\/ $W^\ci$ in $X\t Y,$ and then $W$ is an embedded submanifold of\/ $X\t Y,$ and\/ $e:W\ra X$ and\/ $f:W\ra Y$ act by $e:(x,y)\mapsto x$ and\/~$f:(x,y)\mapsto y$.
\label{cc6thm11}
\end{thm}

\begin{thm} Suppose $g:X\ra Z$ and\/ $h:Y\ra Z$ are c-transverse morphisms in $\Mangcin$. Then a fibre product\/ $W=X\t_{g,Z,h}Y$ exists in $\Mangc,$ with\/ $\dim W=\dim X+\dim Y-\dim Z$. Explicitly, we may write
\begin{equation*}
W=\bigl\{(x,y)\in X\t Y:\text{$g(x)=h(y)$ in $Z$}\bigr\},
\end{equation*}
and then $W$ is an embedded submanifold of\/ $X\t Y,$ and\/ $e:W\ra X$ and\/ $f:W\ra Y$ act by\/ $e:(x,y)\mapsto x$ and\/ $f:(x,y)\mapsto y$. This $W$ is also a fibre product in $\Mangcin,$ and agrees with that in Theorem\/~{\rm\ref{cc6thm11}}.

Furthermore, the following is Cartesian in both\/ $\cMangc$ and\/ $\cMangcin\!:$
\e
\begin{gathered}
\xymatrix@R=13pt@C=90pt{ *+[r]{C(W)} \ar[r]_{C(f)} \ar[d]^{C(e)} & *+[l]{C(Y)}
\ar[d]_{C(h)} \\ *+[r]{C(X)} \ar[r]^{C(g)} & *+[l]{C(Z).\!\!} }
\end{gathered}
\label{cc6eq34}
\e
Equation \eq{cc6eq34} has a grading-preserving property, in that if\/ $(w,\be)\in C_i(W)$ with\/ $C(e)(w,\be)=(x,\ga)\in C_j(X),$ and\/ $C(f)(w,\be)=(y,\de)\in C_k(Y),$ and\/ $C(g)(x,\ga)\ab=C(h)(y,\de)=(z,\ep)\in C_l(Z),$ then\/ $i+l=j+k$. Hence
\begin{equation*}
C_i(W)\cong \ts\coprod_{j,k,l\ge 0: i=j+k-l} C_j^l(X)\t_{C(g)\vert_{\cdots},C_l(Z),C(h)\vert_{\cdots}}C_k^l(Y),
\end{equation*}
where $C_j^l(X)=C_j(X)\cap C(g)^{-1}(C_l(Z))$ and\/ $C_k^l(Y)=C_k(Y)\cap C(h)^{-1}(C_l(Z)),$ open and closed in $C_j(X),C_k(Y)$. When $i=1,$ this gives a formula for~$\pd W$.
\label{cc6thm12}
\end{thm}

\subsubsection{$F_\Mangcin^\CSchcto,F_\Mangc^\CSchctoex$ preserve b-, c-transverse fibre products}
\label{cc672}

We generalize the last part of Theorem \ref{cc2thm4} to the corners case.

\begin{thm}
\label{cc6thm13}
{\bf(a)} The functor $F_\Mangcin^\CSchcto:\Mangcin\ra\CSchcto$ takes b- and c-transverse fibre products in $\Mangcin$ to fibre products in $\CSchcto$.
\smallskip

\noindent{\bf(b)} The functor $F_\Mangc^\CSchctoex:\Mangc\ra\CSchctoex$ takes c-transverse fibre products in $\Mangc$ to fibre products in $\CSchctoex$.
\end{thm}

\begin{proof} For (a), suppose $g:X\ra Z$ and $h:Y\ra Z$ are b-transverse in $\Mangcin$, and let $W=X\t_{g,Z,h}Y$ be the fibre product in $\Mangcin$ given by Theorem \ref{cc6thm11}, with projections $e:W\ra X$ and $f:W\ra Y$. Write $\bW,\ldots,\bs h$ for the images of $W,\ldots,h$ under $F_\Mangcin^\CSchcto$. Let $\bs{\ti W}=\bX\t_{\bs g,\bZ,\bs h}\bY$ be the fibre product in $\CSchcto$, with projections $\bs{\ti e}:\bs{\ti W}\ra\bX$, $\bs{\ti f}:\bs{\ti W}\ra\bY$, which exists by Theorem \ref{cc5thm10}. As $\bs g\ci\bs e=\bs h\ci\bs f$, the universal property of $\bs{\ti W}$ gives a unique morphism $\bs b:\bW\ra\bs{\ti W}$ with $\bs e=\bs{\ti e}\ci\bs b$ and $\bs f=\bs{\ti f}\ci\bs b$. We must prove $\bs b$ is an isomorphism.

First note that as morphisms $\bs b:\bW\ra\bs{\ti W}$ form a sheaf, this question is local in $\bs{\ti W}$ and hence in $\bX,\bY,\bZ$, so it is enough to prove it over small open neighbourhoods of points $x\in\bX$, $y\in\bY$, $z\in\bZ$ with $\bs g(x)=\bs h(y)=z$. Thus we may replace $X,Y,Z$ by small open neighbourhoods of $x,y,z$ in $X,Y,Z$, or equivalently, by small open neighbourhoods of $(\de_0,0)$ in the local models $X_Q\t\R^l$ for manifolds with g-corners, where $Q$ is a toric monoid, as in Remark \ref{cc3rem2}. This also replaces $W$ by a small open neighbourhood of $w=(x,y)$ in $W$. As small open balls about $(\de_0,0)$ in $X_Q\t\R^l$ are diffeomorphic to $X_Q\t\R^l$, this means we can suppose that $W,\ldots,Z$ are of the form $X_Q\t\R^l$. So we can take
\e
W\cong X_P\t\R^k,\quad X\cong X_Q\t\R^l,\quad Y\cong X_R\t\R^m,\quad Z\cong X_S\t\R^n,
\label{cc6eq35}
\e
for toric monoids $P,Q,R,S$ and $k,l,m,n\ge 0$, where $e,f,g,h:(\de_0,0)\mapsto(\de_0,0)$.

Then $W,X,Y,Z$ are manifolds with g-faces, so by Theorem \ref{cc5thm5}(a) we may take $\bW=\Speccin\bs C^\iy_\rin(W)$, and similarly for $\bX,\bY,\bZ$. Then $\bs e=\Speccin (e^*,e_\rex^*)$, where $e^*:C^\iy(X)\ra C^\iy(W)$, $e_\rex^*:\In(X)\amalg\{0\}\ra \In(W)\amalg\{0\}$ are the pullbacks, and similarly for $\bs f,\bs g,\bs h$. Form a commutative diagram in~$\CRingscto$:
\begin{equation*}
\xymatrix@C=50pt@R=15pt{
& \bs C^\iy_\rin(Y) \ar[dr]_{\bs\psi} \ar@/^.7pc/[drr]^{(f^*,f_\rex^*)} \\
\bs C^\iy_\rin(Z) \ar@/^.5pc/[ur]^{(h^*,h_\rex^*)} \ar@/_.5pc/[dr]_{(g^*,g_\rex^*)} && \bfC \ar[r]^(0.3){\bs\be} & \bs C^\iy_\rin(W), \\
& \bs C^\iy_\rin(X) \ar[ur]^{\bs\phi} \ar@/_.7pc/[urr]_{(e^*,e_\rex^*)} }
\end{equation*}
where $\bfC$ is the pushout $\bs C^\iy_\rin(X)\amalg_{\bs C^\iy_\rin(Z)}\bs C^\iy_\rin(Y)$ in $\CRingscto$, which exists by Proposition \ref{cc4prop14}, and $\bs\phi,\bs\psi$ are the projections to the pushout, and $\bs\be$ exists by the universal property of pushouts. Since $\Speccin:(\CRingscto)^{\bf op}\ra\CSchcto$ preserves limits, it follows that $\bs{\ti W}=\Speccin\bfC$ and $\bs b=\Speccin\bs\be$, and it is enough to prove that $\bs\be$ is an isomorphism.

The pushout $\bfC$ is equivalent to a coequalizer diagram in~$\CRingscto$:
\e
\xymatrix@C=25pt{ \bs C^\iy_\rin(Z) \ar@<.5ex>[rrr]^(0.4){(g^*,g_\rex^*)\ci\bs i_{\bs C^\iy_\rin(X)}} \ar@<-.5ex>[rrr]_(0.4){(h^*,h_\rex^*)\ci\bs i_{\bs C^\iy_\rin(Y)}} &&& \bs C^\iy_\rin(X) \ot_\iy^{\rm to}\bs C^\iy_\rin(Y) \ar[rr] && \bfC }
\label{cc6eq36}
\e 
by properties of colimits, for $\bs C^\iy_\rin(X) \ot_\iy^{\rm to}\bs C^\iy_\rin(Y)$ the coproduct in $\CRingscto$. Now one can show that the coproduct $\bs C^\iy_\rin(X) \ot_\iy\bs C^\iy_\rin(Y)$ in $\CRingsc$ is $\bs C^\iy_\rin(X\t Y)$. Since this is toric, it coincides with the coproduct in $\CRingscto$. Thus we may replace \eq{cc6eq36} by the coequalizer diagram in~$\CRingscto$:
\e
\xymatrix@C=32pt{ \bs C^\iy_\rin(Z) \ar@<.5ex>[rrr]^(0.45){((g\ci\pi_X)^*,(g\ci\pi_X)_\rex^*)} \ar@<-.5ex>[rrr]_(0.45){((h\ci\pi_Y)^*,(h\ci\pi_Y)_\rex^*)} &&& \bs C^\iy_\rin(X\t Y) \ar[rr] && \bfC. }
\label{cc6eq37}
\e
We deduce that
\e
\bfC\cong\Pi_\rin^\sa\bigl(\bs C^\iy_\rin(X\t Y)\big/\bigl[(g\ci\pi_X)^*(c')=(h\ci\pi_Y)^*(c'),\;\; c'\in \In(Z) \bigr]\bigr).
\label{cc6eq38}
\e
Here the term $\bs C^\iy_\rin(X\t Y)/[\cdots]$ is the coequalizer of \eq{cc6eq37} in $\CRingsc$. There is no need to also impose $\R$-type relations $(g\ci\pi_X)^*(c)=(h\ci\pi_Y)^*(c)$ for $x\in C^\iy(Z)$, as these are implied by $(g\ci\pi_X)^*(c')=(h\ci\pi_Y)^*(c')$ for $c'=\Psi_{\exp}(c)$ in $\In(X)$. The proof of Theorem \ref{cc4thm3}(b) then implies that $\Pi_\rin^\sa(\cdots)$ is the coequalizer of \eq{cc6eq37} in~$\CRingscto$.

In \eq{cc6eq38}, we think of $\Pi_\rin^\sa$ as mapping $\CRingscfiin\ra\CRingscto$. It does not change $\bs C^\iy_\rin(X\t Y)$ as this is already toric, but it can add extra relations. From the construction of $\Pi_\rin^\sa$ in the proof of Theorem \ref{cc4thm5}, we can prove that
\e
\begin{split}
\bfC\cong\bs C^\iy_\rin(X\t Y)\big/\bigl[\text{$c_1'=c_2'$: $c_1',c_2'\in \In(X\t Y)$, for some $c'\in\In(Z)$}&\\
\text{and $n\ge 1$ we have $c_1'(g\ci\pi_X)^*(c')^n=c_2'(h\ci\pi_Y)^*(c')^n$}&\bigr].
\end{split}
\label{cc6eq39}
\e
Here imposing relations $c_1'=c_2'$ if $c_1'(g\ci\pi_X)^*(c')=c_2'(h\ci\pi_Y)^*(c')$ for $c'\in\In(Z)$ has the effect of making \eq{cc6eq39} integral, that is, it applies $\Pi_\rin^\Z$. (Note that taking $c_1'=(h\ci\pi_Y)^*(c')$ and $c_2'=(g\ci\pi_X)^*(c')$ recovers the relations in \eq{cc6eq38}, but $c_1'(g\ci\pi_X)^*(c')=c_2'(h\ci\pi_Y)^*(c')$ also includes relations in the integral completion.) Then forcing $c_1'=c_2'$ when $c_1'(g\ci\pi_X)^*(c')^n=c_2'(h\ci\pi_Y)^*(c')^n$ for $n\ge 1$ also makes \eq{cc6eq39} torsion-free and saturated, that is, it applies $\Pi_\tf^\sa\ci\Pi_\Z^\tf$, so overall we apply~$\Pi_\rin^\sa=\Pi_\tf^\sa\ci\Pi_\Z^\tf\ci\Pi_\rin^\Z$.

Using \eq{cc6eq39}, we now show that $\bs\be=(\be,\be_\rex):\bfC\ra\bs C^\iy_\rin(W)$ is an isomorphism, where $W$ is an embedded submanifold in $X\t Y$ as in Theorem \ref{cc6thm11}. To show that $\be,\be_\rex$ are surjective, it is enough to show that the restriction maps $C^\iy(X\t Y)\ra C^\iy(W)$, $\In(X\t Y)\ra\In(W)$ are surjective. That is, all smooth or interior maps $W\ra\R,[0,\iy)$ should extend to $X\t Y$. Locally near $(x,y)$ in $W,X\t Y$ this follows from extension properties of smooth maps on $X_P\t\R^l$ in \cite[Prop.~3.14]{Joyc6}, and globally it follows from \eq{cc6eq35}, which ensures there are no obstructions from multiplicity functions $\mu_{c'}$ as in Remark~\ref{cc3rem3}.

To show $\be$ (or $\be_\rex$) is injective, we need to show that if $c_1,c_2\in C^\iy(X\t Y)$ (or $c'_1,c'_2\in \In(X\t Y)$), then $c_1\vert_W=c_2\vert_W$ (or $c_1'\vert_W=c_2'\vert_W$) if and only if $c_1-c_2$ lies in the ideal $I$ in $C^\iy(X\t Y)$ generated by the relations in \eq{cc6eq39} (or $c_1'\simc c_2'$, where $\simc$ is the equivalence relation on $\In(X\t Y)$ generated by the relations in \eq{cc6eq39}). Since $W$ is an embedded submanifold defined as a subset of $X\t Y$ by the b-transverse relations $(g\ci\pi_X)^*(c')=(h\ci\pi_Y)^*(c')$ for $c'\in\In(Z)$, this is indeed true, though proving it carefully requires a little work on ideals in $C^\iy(X\t Y)$. Therefore $\bs\be$ is an isomorphism, completing part~(a). 

For (b), suppose $g:X\ra Z$ and $h:Y\ra Z$ are c-transverse in $\Mangcin$, which implies $g,h$ b-transverse, and let $W=X\t_{g,Z,h}Y$ be the fibre product in both $\Mangcin$ and $\Mangc$ given by Theorems \ref{cc6thm11}--\ref{cc6thm12}. Then \eq{cc6eq34} is Cartesian in $\Mangcin$ by Theorem \ref{cc6thm12}. Write $\bW,\ldots,\bs h$ for the images of $W,\ldots,h$ under $F_\Mangcin^\CSchcto$. Then part (a), Proposition \ref{cc6prop2}, and Theorem \ref{cc6thm7} imply that $\bW=\bX\t_\bZ\bY$ in $\CSchcto$, and \eq{cc6eq32} is Cartesian in $\CSchcto$. Hence Theorem \ref{cc6thm9}(d) says $\bW=\bX\t_\bZ\bY$ in $\CSchctoex$, and the theorem follows.	
\end{proof}

\begin{ex}
\label{cc6ex10}
Let $X=Y=Z=[0,\iy)$, and $g:X\ra Z$, $h:Y\ra Z$ map $g(x)=x^2$, $h(y)=y^2$. Then $g,h$ are b- and c-transverse, and the fibre product $W=X\t_{g,Z,h}Y$ in $\Mangcin$ and $\Mangc$ is $W=[0,\iy)$ with projections $e:W\ra X$, $f:W\ra Y$ mapping $e(w)=w$, $f(w)=w$.

Let $\bW,\ldots,\bs h$ be the images of $W,\ldots,h$ under $F_\Manc^\CSchc$. Then Theorem \ref{cc6thm13} implies that $\bW=\bX\t_{\bs g,\bZ,\bs h}\bY$ in $\CSchcto$ and $\CSchctoex$. But what about fibre products $\bX\t_{\bs g,\bZ,\bs h}\bY$ in other categories?

As $\bX,\bY,\bZ=\Speccin\bs C_\rin^\iy([0,\iy))$ are affine, we can show that the fibre product $\bs{\hat W}=\bX\t_\bZ\bY$ in $\CSchcfiin$ is $\Speccin\bfC$, where $\bfC=\R[x,y]/[x^2=y^2]$. Then $\fC_\rin^\sh\cong \Z\t\Z_2$, identifying $[x]=(1,0)$ and $[y]=(1,1)$ in $\Z\t\Z_2$. We see that $\bfC$ is firm and integral, but not torsion-free (as $\fC_\rin^\sh$ has torsion $\Z_2$) or saturated (as $x,y\in\fC_\rin$ with $x\ne y$ but $x^2=y^2$). Thus $\bs{\hat W}$ is also firm and integral, but not torsion-free or saturated, and~$\bW\not\cong\bs{\hat W}$.

This shows that $F_\Mangcin^\CSchcZ:\Mangcin\!\ra\!\CSchcZ$ and $F_\Mangcin^\CSchcfiin:\Mangcin\!\ra\!\CSchcfiin$ need not preserve b-transverse fibre products. Example \ref{cc6ex8} above is similar, but $\bs{\hat W}$ is not integral. One conclusion is that in Theorem \ref{cc6thm13} it is essential 
to work in the categories $\CSchcto,\CSchctoex$ of {\it toric\/} $C^\iy$-schemes with corners, rather than larger categories such as $\CSchcfitf,\ab\CSchcfiZ,\ab\CSchcfiin$ discussed in \S\ref{cc56}--\S\ref{cc57}. Thus, $\CSchcto,\CSchctoex$ may be regarded as natural generalizations of~$\Mangcin,\Mangc$.
\end{ex}

\subsubsection{Restriction to manifolds with corners}
\label{cc673}

There is an easy way to restrict Theorem \ref{cc6thm13} to $\Mancin,\Manc$:

\begin{dfn}
\label{cc6def6}
Let $g:h\ra Z$ and $h:Y\ra Z$ be morphisms in $\Mancin$. Following the second author \cite[\S 2.5.4]{Joyc8}, we call $g,h$ {\it sb-transverse\/} or {\it sc-transverse\/} (short for {\it strictly b-transverse\/} or {\it strictly c-transverse\/}) if they are b-transverse or c-transverse in $\Mangcin$, respectively, and the fibre product $W=X\t_{g,Z,h}Y$ in $\Mangcin$ is a manifold with corners, not just g-corners. In \cite[\S 2.5.4]{Joyc8} we explain the sb- and sc-transverse conditions explicitly. Theorems \ref{cc6thm11}--\ref{cc6thm12} imply that sb- and sc-transverse fibre products exist in $\Mancin$ and $\Manc$, respectively.
\end{dfn}

Then Theorem \ref{cc6thm13} implies:

\begin{cor}
\label{cc6cor1}
{\bf(a)} The functor $F_\Mancin^\CSchcto:\Mancin\ra\CSchcto$ takes sb- and sc-transverse fibre products in $\Mancin$ to fibre products in $\CSchcto$.
\smallskip

\noindent{\bf(b)} The functor $F_\Manc^\CSchctoex:\Manc\ra\CSchctoex$ takes sc-transverse fibre products in $\Manc$ to fibre products in~$\CSchctoex$.
\end{cor}

\section{Modules, and sheaves of modules}
\label{cc7}

Next we discuss modules over $C^\iy$-rings with corners, sheaves of $\bO_X$-modules on $C^\iy$-schemes with corners $\bX$, and (b-)cotangent modules and sheaves.

\subsection{\texorpdfstring{Modules over $C^\iy$-rings with corners, (b-)cotangent modules}{Modules over C∞-rings with corners, (b-)cotangent modules}}
\label{cc71} 

In \S\ref{cc22} we discussed modules over $C^\iy$-rings. Here is the corners analogue:

\begin{dfn}
\label{cc7def1}
Let $\bfC=(\fC,\fC_\rex)$ be a $C^\iy$-ring with corners. A {\it module\/ $M$ over\/} $\bfC$, or $\bfC$-{\it module}, is a module over $\fC$ regarded as a commutative $\R$-algebra as in Definition \ref{cc2def7}, and morphisms of $\bfC$-modules are morphisms of $\R$-algebra modules. Then $\bfC$-modules form an abelian category, which we write as
$\bfCmod$.

The basic theory of \S\ref{cc22} extends trivially to the corners case. So if $\bs\phi=(\phi,\phi_\rex):\bfC\ra\bfD$ is a morphism in $\CRingsc$ then using $\phi:\fC\ra\fD$ we get functors $\bs\phi_*:\bfCmod\ra\bfDmod$ mapping $M\mapsto M\ot_\fC\fD$, and $\bs\phi^*:\bfDmod\ra\bfCmod$ mapping~$N\mapsto N$.
\end{dfn}

One might expect that modules over $\bfC=(\fC,\fC_\rex)$ should also include some kind of monoid module over $\fC_\rex$, but we do not do this.

\begin{ex}
\label{cc7ex1}
Let $X$ be a manifold with corners (or g-corners) and $E\ra X$ be a vector bundle, and write $\Ga^\iy(E)$ for the vector space of smooth sections $e$ of $E$. This is a module over $C^\iy(X)$, and hence over the $C^\iy$-rings with corners $\bs C^\iy(X)$ and $\bs C^\iy_\rin(X)$ from Example~\ref{cc4ex3}.
\end{ex}

Section \ref{cc22} discussed cotangent modules of $C^\iy$-rings, the analogues of cotangent bundles of manifolds. As in \S\ref{cc35} manifolds with corners $X$ have cotangent bundles $T^*X$ which are functorial over smooth maps, and b-cotangent bundles ${}^bT^*X$, which are functorial only over interior maps. In a similar way, for a $C^\iy$-ring with corners $\bfC$ we will define the {\it cotangent module\/} $\Om_\bfC$, and if $\bfC$ is interior we will also define the {\it b-cotangent module\/}~${}^b\Om_\bfC$.

\begin{dfn}
\label{cc7def2}
Let $\bfC=(\fC,\fC_\rex)$ be a $C^\iy$-ring with corners. Define the {\it cotangent module\/} $\Om_\bfC$ of $\bfC$ to be the cotangent module $\Om_\fC$ of \S\ref{cc22}, regarded as a $\bfC$-module. If $\bs\phi=(\phi,\phi_\rex):\bfC\ra\bfD$ is a morphism in $\CRingsc$ then from $\Om_\phi$ in \S\ref{cc22} we get functorial morphisms $\Om_{\bs\phi}:\Om_\bfC\ra\Om_\bfD=\bs\phi^*(\Om_\bfD)$ in $\bfCmod$ and $(\Om_\phi)_*:\bs\phi_*(\Om_\bfC)=\Om_\bfC\ot_\fC\fD\ra\Om_\bfD$ in $\bfDmod$.
\end{dfn}

\begin{dfn}
\label{cc7def3} 
Let $\bfC=(\fC,\fC_\rex)$ be an interior $C^\iy$-ring with corners, so that $\fC_\rex=\fC_\rin\amalg\{0_{\fC_\rex}\}$ with $\fC_\rin$ a monoid. Let $M$ be a $\bfC$-module. A {\it b-derivation} is a map $\d_\rin:\fC_\rin\ra M$ satisfying the following corners analogue of \eq{cc2eq4}: for any interior $g:[0,\iy)^n\ra[0,\iy)$ and elements $c'_1,\ldots,c'_n\in\fC_\rin$, then
\e
\d_\rin\bigl(\Psi_g(c'_1,\ldots,c'_n)\bigr)-\ts\sum\limits_{i=1}^n\Phi_{g^{-1}x_i\frac{\pd g}{\pd x_i}}(c'_1,\ldots,c'_n)\cdot \d_\rin c'_i=0.
\label{cc7eq1}
\e
Here $g$ interior implies that $g^{-1}x_i\frac{\pd g}{\pd x_i}:[0,\iy)^n\ra\R$ is smooth.

It is easy to show that \eq{cc7eq1} implies:
\begin{itemize}
\setlength{\itemsep}{0pt}
\setlength{\parsep}{0pt}
\item[(i)] $\d_\rin:\fC_\rin\ra M$ is a monoid morphism, where $M$ is a monoid over addition.
\item[(ii)] $\d=\d_\rin\ci\Psi_{\exp}:\fC\ra M$ is a $C^\iy$-derivation in the sense of Definition \ref{cc2def8}.
\item[(iii)] $\d_\rin\ci\Psi_{\exp}(\Phi_i(c'))=\Phi_i(c')\cdot\d_\rin c'$ for all $c'\in\fC_\rin$. 
\end{itemize}
As interior $g:[0,\iy)^n\ra[0,\iy)$ are of the form $g=x_1^{a_1}\cdots x_n^{a_n}\exp f(x_1,\ldots x_n)$, with $\Psi_g(c_1',\ldots,c_n')=c_1^{\prime a_1}\cdots c_n^{\prime a_n}\Psi_{\exp}(\Phi_f(c_1',\ldots,c_n'))$ for $c_1',\ldots,c_n'\in\fC_\rin$, we can show that (i)--(iii) are equivalent to~\eq{cc7eq1}.

We call such a pair $(M,\d_\rin)$ a {\it b-cotangent module\/} for $\bfC$ if it has the universal property that for any b-derivation $\d'_\rin:\fC_\rin\ra M'$, there exists a unique morphism of $\bfC$-modules ${}^b\la:M\ra M'$ with $\d'_\rin={}^b\la\ci\d_\rin$. 

There is a natural construction for a b-cotangent module: we take $M$ to be the quotient of the free $\bfC$-module with basis of symbols $\d_\rin c'$ for $c'\in\fC_\rin$ by the $\fC$-submodule spanned by all expressions of the form \eq{cc7eq1}. Thus b-cotangent modules exist, and are unique up to unique isomorphism. When we speak of `the' b-cotangent module, we mean that constructed above, and we write it as~$\d_{\bfC,\rin}:\fC_\rin\ra{}^b\Om_\bfC$.

Since $\d_{\bfC,\rin}\ci\Psi_{\exp}:\fC\ra{}^b\Om_\bfC$ is a $C^\iy$-derivation, the universal property of $\Om_\fC=\Om_\bfC$ in \S\ref{cc22} implies that there is a unique $\bfC$-module morphism $I_\bfC:\Om_\bfC\ra{}^b\Om_\bfC$ with~$\d_{\bfC,\rin}\ci\Psi_{\exp}=I_\bfC\ci\d_\fC:\fC\ra{}^b\Om_\bfC$.

Let $\bs\phi:\bfC\ra\bfD$ be a morphism in $\CRingscin$. Then we have a monoid morphism $\phi_\rin:\fC_\rin\ra\fD_\rin$. Regarding ${}^b\Om_\bfD$ as a $\bfC$-module using $\phi:\fC\ra\fD$, then $\d_{\bfD,\rin}\ci\phi_\rin:\fC_\rin\ra{}^b\Om_\bfD$ becomes a b-derivation. Thus by the universal property of ${}^b\Om_\bfC$, there exists a unique $\bfC$-module morphism ${}^b\Om_{\bs\phi}:{}^b\Om_\bfC\ra{}^b\Om_\bfD$ with $\d_{\bfD,\rin}\ci\phi_\rin={}^b\Om_{\bs\phi}\ci\d_{\bfC,\rin}$. This then induces a morphism of $\bfD$-modules $({}^b\Om_{\bs\phi})_*:\bs\phi_*({}^b\Om_\bfC)={}^b\Om_\bfC\ot_\fC\fD\ra{}^b\Om_\bfD$. If $\bs\phi:\bfC\ra\bfD$, $\bs\psi:\bfD\ra\bfE$ are morphisms in $\CRingscin$ then~${}^b\Om_{\bs\psi\ci\bs\phi}={}^b\Om_{\bs\psi}\ci{}^b\Om_{\bs\phi}:{}^b\Om_\bfC\ra{}^b\Om_\bfE$.
\end{dfn}

\begin{rem}
\label{cc7rem1}
{\bf(a)} We may think of $\d_\rin c'$ above as $\d(\log c')$. As in Example \ref{cc7ex2}(b), if $X$ is a manifold with corners and $c':X\ra[0,\iy)$ is interior, then $\d(\log c')$ is well defined section of ${}^bT^*X$. This holds even at $\pd X$ where we allow $c'=0$ and $\log c'$ is undefined, as ${}^bT^*X$ is spanned by $x_1^{-1}\d x_1,\ab\ldots,\ab x_k^{-1}\d x_k,\ab\d x_{k+1},\ab\ldots,\ab\d x_n$ in local coordinates $(x_1,\ldots,x_n)\in\R^n_k$, and the factors $x_1^{-1},\ldots,x_k^{-1}$ control the singularities of~$\d\log c'$.
\smallskip

\noindent{\bf(b)} In Definition \ref{cc7def3} we could have omitted the condition that $\bfC$ be interior, and considered monoid morphisms $\d_\rex:\fC_\rex\ra M$ such that $\d=\d_\rex\ci\Psi_{\exp}:\fC\ra M$ is a $C^\iy$-derivation and $\d_\rex\ci\Psi_{\exp}(\Phi_i(c'))=\Phi_i(c')\d_\rex c'$ for all $c'\in\fC_\rex$. However, since $c'\cdot 0_{\fC_\rex}=0_{\fC_\rex}$ for all $c'\in\fC_\rex$ we would have $\d_\rex c'+\d_\rex 0_{\fC_\rex}=\d_\rex 0_{\fC_\rex}$ in $M$, so $\d_\rex c'=0$, and this modified definition would give ${}^b\Om_\bfC=0$ for any $\bfC$. To get a nontrivial definition we took $\bfC$ to be interior, and defined $\d_\rin$ only on~$\fC_\rin=\fC_\rex\sm\{0_{\fC_\rex}\}$.

If $\bfC$ lies in the image of $\Pi_\CRingsc^\CRingscin:\CRingsc\hookra\CRingscin$ from Definition \ref{cc4def5} then ${}^b\Om_\bfC=0$, since $\fC_\rin$ then contains a zero element $0_{\fC_\rin}$ with $c'\cdot 0_{\fC_\rin}=0_{\fC_\rin}$ for all~$c'\in\fC_\rin$.

If $\d_\rin:\fC_\rin\ra M$ is a b-derivation then it is a morphism from a monoid to an abelian group, and so factors through $\pi^\gp:\fC_\rin\ra(\fC_\rin)^\gp$. This suggests that b-cotangent modules may be most interesting for {\it integral\/} $C^\iy$-rings with corners, as in \S\ref{cc47}, for which $\pi^\gp:\fC_\rin\ra(\fC_\rin)^\gp$ is injective.
\end{rem}

\subsection{(B-)cotangent modules for manifolds with (g-)corners}
\label{cc72}

\begin{ex}{\bf(a)} Let $X$ be a manifold with corners. Then $\Ga^\iy(T^*X)$ is a module over the $\R$-algebra $C^\iy(X)$, and so over the $C^\iy$-rings with corners $\bs C^\iy(X)$ and $\bs C^\iy_\rin(X)$ from Example \ref{cc4ex3}. The exterior derivative $\d:C^\iy(X)\ra\Ga^\iy(T^*X)$ is a $C^\iy$-derivation, and there is a unique morphism $\la:\Om_{\bs C^\iy(X)}\ra\Ga^\iy(T^*X)$ such that $\d=\la\ci\d$.
\smallskip

\noindent{\bf(b)} Let $X$ be a manifold with corners (or g-corners), with b-co\-tang\-ent bundle ${}^bT^*X$ as in \S\ref{cc35}. Example \ref{cc4ex3}(b) defines a $C^\iy$-ring with corners $\bs C^\iy_\rin(X)=\bfC$ with $\fC_\rin=\In(X)$, the monoid of interior maps $c':X\ra[0,\iy)$. We have a $C^\iy(X)$-module $\Ga^\iy({}^bT^*X)$. Define $\d_\rin:\In(X)\ra \Ga^\iy({}^bT^*X)$ by
\e
\d_\rin(c')=c^{\prime -1}\cdot {}^b\d c'={}^b\d(\log c'),
\label{cc7eq2}
\e
where ${}^b\d=I_X^*\ci\d$ is the composition of the exterior derivative $\d:C^\iy(X)\ra\Ga^\iy(T^*X)$ with the projection $I_X^*:T^*X\ra{}^bT^*X$. Here \eq{cc7eq2} makes sense on the interior $X^\ci$ where $c'>0$, but has a unique smooth extension over~$X\sm X^\ci$.

We can now show that $\d_\rin:\In(X)\ra \Ga^\iy({}^bT^*X)$ is a b-derivation in the sense of Definition \ref{cc7def3}, so there is a unique morphism ${}^b\la:{}^b\Om_{\bs C_\rin^\iy(X)}\ra \Ga^\iy({}^bT^*X)$ such that $\d_\rin={}^b\la \ci \d_\rin$.
\label{cc7ex2}
\end{ex}

\begin{thm}
\label{cc7thm1}
{\bf(a)} Let\/ $X$ be a manifold with faces, as in Definition\/ {\rm\ref{cc3def8},} such that\/ $\pd X$ has finitely many connected components. Then $\Ga^\iy(T^*X)$ is the cotangent module of\/ $C^\iy(X)$. That is, $\la$ from Example\/ {\rm\ref{cc7ex2}(a)} is an isomorphism.

\smallskip

\noindent{\bf(b)} Let $X$ be a manifold with faces, or a manifold with g-faces, as in Definition\/ {\rm\ref{cc4def12},} such that\/ $\pd X$ has finitely many connected components. Then $\Ga^\iy({}^bT^*X)$ is the b-cotangent module of\/ $\bs C^\iy_\rin(X)$. That is, ${}^b\la$ from Example\/ {\rm\ref{cc7ex2}(b)} is an isomorphism.
\end{thm} 

\begin{proof} The first part of the proof is common to both (a) and (b). Let $X$ be a manifold with faces, or with g-faces, and suppose $\pd X$ has finitely many connected components, which we write $\pd_1X,\ldots,\pd_NX$. We will show that the sharpened monoid $C^\iy_\rin(X)^\sh$ is finitely generated. 

Write $[c']\in C^\iy_\rin(X)^\sh$ for the $C^\iy_\rin(X)^\t$-orbit of $c'\in C^\iy_\rin(X)$. Each such $c'$ vanishes to some order $a_i\in\N$ on $\pd_iX$ for $i=1,\ldots,N$. The map $C^\iy_\rin(X)\ra\N^N$, $c'\mapsto(a_1,\ldots,a_N)$ is a monoid morphism, which descends to a morphism $C^\iy_\rin(X)^\sh\ra\N^N$, $[c']\mapsto(a_1,\ldots,a_N)$, as $\N^N$ is sharp. 

It is easy to show that this morphism $C^\iy_\rin(X)^\sh\ra\N^N$ is injective. If $X$ is a manifold with faces it is also surjective, so $C^\iy_\rin(X)^\sh\cong\N^N$. If $X$ is a manifold with g-faces, the image of $C^\iy_\rin(X)^\sh$ is the submonoid of $\N^N$ satisfying some linear equations coming from the higher-dimensional g-corner strata. Therefore $C^\iy_\rin(X)^\sh$ is a toric monoid, and so finitely generated.

Choose $x_1,\ldots,x_k\in C^\iy_\rin(X)$ whose images $[x_1],\ldots,[x_k]$ generate $C^\iy_\rin(X)^\sh$. Following the proof of the Whitney Embedding Theorem, we may choose a smooth map $(x_{k+1},\ldots,x_l):X\ra\R^{l-k}$ for some $l$ with $l-k\ge 2\dim X$ which is a closed embedding (in a weak sense, in which the inclusion $[0,\iy)\hookra\R$ counts as an embedding). Then $F:=(x_1,\ldots,x_k,x_{k+1},\ldots,x_l):X \hookra[0,\iy)^k\t\R^{l-k}=\R^l_k$ is an interior closed embedding in a stronger sense, which embeds the corner strata of $X$ into corner strata of~$\R^l_k$. 

Here as $X$ has (g-)faces, by Proposition \ref{cc4prop9} and Definition \ref{cc4def12} the natural morphism $\bs\la_x:(\bs C^\iy(X)){}_{x_*}\ra\bs C_x^\iy(X)$ is an isomorphism for all $x\in X$, so as $[x_1],\ldots,[x_k]$ generate $C^\iy_\rin(X)^\sh$, they also generate $C^\iy_{x,\rin}(X)^\sh$ for each $x\in X$. Geometrically this means that $X$ near $x$ is nicely embedded in $\R^l_k$ near $F(x)$, in such a way that the corner strata of $X$ near $x$ all come from corner strata of $\R^l_k$ near $F(x)$. Write $\ti X=F(X)$, as a closed, embedded submanifold of~$\R^l_k$.

Thus we have a morphism $F^*:\bs C^\iy_\rin(\R^l_k)\ra\bs C^\iy_\rin(X)$ in $\CRingscin$. We claim that $F^*$ is surjective. To see this, note that $F^*$ is surjective on $C^\iy$-rings as it is a closed embedding, so every smooth function $F(X)\ra\R$ extends to a smooth function $\R^l_k\ra\R$ in the usual way. Also $(F^*_\rin)^\sh:C^\iy_\rin(\R^l_k)^\sh\cong \N^k\ra C^\iy_\rin(X)^\sh$ is surjective as it maps $(a_1,\ldots,a_k)\mapsto[x_1^{a_1}\cdots x_k^{a_k}]$, where $[x_1],\ldots,[x_k]$ generate $C^\iy_\rin(X)^\sh$. Surjectivity of $F^*_\rin:C^\iy_\rin(\R^l_k)\ra C^\iy_\rin(X)$ follows.

Therefore $\bs C^\iy_\rin(X)$ is isomorphic to the quotient of $\bs C^\iy_\rin(\R^l_k)$ by relations generating the kernel of $F^*_\rin$, or equivalently, by equations defining $\ti X$ as a submanifold of $\R^l_k$. If $X$ is a manifold with faces then relations of type $f=0$ for $f\in C^\iy(\R^l_k)$ are sufficient, so in the notation of Definition \ref{cc4def8} we have
\e
\bs C^\iy_\rin(X)\cong \bs C^\iy_\rin(\R^l_k)\big/\bigl(f=0: f\in C^\iy(\R^l_k),\; f\vert_{\ti X}=0\bigr).
\label{cc7eq3}
\e
If $X$ is a manifold with g-faces we may also need relations of type $g=h$ for $g,h\in C^\iy_\rin(\R^l_k)$, giving
\e
\begin{split}
\bs C^\iy_\rin(X)\cong \bs C^\iy_\rin(\R^l_k)\big/&\bigl(f=0: f\in C^\iy(\R^l_k),\; f\vert_{\ti X}=0\bigr)\\
&\bigl[g=h:g,h\in C^\iy_\rin(\R^l_k),\;g\vert_{\ti X}=h\vert_{\ti X}\bigr].
\end{split}
\label{cc7eq4}
\e

We now prove part (a), so suppose $X$ is a manifold with faces and $\pd X$ has finitely many boundary components, so that \eq{cc7eq3} holds. Consider the commutative diagram of $C^\iy(X)$-modules with exact rows:
\e
\begin{gathered}
\xymatrix@C=1.5pt@R=15pt{
0 \ar[rr] && K\!=\!\Ker(\Om_{F^*}) \ar@{->>}[d]^\mu \ar[rrr] &&& {\begin{subarray}{l}\ts \Om_{\bs C^\iy_\rin(\R^k_l)}\!\ot_{C^\iy(\R^k_l)}\!C^\iy(X)\!= \\ \ts \;\> \an{\d x_1,\ldots,\d x_l}\!\ot_\R\! C^\iy(X)\end{subarray}} \ar[d]_\cong \ar[rrr]_(0.68){\Om_{F^*}} &&&   \Om_{\bs C^\iy_\rin(X)} \ar[d]_\la \ar[rr] && 0 \\
0 \ar[rr] && \Ga^\iy(\Ker(\d F^*)) \ar[rrr] &&& {\begin{subarray}{l}\ts \qquad \Ga^\iy(F^*(T^*\R^k_l))= \\ \ts \an{\d x_1,\ldots,\d x_l}\!\ot_\R\! C^\iy(X)\end{subarray}} \ar[rrr]^(0.63){\Ga^\iy(\d F^*)} &&&   \Ga^\iy(T^*X) \ar[rr] && 0.\! }\!\!\!\!\!\!\!\!\!
\end{gathered}
\label{cc7eq5}
\e
Here as $F^*:C^\iy(\R^k_l)\ra C^\iy(X)$ is surjective, $\Om_{F^*}$ is surjective, and the top row is exact. As the exact sequence $0\ra\Ker(\d F^*)\ra F^*(T^*\R^k_l)\ra T^*X\ra 0$ of vector bundles on $X$ splits, the bottom row is exact. It is easy to check that $\Om_{\bs C^\iy_\rin(\R^k_l)}=\Om_{C^\iy(\R^k_l)}=\an{\d x_1,\ldots,\d x_l}\ot_\R C^\iy(\R^k_l)$ satisfies the universal property for cotangent modules. Since $\Ga^\iy(T^*\R^k_l)=\an{\d x_1,\ldots,\d x_l}_{C^\iy(\R^k_l)}$, we see that the two centre terms are $\an{\d x_1,\ldots,\d x_l}\ot_\R C^\iy(X)$, and are isomorphic. Clearly the right hand square commutes. Hence by exactness there is a unique map $\mu:K\ra \Ga^\iy(\Ker(\d F^*))$ making the diagram commute.

As the right hand square of \eq{cc7eq5} commutes, with the central column an isomorphism, we see that $\la$ is surjective. To prove that $\la$ is injective, and hence an isomorphism, it is enough to show that $\mu$ is surjective. Since $\ti X$ is cut out transversely in $\R^k_l$ by equations $f=0$ for $f\in C^\iy(\R^l_k)$ with $f\vert_{\ti X}=0$, as in \eq{cc7eq3}, we see that the vector bundle $\Ker(\d F^*)\ra X$ is spanned over $C^\iy(X)$ by sections $F^*(\d f)$ for $f\in C^\iy(\R^l_k)$ with $f\vert_{\ti X}=0$. But then $(\d f)\ot 1$ lies in $K=\Ker(\Om_{F^*})$ with $F^*(\d f)=\mu((\d f)\ot 1)$, so the image of $\mu$ contains a spanning set, and $\mu$ is surjective. This proves~(a).

For (b), allow $X$ to have faces or g-faces, so that \eq{cc7eq4} holds, and replace \eq{cc7eq5} by the commutative diagram with exact rows
\begin{equation*}
\xymatrix@C=3.7pt@R=15pt{
0 \ar[rr] && {}^bK\!=\!\Ker({}^b\Om_{F^*}) \ar@{->>}[d]^{{}^b\mu} \ar[rrr] &&& {\begin{subarray}{l}\ts {}^b\Om_{\bs C^\iy_\rin(\R^k_l)}\!\ot_{C^\iy(\R^k_l)}\!C^\iy(X) \\ \ts =\!\langle x_1^{-1}\d x_1,\ldots,x_k^{-1}\d x_k,\\ \ts \d x_{k+1},\ldots,\d x_l\rangle\!\ot_\R\! C^\iy(X)\!\!\!\!\!\end{subarray}} \ar[d]_\cong \ar[rrr]_(0.68){{}^b\Om_{F^*}} &&&   {}^b\Om_{\bs C^\iy_\rin(X)} \ar[d]_{{}^b\la} \ar[rr] && 0 \\
0 \ar[rr] && \Ga^\iy(\Ker({}^b\d F^*)) \ar[rrr] &&& {\begin{subarray}{l}\ts  \Ga^\iy(F^*({}^bT^*\R^k_l)) \\ \ts =\!\langle x_1^{-1}\d x_1,\ldots,x_k^{-1}\d x_k,\\ \ts \d x_{k+1},\ldots,\d x_l\rangle\!\ot_\R\! C^\iy(X)\!\!\!\!\!\!\!\!\end{subarray}} \ar[rrr]^(0.59){\Ga^\iy({}^b\d F^*)} &&&   \Ga^\iy({}^bT^*X) \ar[rr] && 0.\! }
\end{equation*}
Here ${}^b\Om_{\bs C^\iy_\rin(\R^k_l)}=\an{x_1^{-1}\d x_1,\ldots,x_k^{-1}\d x_k,\d x_{k+1},\ldots,\d x_l}\ot_\R C^\iy(\R^k_l)$ satisfies the universal property for b-cotangent modules. The argument above gives a unique map ${}^b\mu$ making the diagram commute, and shows ${}^b\la$ is surjective. To prove ${}^b\la$ is injective, and hence an isomorphism, it is enough to show that ${}^b\mu$ is surjective.

Since $\ti X$ is cut out b-transversely in $\R^k_l$ by equations $f=0$ for $f\in C^\iy(\R^l_k)$ with $f\vert_{\ti X}=0$, and (if $X$ has g-faces) equations $g=h$ for $g,h\in C^\iy_\rin(\R^l_k)$ with $g\vert_{\ti X}=h\vert_{\ti X}$, we see that the vector bundle $\Ker({}^b\d F^*)\ra X$ is spanned over $C^\iy(X)$ by sections $F^*({}^b\d \exp(f))$ for $f\in C^\iy(\R^l_k)$ with $f\vert_{\ti X}=0$ and $F^*({}^b\d g-{}^b\d h)$ for $g,h\in C^\iy_\rin(\R^l_k)$ with $g\vert_{\ti X}=h\vert_{\ti X}$. But then ${}^b\d \exp(f)\ot 1$ and $({}^b\d g-{}^b\d h)\ot 1$ lie in ${}^bK=\Ker({}^b\Om_{F^*})$ with $F^*({}^b\d \exp(f))={}^b\mu({}^b\d \exp(f)\ot 1)$ and $F^*({}^b\d g-{}^b\d h)={}^b\mu(({}^b\d g-{}^b\d h)\ot 1)$, so the image of ${}^b\mu$ contains a spanning set, and ${}^b\mu$ is surjective. This proves~(b).
\end{proof}

\begin{rem}
\label{cc7rem2}
The first author \cite[Prop.s 4.7.5 \& 4.7.9]{Fran} proves by a different method that if we allow $\pd X$ to have infinitely many connected components, then $\la$ is still an isomorphism in Theorem \ref{cc7thm1}(a), and ${}^b\la$ is surjective in (b).
\end{rem}

\subsection{Some computations of b-cotangent modules}
\label{cc73}

The next proposition generalizes Example~\ref{cc2ex6}:

\begin{prop}
\label{cc7prop1}
{\bf(a)} Let $A,A_\rin$ be sets and\/ $\bfF^{A,A_\rin}$ the free interior $C^\iy$-ring with corners from Definition\/ {\rm\ref{cc4def7},} with generators $x_a\in\fF^{A,A_\rin},$ $y_{a'}\in\fF^{A,A_\rin}_\rin$ for $a\in A,$ $a'\in A_\rin$. Then there is a natural isomorphism
\e
{}^b\Om_{\bfF^{A,A_\rin}}\cong\ban{\d_\rin\Psi_{\exp}(x_a),\d_\rin y_{a'}:a\in A,\; a'\in A_\rin}_\R\ot_\R\fF^{A,A_\rin}.
\label{cc7eq6}
\e

\noindent{\bf(b)} Suppose $\bfC$ is defined by a coequalizer diagram \eq{cc4eq13} in $\CRingscin$. Then writing $(x_a)_{a\in A},(y_{a'})_{a'\in A_\rin}$ and\/ $(\ti x_b)_{b\in B},(\ti y_{b'})_{b'\in B_\rin}$ for the generators of\/ $\bfF^{A,A_\rin}$ and\/ $\bfF^{B,B_\rin},$ we have an exact sequence
\e
\xymatrix@C=25pt{{\begin{subarray}{l}\ts  \bigl\langle \d_\rin\Psi_{\exp}(\ti x_b),\; b\in B,\\ \ts \d_\rin \ti y_{b'},\; b'\!\in\! B_\rin\bigr\rangle_\R\!\ot_\R\!\fC \end{subarray}} 
\ar[r]^{\ga} & {\begin{subarray}{l}\ts  \bigl\langle \d_\rin\Psi_{\exp}(x_a),\; a\in A,\\ \ts \d_\rin y_{a'},\; a'\!\in\! A_\rin\bigr\rangle_\R\!\ot_\R\!\fC \end{subarray}} \ar[r]^(0.7)\de & {}^b\Om_\bfC \ar[r] & 0 }
\label{cc7eq7}
\e
in\/ $\bfCmod,$ where if\/ $\bs\al=(\al,\al_\rex)$ and\/ $\bs\be=(\be,\be_\rex)$ in \eq{cc4eq13} map
\e
\begin{aligned}
\al&:\ti x_b\!\mapsto\! e_b\bigl((x_a)_{a\in A},(y_{a'})_{a'\in A_\rin}\bigr), & \al_\rex&:\ti y_{b'}\!\mapsto\! g_{b'}\bigl((x_a)_{a\in A},(y_{a'})_{a'\in A_\rin}\bigr), \\
\be&:\ti x_b\!\mapsto\! f_b\bigl((x_a)_{a\in A},(y_{a'})_{a'\in A_\rin}\bigr),& \be_\rex&:\ti y_{b'}\!\mapsto\! h_{b'}\bigl((x_a)_{a\in A},(y_{a'})_{a'\in A_\rin}\bigr),
\end{aligned}
\label{cc7eq8}
\e
for\/ $e_b,f_b,g_{b'},h_{b'}$ depending only on finitely many\/ $x_a,y_{a'},$ then\/ $\ga,\de$ are given by
\e
\begin{split}
\ga(\d_\rin\Psi_{\exp}(\ti x_b))&=\sum_{a\in A}\begin{aligned}[t]&\phi\bigl(\ts\frac{\pd e_b}{\pd x_a}\bigl((x_a)_{a\in A},(y_{a'})_{a'\in A_\rin}\bigr)\\
&-\ts\frac{\pd f_b}{\pd x_a}\bigl((x_a)_{a\in A},(y_{a'})_{a'\in A_\rin}\bigr)\bigr)\d_\rin\Psi_{\exp}(x_a)\end{aligned}
\\
&+\sum_{a'\in A_\rin}\begin{aligned}[t]&\phi\bigl(\ts y_{a'}\frac{\pd e_b}{\pd y_{a'}}\bigl((x_a)_{a\in A},(y_{a'})_{a'\in A_\rin}\bigr)\\
&-\ts y_{a'}\frac{\pd f_b}{\pd y_{a'}}\bigl((x_a)_{a\in A},(y_{a'})_{a'\in A_\rin}\bigr)\bigr)\d_\rin y_{a'},\end{aligned}
\\
\ga(\d_\rin \ti y_{b'})&=\sum_{a\in A}\begin{aligned}[t]&\phi\bigl(\ts g_{b'}^{-1}\frac{\pd g_{b'}}{\pd x_a}\bigl((x_a)_{a\in A},(y_{a'})_{a'\in A_\rin}\bigr)\\
&-\ts h_{b'}^{-1}\frac{\pd h_{b'}}{\pd x_a}\bigl((x_a)_{a\in A},(y_{a'})_{a'\in A_\rin}\bigr)\bigr)\d_\rin\Psi_{\exp}(x_a)\end{aligned}
\\
&+\sum_{a'\in A_\rin}\begin{aligned}[t]&\phi\bigl(\ts y_{a'}g_{b'}^{-1}\frac{\pd g_{b'}}{\pd y_{a'}}\bigl((x_a)_{a\in A},(y_{a'})_{a'\in A_\rin}\bigr)\\
&-\ts y_{a'}h_{b'}^{-1}\frac{\pd h_{b'}}{\pd y_{a'}}\bigl((x_a)_{a\in A},(y_{a'})_{a'\in A_\rin}\bigr)\bigr)\d_\rin y_{a'},\end{aligned}
\\
\de(\d_\rin\Psi_{\exp}(x_a))&=\d_\rin\ci\phi_\rin(\Psi_{\exp}(x_a)),\qquad \de(\d_\rin y_{a'})=\d_\rin\ci \phi_\rin(y_{a'}).
\end{split}
\label{cc7eq9}
\e
Hence if\/ $\bfC$ is finitely generated (or finitely presented) in the sense of Definition\/ {\rm\ref{cc4def13},} then ${}^b\Om_\bfC$ is finitely generated (or finitely presented).
\end{prop}

\begin{proof} Part (a) follows from (b) with $B=B_\rin=\es$, so it is enough to prove (b).

For (b), first note that by Definition \ref{cc7def3}, we have
\e
{}^b\Om_\bfC=\ban{\d_\rin c':c'\in \fC_\rin}_\R\ot_\R\fC\big/\an{\text{all relations \eq{cc7eq1}}}.
\label{cc7eq10}
\e
Now let $S\subseteq\fC_\rin$ be a generating subset, that is, every $c'\in\fC_\rin$ may be written as $c'=\Psi_g(s_1,\ldots,s_n)$ for some interior $g:[0,\iy)^n\ra[0,\iy)$ and $s_1,\ldots,s_n\in S$. For each $c'\in\fC_\rin\sm S$, choose interior $g^{c'}:[0,\iy)^{n^{c'}}\ra[0,\iy)$ and $s_1^{c'},\ldots,s_{n^{c'}}^{c'}\in S$ with $c'=\Psi_{g^{c'}}(s_1^{c'},\ldots,s_{n^{c'}}^{c'})$, using the Axiom of Choice. Then \eq{cc7eq1} gives
\e
\d_\rin c'=\ts\sum\limits_{i=1}^{n^{c'}}\Phi_{(g^{c'})^{-1}x_i\frac{\pd g^{c'}}{\pd x_i}}(s_1^{c'},\ldots,s_{n^{c'}}^{c'})\cdot \d_\rin s^{c'}_i.
\label{cc7eq11}
\e
Hence ${}^b\Om_\bfC$ is spanned over $\fC$ by $\d_\rin s$ for $s\in S$, and we may write ${}^b\Om_\bfC$ as $\an{\d_\rin s:s\in S}_\R\ot_\R\fC/($relations$)$. The relations required are the result of taking all relations \eq{cc7eq1} in \eq{cc7eq10}, and using \eq{cc7eq11} to substitute for all terms $\d_\rin c'$ for $c'\in\fC_\rin\sm S$, so that the resulting relation is a finite $\fC$-linear combination of $\d_\rin s$ for $s\in S$. A calculation using Definition \ref{cc4def2}(ii) shows that this relation is of the form \eq{cc7eq1} with $c_1',\ldots,c_n'\in S$. Hence
\e
{}^b\Om_\bfC\cong\ban{\d_\rin c':c'\in S}_\R\ot_\R\fC\big/\ban{\text{all relations \eq{cc7eq1} with $c_1',\ldots,c_n'\in S$}}.
\label{cc7eq12}
\e

Now let $\bfC$ be as in \eq{cc4eq13}. Then $\bfC$ has $\R$ type generators $\phi(x_a)$ for $a\in A$ and $[0,\iy)$ type generators $\phi_\rin(y_{a'})$ for $a'\in A_\rin$. But \eq{cc7eq12} only allows for $[0,\iy)$ type generators. To convert $\R$ type generators to $[0,\iy)$ type generators observe that, in a similar way to the proof of Proposition \ref{cc4prop3}, any smooth function $\R\ra\R$ or $\R\ra[0,\iy)$ may be extended to a smooth function $[0,\iy)^2\ra\R$ or $[0,\iy)^2\ra[0,\iy)$ by mapping $t\in\R$ to $(x,y)=(e^t,e^{-t})\in[0,\iy)^2$, which then satisfies $xy=1$. Thus, we may replace $\R$ type generators $\phi(x_a)\in\fC$ by pairs of $[0,\iy)$ type generators $\phi_\rin\ci\Psi_{\exp}(x_a),\phi_\rin\ci\Psi_{\exp}(-x_a)$ in $\fC_\rin$, satisfying the relation $(\phi_\rin\ci\Psi_{\exp}(x_a))\cdot(\phi_\rin\ci\Psi_{\exp}(-x_a))=1$. But applying $\d_\rin$ to this relation yields $\d_\rin(\phi_\rin\ci\Psi_{\exp}(x_a))+\d_\rin(\phi_\rin\ci\Psi_{\exp}(-x_a))=0$. Hence by taking $S=\bigl\{\phi_\rin\ci\Psi_{\exp}(\pm x_a):a\in A\bigr\}\cup\bigl\{\phi_\rin(y_{a'}):a'\in A_\rin\bigr\}$ in \eq{cc7eq12}, we see that
\e
{}^b\Om_\bfC\!\cong\!\frac{\ban{\d_\rin(\phi_\rin\ci\Psi_{\exp}(\pm x_a)):a\in A,\;\>
\d_\rin(\phi_\rin(y_{a'})):a'\in A_\rin}_\R\ot_\R\fC}{\begin{subarray}{l}\ts\bigl\langle\d_\rin(\phi_\rin\ci\Psi_{\exp}(x_a))+\d_\rin(\phi_\rin\ci\Psi_{\exp}(-x_a))=0,\;\> a\in A,\\ \ts\text{ other relations \eq{cc7eq1} with $c_i'=\phi_\rin\ci\Psi_{\exp}(\pm x_a),\phi_\rin(y_{a'})$}\bigr\rangle\end{subarray}}\,.
\label{cc7eq13}
\e

Clearly we can omit the generator $\d_\rin\phi_\rin\ci\Psi_{\exp}(-x_a)$ and 
the first relation on the bottom of \eq{cc7eq13} for $a\in A$. We can also omit the `$\phi_\rin$' in $\d_\rin(\phi_\rin\ci\Psi_{\exp}(\pm x_a))$, $
\d_\rin(\phi_\rin(y_{a'}))$, as these are only formal symbols anyway. Furthermore, in $\fF^{A,A_\rin}_\rin$ in \eq{cc4eq13}, the only relations between the generators $\Psi_{\exp}(\pm x_a),$ $a\in A$, $y_{a'}$, $a'\in A_\rin$, are generated by $\Psi_{\exp}(x_a)\cdot \Psi_{\exp}(-x_a)=1$ for $a\in A$. Hence in \eq{cc7eq13}, the `other relations' must all come from $\bfF^{B,B_\rin}$. But the relations in $\fF^{A,A_\rin}_\rin$ from $\bfF^{B,B_\rin}$ are generated by $\al_\rin(\Psi_{\exp}(\pm\ti x_b))=\be_\rin(\Psi_{\exp}(\pm\ti x_b))$, $b\in B$, and $\al_\rin(\ti y_{b'})=\be_\rin(\ti y_{b'})$, $b'\in B_\rin$. Putting all this together and using \eq{cc7eq8} yields
\e
{}^b\Om_\bfC\!\cong\!\frac{\ban{\d_\rin\Psi_{\exp}(x_a):a\in A,\;\>
\d_\rin y_{a'}:a'\in A_\rin}_\R\ot_\R\fC}{\begin{subarray}{l}\ts\bigl\langle
\d_\rin\ci \Psi_{\exp}\bigl(e_b\bigl((x_a)_{a\in A},(y_{a'})_{a'\in A_\rin}\bigr)\bigr)\\
\ts\quad -\d_\rin\ci \Psi_{\exp}\bigl(f_b\bigl((x_a)_{a\in A},(y_{a'})_{a'\in A_\rin}\bigr)\bigr)=0,\; b\in B,
\\ \ts \;\d_\rin\bigl(g_{b'}\bigl((x_a)_{a\in A},(y_{a'})_{a'\in A_\rin}\bigr)\bigr)\\
\ts \quad -\d_\rin\bigl(h_{b'}\bigl((x_a)_{a\in A},(y_{a'})_{a'\in A_\rin}\bigr)\bigr)=0,\; b'\in B_\rin\bigr\rangle\end{subarray}}\,.
\label{cc7eq14}
\e

By \eq{cc7eq1} and \eq{cc7eq9}, the relations in \eq{cc7eq14} are $\ga(\d_\rin\Psi_{\exp}(\ti x_b))=0$, $b\in B$, and $\ga(\d_\rin \ti y_{b'})=0$, $b'\in B_\rin$. Hence \eq{cc7eq14} shows that ${}^b\Om_\bfC\cong\Coker\ga$. Since this isomorphism is realized by mapping
$\d_\rin\Psi_{\exp}(x_a)\mapsto\d_\rin\ci\phi_\rin(\Psi_{\exp}(x_a)$, $\d_\rin y_{a'}\mapsto\d_\rin\ci \phi_\rin(y_{a'})$, which is $\de$ in \eq{cc7eq9}, we see that \eq{cc7eq7} is exact, proving~(b).
\end{proof}

We can compute the b-cotangent module of an interior $C^\iy$-ring with corners $\bfC$ modified by generators and relations, as in~\S\ref{cc45}.

\begin{prop}
\label{cc7prop2}
Suppose $\bfC\in\CRingscin,$ and we are given sets $A,\ab A_\rin,\ab B,\ab B_\rin,$ so that\/ $\bfC(x_a:a\in A)[y_{a'}:a'\in A_\rin]\in\CRingscin,$ and elements $f_b\in\fC(x_a:a\in A)[y_{a'}:a'\in A_\rin]$ and\/ $g_{b'},h_{b'}\in\fC(x_a:a\in A)[y_{a'}:a'\in A_\rin]_\rin$ for $b\in B$ and\/ $b'\in B_\rin$. Define $\bfD\in\CRingscin$ by
\e
\bfD=\bigl(\bfC(x_a:a\!\in\! A)[y_{a'}:a'\!\in\! A_\rin]\bigr)\big/\bigl(f_b\!=\!0,\,\, b\!\in\! B\bigr)\bigl[g_{b'}\!=\!h_{b'},\,\, b'\!\in\! B_\rin\bigr],
\label{cc7eq15}
\e
in the notation of Definition\/ {\rm\ref{cc4def8},} with natural morphism $\bs\pi:\bfC\ra\bfD$. Then we have an exact sequence in {\rm$\bfDmod$:}
\e
\text{\begin{small}$
\xymatrix@C=7.5pt{
{\begin{subarray}{l}\ts
\an{\d_\rin\Psi_{\exp}(\ti x_b),\, b\!\in\! B}_\R\!\ot_\R\!\fD \\
\ts \op\an{\d_\rin(\ti y_{b'}),\, b'\!\in\! B_\rin}_\R\!\ot_\R\!\fD\end{subarray}}  
\ar[rr]^(0.48){\al} && {\begin{subarray}{l}\ts {}^b\Om_\bfC\ot_\fC\fD\op \\
\ts\an{ \d_\rin\Psi_{\exp}(x_a),\, a\!\in\! A}_\R\!\ot_\R\!\fD \\
\ts \op\an{\d_\rin(y_{a'}),\, a'\!\in\! A_\rin}_\R\!\ot_\R\!\fD
\end{subarray}} \ar[rr]^(0.73)\be && {}^b\Om_\bfD \ar[r] & 0. }$\end{small}}
\label{cc7eq16}
\e
Here  we regard $\al$ as mapping~to
\begin{align*}
&{}^b\Om_{\bfC(x_a:a\in A)[y_{a'}:a'\in A_\rin]}\ot_{\fC(x_a:a\in A)[y_{a'}:a'\in A_\rin]}\fD={}^b\Om_{\bfC\ot_\iy\bfF^{A,A_\rin}}\ot_{\fC\ot_\iy\fF^{A,A_\rin}}\fD\\
&\cong\bigl({}^b\Om_\bfC\ot_\fC(\fC\ot_\iy\fF^{A,A_\rin})\op{}^b\Om_{\bfF^{A,A_\rin}}\ot_{\fF^{A,A_\rin}}(\fC\ot_\iy\fF^{A,A_\rin})\bigr)\ot_{\fC\ot_\iy\fF^{A,A_\rin}}\fD\\
&\cong{}^b\Om_\bfC\ot_\fC\fD\op{}^b\Om_{\bfF^{A,A_\rin}}\ot_{\fF^{A,A_\rin}}\fD\\
&\cong{}^b\Om_\bfC\ot_\fC\fD\op \an{p_a,\, a\in A}_\R\!\ot_\R\!\fD\op \an{q_{a'},\, a'\in A_\rin}_\R\!\ot_\R\!\fD,
\end{align*}
using Proposition\/ {\rm\ref{cc7prop1}(a)} in the last step, and then\/ $\al,\be$ act by
\e
\al=\begin{pmatrix}
\d_\rin\Psi_{\exp}(\ti x_b)\longmapsto \d_\rin\ci\Psi_{\exp}(f_b) \\ 
\d_\rin(\ti y_{b'})\longmapsto \d_\rin(g_{b'})-\d_\rin(h_{b'})
\end{pmatrix},\qquad
\be=\begin{pmatrix}
({}^b\Om_{\bs\pi})_* \\
\id \\ 
\id
\end{pmatrix}.
\label{cc7eq17}
\e
\end{prop}

\begin{proof} Proposition \ref{cc4prop6}(b) gives a coequalizer diagram for $\bfC$ in $\CRingscin$:
\e
\xymatrix@C=70pt{  \bfF^{D,D_\rin} \ar@<.1ex>@/^.2pc/[r]^{\bs\ga} \ar@<-.1ex>@/_.2pc/[r]_{\bs\de} & \bfF^{C,C_\rin} \ar[r]^{\bs\phi} & \bfC. }
\label{cc7eq18}
\e
We can extend this to a coequalizer diagram for $\bfD$:
\e
\xymatrix@C=50pt{  \bfF^{B\amalg D,B_\rin\amalg D_\rin} \ar@<.1ex>@/^.2pc/[r]^{\bs\ep} \ar@<-.1ex>@/_.2pc/[r]_{\bs\ze} & \bfF^{A\amalg C,A_\rin\amalg C_\rin} \ar[r]^{\bs\psi} & \bfD. }
\label{cc7eq19}
\e
Here $\bs\ep,\bs\ze$ act as $\bs\ga,\bs\de$ do on the generators from $D,D_\rin$, and as $\ep(\ti x_b)=\ti f_b$, $\ep_\rin(\ti y_{b'})=\ti g_{b'}$, $\ze(\ti x_b)=0$, $\ze_\rin(\ti y_{b'})=\ti h_{b'}$ on the generators from $B,B_\rin$, where $\ti f_b,\ti g_{b'},\ti h_{b'}$ are lifts of $f_b,g_{b'},h_{b'}$ along the surjective morphism $\bfF^{A\amalg C,A_\rin\amalg C_\rin}\ra\bfC(x_a:a\in A)[y_{a'}:a'\in A_\rin]$ induced by $\bs\phi$. Also $\bs\psi$ acts as $\bs\phi$ does on the generators from $C,C_\rin$, and as the identity on generators $x_a,y_{a'}$ from $A,A_\rin$.

Thus the effect of enlarging $\bfF^{C,C_\rin}$ to $\bfF^{A\amalg C,A_\rin\amalg C_\rin}$ is to enlarge $\bfC$ to $\bfC(x_a:a\in A)[y_{a'}:a'\in A_\rin]$, and the effect of enlarging $ \bfF^{D,D_\rin}$ to $\bfF^{B\amalg D,B_\rin\amalg D_\rin}$ is to quotient $\bfC(x_a:a\in A)[y_{a'}:a'\in A_\rin]$ by extra relations $\ep(\ti x_b)=\ze(\ti x_b)$ for $b\in B$, which means $f_b=0$, and $\ep_\rin(\ti y_{b'})=\ze_\rin(\ti y_{b'})$ for $b'\in B_\rin$, which means $g_{b'}=h_{b'}$. So the result is $\bfD$ in~\eq{cc7eq15}.

Now consider the commutative diagram:
\begin{equation*}
\text{\begin{footnotesize}$
\xymatrix@!0@C=54pt@R=18pt{
&&& *+[l]{\begin{subarray}{l}\ts  \bigl\langle \d_\rin\Psi_{\exp}(\ti v_d),\; d\in D,\\ \ts \d_\rin \ti w_{d'},\; d'\!\in\! D_\rin\bigr\rangle_\R\!\ot_\R\!\fD \end{subarray}} \ar[dddd]^(0.45){\bigl(\begin{smallmatrix}  0 \\  * \end{smallmatrix}\bigr)} \ar[rr]^(0.28){\bigl(\begin{smallmatrix}  0 \\  \id \end{smallmatrix}\bigr)}
&& {\begin{subarray}{l} \ts \bigl\langle \d_\rin\Psi_{\exp}(\ti x_b),\; b\in B,\\ \ts \d_\rin \ti y_{b'},\; b'\!\in\! B_\rin\bigr\rangle_\R\!\ot_\R\!\fD \op \\
\ts  \bigl\langle \d_\rin\Psi_{\exp}(\ti v_d),\; d\in D,\\ \ts \d_\rin \ti w_{d'},\; d'\!\in\! D_\rin\bigr\rangle_\R\!\ot_\R\!\fD
\end{subarray}} \ar[dddd]^(0.45){\bigl(\begin{smallmatrix}  * & 0 \\ * & * \end{smallmatrix}\bigr)} \\
\\
\\
\\
&&&
*+[l]{\begin{subarray}{l}\ts  \bigl\langle \d_\rin\Psi_{\exp}(x_a),\; a\in A,\\ \ts \d_\rin y_{a'},\; a'\!\in\! A_\rin\bigr\rangle_\R\!\ot_\R\!\fD \\
\ts  \bigl\langle \d_\rin\Psi_{\exp}(v_c),\; c\in C,\\ \ts \d_\rin w_{c'},\; c'\!\in\! C_\rin\bigr\rangle_\R\!\ot_\R\!\fD \end{subarray}} \ar[ddd]^(0.55){\bigl(\begin{smallmatrix}  0 & * \\ \id & 0 \end{smallmatrix}\bigr)} \ar@{=}[rr] &&
{\begin{subarray}{l}\ts  \bigl\langle \d_\rin\Psi_{\exp}(x_a),\; a\in A,\\ \ts \d_\rin y_{a'},\; a'\!\in\! A_\rin\bigr\rangle_\R\!\ot_\R\!\fD \\
\ts  \bigl\langle \d_\rin\Psi_{\exp}(v_c),\; c\in C,\\ \ts \d_\rin w_{c'},\; c'\!\in\! C_\rin\bigr\rangle_\R\!\ot_\R\!\fD \end{subarray}} \ar[ddd]^(0.6){(* \, *)}
\\
\\
\\
*+[r]{\begin{subarray}{l}\ts
\an{\d_\rin\Psi_{\exp}(\ti x_b),\, b\!\in\! B}_\R\!\ot_\R\!\fD \\
\ts \op\an{\d_\rin(\ti y_{b'}),\, b'\!\in\! B_\rin}_\R\!\ot_\R\!\fD\end{subarray}}  
\ar[rrr]^(0.61){\al} &&& {\begin{subarray}{l}\ts {}^b\Om_\bfC\ot_\fC\fD\op \\
\ts\an{ \d_\rin\Psi_{\exp}(x_a),\, a\!\in\! A}_\R\!\ot_\R\!\fD \\
\ts \op\an{\d_\rin(y_{a'}),\, a'\!\in\! A_\rin}_\R\!\ot_\R\!\fD
\end{subarray}} \ar[dd] \ar[rr]^(0.7){\be} && {}^b\Om_\bfD \ar[dd] \ar[r] & *+[l]{0} 
\\
\\
&&& 0 && 0.\!\! }$\end{footnotesize}}
\end{equation*}
The middle column is the direct sum of $-\ot_\fC\fD$ applied to \eq{cc7eq7} from Proposition \ref{cc7prop1}(b) for \eq{cc7eq18} and the identity map on $\an{\d_\rin\Psi_{\exp}(x_a),\d_\rin y_{a'}}_\R\ot_\R\fD$, and the right hand column is \eq{cc7eq7} for \eq{cc7eq19}. Thus the columns are exact. Some diagram-chasing in exact sequences and the snake lemma then shows the third row is exact. But this is \eq{cc7eq16}, so the proposition follows.
\end{proof}

Here is an analogue of Proposition~\ref{cc2prop7}:

\begin{prop}
\label{cc7prop3}
Let\/ $\bfC$ be an interior $C^\iy$-ring with corners and\/ $A\subseteq\fC$, $A_\rin\subseteq\fC_\rin$ be subsets, and write  $\bfD=\bfC(a^{-1}:a\in A)[a^{\prime -1}:a'\in A_\rin]$ for the localization in Definition\/ {\rm\ref{cc4def9},} with projection $\bs\pi:\bfC\ra\bfD$. Then $({}^b\Om_{\bs\pi})_*:{}^b\Om_\bfC\ot_\fC\fD\ra{}^b\Om_\bfD$ is
an isomorphism of\/ $\bfD$-modules.	
\end{prop}

\begin{proof} We may write $\bfD$ by adding generators and relations to $\bfC$ as in \eq{cc4eq16}, where we add an $\R$-generator $x_a$ and $\R$-relation $a\cdot x_a=1$ for each $a\in A$, and a $[0,\iy)$-generator $y_{a'}$ and $[0,\iy)$-relation $a'\cdot y_{a'}=1$ for each $a'\in A_\rin$. Then Proposition \ref{cc7prop2} gives an exact sequence \eq{cc7eq16} computing ${}^b\Om_\bfD$. But in this case the contributions from the generators $x_a,y_{a'}$ in the middle of \eq{cc7eq16} are exactly cancelled by the contributions from the relations $a\cdot x_a=1$, $a'\cdot y_{a'}=1$ on the left of \eq{cc7eq16}. That is, $\al$ in \eq{cc7eq16} induces an isomorphism from the $\ti x_b,\ti y_{b'}$ terms to the $x_a,y_{a'}$ terms. So by exactness we see that the top component of $\be$ in \eq{cc7eq16}, which is 
$({}^b\Om_{\bs\pi})_*:{}^b\Om_\bfC\ot_\fC\fD\ra{}^b\Om_\bfD$, is an isomorphism.
\end{proof}

We show that the reflection functors $\Pi_\rin^\Z,\Pi_\Z^\tf,\Pi_\tf^\sa,\Pi_\rin^\sa$ in Theorem \ref{cc4thm5} do not change b-cotangent modules, in a certain sense.

\begin{prop}
\label{cc7prop4}
Theorem\/ {\rm\ref{cc4thm5}} constructs $\Pi_\rin^\Z:\CRingscin\ra\CRingscZ$ left adjoint to $\inc:\CRingscZ\hookra\CRingscin$. For each\/ $\bfC\in\CRingscin,$ there is a natural surjective morphism $\bs\pi:\bfC\ra\Pi_\rin^\Z(\bfC)=:\bfD$ in $\CRingscin,$ the unit of the adjunction. Thus Definition\/ {\rm\ref{cc7def3}} gives a morphism $({}^b\Om_{\bs\pi})_*:\bs\pi_*({}^b\Om_\bfC)={}^b\Om_\bfC\ot_\fC\fD\ra{}^b\Om_\bfD$ in $\bfDmod$. Then $({}^b\Om_{\bs\pi})_*$ is an isomorphism.

The analogue holds for the other functors $\Pi_\Z^\tf,\Pi_\tf^\sa,\Pi_\rin^\sa$ in Theorem\/~{\rm\ref{cc4thm5}}.
\end{prop}

\begin{proof} The definition of $\bfD=\Pi_\rin^\Z(\bfC)$ in the proof of Theorem \ref{cc4thm5} involves inductively constructing a sequence $\bfC=\bfC^0\twoheadrightarrow\bfC^1 \twoheadrightarrow\bfC^2\twoheadrightarrow\cdots$ and setting $\bfD=\underrightarrow\lim_{n=0}^\iy\bfC^n$. Here $\bfC^{n+1}=\bfC^n/[g_{b'}=h_{b'}$, $b'\in B^n_\rin]$, where the relations are $g_{b'}=h_{b'}$ if $g_{b'},h_{b'}\in \fC^n_\rin$ with $\pi^\gp(g_{b'})=\pi^\gp(h_{b'})$. This is equivalent to the existence of $i_{b'}\in\fC^n_\rin$ with $g_{b'}i_{b'}=h_{b'}i_{b'}$. But then in ${}^b\Om_{\bfC^n}$ we have
\begin{equation*}
\d_\rin g_{b'}+\d_\rin i_{b'}=\d_\rin(g_{b'}i_{b'})=\d_\rin(h_{b'}i_{b'})=\d_\rin h_{b'}+\d_\rin i_{b'},
\end{equation*}
so $\d_\rin g_{b'}=\d_\rin h_{b'}$. Applying Proposition \ref{cc7prop2} with $\bfC^n$ in place of $\bfC$ and $A=A_\rin=B=\es$ gives an exact sequence in \hbox{$\fC^{n+1}$-mod:}
\begin{equation*}
\xymatrix@C=14pt{
\an{\d_\rin(\ti y_{b'}),\, b'\!\in\! B^n_\rin}_\R\!\ot_\R\!\fD 
\ar[rr]^(0.57){\al} && {}^b\Om_{\bfC^n}\ot_{\fC^n}\fC^{n+1} \ar[rr]^(0.6)\be && {}^b\Om_{\bfC^{n+1}} \ar[r] & 0. }
\end{equation*}
Here $\al$ maps $\d_\rin(\ti y_{b'})\mapsto \d_\rin g_{b'}-\d_\rin h_{b'}=0$ by \eq{cc7eq17}, so $\al=0$, and $\be$ is an isomorphism. Applying $-\ot_{\fC^{n+1}}\fD$ shows that the natural map ${}^b\Om_{\bfC^n}\ot_{\fC^n}\fD \ra{}^b\Om_{\bfC^{n+1}}\ot_{\fC^{n+1}}\fD$ is an isomorphism. Thus by induction on $n$ and $\bfC^0=\bfC$ we see that ${}^b\Om_{\bfC}\ot_{\fC}\fD\cong{}^b\Om_{\bfC^n}\ot_{\fC^n}\fD$ for all~$n\ge 0$.

Now $\bfD\cong\bfC/[\ti g_{b'}=\ti h_{b'}$, $b'\in\coprod_{n\ge 0} B^n_\rin]$, where the relations $\ti g_{b'}=\ti h_{b'}$ are lifts to $\bfC$ of all relations $g_{b'}=h_{b'}$ in $\bfC^{n+1}=\bfC^n/[g_{b'}=h_{b'}$, $b'\in B^n_\rin]$ for all $n\ge 0$. Since $\d_\rin g_{b'}=\d_\rin h_{b'}$ in ${}^b\Om_{\bfC^n}$, and hence in ${}^b\Om_{\bfC^n}\ot_{\fC^n}\fD$, and ${}^b\Om_{\bfC}\ot_{\fC}\fD\cong{}^b\Om_{\bfC^n}\ot_{\fC^n}\fD$, we see that $\d_\rin\ti g_{b'}=\d_\rin\ti h_{b'}$ in ${}^b\Om_{\bfC}\ot_{\fC}\fD$ for all such $\ti g_{b'},\ti h_{b'}$. Thus in \eq{cc7eq16} for $\bfC,\bfD$ with $A,A_\rin,B=\es$ and $B_\rin=\coprod_{n\ge 0} B^n_\rin$ we again have $\al=0$, so $\be=({}^b\Om_{\bs\pi})_*:{}^b\Om_\bfC\ot_\fC\fD\ra{}^b\Om_\bfD$ is an isomorphism.

For $\Pi_\Z^\tf$ a similar proof works. In this case $\bfC^{n+1}=\bfC^n/[g_{b'}=h_{b'}$, $b'\in B^n_\rin]$, with relations $g_{b'}=h_{b'}$ if $g_{b'},h_{b'}\in \fC^n_\rin$ with $\pi^\tf(g_{b'})=\pi^\tf(h_{b'})$, which (as $\bfC^n$ is integral) is equivalent to $g_{b'}^n=h_{b'}^n$ for some $n\ge 1$. But then $n\,\d_\rin g_{b'}=\d_\rin(g_{b'}^n)=\d_\rin(h_{b'}^n)=n\,\d_\rin h_{b'}$, so $\d_\rin g_{b'}=\d_\rin h_{b'}$, and the rest of the proof is unchanged.

For $\Pi_\tf^\sa$, now $\bfC^{n+1}=\bfC^n[y_{b'}:b'\in B_\rin^n]/[y_{b'}^{n_{b'}}=h_{b'}$, $b'\in B^n_\rin]$, for $h_{b'}$ in $\fC^n_\rin\subseteq(\fC^n_\rin)^\gp$ and $n_{b'}\ge 2$ for which there exists (unique) $c''\in (\fC^n_\rin)^\gp\sm \fC^n_\rin$ with $n_{b'}\cdot c''=h_{b'}$. That is, we introduce generators $y_{b'}$ and relations $y_{b'}^{n_{b'}}=h_{b'}$ in pairs. Since $y_{b'}^{n_{b'}}=h_{b'}$ implies that $n_{b'}\cdot\d_\rin y_{b'}=\d_\rin h_{b'}$, again \eq{cc7eq16} implies that ${}^b\Om_{\bfC^n}\ot_{\fC^n}\fC^{n+1} \ra{}^b\Om_{\bfC^{n+1}}$ is an isomorphism, and the rest of the argument is similar. As $\Pi_\rin^\sa=\Pi_\tf^\sa\ci\Pi_\Z^\tf\ci\Pi_\rin^\Z$, the result for $\Pi_\rin^\sa$ follows.
\end{proof}

\subsection{B-cotangent modules of pushouts}
\label{cc74}

The next theorem generalizes Theorem~\ref{cc2thm1}.

\begin{thm}
\label{cc7thm2}
Suppose we are given a pushout diagram in\/ $\CRingscin,\ab\CRingscZ,\ab\CRingsctf,$ or\/~{\rm$\CRingscsa$:}
\e
\begin{gathered}
\xymatrix@C=60pt@R=14pt{ *+[r]{\bfC} \ar[r]_{\bs\be} \ar[d]^{\bs\al} & *+[l]{\bfE} \ar[d]_{\bs\de} \\
*+[r]{\bfD} \ar[r]^{\bs\ga} & *+[l]{\bfF,\!} }
\end{gathered}
\label{cc7eq20}
\e
so that\/ $\bfF=\bfD\amalg_\bfC\bfE$. Then the following sequence of\/
$\fF$-modules is exact:
\e
\xymatrix@C=16pt{ {}^b\Om_\bfC\ot_{\fC,\ga\ci\al}\fF
\ar[rrr]^(0.5){({}^b\Om_{\bs\al})_*\op -({}^b\Om_{\bs\be})_*} &&&
{\raisebox{5pt}{$\begin{subarray}{l}\ts {}^b\Om_\bfD\ot_{\fD,\ga}\fF \op\\
\ts\,\,\,{}^b\Om_\bfE\ot_{\fE,\de}\fF \end{subarray}$}}
\ar[rrr]^(0.59){({}^b\Om_{\bs\ga})_*\op({}^b\Om_{\bs\de})_*} &&& {}^b\Om_\bfF \ar[r] & 0. }\!\!\!
\label{cc7eq21}
\e
Here\/ $({}^b\Om_{\bs\al})_*:{}^b\Om_\bfC\ot_{\fC,\ga\ci\al}\fF\ra
{}^b\Om_\bfD\ot_{\fD,\ga}\fF$ is induced by\/ ${}^b\Om_{\bs\al}:{}^b\Om_\bfC\ra
{}^b\Om_\bfD,$ and so on. Note the sign of\/ $-({}^b\Om_{\bs\be})_*$
in\/~\eq{cc7eq21}.
\end{thm}

\begin{proof} First suppose \eq{cc7eq20} is a pushout in $\CRingscin$. By ${}^b\Om_{\bs\psi\ci\bs\phi}={}^b\Om_{\bs\psi}\ci{}^b\Om_{\bs\phi}$ in Definition \ref{cc7def3} and commutativity of \eq{cc7eq20} we have
${}^b\Om_{\bs\ga}\ci{}^b\Om_{\bs\al}={}^b\Om_{\bs\ga\ci\bs\al}={}^b\Om_{\bs\de\ci\bs\be}={}^b\Om_{\bs\de}\ci
{}^b\Om_{\bs\be}:{}^b\Om_\bfC\ra{}^b\Om_\bfF$. Tensoring with $\fF$ then gives
$({}^b\Om_{\bs\ga})_*\ci({}^b\Om_{\bs\al})_*=({}^b\Om_{\bs\de})_*\ci({}^b\Om_{\bs\be})_*:
{}^b\Om_\bfC\ot_\fC\fF\ra{}^b\Om_\bfF$. As the composition of morphisms in
\eq{cc7eq21} is $({}^b\Om_{\bs\ga})_*\ci({}^b\Om_{\bs\al})_*-({}^b\Om_{\bs\de})_*\ci
({}^b\Om_{\bs\be})_*$, this implies \eq{cc7eq21} is a complex.

Apply Proposition \ref{cc4prop10} to \eq{cc7eq20}, using Proposition \ref{cc4prop6}(b) to choose coequalizer presentations \eq{cc4eq22}--\eq{cc4eq24} for $\bfC,\bfD,\bfE$ in $\CRingscin$, so we have a presentation \eq{cc4eq25} for $\bfF$. Thus Proposition \ref{cc7prop1}(b) gives exact sequences
\ea
\xymatrix@C=15pt{ {\begin{subarray}{l}\ts  \bigl\langle \d_\rin\Psi_{\exp}(\ti s_b),\; b\in B,\\ \ts \d_\rin \ti t_{b'},\; b'\!\in\! B_\rin\bigr\rangle_\R\!\ot_\R\!\fF \end{subarray}} 
\ar[r]^{\ep_1} & {\begin{subarray}{l}\ts \bigl\langle \d_\rin\Psi_{\exp}(s_a),\; a\in A,\\ \ts \d_\rin t_{a'},\; a'\!\in\! A_\rin\bigr\rangle_\R\!\ot_\R\!\fF \end{subarray}} \ar[r]^(0.64){\ze_1} & {}^b\Om_\bfC\ot_\fC\fF \ar[r] & 0, }
\label{cc7eq22}\\
\xymatrix@C=15pt{{\begin{subarray}{l}\ts  \bigl\langle \d_\rin\Psi_{\exp}(\ti u_d),\; d\in D,\\ \ts \d_\rin \ti v_{b'},\; d'\!\in\! D_\rin\bigr\rangle_\R\!\ot_\R\!\fF \end{subarray}} 
\ar[r]^{\ep_2} & {\begin{subarray}{l}\ts  \bigl\langle \d_\rin\Psi_{\exp}(u_c),\; c\in C,\\ \ts \d_\rin v_{c'},\; c'\!\in\! C_\rin\bigr\rangle_\R\!\ot_\R\!\fF \end{subarray}} \ar[r]^(0.64){\ze_2} & {}^b\Om_\bfD\ot_\fD\fF \ar[r] & 0, }
\label{cc7eq23}\\
\xymatrix@C=15pt{{\begin{subarray}{l}\ts  \bigl\langle \d_\rin\Psi_{\exp}(\ti w_f),\; f\in F,\\ \ts \d_\rin \ti x_{f'},\; f'\!\in\! F_\rin\bigr\rangle_\R\!\ot_\R\!\fF \end{subarray}} 
\ar[r]^{\ep_3} & {\begin{subarray}{l}\ts  \bigl\langle \d_\rin\Psi_{\exp}(w_e),\; e\in E,\\ \ts \d_\rin x_{e'},\; e'\!\in\! E_\rin\bigr\rangle_\R\!\ot_\R\!\fF \end{subarray}} \ar[r]^(0.64){\ze_3} & {}^b\Om_\bfE\ot_\fE\fF \ar[r] & 0, }
\label{cc7eq24}\\
\xymatrix@C=22pt{{\begin{subarray}{l}\ts  \bigl\langle \d_\rin\Psi_{\exp}(\ti y_h),\; h\in H,\\ \ts \d_\rin \ti z_{h'},\; h'\!\in\! H_\rin\bigr\rangle_\R\!\ot_\R\!\fF \end{subarray}} 
\ar[r]^{\ep_4} & {\begin{subarray}{l}\ts  \bigl\langle \d_\rin\Psi_{\exp}(y_g),\; g\in G,\\ \ts \d_\rin z_{g'},\; g'\!\in\! G_\rin\bigr\rangle_\R\!\ot_\R\!\fF \end{subarray}} \ar[r]^(0.7){\ze_4} & {}^b\Om_\bfF \ar[r] & 0, }
\label{cc7eq25}
\ea
where for \eq{cc7eq22}--\eq{cc7eq24} we have tensored \eq{cc7eq7} for
$\bfC,\bfD,\bfE$ with $\fF$, and we write $s_1,\ldots,\ti z_{h'}$ for the generators in $\fF^{A,A_\rin},\ldots,\fF^{H,H_\rin}$ as shown.

Now consider the diagram
\e
\begin{gathered}
\text{\begin{footnotesize}$
\!\!\!\!\!\!\!\!\!\!\!\!\xymatrix@!0@C=36pt@R=60pt{
*+[r]{\qquad\!\!\!\begin{subarray}{l}\ts \bigl\langle \d_\rin\Psi_{\exp}(s_a),\; a\in A,\\ \ts \d_\rin t_{a'},\; a'\!\in\! A_\rin\bigr\rangle_\R\!\ot_\R\!\fF \op \\
\ts  \bigl\langle \d_\rin\Psi_{\exp}(\ti u_d),\; d\in D,\\ \ts \d_\rin \ti v_{b'},\; d'\!\in\! D_\rin\bigr\rangle_\R\!\ot_\R\!\fF\op \\
\ts  \bigl\langle \d_\rin\Psi_{\exp}(\ti w_f),\; f\in F,\\ \ts \d_\rin \ti x_{f'},\; f'\!\in\! F_\rin\bigr\rangle_\R\!\ot_\R\!\fF
\end{subarray}}
\ar[rrrr]_(0.68){\begin{subarray}{l} \ep_4= \\ \left(\begin{subarray}{l} \phantom{-}\th_1 \,\, \ep_2 \,\, 0 
\\ -\th_2 \,\, 0 \,\, \ep_3 \end{subarray}\right)\end{subarray}}
\ar[d]^(0.73){\left(\ze_1 \,\, 0 \,\, 0 \right)} &&&&
*+[r]{\begin{subarray}{l}\ts  \bigl\langle \d_\rin\Psi_{\exp}(u_c),\; c\in C,\\ \ts \d_\rin v_{c'},\; c'\!\in\! C_\rin\bigr\rangle_\R\!\ot_\R\!\fF \op \\
\ts  \bigl\langle \d_\rin\Psi_{\exp}(w_e),\; e\in E,\\ \ts \d_\rin x_{e'},\; e'\!\in\! E_\rin\bigr\rangle_\R\!\ot_\R\!\fF
\end{subarray}}
 \ar[rrr]_(0.75){\ze_4}
\ar[d]^(0.55){\left(\begin{subarray}{l} \ze_2 \,\, 0 \\ 0 \,\,
\ze_3\end{subarray}\right)} &&& {}^b\Om_\bfF \ar[r]
\ar@{=}[d]^{{}\,\,\id_{{}^b\Om_\bfF}} & 0 
\\
{{}^b\Om_\bfC\ot_\fC\fF}
\ar[rrrr]^(0.5){\left(\begin{subarray}{l} ({}^b\Om_{\bs\al})_*\\
-({}^b\Om_{\bs\be})_*\end{subarray}\right)} &&&&
{\begin{subarray}{l}\ts {}^b\Om_\bfD\ot_\fD\fF\, \op\\
\ts{}^b\Om_\bfE\ot_\fE\fF \end{subarray}}
\ar[rrr]^(0.55){\left(({}^b\Om_{\bs\ga})_*\,\, ({}^b\Om_{\bs\de})_*\right)} &&&
{}^b\Om_\bfF \ar[r] & 0. }\!\!\!\!\!\!\!\!\!\!\!\!\!\!$\end{footnotesize}}
\end{gathered}
\label{cc7eq26}
\e
Here the top row is \eq{cc7eq25}, except that we have substituted for $G,G_\rin,H,H_\rin$ in terms of $A,\ldots,F_\rin$ as in \eq{cc4eq25}, and labelled the corresponding generators as in \eq{cc7eq22}--\eq{cc7eq24}. Then $\ep_4$ in \eq{cc7eq25} is written in matrix form in \eq{cc7eq26} as shown, where $\th_1,\th_2$ are induced by the morphism $\bfF^{A,A_\rin}\ra\bfF^{C,C_\rin}\ot_\iy\bfF^{E,E_\rin}\cong\bfF^{C\amalg E,C_\rin\amalg E_\rin}=\bfF^{G,G_\rin}$ lifting $\bfF^{A,A_\rin}\ra\bfD\ot_\iy\bfE$ chosen during the proof of Proposition \ref{cc4prop10}. The bottom line is the complex~\eq{cc7eq21}.

The left hand square of \eq{cc7eq26} commutes as $\ze_2\ci\ep_2=\ze_3\ci\ep_3=0$ by
exactness of \eq{cc7eq23}--\eq{cc7eq24} and
$\ze_2\ci\th_1=({}^b\Om_{\bs\al})_*\ci\ze_1$, $\ze_3\ci\th_2=({}^b\Om_{\bs\be})_*\ci\ze_1$ follow from the definition of $\th_1,\th_2$. The right hand square commutes as $\ze_4$ and $({}^b\Om_{\bs\ga})_*\ci\ze_2$ map $ 
\d_\rin\Psi_{\exp}(u_c)\mapsto \d_\rin\Psi_{\exp}\ci\ga\ci\chi(u_c)$ and $ 
\d_\rin v_{c'}\mapsto \d_\rin\ga_\rin\ci\chi_\rin(v_{c'})$, and similarly for $({}^b\Om_{\bs\de})_*\ci\ze_3$. Hence \eq{cc7eq26} commutes. The columns are surjective since
$\ze_1,\ze_2,\ze_3$ are surjective as \eq{cc7eq22}--\eq{cc7eq24} are
exact and identities are surjective.

The bottom right morphism $\bigl(({}^b\Om_{\bs\ga})_*\,({}^b\Om_{\bs\de})_*\bigr)$ in
\eq{cc7eq26} is surjective as $\ze_4$ is and the right hand square
commutes. Also surjectivity of the middle column implies that it
maps $\Ker\ze_4$ surjectively onto $\Ker\bigl(({}^b\Om_{\bs\ga})_*\,
({}^b\Om_{\bs\de})_*\bigr)$. But $\Ker\ze_4=\Im\ep_4$ as the top row is
exact, so as the left hand square commutes we see that
$\bigl(({}^b\Om_{\bs\al})_*\, -({}^b\Om_{\bs\be})_*\bigr){}^T$ surjects onto
$\Ker\bigl(({}^b\Om_{\bs\ga})_*\, ({}^b\Om_{\bs\de})_*\bigr)$, and the bottom row of
\eq{cc7eq26}, which is \eq{cc7eq21}, is exact. This proves the theorem for~$\CRingscin$.

Next suppose \eq{cc7eq20} is a pushout diagram in $\CRingscZ$. Let $\bs{\ti\fF}=\bfD\amalg_\bfC\bfE$ be the pushout in $\CRingscin$, which exists by Theorem \ref{cc4thm3}(b), with projections $\bs{\ti\ga}:\bfD\ra\bs{\ti\fF}$, $\bs{\ti\de}:\bfE\ra\bs{\ti\fF}$. Since $\Pi_\rin^\Z:\CRingscin\ra\CRingscZ$ in Theorem \ref{cc4thm5} preserves colimits as it is a left adjoint, and acts as the identity on $\bfC,\bfD,\bfE$ as they are integral, we see that $\bfF\cong \Pi_\rin^\Z(\bs{\ti\fF})$. The natural projection $\bs\pi:\bs{\ti\fF}\ra\Pi_\rin^\Z(\bs{\ti\fF})\cong\bfF$ is the morphism induced by \eq{cc7eq20} and the pushout property of $\bs{\ti\fF}$, so that $\bs\ga=\bs\pi\ci\bs{\ti\ga}$ and~$\bs\de=\bs\pi\ci\bs{\ti\de}$.

The first part implies that \eq{cc7eq21} with $\bs{\ti\fF},\bs{\ti\ga},\bs{\ti\de}$ in place of $\bfF,\bs\ga,\bs\de$ is exact. Applying $-\ot_{\ti\fF}\fF$ to this shows the following is exact:
\e
\xymatrix@C=14pt{ {}^b\Om_\bfC\!\ot_{\fC,\ga\ci\al}\!\fF
\ar[rrr]^(0.52){({}^b\Om_{\bs\al})_*\op -({}^b\Om_{\bs\be})_*} &&&
{\raisebox{5pt}{$\begin{subarray}{l}\ts {}^b\Om_\bfD\!\ot_{\fD,\ga}\!\fF \op\\
\ts\,\,\,{}^b\Om_\bfE\ot_{\fE,\de}\fF \end{subarray}$}}
\ar[rrr]^(0.54){({}^b\Om_{\bs{\ti\ga}})_*\op({}^b\Om_{\bs{\ti\de}})_*} &&& {}^b\Om_{\bs{\ti\fF}}\!\ot_{\ti\fF}\!\fF \ar[r] & 0. }\!\!\!
\label{cc7eq27}
\e
Now Proposition \ref{cc7prop4} shows that $({}^b\Om_{\bs\pi})_*:{}^b\Om_{\bs{\ti\fF}}\ot_{\ti\fF}\fF \ra{}^b\Om_\bfF$ is an isomorphism. Combining this with \eq{cc7eq27} and $({}^b\Om_{\bs\ga})_*=({}^b\Om_{\bs\pi})_*\ci({}^b\Om_{\bs{\ti\ga}})_*$, $({}^b\Om_{\bs\ga})_*=({}^b\Om_{\bs\pi})_*\ci({}^b\Om_{\bs{\ti\ga}})_*$ as $\bs\ga=\bs\pi\ci\bs{\ti\ga}$, $\bs\de=\bs\pi\ci\bs{\ti\de}$ shows that \eq{cc7eq21} is exact. The proofs for $\CRingsctf,\CRingscsa$ are the same, using $\Pi_\Z^\tf,\Pi_\tf^\sa$ in Theorem~\ref{cc4thm5}.
\end{proof}

\subsection{\texorpdfstring{Sheaves of $\bO_X$-modules, and (b-)cotangent sheaves}{Sheaves of OX-modules, and (b-)cotangent sheaves}}
\label{cc75} 

We define sheaves of $\bO_X$-modules on a $C^\iy$-ringed space with corners, as in~\S\ref{cc26}.

\begin{dfn}
\label{cc7def4}
Let $(X,\bO_X)$ be a $C^\iy$-ringed space with corners. A {\it sheaf of\/ $\bO_X$-modules}, or simply an $\bO_X$-{\it module}, $\cE$ on $X$, is a sheaf of $\O_X$-modules on $X$, as in Definition \ref{cc2def20}, where $\O_X$ is the sheaf of $C^\iy$-rings in $\bO_X$. A {\it morphism of sheaves of\/ $\bO_X$-modules\/} is a morphism of sheaves of $\O_X$-modules. Then $\bO_X$-modules form an abelian category, which we write as $\bOXmod$. An $\bO_X$-module $\cE$ is called a {\it vector bundle of rank\/} $n$ if we may cover $X$ by open $U\subseteq X$ with $\cE\vert_U\cong\O_X\vert_U\ot_\R\R^n$.
\end{dfn}

Pullback sheaves are defined analogously to Definition~\ref{cc2def21}.

\begin{dfn} 
\label{cc7def5}
Let $\bs f=(f,f^\sh,f^\sh_\rex):(X,\O_X,\O_{X}^\rex)\ra(Y,\O_Y,\O_Y^\rex)$ be a morphism in $\CRSc$, and $\cE$ be an $\bO_Y$-module. Define the {\it pullback\/} $\bs f^*(\cE)$ by $\bs f^*(\cE)=\uf^*(\cE)$, the pullback in Definition \ref{cc2def21} for $\uf=(f,f^\sh):(X,\O_X)\ra(Y,\O_Y)$. If $\phi:\cE\ra\cF$ is a morphism of $\bO_Y$-modules we have a morphism of $\bO_X$-modules~$\bs f^*(\phi)=\uf^*(\phi)=f^{-1}(\phi)\ot\id_{\O_X}:\bs f^*(\cE)\ra\bs f^*(\cF)$.
\end{dfn}

\begin{ex}
\label{cc7ex3}
Let $X$ be a manifold with corners, or with g-corners, and $E\ra X$ a vector bundle. Write $\bX=F_\Mangc^\CSchc(X)$ for the associated $C^\iy$-scheme with corners. For each open set $U\subseteq X$, write $\cE(U)=\Ga^\iy(E\vert_U)$, as a module over $\O_X(U)=C^\iy(U)$, and for open $V\subseteq U\subseteq X$ define $\rho_{UV}:\cE(U)\ra\cE(V)$ by $\rho_{UV}(e)=e\vert_V$. Then $\cE$ is the $\bO_X$-module of smooth sections of~$E$.

If $F\ra X$ is another vector bundle with $\bO_X$-module $\cF$, then vector bundle morphisms $E\ra F$ are in natural 1-1 correspondence with $\bO_X$-module morphisms $\cE\ra\cF$. If $f:W\ra X$ is a smooth map of manifolds with (g-)corners and $\bs f=F_\Mangc^\CSchc:\bW\ra\bX$ then the vector bundle $f^*(E)\ra W$ corresponds naturally to the $\bO_W$-module~$\bs f^*(\cE)$.
\end{ex}

\begin{dfn}
\label{cc7def6}
Let $\bX=(X,\bO_X)$ be a $C^\iy$-ringed space with corners, and $\uX=(X,\O_X)$ the underlying $C^\iy$-ringed space. Define the {\it cotangent sheaf\/ $T^*\bX$ of\/} $\bX$ to be the cotangent sheaf $T^*\uX$ of $\uX$ from Definition~\ref{cc2def22}.

If $U\subseteq X$ is open then we have an equality of $\bO_X\vert_U$-modules
\begin{equation*}
T^*\bU=T^*(U,\bO_X\vert_U)=T^*\bX\vert_U.
\end{equation*}

Let $\bs f=(f,f^\sh,f^\sh_\rex):\bX\ra\bY$ be a morphism of $C^\iy$-ringed spaces. Define $\Om_{\bs f}:\bs f^*(T^*\bY)\ra T^*\bX$ to be a morphism of the cotangent sheaves by $\Om_{\bs f}=\Om_\uf$, where~$\uf=(f,f^\sh):(X,\O_X)\ra(Y,\O_Y)$.
\end{dfn}

\begin{dfn}
\label{cc7def7}
Let $\bX=(X,\bO_X)$ be an interior local $C^\iy$-ringed space with corners. For each open $U\subset X$, let $\d_{U,\rin}:\O_X^\rin(U)\ra{}^b\Om_{\bfC_U}$ be the b-cotangent module associated to the interior $C^\iy$-ring with corners $\bfC_U=(\O_X(U),\O_X^\rin(U)\amalg \{0\})$, where $\amalg$ is the disjoint union. Here $\O_X^\rin(U)$ is the set of all elements of $\O^\rex_X(U)$ that are non-zero in each stalk $\O_{X,x}^\rex$ for all $x\in U$. Note that $\O_X^\rin$ is a sheaf of monoids on $X$. 

For each open $U\subseteq X$, the b-cotangent modules ${}^b\Om_{\bfC_U}$ define a presheaf $\bs{\cP}{}^bT^*\bX$ of $\bO_X$-modules, with restriction map ${}^b\Om_{\rho_{UV}}:{}^b\Om_{\bfC_U}\ra {}^b\Om_{\bfC_V}$ defined in Definition \ref{cc7def3}. Denote the sheafification of this presheaf the \/{\it b-cotangent sheaf}\/ ${}^bT^*\bX$ of $\bX$. Properties of sheafification imply that, for each open set $U\subseteq X$, there is a canonical morphism ${}^b\Om_{\bfC_U}\ra {}^bT^*\bX(U)$, and we have an equality of $\bO_X\vert_U$-modules
${}^bT^*(U,\bO_X\vert_U)={}^bT^*\bX\vert_U$. Also, for each $x\in X$, the stalk ${}^bT^*\bX\vert_x\cong {}^b\Om_{\bO_{X,x}}$, where ${}^b\Om_{\bO_{X,x}}$ is the b-cotangent module of the interior local $C^\iy$-ring with corners~$\bO_{X,x}$. 

For a morphism $\bs f=(f,f^\sh,f^\sh_\rex):\bX\ra\bY$ of interior local $C^\iy$-ringed spaces, we define the morphism of b-cotangent sheaves $ {}^b\Om_{\bs f}:\bs f^*({}^bT^*\bY)\ra {}^bT^*\bX$ by firstly noting that $\bs f^*({}^bT^*\bY)$ is the sheafification of the presheaf $\bs{\cP}(\bs f^*({}^bT^*\bY))$ acting by
\begin{equation*}
U\longmapsto\bs{\cP}(\bs f^*({}^bT^*\bY))(U)=
\ts\lim_{V\supseteq f(U)}{}^b\Om_{\bO_Y(V)}\ot_{\bO_Y(V)}\bO_X(U),
\end{equation*}
as in Definition \ref{cc2def22}. Then, following Definition \ref{cc2def22}, define a morphism of presheaves $\bs{\cP}{}^b\Om_{\bs f}:\bs{\cP}(\bs f^*({}^bT^*\bY))\ra\bs{\cP} T^*\bX$ on $X$ by
\begin{equation*}
(\bs{\cP}{}^b\Om_{\bs f})(U)=\ts\lim_{V\supseteq f(U)}({}^b\Om_{\bs{\rho}_{f^{-1}(V)\,U}\ci \bs f_\sh(V)})_*,
\end{equation*}
where $({}^b\Om_{\bs{\rho}_{f^{-1}(V)\,U}\ci \bs f_\sh(V)})_*:{}^b\Om_{\bO_Y(V)}\!\ot_{\O_Y(V)}\!\O_X(U)\!\ra\!{}^b\Om_{\bO_X(U)}\!=\!(\bs{\cP}{}^b T^*\bX)(U)$ is constructed as in Definition \ref{cc2def8} from the $C^\iy$-ring with corners morphisms $\bs f_\sh(V):\bO_Y(V)\ra\bO_X(f^{-1}(V))$ from $\bs f_\sh:\bO_Y\ra f_*(\bO_X)$ corresponding to $\bs f^\sh$ in $\bs f$ as in \eq{cc2eq8}, and $\bs{\rho}_{f^{-1}(V)\,U}:\bO_X(f^{-1}(V))\ra\bO_X(U)$ in $\bO_X$. Define ${}^b\Om_{\bs f}:\bs f^*({}^bT^*\bY)\ra {}^bT^*\bX$ to be the induced morphism of the associated sheaves.
\end{dfn}

We prove an analogue of Theorem~\ref{cc2thm5}(e):

\begin{prop}
\label{cc7prop5}
Let\/ $\bfC\in\CRingscin,$ and set\/ $\bX=\Speccin\bfC$. Then there is a canonical isomorphism ${}^bT^*\bX\cong\MSpec{}^b\Om_\bfC$ in $\bOXmod$.	
\end{prop}

\begin{proof} By Definitions \ref{cc2def23} and \ref{cc7def7}, $\MSpec{}^b\Om_\bfC$ and ${}^bT^*\bX$ are sheafifications of presheaves $\cP\MSpec{}^b\Om_\bfC,\cP {}^bT^*\bX$, where for open $U\subseteq X$ we have
\begin{equation*}
\cP\MSpec{}^b\Om_\bfC(U)={}^b\Om_\bfC\ot_\fC\O_X(U)\quad\text{and}\quad \cP{}^bT^*\bX(U)={}^b\Om_{\bO^\rin_X(U)},
\end{equation*}
writing $\bO^\rin_X(U)=(\O_X(U),\O_X^\rin(U)\amalg\{0\})$. We have morphisms $\bs\Xi_\bfC:\bfC\ra\bO^\rin_X(X)$ in $\CRingscin$ from Definition \ref{cc5def6}, and restriction $\bs\rho_{XU}:\bO^\rin_X(X)\ra\bO^\rin_X(U)$ from $\bO_X$, and so as in Definition \ref{cc7def2} a morphism of $\O_X(U)$-modules $\cP\rho(U):=(\bs\rho_{XU}\ci\bs\Xi_\bfC)_*:{}^b\Om_\bfC\ot_\fC\O_X(U)\ra{}^b\Om_{\bO^\rin_X(U)}$. This defines a morphism of presheaves $\cP\rho:\cP\MSpec{}^b\Om_\bfC\ra\cP{}^bT^*\bX$, and so sheafifying induces a morphism~$\rho:\MSpec{}^b\Om_\bfC\ra{}^bT^*\bX$.

The induced morphism on stalks at $x\in X$ is $\rho_x=(\bs\pi_x)_*:{}^b\Om_\bfC\ot_\fC\fC_x\ra{}^b\Om_{\bfC_x}$, where $\bs\pi_x:\bfC\ra\bfC_x\cong\bO_{X,x}$ is projection to the local $C^\iy$-ring with corners $\bfC_x$. But $\bfC_x$ is the localization \eq{cc4eq17}, so Proposition \ref{cc7prop3} says $(\bs\pi_x)_*$ is an isomorphism. Hence $\rho:\MSpec\Om_\fC\ra T^*\uX$ is a sheaf morphism which induces isomorphisms on stalks at all $x\in X$, so $\rho$ is an isomorphism.
\end{proof}

The next theorem shows that $T^*\bX,{}^bT^*\bX$ are good analogues of the \hbox{(b-)}\nobreak co\-tang\-ent bundles of manifolds with (g-)corners.

\begin{thm}
\label{cc7thm3}
{\bf(a)} Let\/ $X$ be a manifold with corners, and\/ $\bX=F_\Manc^\CSchc(X)$ the associated\/ $C^\iy$-scheme with corners. Then there is a natural, functorial isomorphism between the cotangent sheaf\/ $T^*\bX$ of\/ $\bX,$ and the $\bO_X$-module\/ $\bs{T^*X}$ associated to the vector bundle $T^*X\ra X$ in Example\/~{\rm\ref{cc7ex3}}.
\smallskip

\noindent{\bf(b)} Let\/ $X$ be a manifold with corners, or with g-corners, and\/ $\bX=F_\Mangc^\CSchc(X)$ the associated\/ $C^\iy$-scheme with corners. Then there is a natural, functorial isomorphism between the b-cotangent sheaf\/ ${}^bT^*\bX$ of\/ $\bX,$ and the $\bO_X$-module $\bs{{}^bT^*X}$ associated to the vector bundle ${}^bT^*X\ra X$ in Example\/~{\rm\ref{cc7ex3}}.
\end{thm}

\begin{proof} For (a), every $x\in X$ has arbitrarily small open neighbourhoods $x\in U\subseteq X$ such that $U$ is a manifold with faces, and $\pd U$ has finitely many boundary components (for example, we can take $U\cong\R^n_k$). Then we have
\begin{equation*}
\cP T^*\uX(U)=\Om_{\O_X(U)}=\Om_{C^\iy(U)}\cong \Ga^\iy(T^*X\vert_U)=\bs{T^*X}(U).
\end{equation*}
Here in the first step $\cP T^*\uX$ is the presheaf in Definition \ref{cc2def22}, whose sheafification is $T^*\bX=T^*\uX$ by Definition \ref{cc7def6}. In the second step $\O_X(U)=C^\iy(U)$ by Definition \ref{cc5def9}. The third step holds by Theorem \ref{cc7thm1}(a), and the fourth by Example \ref{cc7ex3}. Hence we have natural isomorphisms $\cP T^*\uX(U)\cong\bs{T^*X}(U)$ for arbitrarily small open neighbourhoods $U$ of each $x\in X$. These isomorphisms are compatible with the restriction morphisms in $\cP T^*\uX$ and $\bs{T^*X}$. Thus properties of sheafification give a canonical isomorphism $T^*\bX\cong\bs{T^*X}$. Part (b) is proved in the same way, using Definition \ref{cc7def7}, Example \ref{cc4ex6}, and Theorem~\ref{cc7thm1}(b).
\end{proof}

Here is a corners analogue of Theorem \ref{cc2thm6}. It is proved following the proof of Theorem \ref{cc2thm6} in \cite[\S 5.6]{Joyc9}, but using as input the identity ${}^b\Om_{\bs\psi\ci\bs\phi}={}^b\Om_{\bs\psi}\ci{}^b\Om_{\bs\phi}$ in Definition \ref{cc7def3} for (a), and Theorem \ref{cc7thm2} for (b). We restrict to {\it firm\/} $C^\iy$-schemes with corners so that Proposition \ref{cc5prop3} applies.

\begin{thm}
\label{cc7thm4}
{\bf(a)} Let\/ $\bs f:\bX\ra\bY$ and\/ $\bs g:\bY\ra\bZ$ be
morphisms in $\CSchcfiin$. Then in $\bOXmod$ we have
\begin{equation*}
{}^b\Om_{\bs g\ci\bs f}={}^b\Om_{\bs f}\ci \bs f^*({}^b\Om_{\bs g}):(\bs g\ci\bs f)^*(T^*\bZ)\longra T^*\bX.
\end{equation*}
 
\noindent{\bf(b)} Suppose we are given a Cartesian square in any of\/ $\CSchcfiin,$ $\CSchcfiZ,$ $\CSchcfitf$ or\/~{\rm $\CSchcto$:}
\begin{equation*}
\xymatrix@C=80pt@R=14pt{ *+[r]{\bW} \ar[r]_{\bs f} \ar[d]^{\bs e} & *+[l]{\bY} \ar[d]_{\bs h} \\
*+[r]{\bX} \ar[r]^{\bs g} & *+[l]{\bZ,\!} }
\end{equation*}
so that\/ $\bW=\bX\t_\bZ\bY$. Then the following is exact in $\bOWmod\!:$
\e
\xymatrix@C=13pt{ (\bs g\ci\bs e)^*({}^bT^*\bZ)
\ar[rrrr]^(0.5){\bs e^*({}^b\Om_{\bs g})\op -\bs f^*({}^b\Om_{\bs h})} &&&&
{\raisebox{5pt}{$\begin{subarray}{l}\ts \;\>\bs e^*({}^bT^*\bX) \\ \ts {}\op\bs f^*({}^bT^*\bY)\end{subarray}$}} \ar[rr]^(0.55){{}^b\Om_{\bs e}\op {}^b\Om_{\bs f}}
&& {}^bT^*\bW \ar[r] & 0.}
\label{cc7eq28}
\e
\end{thm}

\begin{ex}
\label{cc7ex4}
Define morphisms $g:X\ra Z$, $h:Y\ra Z$ in $\Mancin$ by $X=[0,\iy)$, $Y=*$, $Z=\R$, $g(x)=x^n$ for some $n>0$, and $h(*)=0$. Write $\bX,\bY,\bZ,\bs g,\bs h$ for the images of $X,\ldots,h$ under $F_\Mancin^\CSchcto$. Write $\bW=\bX\t_{\bs g,\bZ,\bs h}\bY$ for the fibre product in $\CSchcto$, which exists by Theorem \ref{cc5thm10} (it is also the fibre product in $\CSchcfitf,\CSchcfiZ$ and~$\CSchcfiin$).

Now ${}^bT^*X=\an{x^{-1}\d x}_\R$, ${}^bT^*Y=0$ and ${}^bT^*Z=\an{\d z}_\R$, where ${}^b\Om_g$ maps $\d z\mapsto nx^n\cdot x^{-1}\d x$. So by Theorem \ref{cc7thm3}(b), ${}^bT^*\bX,{}^bT^*\bZ$ are trivial line bundles, such that ${}^b\Om_{\bs g}$ is multiplication by $nx^n$ on $\bX$, and ${}^b T^*\bY=0$. Since $\bW$ is defined by $x^n=0$ in $\bX$, we see that $\bs e^*({}^b\Om_{\bs g})=0$ on $\bW$. Thus \eq{cc7eq28} implies that ${}^bT^*\bW\cong\bs e^*({}^bT^*\bX)$, so ${}^bT^*\bW$ is a rank 1 vector bundle on~$\bW$.

We might have hoped that if $\bW$ is an interior $C^\iy$-scheme with corners such that ${}^bT^*\bW$ is a vector bundle, and $\bW$ satisfies a few other obvious necessary conditions (e.g.\ $\bW$ should be Hausdorff, second countable, toric, locally finitely presented) then $\bW$ should be a manifold with g-corners. However, in this example ${}^bT^*\bW$ is a vector bundle, but $\bW$ is not a manifold with g-corners, and there do not seem to be obvious conditions to impose which exclude~$\bW$.
\end{ex}

\subsection{(B-)cotangent sheaves and the corner functor}
\label{cc76}

If $X$ is a manifold with (g-)corners, then as in \cite[\S 3.5]{Joyc6} and Remark \ref{cc3rem5}(d) we have an exact sequence of vector bundles on $C(X)$:
\e
\smash{\xymatrix@C=18pt{ 0 \ar[r] & {}^bT^*(C(X)) \ar[rr]^(0.45){{}^b\pi^*_T} && \Pi_X^*({}^bT^*X) \ar[rr]^(0.5){{}^bi^*_T} && {}^bN^*_{C(X)} \ar[r] & 0. }}
\label{cc7eq29}
\e
Note that ${}^b\pi^*_T$ is {\it not\/} the morphism ${}^b\d\Pi_X^*$ associated to $\Pi_X:C(X)\ra X$: this is not defined as $\Pi_X$ is not interior, and if it were it would map in the opposite direction. Both ${}^b\pi^*_T,{}^bi^*_T$ are specially constructed using the geometry of corners.

We will construct an analogue of \eq{cc7eq29} for (firm, interior) $C^\iy$-schemes with corners. First we define an analogue for (interior) $C^\iy$-rings with corners.

\begin{dfn}
\label{cc7def8}
Let $\bfC$ be an interior $C^\iy$-ring with corners, and $P\in\Pr_\bfC$ be a prime ideal in $\fC_\rex$. Then Example \ref{cc4ex4}(b) defines a $C^\iy$-ring with corners $\bfD=\bfC/\simc_P$, which is also interior, with (non-interior) projection $\bs\pi_P:\bfC\ra\bfD$. Writing $P_\rin=P\sm\{0\}\subset\fC_\rin$, define a monoid
\e
\fC_\rin^P=\fC_\rin\big/\bigl[c'=1:c'\in \fC_\rin\sm P_\rin\bigr]\cong \fC_\rin^\sh/\bigl[c''=1:c''\in \fC^\sh_\rin\sm P_\rin^\sh\bigr],
\label{cc7eq30}
\e
with projection $\la^P:\fC_\rin\ra \fC_\rin^P$. We will construct an exact sequence in~$\bfDmod$:
\e
\xymatrix@C=13pt{
{}^b\Om_\bfD \ar[rrr]^(0.4){{}^b\pi_{\bfC,P}} &&& {}^b\Om_\bfC \ot_\fC\fD \ar[rrr]^{{}^bi_{\bfC,P}} &&& \fC_\rin^P\ot_\N\fD \ar[rr] && 0.
}
\label{cc7eq31}
\e

Beware of a possible confusion here: in \eq{cc7eq30} we regard $\fC_\rin^P$ as a monoid under multiplication, with identity 1, but in \eq{cc7eq31} it is better to think of $\fC_\rin^P$ as a monoid under addition, so that the identity $1\in\fC_\rin^P$ is mapped to 0 in the $\bfD$-module $\fC_\rin^P\ot_\N\fD$, i.e. $1\ot 1=0$. This arises as for b-cotangent modules we think of $\d_\rin=\d\ci\log$, where $\log$ maps multiplication to addition and~$1\mapsto 0$. 

As in Example \ref{cc4ex4}(b), $\fD_\rex=\fC_\rex/\simc_P$, where for $c_1',c_2'\in\fC_\rex$ we have $c_1'\simc_P c_2'$ if either $c_1',c_2'\in P$ or $c_1'=\Psi_{\exp}(c)c_2'$ for $c$ in the ideal $\an{\Phi_i(P)}$ in $\fC$ generated by $\Phi_i(P)$. Thus $\fC_\rin\sm P_\rin$ is mapped surjectively to $\fD_\rin=\fD_\rex\sm\{0\}$. 

Define a map $\d'_\rin:\fD_\rin\ra{}^b\Om_\bfC \ot_\fC\fD$ by $\d'_\rin[c']=(\d_\rin c')\ot 1$, where $c'\in\fC_\rin\sm P_\rin$, and $[c']\in\fD_\rin$ is its $\simc_P$-equivalence class, and $\d_\rin c'\in{}^b\Om_\bfC$, so that $(\d_\rin c')\ot 1$ lies in ${}^b\Om_\bfC \ot_\fC\fD$. To show this is well defined, note that if $c_1',c_2'\in \fC_\rin\sm P_\rin$ with $[c_1']=[c_2']$ then $c_1'=\Psi_{\exp}(c)c_2'$ for $c\in\an{\Phi_i(P)}$. We may write $c=a_1\Phi_i(b_1')+\cdots+a_n\Phi_i(b_n')$ for $a_1,\ldots,a_n\in\fC$ and $b_1',\ldots,b_n'\in P$. Hence
\begin{equation*}
\d_\rin c_1'=\Psi_{\exp}(c)\cdot\d_\rin c_2'+\ts\sum_{j=1}^nc_2'\Phi_i(b_j')\bigl(\d_\rin\Psi_{\exp}(a_j)+a_j\cdot\d_\rin b_j'\bigr)
\end{equation*}
in ${}^b\Om_\bfC$, using Definition \ref{cc7def3}. Applying $-\ot 1:{}^b\Om_\bfC\ra{}^b\Om_\bfC \ot_\fC\fD$, we have $\pi_P(\Psi_{\exp}(c))=1$ as $\fD=\fC/I$ with $c\in I$, and $\pi_P(c_2'\Phi_i(b_j'))=0$ as $b_j'\in P$ so $c_2'\Phi_i(b_j')\in\an{\Phi_i(P)}$. Hence $\d_\rin c_1'\ot 1=1\cdot (\d_\rin c_2'\ot 1)+0=\d_\rin c_2'\ot 1$, so $\d'_\rin$ is well defined. As $\d_\rin$ is a b-derivation, $\d_\rin'$ is too. Thus by the universal property of ${}^b\Om_\bfD$ there is a unique morphism ${}^b\pi_{\bfC,P}:{}^b\Om_\bfD\ra{}^b\Om_\bfC \ot_\fC\fD$ in $\bfDmod$ with $\d'_\rin={}^b\pi_{\bfC,P}\ci\d_\rin$, as required in~\eq{cc7eq31}.

Next, define a map $\d''_\rin:\fC_\rin\ra\fC_\rin^P\ot_\N\fD$ by $\d''_\rin(c')=\la^P(c')\ot 1$. Regard $\fC_\rin^P\ot_\N\fD$ as a $\bfC$-module via $\pi_P:\fC\ra\fD$, with trivial $\fC$-action on the $\fC_\rin^P$ factor, so that $c\cdot(\bar c\ot d)=\bar c\ot (\pi_P(c)d)$ for $c\in\fC$, $d\in\fD$ and $\bar c\in\fC_\rin^P$. We claim that $\d''_\rin$ is a b-derivation. To prove this, let $g:[0,\iy)^n\ra[0,\iy)$ be interior and $c'_1,\ldots,c'_n\in\fC_\rin$. Write $g(x_1,\ldots,x_n)=x_1^{a_1}\cdots x_n^{a_n}\exp(f(x_1,\ldots,x_n))$, for $a_1,\ldots,a_n\ge 0$ and $f:[0,\iy)^n\ra\R$ smooth. Then
\ea
&\d_\rin''\bigl(\Psi_g(c'_1,\ldots,c'_n)\bigr)=\d_\rin''\bigl(c_1^{\prime a_1}\cdots c_n^{\prime a_n}\Psi_{\exp}(f(x_1,\ldots,x_n))\bigr)
\nonumber\\
&=\la^P\bigl(c_1^{\prime a_1}\cdots c_n^{\prime a_n}\Psi_{\exp}(f(x_1,\ldots,x_n))\bigr)\ot 1
\label{cc7eq32}\\
&=\ts\sum_{j=1}^na_j\la^P(c_j')\ot 1+\la^P(\Psi_{\exp}(f(x_1,\ldots,x_n))\bigr)\ot 1=\ts\sum_{j=1}^na_j\d_\rin''(c_j')+0,
\nonumber
\ea
using in the third step that $\la^P$ is a monoid morphism, and in the fourth that $\la^P(\Psi_{\exp}(f(x_1,\ldots,x_n)))=1$ by \eq{cc7eq30} as $\Psi_{\exp}(f(x_1,\ldots,x_n))\in\fC_\rin\sm P_\rin$, so $\la^P(\Psi_{\exp}(f(x_1,\ldots,x_n)))\ot 1=0$ as $-\ot 1$ maps $1\mapsto 0$. Now
\begin{equation*}
g^{-1}x_j\frac{\pd g}{\pd x_j}=a_j+x_j\frac{\pd f}{\pd x_j}(x_1,\ldots,x_n).
\end{equation*}
Hence
\e
\Phi_{g^{-1}x_i\frac{\pd g}{\pd x_i}}(c'_1,\ldots,c'_n)=a_j+\Phi_i(c_j')\Phi_{\frac{\pd f}{\pd x_j}}(c_1',\ldots,c_n').
\label{cc7eq33}
\e
Combining \eq{cc7eq32} with $\pi_P$ applied to \eq{cc7eq33} we see that
\ea
\d_\rin''\bigl(\Psi_g(c'_1,\ldots,c'_n)\bigr)&=\ts\sum_{j=1}^n\pi_P\bigl(\Phi_{g^{-1}x_j\frac{\pd g}{\pd x_j}}(c'_1,\ldots,c'_n)\bigr)\cdot \d''_\rin c'_j
\label{cc7eq34}\\
&-\ts\sum_{j=1}^n\pi_P\bigl(\Phi_i(c_j')\bigr)\pi_P\bigl(\Phi_{\frac{\pd f}{\pd x_j}}(c_1',\ldots,c_n')\bigr)\cdot\d_\rin''c_j'.
\nonumber
\ea
For each $j=1,\ldots,n$, if $c_j'\in P$ then $\Phi_i(c_j')\in\an{\Phi_i(P)}$ so $\pi_P(\Phi_i(c_j'))=0$, and if $c_j'\notin P$ then $\la^P(c_j')=1$ so $\d_\rin''c_j'=0$. Hence the second line of \eq{cc7eq34} is zero. So $\d_\rin''$ satisfies \eq{cc7eq1}, and is a b-derivation.

Thus there exists a unique morphism ${}^bi'_{\bfC,P}:{}^b\Om_\bfC\ra\fC_\rin^P\ot_\N\fD$ in $\bfCmod$ with $\d_\rin''={}^bi'_{\bfC,P}\ci\d_\rin$, by the universal property of ${}^b\Om_\bfC$. Applying $-\ot_\fC\fD$ to ${}^bi'_{\bfC,P}$ and noting that $\fD\ot_\fC\fD\cong\fD$ as $\pi_P:\fC\ra\fD$ is surjective gives a morphism ${}^bi_{\bfC,P}:{}^b\Om_\bfC\ot_\fC\fD\ra\fC_\rin^P\ot_\N\fD$, as required in \eq{cc7eq31}. Proposition \ref{cc7prop6} shows \eq{cc7eq31} is an exact sequence.

Now suppose also that $\bs\phi:\bfC\ra\bfE$ is a morphism in $\CRingscin$, and $Q\subset \fE_\rex$ is a prime ideal with $P=\phi_\rex^{-1}(Q)$, and write $\bfF=\bfE/\simc_Q$. As $\phi_\rex(P)\subseteq Q$ there is a unique morphism $\bs\psi:\bfD\ra\bfF$ with $\bs\psi\ci\bs\pi_P=\bs\pi_Q\ci\bs\phi:\bfC\ra\bfF$. From the definitions, it is not difficult to show that the following diagram commutes:
\e
\begin{gathered}
\xymatrix@C=13pt@R=18pt{
{}^b\Om_\bfD\ot_\fD\fF \ar[d]^{({}^b\Om_{\bs\psi})_*} \ar[rrr]_(0.45){{}^b\pi_{\bfC,P}\ot\id} &&& {}^b\Om_\bfC \ot_\fC\fF  \ar[d]^{{}^b\Om_{\bs\phi}\ot\id} \ar[rrr]_{{}^bi_{\bfC,P}\ot\id} &&& \fC_\rin^P\ot_\N\fF  \ar[d]^{(\phi_\rin)_*\ot\id} \ar[rr] && 0 \\
{}^b\Om_\bfF \ar[rrr]^(0.4){{}^b\pi_{\bfE,Q}} &&& {}^b\Om_\bfE \ot_\fE\fF \ar[rrr]^{{}^bi_{\bfE,Q}} &&& \fE_\rin^Q\ot_\N\fF \ar[rr] && 0.\! }
\end{gathered}
\label{cc7eq35}
\e

\end{dfn}

\begin{prop}
\label{cc7prop6}
Equation\/ \eq{cc7eq31} is an exact sequence.
\end{prop}

\begin{proof} Work in the situation of Definition \ref{cc7def8}. Then ${}^b\Om_\bfD$ is generated over $\fD$ by $\d_\rin[c']$ for $c'\in\fC_\rin\sm P_\rin$. But
\begin{align*}
{}^bi_{\bfC,P}\ci{}^b\pi_{\bfC,P}(\d_\rin[c'])&={}^bi_{\bfC,P}\ci\d'_\rin[c']={}^bi_{\bfC,P}((\d_\rin c')\ot 1)={}^bi'_{\bfC,P}(\d_\rin c')\\
&=\d''_\rin(c')=\la^P(c')\ot 1=1\ot 1=0.
\end{align*}
Hence ${}^bi_{\bfC,P}\ci{}^b\pi_{\bfC,P}=0$, so \eq{cc7eq31} is a complex.

By Proposition \ref{cc4prop6}(b), $\bfC$ fits into a coequalizer diagram in $\CRingscin$:
\e
\xymatrix@C=50pt{  \bfF^{B,B_\rin} \ar@<.1ex>@/^.3pc/[r]^{\bs\al} \ar@<-.1ex>@/_.3pc/[r]_{\bs\be} & \bfF^{A,A_\rin} \ar[r]^{\bs\phi} & \bfC. }
\label{cc7eq36}
\e
Define prime ideals $Q=\phi_\rex^{-1}(P)\subset\fF^{A,A_\rin}_\rex$ and
\begin{equation*}
R=\al_\rex^{-1}(Q)=(\phi_\rex\ci\al_\rex)^{-1}(P)=(\phi_\rex\ci\be_\rex)^{-1}(P)=\be_\rex^{-1}(Q)	\subset\fF^{B,B_\rin}_\rex.
\end{equation*}
As $\bfF^{A,A_\rin},\bfF^{B,B_\rin}$ are free, prime ideals in $\fF^{A,A_\rin}_\rex,\fF^{B,B_\rin}_\rex$ are parametrized by subsets of $A_\rin,B_\rin$. Thus there are unique subsets $\dot A_\rin\subseteq A_\rin$, $\dot B_\rin\subseteq B_\rin$ with $Q=\an{y_{a'}:a'\in\dot A_\rin}$ and $R=\an{\ti y_{b'}:b'\in\dot B_\rin}$. Set $\ddot A_\rin=A_\rin\sm \dot A_\rin$ and $\ddot B_\rin=B_\rin\sm \dot B_\rin$. Then
\begin{align*}
\bfF^{A,A_\rin}/\simc_Q&=\bfF^{A,A_\rin}/[y_{a'}=0:a'\in\dot A_\rin]\cong \bfF^{A,\ddot A_\rin},\\
\bfF^{B,B_\rin}/\simc_R&=\bfF^{B,B_\rin}/[\ti y_{b'}=0:b'\in\dot B_\rin]\cong \bfF^{B,\ddot B_\rin}.
\end{align*}
Thus equation \eq{cc7eq36} descends to a coequalizer diagram for $\bfD$:
\e
\xymatrix@C=50pt{  \bfF^{B,\ddot B_\rin} \ar@<.1ex>@/^.3pc/[r]^{\bs\ga} \ar@<-.1ex>@/_.3pc/[r]_{\bs\de} & \bfF^{A,\ddot A_\rin} \ar[r]^{\bs\psi} & \bfD. }
\label{cc7eq37}
\e

Consider the diagram in $\bfDmod$:
\ea
\begin{gathered}
\text{\begin{footnotesize}$\xymatrix@!0@C=29pt@R=25pt{
{\begin{subarray}{l}\ts  \bigl\langle \d_\rin\Psi_{\exp}(\ti x_b),\, b\!\in\! B,\\ \ts \d_\rin \ti y_{b'},\, b'\!\in\!\ddot B_\rin\bigr\rangle_\R\!\ot_\R\!\fD \end{subarray}} \ar[dd]^{\bs\ga_*-\bs\de_*} \ar[rrrr]^(0.48){\inc} &&&& {\begin{subarray}{l}\ts  \bigl\langle \d_\rin\Psi_{\exp}(\ti x_b),\, b\!\in\! B,\\ \ts \d_\rin \ti y_{b'},\, b'\!\in\! B_\rin\bigr\rangle_\R\!\ot_\R\!\fD \end{subarray}}  \ar[dd]^{\bs\al_*-\bs\be_*} \ar[rrrr]^(0.48){\text{project}} &&&& {\begin{subarray}{l}\ts  \bigl\langle\d_\rin \ti y_{b'},\, b'\!\in\!\dot B_\rin\bigr\rangle_\R\!\ot_\R\!\fD \end{subarray}}  \ar[dd]^{\bs\al_*-\bs\be_*} \ar[rr] && 0 \\
\\
{\begin{subarray}{l}\ts  \bigl\langle \d_\rin\Psi_{\exp}(x_a),\, a\!\in\! A,\\ \ts \d_\rin y_{a'},\, a'\!\in\!\ddot A_\rin\bigr\rangle_\R\!\ot_\R\!\fD \end{subarray}} \ar[dd]^{\bs\psi_*} \ar[rrrr]^(0.48){\inc} &&&& {\begin{subarray}{l}\ts  \bigl\langle \d_\rin\Psi_{\exp}(x_a),\, a\!\in\! A,\\ \ts \d_\rin y_{a'},\, a'\!\in\! A_\rin\bigr\rangle_\R\!\ot_\R\!\fD \end{subarray}} \ar[dd]^{\bs\phi_*} \ar[rrrr]^(0.48){\text{project}} &&&& {\begin{subarray}{l}\ts  \bigl\langle \d_\rin y_{a'},\, a'\!\in\!\dot A_\rin\bigr\rangle_\R\!\ot_\R\!\fD \end{subarray}} \ar[dd]^{\d_\rin y_{a'}\mapsto [y_{a'}]} \ar[rr] && 0 \\
\\
{}^b\Om_\bfD \ar[d] \ar[rrrr]^(0.5){{}^b\pi_{\bfC,P}} &&&& {}^b\Om_\bfC \ot_\fC\fD  \ar[d] \ar[rrrr]^{{}^bi_{\bfC,P}} &&&& \fC_\rin^P\ot_\N\fD  \ar[d] \ar[rr] && 0 \\
0 &&&& 0 &&&& 0.\!\!}$\end{footnotesize}}
\end{gathered}
\nonumber
\\[-15pt]
\label{cc7eq38}
\ea
Here the first column is \eq{cc7eq7} for \eq{cc7eq37}, and so is exact. The maps $\bs\ga_*,\bs\de_*,\bs\psi_*$ are defined as in \eq{cc7eq8}--\eq{cc7eq9} using $\bs\ga,\bs\de,\bs\psi$. The second column is $-\ot_\fC\fD$ applied to \eq{cc7eq7} for \eq{cc7eq36} in the same way, and so is exact. 

For the third column of \eq{cc7eq38}, equation \eq{cc7eq36} and the definition of $Q,R$ induce a coequalizer diagram in~$\Mon$:
\begin{equation*}
\xymatrix@C=40pt{  {\begin{subarray}{l}\ts \fF^{B,B_\rin}_\rin\big/\bigl[\ti y'=1:\\ \ts \ti y'\in \fF^{B,B_\rin}_\rin\!\sm\! R_\rin\bigr]\end{subarray}}
 \ar@<.1ex>@/^.3pc/[r]^{(\al_\rin)_*} \ar@<-.1ex>@/_.3pc/[r]_{(\be_\rin)_*} & {\begin{subarray}{l}\ts \fF^{A,A_\rin}_\rin\big/\bigl[y'=1:\\ \ts y'\!\in\! \fF^{A,A_\rin}_\rin\sm Q_\rin\bigr]\end{subarray}} \ar[r]^{(\phi_\rin)_*} & {\begin{subarray}{l}\ts  \fC_\rin^P\!=\!\fC_\rin\big/\bigl[c'\!=\!1\!:\\ \ts c'\in \fC_\rin\sm P_\rin\bigr].\end{subarray}} }
\end{equation*}
Applying $-\ot_\N\fD$, which preserves coequalizers, gives an exact sequence 
\begin{equation*}
\xymatrix@C=18pt{  {\begin{subarray}{l}\ts \fF^{B,B_\rin}_\rin\big/\bigl[\ti y'=1:\\ \ts \ti y'\in \fF^{B,B_\rin}_\rin\!\sm\! R_\rin\bigr]\!\ot_\N\!\fD\end{subarray}}
\ar[rr]^(0.45){(\al_\rin)_*-(\be_\rin)_*} && {\begin{subarray}{l}\ts \fF^{A,A_\rin}_\rin\big/\bigl[y'=1:\\ \ts y'\!\in\! \fF^{A,A_\rin}_\rin\sm Q_\rin\bigr]\!\ot_\N\!\fD\end{subarray}} \ar[r]^(0.6){(\phi_\rin)_*} & \fC_\rin^P\!\ot_\N\!\fD \ar[r] & 0. }
\end{equation*}
But this is isomorphic to the third column of \eq{cc7eq38}, which is therefore exact.

The top two squares of \eq{cc7eq38} commute by the relation between \eq{cc7eq36} and \eq{cc7eq37}. The bottom two squares commute as they are \eq{cc7eq35} for $\bs\phi:\bfF^{A,A_\rin}\ra\bfC$, using \eq{cc7eq6}. Thus \eq{cc7eq38} commutes. The columns are all exact, and the top two rows are obviously exact as $A_\rin=\dot A_\rin\amalg\ddot A_\rin$, $B_\rin=\dot B_\rin\amalg\ddot B_\rin$. The third row is a complex as above. Therefore by some standard diagram-chasing in exact sequences, the third row, which is \eq{cc7eq31}, is exact.
\end{proof}

\begin{ex}
\label{cc7ex5}
Set $\bfC=\R(x)[y]/[y=y^2$, $e^xy=y]$, and let $P=\an{y}\subset\fC_\rex$. Then $\bfD=\bfC/\simc_P\cong\R(x)$. Proposition \ref{cc7prop2} implies that
\begin{equation*}
{}^b\Om_\bfC\!\cong\!\frac{\an{\d_\rin\Psi_{\exp}(x),\d_\rin y}_\R\ot_\R\fC}{\d_\rin y\!=\!2\d_\rin y, \;\d_\rin\Psi_{\exp}(x)\!+\!\d_\rin y\!=\!\d_\rin y}\!=\!0,\;\> {}^b\Om_\bfD\!\cong\!\an{\d_\rin\Psi_{\exp}(x)}_\R\ot_\R\fD.
\end{equation*}
Also \eq{cc7eq30} implies that $\fC_\rin^P=\{[1],[y]\}$ with $[y]^2=[y]$, so $\fC_\rin^P\ot_\N\fD=0$. Thus in this case \eq{cc7eq31} becomes the exact sequence~$\fD\ra 0\ra 0\ra 0$.
\end{ex}

Note in particular that in Example \ref{cc7ex5}, in contrast to \eq{cc7eq29}, if we add `$0\longra$' to the left of \eq{cc7eq31}, it may no longer be exact. However, this is exact if we also assume $\bfC$ is toric.

\begin{prop}
\label{cc7prop7}
In Definition\/ {\rm\ref{cc7def8},} suppose\/ $\bfC$ is toric. Then we may extend\/ \eq{cc7eq31} to an exact sequence in\/~{\rm$\bfDmod$:}
\e
\xymatrix@C=9pt{
0 \ar[rr] && {}^b\Om_\bfD \ar[rrr]^(0.4){{}^b\pi_{\bfC,P}} &&& {}^b\Om_\bfC \ot_\fC\fD \ar[rrr]^{{}^bi_{\bfC,P}} &&& \fC_\rin^P\ot_\N\fD \ar[rr] && 0.
}
\label{cc7eq39}
\e
\end{prop}

\begin{proof} We will explain how to modify the proof of Proposition \ref{cc7prop6}. Firstly, as $\bfC$ is toric we may take \eq{cc7eq36} to be a coequalizer diagram in $\CRingscto$ rather than in $\CRingscin$, which implies $A_\rin,B_\rin$ are finite. We choose it so that $\md{A_\rin}$ is as small as possible. This implies that we have a 1-1 correspondence 
\e
\begin{split}
\phi_\rin^\sh:A_\rin&\cong\bigr\{[y_{a'}]:a'\in A_\rin\bigr\}\,{\buildrel\cong\over\longra}\\
&\bigl\{[c']\in \fC_\rin^\sh\sm\{1\}:\text{$[c']\ne [c_1'][c_2']$ for $[c_1'],[c_2']\in\fC_\rin^\sh\sm\{1\}$}\bigr\},
\end{split}
\label{cc7eq40}
\e
where the second line is the unique minimal set of generators of~$\fC_\rin^\sh$. 

Having fixed $A_\rin$, we choose \eq{cc7eq36} so that $\md{B_\rin}$ is also as small as possible, that is, we define $\bfC$ using the fewest possible $[0,\iy)$-type relations in $\CRingscto$. From \eq{cc7eq36} we may form diagrams
\begin{gather}
\xymatrix@C=50pt{  \N^{B_\rin}=(\fF^{B,B_\rin}_\rin)^\sh \ar@<.1ex>@/^.3pc/[r]^{\al_\rin^\sh} \ar@<-.1ex>@/_.3pc/[r]_{\be_\rin^\sh} & \N^{A_\rin}=(\fF^{A,A_\rin}_\rin)^\sh \ar[r]^(0.63){\phi_\rin^\sh} & \fC_\rin^\sh, }
\label{cc7eq41}\\
\xymatrix@C=16pt{  0 \ar[r] & {\!\begin{subarray}{l}\ts \qquad\Q^{B_\rin}= \\ \ts(\fF^{B,B_\rin}_\rin)^\sh\!\ot_\N\!\Q\end{subarray}\!} \ar[rr]^(0.48){\al_\rin^\sh-\be_\rin^\sh} && {\!\begin{subarray}{l}\ts \qquad\Q^{A_\rin}= \\ \ts(\fF^{A,A_\rin}_\rin)^\sh\!\ot_\N\!\Q\end{subarray}\!} \ar[rr]^(0.57){\phi_\rin^\sh} && \fC_\rin^\sh\!\ot_\N\!\Q \ar[r] & 0. }
\label{cc7eq42}
\end{gather}
Here \eq{cc7eq41} is a coequalizer in the category of toric monoids $\Mon_{\bf to}$, as \eq{cc7eq36} is in $\CRingscto$. This implies \eq{cc7eq42} is exact at the third and fourth terms. 

Choosing $\md{B_\rin}$ as small as possible also forces \eq{cc7eq42} to be exact at the second term. Suppose for a contradiction that $0\ne(y_{b'}:b'\in B_\rin)\in\Ker(\al_\rin^\sh-\be_\rin^\sh)$ in \eq{cc7eq42}. By rescaling we can suppose all $y_{b'}$ lie in $\Z\subset\Q$. Write $\al_\rin(y_{b'})=g_{b'}\in \fF^{A,A_\rin}_\rin$ and $\be_\rin(y_{b'})=h_{b'}\in \fF^{A,A_\rin}_\rin$ for $b'\in B_\rin$. Then $(y_{b'}:b'\in B_\rin)\in\Ker(\al_\rin^\sh-\be_\rin^\sh)$ implies that in $(\fF^{A,A_\rin}_\rin)^\sh$ we have
\begin{equation*}
\ts\prod\limits_{b'\in B_\rin:y_{b'}>0}[g_{b'}]^{y_{b'}}\prod\limits_{b'\in B_\rin:y_{b'}<0}[h_{b'}]^{-y_{b'}}=\prod\limits_{b'\in B_\rin:y_{b'}>0}[h_{b'}]^{y_{b'}}\prod\limits_{b'\in B_\rin:y_{b'}<0}[g_{b'}]^{-y_{b'}}.
\end{equation*}

Therefore there is a unique $f\in \fF^{A,A_\rin}$ such that in $\fF^{A,A_\rin}_\rin$ we have
\e
\ts\prod\limits_{b'\in B_\rin:y_{b'}>0\!\!\!\!\!\!\!}g_{b'}^{y_{b'}}\prod\limits_{b'\in B_\rin:y_{b'}<0\!\!\!\!\!\!\!}h_{b'}^{-y_{b'}}=\Psi_{\exp}(f)\prod\limits_{b'\in B_\rin:y_{b'}>0\!\!\!\!\!\!\!}h_{b'}^{y_{b'}}\prod\limits_{b'\in B_\rin:y_{b'}<0\!\!\!\!\!\!\!}g_{b'}^{-y_{b'}}.
\label{cc7eq43}
\e
As $g_{b'}=h_{b'}$ in $\fC_\rin$, and $\bfC$ is toric, \eq{cc7eq43} implies that $\phi_\rin\ci\Psi_{\exp}(f)=1$ in $\fC_\rin$, so $\phi(f)=0$ in $\fC$. Conversely, if we impose the relation $f=0$ then the other relations involved in \eq{cc7eq43} become dependent. Pick $b'\in B_\rin$ with $y_{b'}\ne 0$. Then we may modify $\fF^{B,B_\rin}_\rin,\bs\al,\bs\be$ in \eq{cc7eq36}, replacing $B,B_\rin$ by $B\amalg\{\ti b\},B_\rin\sm\{b\}$, where $\bs\al,\bs\be$ act by $\al(x_{\ti b})=f$ and $\be(x_{\ti b})=0$. That is, we replace the $[0,\iy)$-type relation $g_{b'}=h_{b'}$ by the $\R$-type relation $f=0$. The modified equation \eq{cc7eq36} is still a coequalizer diagram in $\CRingscto$, but we have decreased $\md{B_\rin}$, a contradiction. Thus \eq{cc7eq42} is exact, and applying $-\ot_\Q\R$ to it is also exact.

Observe that the 1-1 correspondence \eq{cc7eq40} for $A_\rin$ also determines $\dot A_\rin,\ddot A_\rin$    uniquely, as $\dot A_\rin$ is identified with the intersection of the second line of \eq{cc7eq40} with $P_\rin^\sh$. With $\md{A_\rin},\md{B_\rin}$ fixed and minimal, we choose \eq{cc7eq36} so that $\md{\dot B_\rin}$ is as small as possible, or equivalently $\md{\ddot B_\rin}$ is as large as possible. In \eq{cc7eq42} we have $\Q^{A_\rin}=\Q^{\dot A_\rin}\op\Q^{\ddot A_\rin}$ and $\Q^{B_\rin}=\Q^{\dot B_\rin}\op\Q^{\ddot B_\rin}$, with $\Ker\phi_\rin^\sh\subset\Q^\rin$. The definitions imply that $(\al_\rin^\sh-\be_\rin^\sh)(\Q^{\ddot B_\rin})\subseteq \Q^{\ddot A_\rin}\cap\Ker\phi_\rin^\sh$. Choosing $\md{\ddot B_\rin}$ as large as possible forces $(\al_\rin^\sh-\be_\rin^\sh)(\Q^{\ddot B_\rin})=\Q^{\ddot A_\rin}\cap\Ker\phi_\rin^\sh$. Taking the quotient of \eq{cc7eq42} by the subsequence generated by $\Q^{\ddot A_\rin},\Q^{\ddot B_\rin}$ gives an exact sequence
\e
\xymatrix@C=23pt{  0 \ar[r] & \Q^{\dot B_\rin} \ar[rr]^(0.48){\dot\al_\rin^\sh-\dot\be_\rin^\sh} && \Q^{\dot A_\rin} \ar[rr]^(0.45){\dot\phi_\rin^\sh} && \fC_\rin^P\ot_\N\Q \ar[r] & 0, }
\label{cc7eq44}
\e
where exactness at $\Q^{\dot B_\rin}$ holds as~$(\al_\rin^\sh-\be_\rin^\sh)(\Q^{\ddot B_\rin})=\Q^{\ddot A_\rin}\cap\Ker\phi_\rin^\sh$.

Now consider the following commutative diagram extending~\eq{cc7eq38}:
\ea
\begin{gathered}
\text{\begin{footnotesize}$\xymatrix@!0@C=33.2pt@R=25pt{
&&&&&&&& 0 \ar[d] \\
0 \ar[rr] && {\begin{subarray}{l}\ts  \bigl\langle \d_\rin\Psi_{\exp}(\ti x_b),\, b\!\in\! B,\\ \ts \d_\rin \ti y_{b'},\, b'\!\in\!\ddot B_\rin\bigr\rangle_\R\!\ot_\R\!\fD\!\! \end{subarray}} \ar[dd]^{\bs\ga_*-\bs\de_*} \ar[rrr] &&& {\begin{subarray}{l}\ts  \bigl\langle \d_\rin\Psi_{\exp}(\ti x_b),\, b\!\in\! B,\\ \ts \d_\rin \ti y_{b'},\, b'\!\in\! B_\rin\bigr\rangle_\R\!\ot_\R\!\fD\!\! \end{subarray}}  \ar[dd]^{\bs\al_*-\bs\be_*} \ar[rrr] &&& {\begin{subarray}{l}\ts  \bigl\langle\d_\rin \ti y_{b'},\, b'\!\in\!\dot B_\rin\bigr\rangle_\R\!\ot_\R\!\fD \end{subarray}}  \ar[dd]^{\bs\al_*-\bs\be_*} \ar[rr] && 0 \\
\\
0 \ar[rr] &&  {\begin{subarray}{l}\ts  \bigl\langle \d_\rin\Psi_{\exp}(x_a),\, a\!\in\! A,\\ \ts \d_\rin y_{a'},\, a'\!\in\!\ddot A_\rin\bigr\rangle_\R\!\ot_\R\!\fD\!\! \end{subarray}} \ar[dd]^{\bs\psi_*} \ar[rrr] &&& {\begin{subarray}{l}\ts  \bigl\langle \d_\rin\Psi_{\exp}(x_a),\, a\!\in\! A,\\ \ts \d_\rin y_{a'},\, a'\!\in\! A_\rin\bigr\rangle_\R\!\ot_\R\!\fD\!\! \end{subarray}} \ar[dd]^{\bs\phi_*} \ar[rrr] &&& {\begin{subarray}{l}\ts  \bigl\langle \d_\rin y_{a'},\, a'\!\in\!\dot A_\rin\bigr\rangle_\R\!\ot_\R\!\fD \end{subarray}} \ar[dd]^{\d_\rin y_{a'}\mapsto [y_{a'}]} \ar[rr] && 0 \\
\\
0 \ar[rr] &&  {}^b\Om_\bfD \ar[d] \ar[rrr]^(0.5){{}^b\pi_{\bfC,P}} &&& {}^b\Om_\bfC \ot_\fC\fD  \ar[d] \ar[rrr]^{{}^bi_{\bfC,P}} &&& \fC_\rin^P\ot_\N\fD  \ar[d] \ar[rr] && 0 \\
&& 0 &&& 0 &&& 0,\!\!}$\end{footnotesize}}
\end{gathered}
\nonumber
\\[-15pt]
\label{cc7eq45}
\ea
where we have added `$0\ra$' to the rows and third column of \eq{cc7eq38}. The first and second rows are clearly exact. As for \eq{cc7eq38}, the first and second columns of \eq{cc7eq45} come from \eq{cc7eq7} for \eq{cc7eq37} and \eq{cc7eq36}. In this case \eq{cc7eq36}--\eq{cc7eq37} are coequalizers in $\CRingscto$, not in $\CRingscin$. But applying Proposition \ref{cc7prop4} as in the last part of the proof of Theorem \ref{cc7thm2}, we can show that Proposition \ref{cc7prop1} also holds for coequalizers in $\CRingscto$. Thus the first and second columns of \eq{cc7eq45} are exact. The third column is exact as it is $-\ot_\Q\fD$ applied to \eq{cc7eq44}. Therefore by some standard diagram-chasing in exact sequences, the third row, which is \eq{cc7eq39}, is exact.
\end{proof}

Next we generalize \eq{cc7eq29} to $C^\iy$-schemes with corners. We restrict to {\it firm\/} interior $C^\iy$-schemes with corners $\CSchcfiin$ so that we can use Theorem \ref{cc6thm8} and Proposition \ref{cc6prop4}, though one can also prove slightly weaker results for $\LCRScin$ and~$\CSchcin$.

\begin{dfn}
\label{cc7def9}
Let $\bfC$ be a firm, interior $C^\iy$-ring with corners, and write $\bX=\Speccin\bfC$. Let $P\in\Pr_\bfC$ be a prime ideal in $\fC_\rex$. Then Example \ref{cc4ex4}(b) defines a $C^\iy$-ring with corners $\bfD=\bfC/\simc_P$, which is also firm and interior, with projection $\bs\pi_P:\bfC\ra\bfD$. Theorem \ref{cc6thm8} shows there is an open and closed $C^\iy$-subscheme $C(\bX)^P$ in $C(\bX)$ with an isomorphism $\bY=\Speccin\bfD\ra C(\bX)^P$.

As $\bX$ is interior, Definition \ref{cc6def3} defines a sheaf of monoids
$\check M^\rin_{C(X)}$ on $C(X)$. The proof of Proposition \ref{cc6prop4} shows that $\check M^\rin_{C(X)}\vert_{C(X)^P}$ is a constant sheaf on $C(X)^P$ with fibre $\fC_\rin^P$ in \eq{cc7eq30}.

Consider the diagram in $\bO_{C(X)^P}\text{-mod}\cong\bOYmod$:
\e
\begin{gathered}
\xymatrix@!0@C=36pt@R=40pt{
\MSpec{}^b\Om_\bfD \ar[d]^\cong \ar[rrr]_(0.45){\raisebox{-10pt}{$\st\MSpec{}^b\pi_{\bfC,P}$}} &&& \MSpec({}^b\Om_\bfC \!\ot_\fC\!\fD)\ar[d]^\cong \ar[rrr]_{\raisebox{-10pt}{$\st\MSpec{}^bi_{\bfC,P}$}} &&& \MSpec(\fC_\rin^P\!\ot_\N\!\fD) \ar[d]^\cong \ar[rr] && 0 \\
{}^bT^*C(\bX)^P \ar[rrr]^(0.4){{}^b\pi_\bX\vert_{C(X)^P}} &&& \bs\Pi_\bX^*({}^bT^*\bX)\vert_{C(X)^P} \ar[rrr]^(0.48){{}^bi_\bX\vert_{C(X)^P}} &&& \check M^\rin_{C(X)}\!\ot_\N\!\O_{C(X)^P} \ar[rr] && 0.\! }\!\!\!\!\!\!\!\!
\end{gathered}
\label{cc7eq46}
\e
Here the top row is the exact functor $\MSpec:\bfDmod\ra\bOYmod$ from Definition \ref{cc2def23} applied to \eq{cc7eq31}, and so is exact. The isomorphism in the left column comes from Proposition \ref{cc7prop5} for $\bfD$, as $\Speccin\bfD\cong C(\bX)^P$. The isomorphism in the middle column is $\bs\Pi_\bX^*$ applied to the isomorphism $\MSpec{}^b\Om_\bfC\cong{}^bT^*\bX$ from Proposition \ref{cc7prop5}. The isomorphism in the right hand column holds as $\check M^\rin_{C(X)}$ is the constant sheaf on $C(X)^P$ with fibre $\fC_\rin^P$, and $\MSpec\fD\cong\O_{C(X)^P}$. 

Thus there exist unique morphisms ${}^b\pi_\bX\vert_{C(X)^P},{}^bi_\bX\vert_{C(X)^P}$ making \eq{cc7eq46} commute, and the bottom line of \eq{cc7eq46} is exact. Since $C(\bX)=\coprod_{P\in\Prc_\bfC}C(\bX)^P$ by \eq{cc6eq20}, there is a unique exact sequence in $\bO_{C(X)}\text{-mod}$: 
\e
\xymatrix@C=4pt{
{}^bT^*C(\bX) \ar[rrr]^(0.45){{}^b\pi_\bX} &&& \bs\Pi_\bX^*({}^bT^*\bX) \ar[rrr]^(0.34){{}^bi_\bX} &&& {}^b\check N_{C(X)}\!=\!\check M^\rin_{C(X)}\!\ot_\N\!\O_{C(X)} \ar[rr] && 0, }\!\!
\label{cc7eq47}
\e
which restricts to the bottom line of \eq{cc7eq46} on $C(X)^P$ for each $P$. Here we define ${}^b\check N_{C(X)}=\check M^\rin_{C(X)}\ot_\N\O_{C(X)}$. As $\check M^\rin_{C(X)}$ is a locally constant sheaf of finitely generated monoids, ${}^b\check N_{C(X)}$ is a finite rank locally free sheaf, that is, a vector bundle on $C(\bX)$, which we call the {\it b-conormal bundle\/} of~$C(\bX)$.

We have constructed an exact sequence \eq{cc7eq47} in $\bO_{C(X)}\text{-mod}$ whenever $\bfC$ is a firm, interior $C^\iy$-ring with corners and $\bX=\Speccin\bfC$. We now claim that if $\bX\in\CSchcfiin$ then there is a unique exact sequence \eq{cc7eq47} in $\bO_{C(X)}\text{-mod}$, which restricts to the sequence constructed above on $C(\bU)$ for any open $\bU\subseteq\bX$ with $\bU\cong\Speccin\bfC$ for some~$\bfC\in\CRingscfiin$. 

To see this, it is enough to show that the morphisms $({}^b\pi_\bX)_1,({}^bi_\bX)_1$ and $({}^b\pi_\bX)_2,({}^bi_\bX)_2$ constructed above on affine open subsets $C(\bs U_1),C(\bs U_2)\subset C(\bX)$ agree on overlaps $C(\bs U_1)\cap C(\bs U_2)$ for any two such $\bs U_1,\bs U_2\subset\bX$ with $\bs U_i\cong\Speccin\bfC_i$ for $i=1,2$. Further, it is enough to show these morphisms agree on stalks at each $(x,P_x)\in C(\bs U_1)\cap C(\bs U_2)$. But the restriction of \eq{cc7eq47} to the stalk at $(x,P_x)$ is equation \eq{cc7eq31} for the local $C^\iy$-rings with corners $\bfC_{1,x}\cong\bO_{X,x}\cong \bfC_{2,x}$ and $\bfD_{1,x}\cong\bO_{C(X),(x,P_x)}\cong \bfD_{2,x}$. As this depends only on $\bO_{X,x},\bO_{C(X),(x,P_x)}$, it is independent of the choice of $\bfC_1,\bfC_2$, as we want.

If we suppose that $\bX\in\CSchcto$ then we can use \eq{cc7eq39} instead of \eq{cc7eq31}, and so get an extended exact sequence in~$\bO_{C(X)}\text{-mod}$: 
\e
\xymatrix@C=9.5pt{
0 \ar[rr] && {}^bT^*C(\bX) \ar[rrr]^(0.45){{}^b\pi_\bX} &&& \bs\Pi_\bX^*({}^bT^*\bX) \ar[rrr]^(0.56){{}^bi_\bX} &&& {}^b\check N_{C(X)} \ar[rr] && 0. }\!\!
\label{cc7eq48}
\e
Applying $\MSpec$ to Example \ref{cc7ex5} shows that if $\bX$ is not toric, then \eq{cc7eq48} may not be exact at the second term.

Now let $\bs f:\bX\ra\bY$ be a morphism in $\CSchcfiin$. Consider the diagram
\e
\begin{gathered}
\xymatrix@C=3pt@R=20pt{
C(\bs f)^*({}^bT^*C(\bY)) \ar[d]^{{}^b\Om_{C(\bs f)}} \ar[rrr]_(0.45){\raisebox{-10pt}{$\st C(\bs f)^*({}^b\pi_\bY)$}} &&& {\begin{subarray}{l}\ts C(\bs f)^*\!\ci\!\bs\Pi_\bY^*({}^bT^*\bY)\!=\\ \ts \quad \bs\Pi_\bY^*\!\ci\!\bs f^*({}^bT^*\bY)\end{subarray}} \ar[d]^{\bs\Pi_\bX^*({}^b\Om_{\bs f})} \ar[rrr]_(0.56){\raisebox{-10pt}{$\st C(\bs f)^*({}^bi_\bY)$}} &&& C(\bs f)^*({}^b\check N_{C(Y)}) \ar[d]^{{}^b\check N_{C(f)}} \ar[rr] && 0 \\
{}^bT^*C(\bX) \ar[rrr]^(0.45){{}^b\pi_\bX} &&& \bs\Pi_\bX^*({}^bT^*\bX) \ar[rrr]^(0.5){{}^bi_\bX} &&& {}^b\check N_{C(X)} \ar[rr] && 0,\! }\!\!\!\!\!\!\!\!\!\!\!
\end{gathered}
\label{cc7eq49}
\e
where we define
\begin{align*}
{}^b\check N_{C(f)}=\check M^\rex_{C(f)}\ot_\N\id:C(\bs f)^*({}^b\check N_{C(Y)})\cong C(f)^{-1}(\check M^\rex_{C(Y)})\ot_\N\O_{C(X)}&\\
\longra {}^b\check N_{C(X)}=\check M^\rin_{C(X)}\ot_\N\O_{C(X)}&,
\end{align*}
for $\check M^\rex_{C(f)}$ as in \eq{cc6eq26}. Over affine subsets of $\bX,\bY$ and $C(\bX),C(\bY)$, equation \eq{cc7eq49} may be identified with $\MSpec$ applied to \eq{cc7eq35}, using \eq{cc7eq46}. Thus \eq{cc7eq49} commutes on affine open subsets covering $C(\bX)$, and hence it commutes on $C(\bX)$. So the exact sequence \eq{cc7eq47} is functorial over~$\CSchcfiin$.
\end{dfn}

Let $X$ be a manifold with corners, and set $\bX=F_\Mancin^\CSchcto(X)$ as in Definition \ref{cc5def9}. Proposition \ref{cc6prop2} says $C(\bX)\cong F_\cManc^\CSchcto(C(X))$. Theorem \ref{cc7thm3}(b) identifies ${}^bT^*\bX,{}^bT^*C(\bX)$ with the sheaves of sections of ${}^bT^*X,{}^bT^*C(X)$. Example \ref{cc6ex3} identifies the sheaf of sections of $M_{C(X)}^\vee$ with $\check M^\rin_{C(X)}$. As ${}^bN_{C(X)}^*=M_{C(X)}^\vee\ot_\N\R$ on $C(X)$ and ${}^b\check N_{C(X)}=\check M^\rin_{C(X)}\ot_\N\O_{C(X)}$ on $C(\bX)$, the sheaf of sections of ${}^bN_{C(X)}^*$ is identified with ${}^b\check N_{C(X)}$. By comparing definitions, one can show that all these isomorphisms identify the exact sequences \eq{cc7eq29} on $C(X)$ and \eq{cc7eq48} on $C(\bX)$. The analogue holds for manifolds with g-corners.

\section{Further generalizations}
\label{cc8}

We discuss three directions in which Chapters \ref{cc4}--\ref{cc7} can be generalized.

\subsection{\texorpdfstring{$C^\iy$-stacks with corners}{C∞-stacks with corners}}
\label{cc81}

In classical algebraic geometry, schemes are generalized to (Deligne--Mumford or Artin) stacks, as in Olsson \cite{Olss}, Laumon and Moret-Bailly \cite{LaMo}, and de Jong \cite{Jong}. This is useful as many moduli spaces naturally have the structure of stacks, but not schemes. Stacks over a field $\K$ form a 2-category~$\Sta_\K$.

The second author \cite[\S 6--\S A]{Joyc9} extended the theory of $C^\iy$-schemes in \cite[\S 2--\S 5]{Joyc9} to $C^\iy$-{\it stacks}, including {\it Deligne--Mumford\/ $C^\iy$-stacks}, which form 2-categories $\DMCSta\subset\CSta$. One reason this is interesting is that the category of manifolds $\Man$ generalizes to the 2-category of {\it orbifolds\/} $\Orb$, which are roughly Deligne--Mumford stacks in manifolds. The full embedding $\Man\hookra\CSch$ generalizes to a full embedding~$\Orb\hookra\DMCSta$.

Thus, it seems an obvious project to generalize our theory of $C^\iy$-schemes with corners to 2-categories $\DMCStac\subset\CStac$ of (Deligne--Mumford) $C^\iy$-stacks with corners, with a full embedding $\Orbc\hookra\DMCStac$ of the 2-category of orbifolds with corners.

Now most of the generalization from $\CSch$ to $\DMCSta\subset\CSta$ in \cite[\S 6--\S A]{Joyc9} is an exercise in the theory of stacks,  which depends only on a few properties of $C^\iy$-schemes:
\begin{itemize}
\setlength{\itemsep}{0pt}
\setlength{\parsep}{0pt}
\item[(i)] We define a {\it Grothendieck topology\/} $\cJ$ on $\CSch$, making $(\CSch,\cJ)$ into a {\it site}, by the usual open sets and open covers for topological spaces. 

This is a reasonable definition as affine $C^\iy$-schemes are Hausdorff, and every open subset of a $C^\iy$-scheme is a $C^\iy$-subscheme. So there are `enough' ordinary open sets in $C^\iy$-schemes, and we do not need an exotic notion of `open set' to define stacks, as for the \'etale or smooth topologies in ordinary algebraic geometry.
\item[(ii)] Objects and morphisms in $\CSch$ can be glued on covers, and form a sheaf. Thus $(\CSch,\cJ)$ is a {\it subcanonical site}.
\item[(iii)] It is convenient, but not essential, that all fibre products exist in~$\CSch$.
\end{itemize}
Parts (i),(ii) work for $\CSchc$ and $\CSchcin$. For (iii), in parts of the theory requiring fibre products, we should restrict to subcategories such as $\CSchcfi$ or $\CSchcfiin$, in which fibre products exist by Theorem~\ref{cc5thm10}.

Thus we may define 2-categories $\DMCStac\subset\CStac$ of (Deligne--Mumford) $C^\iy$-stacks with corners by following \cite[\S 6--\S 7]{Joyc9} with only cosmetic changes. The material of \S\ref{cc56} gives interesting 2-subcategories such as toric and interior $\DMCStacto\subset\DMCStacin\subset\DMCStac$. Theorem \ref{cc5thm10} implies that fibre products exist in 2-categories of $C^\iy$-stacks with corners corresponding to categories in \eq{cc5eq15}, such as $\DMCStacto,\DMCStacfiin,\DMCStacfi$. Most of the material of \S\ref{cc71}--\S\ref{cc75} also generalizes immediately, following \cite[\S 8]{Joyc9}. Extending Chapter \ref{cc6} and \S\ref{cc76} to stacks will require some work.

\subsection{\texorpdfstring{$C^\iy$-rings and $C^\iy$-schemes with a-corners}{C∞-rings and C∞-schemes with a-corners}}
\label{cc82}

Our notion of $C^\iy$-ring with corners started with the categories $\Mancin\subset\Manc$ of manifolds with corners defined in \S\ref{cc31}, and the full subcategories $\Euccin\subset\Mancin$, $\Eucc\subset\Manc$ with objects $\R^n_k=[0,\iy)^k\t\R^{n-k}$ for~$0\le k\le n$.

As we explained in Remark \ref{cc3rem1}, there are several non-equivalent definitions of categories of manifolds with corners in the literature. So we can consider alternative theories of $C^\iy$-rings and $C^\iy$-schemes with corners starting with one of these categories instead of $\Manc$. In particular:
\begin{itemize}
\setlength{\itemsep}{0pt}
\setlength{\parsep}{0pt}
\item[(a)] If we use the categories $\bf{Man^f_{in}}\subset\bf{Man^f}$ of {\it manifolds with faces\/} in Definition \ref{cc3def8}, following Melrose \cite{Melr1,Melr2,Melr3}, we do not change $\Euccin\subset\Eucc$, so the theory is unchanged.
\item[(b)] The category $\Mancst$ of manifolds with corners and {\it strongly smooth maps}, from Definition \ref{cc3def2} and \cite{Joyc1}, gives a theory which can be embedded as subcategories $\CRings^{\bf c}_{\bf st}\subset\CRingsc$, $\CSch^{\bf c}_{\bf st}\subset\CSchc$, as for the embeddings $\CRingscin\subset\CRingsc$, $\CSchcin\subset\CSchc$, since $\Mancst\subset\Manc$. The theory also seems less interesting (roughly, $\fC_{\text{\rm st-ex}}$ is a set, compared to $\fC_\rex$ being a monoid).
\item[(c)] The category $\Mancwe$ of manifolds with corners and {\it weakly smooth maps}, from Definition \ref{cc3def2}, Cerf \cite[\S I.1.2]{Cerf}, and other authors, yields a theory which is essentially equivalent to ordinary $C^\iy$-schemes. As weakly smooth maps have no compatibility with boundaries, and no notion of `interior', corner functors in Chapter \ref{cc6} and b-cotangent modules in Chapter \ref{cc7} no longer work.
\item[(d)] The second author \cite{Joyc7} defined the category $\Manac$ of {\it manifolds with analytic corners}, or {\it manifolds with a-corners}, which we explain next.
\end{itemize}

We recall the definition of manifolds with a-corners from~\cite[\S 3]{Joyc7}:

\begin{dfn}
\label{cc8def1}
We write $\lb 0,\iy)$ to mean $[0,\iy)$, but the notation emphasizes that $\lb 0,\iy)$ has a different kind of boundary (an {\it a-boundary\/}) at 0, giving a different notion of smooth function on $\lb 0,\iy)$. We write other intervals such as $\lb 0,1\rb=[0,1]$ in the same way. For $0\le k\le m$ we write $\R^{k,m}=\lb 0,\iy)^k\t\R^{m-k}$. 

Let $U\subseteq\R^{k,m}$ be open and $f:U\ra\R$ be continuous. Write points of $U$ as $(x_1,\ldots,x_m)$ with $x_1,\ldots,x_k\in\lb 0,\iy)$ and $x_{k+1},\ldots,x_m\in\R$. The {\it b-derivative\/} of $f$ (if it exists) is a map ${}^b\pd f:U\ra\R^m$, written ${}^b\pd f=({}^b\pd_1f,\ldots,{}^b\pd_mf)$ for ${}^b\pd_if:U\ra\R$, where by definition
\e
{}^b\pd_if(x_1,\ldots,x_m)=\begin{cases} 0, & x_i=0, \;\> i=1,\ldots,k, \\
x_i\frac{\pd f}{\pd x_i}(x_1,\ldots,x_m), & x_i>0, \;\> i=1,\ldots,k, \\
\frac{\pd f}{\pd x_i}(x_1,\ldots,x_m), & i=k+1,\ldots,m. 
\end{cases}
\label{cc8eq1}
\e
We say that ${}^b\pd f$ {\it exists\/} if \eq{cc8eq1} is well defined, that is, if $\frac{\pd f}{\pd x_i}$ exists on $U\cap\{x_i>0\}$ if $i=1,\ldots,k$, and $\frac{\pd f}{\pd x_i}$ exists on $U$ if $i=k+1,\ldots,m$.

We can iterate b-derivatives (if they exist), to get maps ${}^b\pd^lf:U\ra\bigot^l\R^m$ for $l=0,1,\ldots,$ by taking b-derivatives of components of ${}^b\pd^jf$ for $j=0,\ldots,l-1$.
\begin{itemize}
\setlength{\itemsep}{0pt}
\setlength{\parsep}{0pt}
\item[(i)] We say that $f$ is {\it roughly differentiable}, or {\it r-differentiable}, if ${}^b\pd f$ exists and is a continuous map ${}^b\pd f:U\ra\R^m$.
\item[(ii)] We say that $f$ is {\it roughly smooth}, or {\it r-smooth}, if ${}^b\pd^lf:U\ra\bigot^l\R^m$ is r-differentiable for all $l=0,1,\ldots.$
\item[(iii)] We say that $f$ is {\it analytically differentiable}, or {\it a-differentiable}, if it is r-differentiable and for any compact subset $S\subseteq U$ and $i=1,\ldots,k$, there exist positive constants $C,\al$ such that
\begin{equation*}
\bmd{{}^b\pd_if(x_1,\ldots,x_m)}\le Cx_i^\al\qquad\text{for all $(x_1,\ldots,x_m)\in S$.}
\end{equation*}
\item[(iv)] We say that $f$ is {\it analytically smooth}, or {\it a-smooth}, if ${}^b\pd^lf:U\ra\bigot^l\R^m$ is a-differentiable for all $l=0,1,\ldots.$
\end{itemize}

One can show that $f$ is a-smooth if for all $a_1,\ldots,a_m\in\N$ and for any compact subset $S\subseteq U$, there exist positive constants $C,\al$ such that
\begin{align*}
\left\vert\frac{\pd^{a_1+\cdots+a_m}}{\pd x_1^{a_1}\cdots\pd x_m^{a_m}}f(x_1,\ldots,x_m)\right\vert\le C \prod_{i=1,\ldots,k:\, a_i>0} x_i^{\al-a_i} \\
\text{for all $(x_1,\ldots,x_m)\in S$ with $x_i>0$ if $i=1,\ldots,k$ with $a_i>0$,}
\end{align*}
where continuous partial derivatives must exist at the required points.

If $f,g:U\ra\R$ are a-smooth (or r-smooth) and $\la,\mu\in\R$ then $\la f+\mu g
$ and $fg:U\ra\R$ are a-smooth (or r-smooth). Thus, the set $C^\iy(U)$ of a-smooth functions $f:U\ra\R$ is an $\R$-algebra, and in fact a $C^\iy$-ring.

If $I\subseteq\R$ is an open interval, such as $I=(0,\iy)$, we say that a map $f:U\ra I$ is {\it a-smooth}, or just {\it smooth}, if it is a-smooth as a map $f:U\ra\R$.
\end{dfn}

\begin{dfn}
\label{cc8def2}
Let $U\subseteq\R^{k,m}$ and $V\subseteq \R^{l,n}$ be open, and $f=(f_1,\ldots,f_n):U\ra V$ be a continuous map, so that $f_j=f_j(x_1,\ldots,x_m)$ maps $U\ra\lb 0,\iy)$ for $j=1,\ldots,l$ and $U\ra\R$ for $j=l+1,\ldots,n$. Then we say:
\begin{itemize}
\setlength{\itemsep}{0pt}
\setlength{\parsep}{0pt}
\item[(a)] $f$ is {\it r-smooth\/} if $f_j:U\ra\R$ is r-smooth in the sense of Definition \ref{cc8def1} for $j=l+1,\ldots,n$, and every $u=(x_1,\ldots,x_m)\in U$ has an open neighbourhood $\ti U$ in $U$ such that for each $j=1,\ldots,l$, either:
\begin{itemize}
\setlength{\itemsep}{0pt}
\setlength{\parsep}{0pt}
\item[(i)] we may uniquely write $f_j(\ti x_1,\ldots,\ti x_m)=F_j(\ti x_1,\ldots,\ti x_m)\cdot\ti x_1^{a_{1,j}}\cdots\ti x_k^{a_{k,j}}$ for all $(\ti x_1,\ldots,\ti x_m)\in\ti U$, where $F_j:\ti U\ra(0,\iy)$ is r-smooth as in Definition \ref{cc8def1}, and $a_{1,j},\ldots,a_{k,j}\in[0,\iy)$, with $a_{i,j}=0$ if $x_i\ne 0$; or 
\item[(ii)] $f_j\vert_{\smash{\ti U}}=0$.
\end{itemize}
\item[(b)] $f$ is {\it a-smooth\/} if $f_j:U\ra\R$ is a-smooth in the sense of Definition \ref{cc8def1} for $j=l+1,\ldots,n$, and every $u=(x_1,\ldots,x_m)\in U$ has an open neighbourhood $\ti U$ in $U$ such that for each $j=1,\ldots,l$, either:
\begin{itemize}
\setlength{\itemsep}{0pt}
\setlength{\parsep}{0pt}
\item[(i)] we may uniquely write $f_j(\ti x_1,\ldots,\ti x_m)=F_j(\ti x_1,\ldots,\ti x_m)\cdot\ti x_1^{a_{1,j}}\cdots\ti x_k^{a_{k,j}}$ for all $(\ti x_1,\ldots,\ti x_m)\in\ti U$, where $F_j:\ti U\ra(0,\iy)$ is a-smooth as in Definition \ref{cc8def1}, and $a_{1,j},\ldots,a_{k,j}\in[0,\iy)$, with $a_{i,j}=0$ if $x_i\ne 0$; or 
\item[(ii)] $f_j\vert_{\smash{\ti U}}=0$.
\end{itemize}
\item[(c)] $f$ is {\it interior\/} if it is a-smooth, and case (b)(ii) does not occur.
\item[(d)] $f$ is an {\it a-diffeomorphism} if it is an a-smooth bijection with a-smooth inverse.
\end{itemize}
\end{dfn}

\begin{dfn}
\label{cc8def3}
We define the category $\Manac$ of {\it manifolds with a-corners\/} $X$ and {\it a-smooth maps\/} $f:X\ra Y$ following the definition of $\Manc$ in Definition \ref{cc3def2}, but using {\it a-charts\/} $(U,\phi)$ with $U\subseteq\R^{k,m}$ open, where a-charts $(U,\phi)$, $(V,\psi)$ are {\it compatible\/} if $\psi^{-1}\ci\phi$ is an a-diffeomorphism between open subsets of $\R^{k,m},\R^{l,m}$, in the sense of Definition \ref{cc8def2}(d). We also define an a-smooth map $f:X\ra Y$ to be {\it interior\/} if it is locally modelled on interior maps between open subsets of $\R^{k,m},\R^{l,n}$, in the sense of Definition \ref{cc8def2}(c), and we write $\Manacin\subset\Manac$ for the subcategory with interior morphisms.

In \cite[\S 3.5]{Joyc7} the second author also defines the category $\Mancac$ of {\it manifolds with corners and a-corners\/}, which are locally modelled on $\R_l^{k,m}=\lb 0,\iy)^k\t[0,\iy)^l\t\R^{m-k-l}$, with smooth functions defined as in Definitions \ref{cc8def1}--\ref{cc8def2} for the $\lb 0,\iy)$ factors, and as in Definition \ref{cc3def1} for the $[0,\iy)$ factors. There are full embeddings $\Manc\hookra\Mancac$ and $\Manac\hookra\Mancac$. For $X\in\Mancac$, the boundary $\pd X=\pd^{\rm c} X\amalg\pd^{\rm ac}X$ decomposes into the ordinary boundary $\pd^{\rm c} X$ and the a-boundary~$\pd^{\rm ac}X$.
\end{dfn}

\begin{rem}
\label{cc8rem1}
{\bf(a)} Manifolds with a-corners are significantly different to manifolds with corners. Even the simplest examples $\lb 0,\iy)$ in $\Manac$ and $[0,\iy)$ in $\Manc$ have different smooth structures. For example, $x^\al:\lb 0,\iy)\ra\R$ is a-smooth for all $\al\in[0,\iy)$, but $x^\al:[0,\iy)\ra\R$ is only smooth for~$\al\in\N$.
\smallskip

\noindent{\bf(b)} Here are two of the reasons for introducing manifolds with a-corners in~\cite{Joyc7}.

Firstly, boundary conditions for p.d.e.s are usually of two types:
\begin{itemize}
\setlength{\itemsep}{0pt}
\setlength{\parsep}{0pt}
\item[(i)] Boundary `at finite distance'. For example, solve $\frac{\pd u}{\pd x^2}+\frac{\pd u}{\pd y^2}=f$ on the closed unit disc $D^2\subset\R^2$ with boundary condition $u\vert_{\pd D^2}=g$.
\item[(ii)] Boundary `at infinite distance', or `of asymptotic type'. For example, solve $\frac{\pd u}{\pd x^2}+\frac{\pd u}{\pd y^2}=f$ on $\R^2$ with asymptotic condition $u=O((x^2+y^2)^\al)$, $\al<0$.
\end{itemize}
We argue in \cite{Joyc7} that it is natural to write type (i) using ordinary manifolds with corners, but type (ii) using manifolds with a-corners. In the example, we would compactify $\R^2$ to a manifold with a-corners $\bar\R{}^2$ by adding a boundary circle $\cS^1$ at infinity, and then $u$ extends to a-smooth~$\bar u:\bar\R{}^2\ra\R$.

Secondly, and related, several important areas of geometry involve forming moduli spaces $\oM$ of some geometric objects, such that $\oM$ is a manifold with corners, or a more general space such as a Kuranishi space with corners in Fukaya et al.\ \cite{FOOO,FuOn}, where the `boundary' $\pd\oM$ parametrizes singular objects included to make $\oM$ compact. Examples include moduli spaces of Morse flow lines \cite{AuBr}, of $J$-holomorphic curves with boundary in a Lagrangian \cite{FOOO}, and of $J$-holomorphic curves with cylindrical ends in Symplectic Field Theory~\cite{EGH}.

We argue in \cite{Joyc7} that the natural smooth structure to put on $\oM$ in such moduli problems is that of a manifold with a-corners (or Kuranishi space with a-corners, etc.). This resolves a number of technical problems in these theories.
\end{rem}

We can now write down a new version of our theory, starting instead with the full subcategories $\Eucacin\subset\Manacin$, $\Eucac\subset\Manac$ with objects $\R^{k,m}=\lb 0,\iy)^k\t\R^{m-k}$ for $0\le k\le m$. This yields categories of ({\it pre\/}) $C^\iy$-{\it rings with a-corners\/} $\PCRingsac, \CRingsac$ and $C^\iy$-{\it schemes with a-corners\/} $\CSchac$. We write $\bfC\in\PCRingsac$ as $(\fC,\fC_\aex)$ satisfying the analogue of Definition \ref{cc4def2}, but for the operations $\Phi_f,\Psi_g$ we take $f:\R^m\t\lb 0,\iy)^n\ra\R$ and $g:\R^m\t\lb 0,\iy)^n\ra\lb 0,\iy)$ to be a-smooth.

Much of Chapters \ref{cc4}--\ref{cc7} generalizes immediately to the a-corners case, with only obvious changes. Here are some important differences, though:
\begin{itemize}
\setlength{\itemsep}{0pt}
\setlength{\parsep}{0pt}
\item[(a)] There should be a category $\Man^{\bf gac}$ of `manifolds with generalized a-corners', which relates to $\Manac$ as $\Mangc$ relates to $\Manc$. But no such category has been written down yet. So, for the present, we do not generalize the material on manifolds with g-corners to the a-corners case.
\item[(b)] The special classes of  $C^\iy$-rings with corners $\bfC=(\fC,\fC_\rex)$ in \S\ref{cc47}, and hence of $C^\iy$-schemes with corners in \S\ref{cc56}, relied on treating $\fC_\rex$ or $\fC_\rin$ as monoids, and imposing conditions such as $\fC_\rex^\sh$ finitely generated.

For (interior) $C^\iy$-rings with a-corners, $\fC_\aex$ (and $\fC_\ain$) are still monoids. However, requiring $\fC_\aex^\sh$ to be finitely generated is not a useful condition. For example, if $\bfC=\bs C^\iy(\lb 0,\iy))$ then by mapping $[x^\al]\mapsto e^{-\al}$, $[0]\mapsto 0$ we may identify $\fC_\aex^\sh\cong [0,1]$, $\fC_\ain^\sh\cong (0,1]$ as monoids under multiplication, and $[0,1],(0,1]$ are not finitely generated monoids.

Instead, we should regard $\fC_\aex$ and $\fC_\ain$ as $[0,\iy)$-{\it modules}. That is, $\fC_\aex$ and $\fC_\ain$ are monoids under multiplication, with operations $c\mapsto c^\al$ for $c$ in $\fC_\aex$ or $\fC_\ain$ and $\al\in[0,\iy)$ satisfying $(cd)^\al=c^\al d^\al$, $(c^\al)^\be=c^{\al\be}$, and $c^0=1$. The operation $c\mapsto c^\al$ is $\Psi_{x^\al}:\fC_\aex\ra \fC_\aex$. Monoids are to abelian groups as $[0,\iy)$-modules are to real vector spaces.
\item[(c)] Continuing (b), we should for example define $\bfC\in\CRingsac$ to be {\it firm\/} if $\fC_\aex^\sh$ is a {\it finitely generated\/ $[0,\iy)$-module}, that is, there is a surjective $[0,\iy)$-module morphism $[0,\iy)^n\ra\fC_\aex^\sh$. Key properties of firm $C^\iy$-schemes with corners such as Proposition \ref{cc5prop3}, and hence Theorem \ref{cc5thm9}, then extend to the a-corners case.
\item[(d)] Continuing (b), to define the corners $C(\bX)$ in Definition \ref{cc6def1} in the a-corners case, for points $(x,P)\in C(X)$ we should take $P$ to be a prime ideal in $\O_{X,x}^\aex$ considered as a $[0,\iy)$-module, not as a monoid.
\end{itemize}
We leave it to the interested reader to work out the details.

\subsection{Derived manifolds and orbifolds with corners}
\label{cc83}

It is well known that classical Algebraic Geometry, studying schemes and stacks, has been extended to Derived Algebraic Geometry by Lurie \cite{Luri}, To\"en and Vezzosi \cite{Toen,ToVe} and others, studying derived schemes and derived stacks. It is less well known that classical Differential Geometry, studying manifolds and orbifolds, has been extended to Derived Differential Geometry, the study of {\it derived manifolds\/} and {\it derived orbifolds}.

The subject began with a short section in Lurie \cite[\S 4.5]{Luri}, where he sketched how to define an $\iy$-category of {\it derived\/ $C^\iy$-schemes}, including derived manifolds. Lurie's student David Spivak \cite{Spiv} worked out the details of this, defining an $\iy$-category of derived manifolds. Simplifications and extensions of Spivak's theory were given by Borisov and Noel \cite{Bori,BoNo} and the second author~\cite{Joyc3,Joyc4}. 

Actually, one could argue that Fukaya and Ono's {\it Kuranishi spaces\/} \cite{FOOO,FuOn} (see also \cite{Joyc5,Joyc8,Joyc10} for the second author's improved, `derived' 2-categorical definition of Kuranishi spaces) are a prototype of derived orbifolds, which predate the invention of Derived Algebraic Geometry. To understand the relationship, observe that there are two ways to define ordinary manifolds:
\begin{itemize}
\setlength{\itemsep}{0pt}
\setlength{\parsep}{0pt}
\item[(A)] A manifold is a Hausdorff, second countable topological space $X$ equipped with a sheaf $\O_X$ of $\R$-algebras (or $C^\iy$-rings) such that $(X,\O_X)$ is locally isomorphic to $\R^n$ with its sheaf of smooth functions~$\O_{\R^n}$.
\item[(B)] A manifold is a Hausdorff, second countable topological space $X$ equipped with a maximal atlas of charts $\{(U_i,\phi_i):i\in I\}$.
\end{itemize}
If we try to define derived manifolds by generalizing approach (A), we get some kind of derived $C^\iy$-scheme, as in \cite{Bori,BoNo,Joyc3,Joyc4,Luri,Spiv}; if we try to generalize (B), we get something like Kuranishi spaces in~\cite{FOOO,FuOn,Joyc5,Joyc8,Joyc10}.

Derived manifolds and orbifolds are interesting for many reasons, including:
\begin{itemize}
\setlength{\itemsep}{0pt}
\setlength{\parsep}{0pt}
\item[(a)] Much of classical Differential Geometry extends nicely to the derived case.
\item[(b)] Many mathematical objects are naturally derived manifolds, for example:
\begin{itemize}
\setlength{\itemsep}{0pt}
\setlength{\parsep}{0pt}
\item[(i)] The solution set of $f_1(x_1,\ldots,x_n)=\cdots=f_k(x_1,\ldots,x_n)=0$, where $x_1,\ldots,x_n$ are real variables and $f_1,\ldots,f_k$ are smooth functions.
\item[(ii)] (Non-transverse) intersections $X\cap Y$ of submanifolds $X,Y\subset Z$.
\item[(iii)] Moduli spaces $\cM$ of solutions of nonlinear elliptic equations on compact manifolds. Also, if we consider moduli spaces $\cM$ for nonlinear equations which are elliptic modulo symmetries, and restrict to objects with finite automorphism groups, then $\cM$ is a derived orbifold.
\end{itemize}
\item[(c)] A compact, oriented derived manifold (or orbifold) $\bX$ has a {\it virtual class\/} $[\bX]_\vir$ in (Steenrod or \v Cech) homology $H_{\vdim\bX}(X,\Z)$ (or $H_{\vdim\bX}(X,\Q)$), which has deformation invariance properties. Combining this with (b)(iii), we can use derived orbifolds as tools in enumerative invariant theories such as Gromov--Witten invariants in Symplectic Geometry.
\end{itemize}

It is obviously desirable to extend Derived Differential Geometry to theories of {\it derived manifolds with corners}, and {\it derived orbifolds with corners}. This will be done by the second author in \cite{Joyc4} by approach (A) and derived $C^\iy$-schemes with corners, using the theory of this book. See \cite{Joyc5,Joyc8} for a parallel treatment using approach (B) and Kuranishi spaces with corners.

Just as derived manifolds/orbifolds have applications in enumerative invariant theories (when one considers moduli spaces without boundary, which have virtual classes), so derived manifolds/orbifolds with corners have applications in Floer-type theories and Fukaya categories (when one considers moduli spaces with boundary and corners, which have virtual chains). An eventual goal of the second author's Derived Differential Geometry project \cite{Joyc3,Joyc4,Joyc5,Joyc8,Joyc10} is to rewrite and extend the foundations of large areas of Symplectic Geometry.

\medskip

\noindent{\small\sc The Mathematical Institute, Radcliffe
Observatory Quarter, Woodstock Road, Oxford, OX2 6GG, U.K.

\noindent E-mails: {\tt kelli.francis-staite@maths.ox.ac.uk, joyce@maths.ox.ac.uk.}}

\end{document}